\documentclass[10pt,letterpaper]{report}
\usepackage[left=100pt, right=100.00pt, top=90.00pt, bottom=100.00pt]{geometry}
\usepackage[utf8]{inputenc}
\usepackage{amsfonts,amsmath,amssymb,amsthm}
\usepackage[toc,page]{appendix}
\usepackage{siunitx}
\usepackage[T1]{fontenc}
\usepackage{textcomp}
\usepackage{subcaption}
\usepackage{array}
\usepackage{mathtools}
\usepackage{empheq}
\usepackage{upquote}
\usepackage{graphicx}
\usepackage{booktabs}
\usepackage{multirow}

\usepackage[hidelinks]{hyperref}
\hypersetup{colorlinks=false}
\usepackage{cleveref}

\usepackage{tikz}
\usetikzlibrary{arrows,decorations.pathmorphing,decorations.markings,math,arrows.meta,patterns,calc}

\usepackage[
backend=bibtex,
style=ieee,
citestyle=numeric-comp,
sorting=nyt
]{biblatex}

\usepackage{matlab-prettifier}
\newdimen\LineSpace
\tikzset{
	LineSpace/.code={\LineSpace=#1},
	LineSpace=3pt
}

\makeatletter
\pgfdeclarepatternformonly[\LineSpace]{my north west lines}{\pgfqpoint{-1pt}{-1pt}}{\pgfqpoint{\LineSpace}{\LineSpace}}{\pgfqpoint{\LineSpace}{\LineSpace}}%
{
	\pgfsetcolor{\tikz@pattern@color}
	\pgfsetlinewidth{0.4pt}
	\pgfpathmoveto{\pgfqpoint{0pt}{\LineSpace}}
	\pgfpathlineto{\pgfqpoint{\LineSpace + 0.1pt}{-0.1pt}}
	\pgfusepath{stroke}
}
\makeatother

\title{Analysis and controller-design of time-delay systems using \packageName{}. \\ {\large A tutorial and manual.} \\ {\normalsize Version 1.0}}

\author{Pieter Appeltans and Wim Michiels\\Department of Computer Science, NUMA Section, KU Leuven\\
	B-3001 Leuven, Belgium}

\date{\today}
\newcommand{\link}{\url{https://gitlab.kuleuven.be/u0011378/tds-control/-/archive/main/tds-control-main.zip?path=tds-control}}
\newcommand{\R}{\mathbb{R}}
\newcommand{\C}{\mathbb{C}}
\newcommand{\Z}{\mathbb{Z}}

\newcommand{\dins}{d_{\mathrm{INS}}}
\newcommand{\ie}{{\it i.e.,~}}
\newcommand{\eg}{{e.g.,~}}
\newcommand{\hinf}{\mathcal{H}_{\infty}}
\newcommand{\hinfnorm}{$\mathcal{H}_{\infty}$-norm}
\newcommand{\dd}{\mathop{}\!\mathrm{d}}
\DeclareMathOperator{\clos}{clos}
\DeclareMathOperator{\rank}{rank}

\newcommand{\matlabfun}[1]{{\tt \detokenize{#1}}}
\newcommand{\matlab}{\textsc{matlab}}
\newcommand{\packageName}{{\tt TDS-CONTROL}}
\newcommand{\sa}{\mathrm{c}}
\newcommand{\taum}{\tau_{\max}}

\newtheorem{theorem}{Theorem}[chapter]
\newtheorem{proposition}{Proposition}[chapter]
\newtheorem{corollary}{Corollary}[chapter]

\newtheorem{assumption}{Assumption}[chapter]

\theoremstyle{definition}
\newtheorem{definition}{Definition}[chapter]
\newtheorem{remark}{Remark}[chapter]
\newtheorem{examplex}{Example}[chapter]
\crefname{examplex}{example}{examples}
\newenvironment{example}
  {\pushQED{\qed}\examplex}
  {\popQED\endexamplex}

\usepackage{titlepic}

\titlepic{\includegraphics[width=0.3\textwidth]{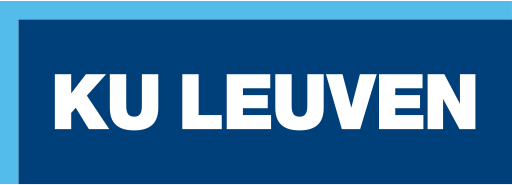}}

\begin{document}
\maketitle

\clearpage

\begin{abstract}

\packageName{} is an integrated \matlab{} package for the analysis and controller-design of linear time-invariant (LTI) dynamical systems with (multiple) discrete delays, supporting both systems of retarded and neutral type. 
 Firstly, the package offers various functionality for analyzing such systems, like methods for computing the spectral abscissa, the \mbox{H-infinity} norm, the pseudospectral abscissa, and the distance to instability. Furthermore, as \packageName{} is designed with neutral time-delay systems in mind, it has the appealing feature that the sensitivity of certain quantities (such as the spectral abscissa) with respect to infinitesimal delay perturbations can explicitly be taken into account. 
Secondly, \packageName{} also allows to design fixed-order dynamic output feedback controllers. 
The corresponding controller-design algorithms are based on minimizing the spectral abscissa, the H-infinity norm, or a combination of both with respect to the free controller parameters by solving a non-smooth, non-convex optimization problem.
As a strictly negative spectral abscissa is a necessary and sufficient condition for stability, the presented design methods are thus not conservative. 
It is also possible to impose structure on the controller, enabling the design of decentralized and proportional-integral-derivative (PID) controllers. Furthermore, by allowing the plant to be described in delay descriptor form (i.e., the system's dynamics are given in terms of delay differential algebraic equations),  acceleration feedback and Pyragas-type and delay-based controllers can also be considered. \\

\noindent \packageName{} is licensed under  \href{https://www.gnu.org/licenses/gpl-3.0.en.html}{GNU GPL v.3.0} and is available for download from \link{}. \\


\end{abstract}
\pagenumbering{roman} 
\clearpage
\setcounter{tocdepth}{1}
\tableofcontents

\clearpage
%
\chapter*{List of Symbols}
\addcontentsline{toc}{chapter}{\protect\numberline{}List of Symbols}
\vspace{-6mm}
	\textbf{Notation related to sets and spaces} 
	\vspace{-4mm}
	\begin{center}
			\begingroup
		\renewcommand*{\arraystretch}{1.2}
	\begin{tabular}{p{0.15\linewidth}p{0.8\linewidth}}
		$\R$ & set of real numbers \\
		$\R^{n}$ & set of real-valued vectors of length $n$  \\
		$\R^{m\times n}$ & set of real-valued matrices with $m$ rows and $n$ columns  \\
		$\C$ & set of complex numbers \\
		$\C^{n}$ & set of complex-valued vectors of length $n$ \\
		$\Lambda$ & spectrum of an LTI differential equation, \ie the set containing its characteristic roots  \\
		$\clos(S)$ & closure of the set $S$ \\
		$X$ & Banach space of continuous functions mapping the interval $[-\tau_{\max},0]$ to $\C^{n}$ equipeped with the supremum norm $\|\cdot\|_{s}$
	\end{tabular}
		\endgroup
	\end{center}
\vspace{2mm}
	\textbf{Notation related to complex numbers}
		\vspace{-4mm}
\begin{center}
				\begingroup
	\renewcommand*{\arraystretch}{1.2}
	\begin{tabular}{p{0.15\linewidth}p{0.8\linewidth}}
		$\Re(\lambda)$ & real part of the complex number $\lambda$ \\
		$\Im(\lambda)$ & imaginary part of the complex number $\lambda$ \\
		$|\lambda|$ & modulus of the complex number $\lambda$ \\
		$\angle(\lambda)$ & argument of the complex number $\lambda$, with range $(-\pi,\pi]$
	\end{tabular}
		\endgroup
\end{center}
\vspace{2mm}	
	\textbf{Notation related to vectors and matrices}
		\vspace{-4mm}
	\begin{center}
		\begingroup
		\renewcommand*{\arraystretch}{1.2}
	\begin{tabular}{p{0.15\linewidth}p{0.8\linewidth}}
		$\|v\|_2$ & Euclidean norm of the vector $v$ \\
		$\|v\|_{\infty}$ & $\infty$-norm of the vector $v$ \\
		$\|A\|_{\mathrm{fro}} $ & Frobenius norm of the matrix $A$\\
		$A^{\top}$ & transpose of the matrix $A$ \\
		$A^{H}$ & Hermitian conjugate of the matrix $A$ \\
		$\rho(A)$ & spectral radius of the square matrix $A$\\
		$0_n$ & vector of length $n$ with all entries equal to zero \\
		$\mathrm{I}_{n}$ & identity matrix of dimension $n$\\
		\end{tabular}
		\endgroup
		\end{center}	
			\vspace{2mm}		
	\textbf{Notation related to time-delay systems}
	\vspace{-4mm}
	\begin{center}
				\begingroup
		\renewcommand*{\arraystretch}{1.15}
	\begin{tabular}{p{0.15\linewidth}p{0.8\linewidth}}
		$c$ & spectral abscissa of an LTI differential equation\\
		$\taum$ & maximal delay value \\
		$\vec{\tau}$ & vector containing all delay values
	\end{tabular}
		\endgroup
\end{center}
		\vspace{2mm}		
	\textbf{Other}
	\vspace{-4mm}
\begin{center}
		\begingroup
	\renewcommand*{\arraystretch}{1.2}
		\begin{tabular}{p{0.15\linewidth}p{0.8\linewidth}}
			$\|\phi\|_s$ & suprenum norm of $\phi \in C([-\tau_{m},0],\C^{n})$, \ie $\|\phi\|_s := \sup_{\theta \in [-\tau_{m},0]} \|\phi(\theta)\|_2$\\
				$\|f\|_{L_2}$ & $L_2$-norm of a $f\in L_2\left([0,\infty);\R^{n}\right)$, \ie $\|f\|_{L_2} := \int_{0}^{+\infty} \|f(t)\|_2^2 \dd t$ 
	\end{tabular}
\endgroup
\end{center}
\chapter*{List of Abbreviations}
\addcontentsline{toc}{chapter}{\protect\numberline{}List of Abbreviations}
\begin{table}[!ht]
    \renewcommand{\arraystretch}{1.2}
	\begin{tabular}{p{0.15\linewidth}p{0.8\linewidth}}
		LTI & Linear Time-Invariant \\
		GEP & Generalized Eigenvalue Problem \\
		ODE & Ordinary Differential Equation\\
		DDE & Delay Differential Equation \\
		PID & Proportional-Integral-Derivative \\
		PIR & Proportional-Integral-Retarded\\
		RDDE & Retarded Delay Differential Equation \\
		NDDE & Neutral Delay Differential Equation \\
		DDAE & Delay Differential Algebraic Equation \\
		SISO & Single Input Single Output
	\end{tabular}
\end{table}

\chapter*{List of Examples}
\addcontentsline{toc}{chapter}{\protect\numberline{}List of Examples}
\begingroup
\renewcommand{\arraystretch}{1.4}
	\begin{tabular}{>{\bfseries}p{0.17\linewidth}p{0.8\linewidth}}
		\multicolumn{2}{l}{\textbf{\large \Cref{sec:stability}}}\\[5pt]
		\Cref{example:retarded} & Stability analysis of a retarded DDE\\
		\Cref{example:warning} & More information on the option \matlabfun{max_size_evp} of \matlabfun{tds_roots}\\
		\Cref{example:cr_quasipolynomial} & Computing the characteristic roots of a quasi-polynomial \\
		\Cref{example:neutral1} & Computing the characteristic roots of a netural DDE \\
		\Cref{example:neutral1b} &  Relation between  the characteristic roots of a netural DDE  and the characteristic roots of its underlying delay difference equation \\
		\Cref{example:neutral_2} & Stability analysis for a neutral DDE with multiple delays \\
		\Cref{example:neutral_3} & Strong stability analysis for a neutral DDE with multiple delays \\
		\Cref{example:roots_NDDE_rhp} & Computing characteristic roots of a NDDE in a given right half-plane \\
		\Cref{example:stability_analysis_DDAE} & (Strong) stability analysis for a delay differential algebraic equation \\
		\Cref{example:roots_qp2} & Computing the characteristic roots of a quasi-polynomial \\
		\Cref{example:computing_transmission_zeros} & Computing the transmission zeros of a SISO time-delay system \\
		\Cref{example:smith_predictor} & Stability analysis of a Smith predictor with delay mismatch\\
		\end{tabular}
	\newline
	\vspace{0.5cm}
	\newline
		\begin{tabular}{>{\bfseries}p{0.17\linewidth}p{0.8\linewidth}}
		\multicolumn{2}{l}{\textbf{\large \Cref{sec:stabilization}}} \\[5pt]
		\Cref{example:stabilization_retarded_2} & Stabilization of a heat-exchanger system \\
		\Cref{example:stab_neutral} & Stabilization of a neutral-like closed-loop system \\
		\Cref{example:robustness_feedback_delay} & Robustness against infinitesimal feedback delay \\
		\Cref{ex:pyragas_feedback} & Design of a Pyragas-type feedback controller \\
		\Cref{ex:delayed_feedback} & Delayed feedback control \\
		\Cref{example:acceleration_feedback} & Acceleration feedback control \\
		\Cref{example:fixed_entries} & Design of structured controllers\\
		\Cref{example:pid_control} & Design of a PID controller \\
		\Cref{ex:decentralized} & Design of decentralized, distributed and overlapping controllers for networked systems \\
		\Cref{example:SISO_transfer_function} & Obtaining a state-space representation from a SISO transfer function 
				\end{tabular}
	\begin{tabular}{>{\bfseries}p{0.17\linewidth}p{0.8\linewidth}}
		\multicolumn{2}{l}{\textbf{\large \Cref{chapter:performance_robust_control}}} \\[5pt]
		\Cref{example:hinfnorm1} & Computing the H-infinity norm of a time-delay system \\
		\Cref{example:hinfnorm2} & Sensitivity of the \hinfnorm{} norm with respect to infinitesimal delay perturbations \\
		\Cref{example:strong_hinfnorm} & Computing the strong \hinfnorm{} \\
		\Cref{example:Hinf_design} & $\hinf$-controller design \\
		\Cref{ex:measurement_noise} & Structured $\hinf$-controller design \\
		\Cref{example:mixed_sensitivity} & Mixed sensitivity controller design \\
		\Cref{example:MOR} & Model-order reduction \\
		\Cref{example:mixed_performance} & Minimization of a mixed performance criterium \\
		\Cref{example:pseudo_definition} & \multirow{2}{\linewidth}{Computing the pseudospectral abscissa and the distance to instability of an uncertain time-delay system }\\[-1.2mm]
		\Cref{example:pseudo_cont} & 
	\end{tabular}
\endgroup
\clearpage
\pagenumbering{arabic} 

\chapter{Introduction}
Delays inherently appear in the study of numerous dynamical systems, especially in control applications~\cite{delays_csm}. In feedback systems, for example, delays arise from the time needed for collecting measurements and computing the control action. Time-lags may however also pop up within (the models for) physical systems themselves, such as in metal cutting~\cite{Insperger2000} or lasers~\cite{erneux2009}, or be purposely introduced to improve performance, see for example the delayed resonator \cite{Olgac1994DelayedResonator} and the proportional-integral-retarded (PIR) control architecture~\cite{ramirez2015}. Yet, as we will see throughout this manual, the analysis and control of such systems is typically more involved than for their delay-free counterparts. It is therefore not surprising that the study of time-delay systems has attracted significant interest  over the last 70 years, as witnessed by the wide variety of books and monographs devoted to the topic \cite{bellman1963dde,Kolmanovskii1986FDE,hale1993,bookdelay}.

Throughout the years, various methodologies have been developed for the analysis and controller-design of time-delay systems. Broadly speaking, these methods can be divided into two categories: the time-domain framework and the frequency-domain framework. Methods in the former category are based on extensions of Lyapunov's second method \cite{krasovskii1963stability,razumikhin1956} and (typically) give rise to analysis and synthesis problems in terms of linear matrix inequalities, see for instance \cite{fridman2014introduction}. The latter category contains the spectrum-based approaches described in \cite{bookdelay}. In contrast to the time-domain approaches, in which is conservatism is introduced by choice of the Lyapynov-Krasovskii functional, the frequency-domain framework allows to provide necessary and sufficient conditions for stability and performance. For example, as a strictly negative spectral abscissa is a necessary and sufficient condition for stability. However, an advantage of the the time-domain framework is that its results and tools can more easily be extended to non-linear time delay systems. Next, with respect to controller-design, the time-domain based methods are typically restricted to unstructured controllers with an order equal to or larger than that of the plant. In contrast, the frequency-domain approach allows to design fixed-order and structured controllers. 


In recent years, several efficient numerical algorithms have been developed to analyze and control linear time-invariant (LTI) time-delay systems within the frequency-domain framework including, among others, \cite{breda2014,vanbiervliet2008,gumussoy2011,roots,borgioli2019novel}. Yet, an integrated framework is lacking. \packageName{} aims at integrating these methods in an easy-to-use and well-documented \matlab{} package for frequency-domain based analysis and controller-design methods aimed at both expert and non-expert users. More specifically, \packageName{} can deal with a wide variety of LTI time-delay systems with point-wise (discrete) delays, namely:
\begin{enumerate}
\setlength{\itemsep}{2pt}
    \item time-delay systems of \textbf{retarded} type,
    \item time-delay systems of \textbf{neutral} type, and
    \item \textbf{delay descriptor systems}, \ie time-delay systems for which the dynamics described by \textbf{delay-differential algebraic equations}.
\end{enumerate}
For the system classes listed above, \packageName{} offers the following functionality:
\begin{enumerate}
    \item \textbf{Stability analysis} (\Cref{sec:stability})
    \begin{itemize}
    	\item \matlabfun{tds_roots} computes the characteristic roots of a time-delay system in a given right half-plane or rectangular region.
    	\item \matlabfun{tds_sa} and \matlabfun{tds_strong_sa} compute the (strong) spectral abscissa of a time-delay system.
    \end{itemize}
    \item \textbf{Stabilization} (\Cref{sec:stabilization})
    \begin{itemize}
    	\item \matlabfun{tds_stabopt_static} and \matlabfun{tds_stabopt_dynamic} synthesize stabilizing static, dynamic and structured output feedback controllers.
    \end{itemize}
    \item \textbf{Performance, robust stability analysis and robust control} (\Cref{chapter:performance_robust_control})
    \begin{itemize}
    \item \matlabfun{tds_hinfnorm} computes the (strong) \hinfnorm{} of a time-delay system.
    \item \matlabfun{tds_hiopt_static} and \matlabfun{tds_hiopt_dynamic} synthesize static, dynamic and structured output feedback controllers based on the robust control framework.
    \item  \matlabfun{tds_psa} computes the pseudospectral abscissa of an uncertain retarded delay differential equation with bounded, structured, real-valued uncertainties on both the system matrices and the delays.
    \item \matlabfun{tds_dist_ins} computes the distance to instability of an uncertain retarded delay differential equation with bounded, structured, real-valued uncertainties on both the system matrices and the delays.
    \end{itemize}

\end{enumerate}
The design functions described above allow to synthesize fixed-order dynamic output feedback controllers (meaning that the order of the controller is specified by the user and can be smaller than the order of the plant). Note that this also includes static output feedback controllers, which are essentially dynamic controllers of order zero. Furthermore, it is also possible to design structured controllers such as delay-based, acceleration feedback, decentralized, and PID controllers (see \Cref{subsec:structured_controllers}). To this end, the presented software package employs a direct optimization approach, \ie finding a suitable controller by directly optimizing the spectral abscissa, the \hinfnorm{}, or a combination of both with respect to the controller parameters. The controller-design part of \packageName{} can thus be seen as an extension of \verb|HIFOO| \cite{HIFOO1,HIFOO2} and \verb|hinfstruct| \cite{hinfstruct} to time-delay systems. As mentioned before, in contrast to time-domain methods, this approach does not introduce conservatism and, for example, a stabilizing controller can be computed whenever it exists.
 This comes however at the cost of having to solve a (small) non-smooth, non-convex optimization problem. \packageName{} therefore relies on the software package \verb|HANSO| \cite{HANSO,HANSO2}, which implements a hybrid algorithm for non-smooth, non-convex optimization. \\

Throughout this manual, we will refrain from proving theoretical results. For such proofs and more background information on time-delay systems, we refer the interested readers to the popular handbooks and monographs on the topic including, among others, \cite{hale1993,gu2003stability,bookdelay,fridman2014introduction}. Instead we will focus on illustrating the presented result using the software package while in the process demonstrating its core working principles and functionality. As such, this manual contains numerous examples and code-snippets, which we encourage the reader to try out themselves. The code for these examples is also available in the \matlabfun{examples} folder. \\

The remainder of this manual is organized as follows. \Cref{sec:stability} will more rigorously introduce the considered class of time-delay systems, illustrate how these systems can be represented in \packageName{} and demonstrate how the exponential stability of their null solution can be analyzed. Special attention will be paid to time-delay systems of neutral type and to the effect of (infinitesimal) delay perturbations on their stability. \Cref{sec:stabilization} deals with designing  stabilizing output feedback controllers. Various examples will illustrate the design of both unstructured and structured control architectures. \Cref{sec:performance} deals with performance and robust stability analysis, thereby introducing measures such as the \hinfnorm{}, the pseudospectral abscissa and the distance to instability. This chapter also demonstrates how \packageName{} can be used to design controllers that minimize the \mbox{\hinfnorm{}} of the resulting closed-loop system. \Cref{sec:installation} will guide you through the installation of the \packageName{} package. Finally, \Cref{sec:documentation} contains more technical background information on the key functions of \packageName{} and their documentation.

\chapter{Stability analysis using a spectrum-based approach}
\label{sec:stability}
This chapter introduces the considered time-delay systems and discusses the stability \mbox{analysis} of their null solution within the eigenvalue-based framework. Hereby we will differentiate between delay differential equations of retarded and neutral type. Furthermore, we will also briefly touch upon delay differential algebraic equations (DDAEs), which combine delay differential and delay algebraic equations. For each system class, we will briefly introduce the necessary notation and the properties, show how these systems can be represented within \packageName{}, and how the exponential stability of their null solution can be examined. To this end, we will first recall some well-known concepts and results for ordinary differential equations (ODEs). Subsequently, we consider retarded delay differential equations (RDDEs) and we show that the obtained stability results are quite similar to the ODE case. Next, we will consider neutral delay differential equations (NDDEs), for which the stability analysis is more involved. Therefore, we will first restrict our attention to NDDEs with a single delay, before considering the multiple delay case. In the latter case, we will see that stability might be fragile with respect to the delays, in the sense that an exponentially stable system might be destabilized by an arbitrary small perturbation on the delays. To resolve this fragility problem, we will consider strong stability, which requires that stability is preserved for all sufficiently small perturbations (on the delays). Thereafter \Cref{subsec:stability_DDAE} introduces the considered class of DDAEs, which (as we will see) encompasses both retarded and neutral systems. Next, to be able to use the stability analysis tool of \packageName{} a state-space representation for the system is necessary. \Cref{subsec:quasi_polynomial} will therefore explain how such a model can be obtained from a frequency-domain representation of the system within \packageName{}. Finally, we will consider several stability analysis problems for which \packageName{} can be used.
\section{Ordinary differential equations}
To familiarize ourselves with the necessary concepts and to highlight the differences between systems with and without delays, we will first consider the following ordinary differential equation: 
\begin{equation}
\label{eq:ode}
\dot{x}(t) = A x(t) \text{ for }t\in [0,+\infty)
\end{equation}
with the state variable $x(t)$ belonging to $\C^{n}$ for all $t\in [0,+\infty)$ and the state matrix $A$ belonging to $\R^{n\times n}$. It is well-known that \eqref{eq:ode} has a unique forward solution once an initial state $x(0)$ is specified. 
Furthermore, by plugging the sample solution $t \mapsto v e^{\lambda t}$ (with $\lambda\in\C$ and $v\in\C^{n}\setminus \{0_{n}\}$) into \eqref{eq:ode}, we obtain the following \textbf{characteristic equation}
\begin{equation}
\label{eq:ode_chf}
\det \big(\lambda \mathrm{I}_{n} - A \big) = 0
\end{equation}
with $\mathrm{I}_{n}$ the identity matrix of dimension $n$, that ``characterizes'' the behavior of the system. As the function at the left-hand side is a polynomial of degree $n$, \eqref{eq:ode_chf} has $n$ solutions, the so-called \textbf{\emph{characteristic roots}}. In this work, the set containing all characteristic roots of a linear-time invariant differential equation, \ie its \textbf{\emph{spectrum}}, will be denoted by $\Lambda$. For example, for \eqref{eq:ode} the corresponding spectrum is given by  
    \begin{equation}
    \label{eq:spectrum_ode}
    \Lambda(A) := \{\lambda \in \C: \det \big(\lambda \mathrm{I}_{n} - A \big) = 0 \}.
    \end{equation}
In \matlab{}, the set in \eqref{eq:spectrum_ode} can be computed using
\begin{lstlisting}
cr = eig(A);
\end{lstlisting}

Let us now examine the relation between these characteristic roots and the exponential stability of the null solution of \eqref{eq:ode}. To this end let us first formally introduce the concept of exponential stability. 
\begin{definition}
    The null solution of \eqref{eq:ode} is (globally)\footnote{For LTI systems, local and global exponential stability coincide. } \textbf{\emph{exponentially stable}} if and only if there exist constants $C>0$ and $\gamma > 0$ such that
    \begin{equation}
    \label{eq:exponential_decay}
    \|x(t;x_0)\|_2 \leq C e^{-\gamma t} \|x_0\|_2 \text{ for all } t\geq 0 \text{ and any } x_0 \in \C^{n},
    \end{equation}
    with $t\mapsto x(t;x_0)$ the unique solution of \eqref{eq:ode} starting from the initial condition $x(0) = x_0$. Or in other words, exponential stability requires that all solutions of \eqref{eq:ode} decay exponentially fast to 0 as $t\rightarrow \infty$. In the remainder of this work, we will (with a slight abuse of terminology) refer to a system as exponential stable if its null solution is exponential stable.
\end{definition}
 Next, we recall the following well-known relation between the exponential stability of \eqref{eq:ode} and the location of its characteristic roots.
\begin{theorem}
\label{th:stability_ode}
The null solution of \eqref{eq:ode} is globally exponentially stable if and only if all its characteristic roots lie in the open left half-plane, \ie
\[
\Re(\lambda) < 0 \text{ for all } \lambda \in \Lambda(A).
\]
\end{theorem}
As a consequence, to asses exponential stability, it suffices to only consider the rightmost characteristic root(s). Let us therefore introduce the notion of the spectral abscissa.
\begin{definition}
\label{def:spectral_abscissa}
The \textbf{\emph{spectral abscissa}} of \eqref{eq:ode} is equal to the real part of its rightmost characteristic root(s) and will be denoted by $c$, \ie{} 
\[
\sa(A) := \max \{\Re(\lambda) : \lambda \in  \Lambda(A)\}.
\]
\end{definition}
The subsequent result now directly follows from \Cref{th:stability_ode}.
\begin{corollary}
The null solution of \eqref{eq:ode} is globally exponentially stable if and only if its spectral abscissa is strictly negative.
\end{corollary}
To conclude this section, we recall that the spectral abscissa, if strictly negative, bounds the exponential decay rate of the solutions of \eqref{eq:ode} for arbitrary initial conditions (in the sense of\eqref{eq:exponential_decay}). Next, we will see that a lot of these concepts and results directly carry over to LTI delay differential equations of retarded type.

\section{Delay differential equations of retarded type}
\label{sec:stability_retarded}

In this section we consider LTI delay differential equations of retarded type with point-wise delays. Such systems have the following general form:
\begin{equation}
\label{eq:ol_ret}
\dot{x}(t) = \sum_{k=1}^{m_A} A_k\, x(t-h_{A,k}) \text{ for }t\in [0,+\infty)
\end{equation}
in which $x(t)\in\C^{n}$ represents the state variable at time $t$, the delays $h_{A,1}, \dots,h_{A,m}$ are non-negative, and the matrices $A_1,\dots,A_{m_A-1}$, and $A_{m_A}$ belong to $\R^{n\times n}$. Furthermore, for ease of notation, let us introduce the vector $\vec{\tau}$ containing all delay values and the scalar $\taum$ denoting the maximal delay value, \ie 
\[
\vec{\tau} = \left[
h_{A,1}, \dots, h_{A,m_A}\right] \text{ and }
\taum = \max\left\{h_{A,1},\dots,h_{A,m_A}\right\}.
\]
Notice that in contrast to the ordinary differential equation in \eqref{eq:ode}, the evolution of the state variable $x$ at time $t$ not only depends on its current value but also on its value at discrete instances in the past. As a consequence, the knowledge of $x(0)$ does not suffice to (uniquely) forward simulate \eqref{eq:ol_ret}. Indeed, to correctly formulate an initial value problem associated with \eqref{eq:ol_ret}, we require an intial function segment $\phi(t)$ that gives the value of the state variable for all $t\in[-\taum,0]$. More specifically, given a function segment $\phi(t)$ 
the initial value problem associated with \eqref{eq:ol_ret} consists of finding a unique solution $t \mapsto x(t;\phi)$ that satisfies 
\[ x(t;\phi) = \phi(t) \text{ for all } t\in[-\taum,0]
\]
and 
\[
\dot{x}(t) = \sum_{k=1}^{m_A} A_k\, x(t-h_{A,k}) \text{ for all }t\in [0,+\infty).
\]
As shown in \cite[Chapters 1 and 2]{hale1993}, the retarded delay differential equation \eqref{eq:ol_ret} admits an unique solution for all $\phi$ belonging to the Banach space $X:=\mathrm{C}\left([-\taum,0],\C^{n}\right)$, \ie the set consisting of the continuous functions that map the interval $[-\taum,0]$ to $\C^{n}$ and that is equipped with the suprenum norm 
\[ \|\phi\|_{s} = \sup_{\theta \in [-\taum,0]} \|\phi(\theta)\|_2.
\] The state of \eqref{eq:ol_ret} at time $t$ is thus the function segment $x_{t} \in X$ with 
\[x_{t}(\theta;\phi) := x(t+\theta;\phi) \text{ for } \theta \in [-\taum,0]. \] 
Note that in constrast to the ODE case, the state is now infinite dimensional.

To examine exponential stability of the null solution of \eqref{eq:ol_ret}, we again plug the sample solution $t\mapsto e^{\lambda t}v$ into the differential equation. We now obtain the following characteristic equation
\begin{equation}
\label{eq:RDEVP}
\det\left(\lambda \mathrm{I}_n  - \sum_{k=1}^{m_A} A_k e^{-\lambda h_{A,k}}\right) = 0.
\end{equation}
Notice that the characteristic function is no longer a polynomial but rather a  quasi-polynomial. Furthermore, the spectrum of \eqref{eq:ol_ret}, \ie{} the set
\[
\Lambda(A_{1},\dots,A_{m_A},\vec{\tau}) := \big\{\lambda \in \C : \det\left(\lambda \mathrm{I}_n  - \textstyle\sum_{k=1}^{m_A} A_k e^{-\lambda h_{A,k}}\right) = 0\big\},
\]
 generally contains infinitely many points. In contrast to the ODE case, it is thus not possible to compute all characteristic roots of \eqref{eq:ol_ret}. However the following important result from \cite[Theorem~1.5]{gu2003stability} states that any right half-plane only contains finitely many characteristic roots.
\begin{proposition}
\label{prop:finite_number}
For each $r\in \R$, the right half-plane $\{z\in \C : \Re(z) \geq r \}$ only contains finitely many (counting multiplicity) characteristic roots of \eqref{eq:ol_ret}.
\end{proposition}
The spectral abscissa of \eqref{eq:ol_ret} can thus be defined in a similar way as for an ODE.
\begin{definition}
\label{def:spectral_abscissa_ret}
The spectral abscissa of \eqref{eq:ol_ret} is equal to the real part of its rightmost characteristic root(s), \ie{}
\begin{equation}
\label{eq:sa_retarded}
\sa(A_{1},\dots,A_{m_A},\vec{\tau}) := \max \big\{\Re(\lambda) : \lambda \in \Lambda(A_{1},\dots,A_{m_A},\vec{\tau})\big\}.
\end{equation}
\end{definition}

Now we can examine the asymptotic behavior of the solutions of \eqref{eq:ol_ret}. We again start with a definition for exponential stability.
\begin{definition}[{\cite[Definition~1.5]{bookdelay}}]
\label{def:exp_stability}
    The null solution of \eqref{eq:ol_ret} is (globally) exponentially stable if and only if there exist constants $C>0$ and $\gamma > 0$ such that
    \[
    \|x(t;\phi)\|_2 \leq C e^{-\gamma t} \|\phi\|_s \text{ for all } \phi \in X.
    \]
    Or in other words, the null solution of \eqref{eq:ol_ret} is exponentially stable if (and only if) all solution of \eqref{eq:ol_ret} converge exponentially fast to 0 for $t\rightarrow \infty$.
\end{definition}
\noindent As in the previous section, exponential stability can be assessed in terms of the spectral abscissa.
\begin{theorem}[{\cite[Theorem~1.5]{gu2003stability}}]
\label{th:exp_stab}
The null solution of \eqref{eq:ol_ret} is exponentially stable if and only if 
\[
\sa(A_{1},\dots,A_{m_A},\vec{\tau}) < 0.
\]
\end{theorem}
Combining this result with \Cref{prop:finite_number}, it follows that we can asses exponential stability by checking the location of the characteristic roots.
\begin{corollary}
The null solution of \eqref{eq:ol_ret} is exponentially stable if and only if all characteristic roots of \eqref{eq:ol_ret} lie in the open left half-plane, \ie
\[
\Re(\lambda) <0 \text{ for all } \lambda \in  \Lambda(A_{1},\dots,A_{m_A},\vec{\tau}).
\]
\end{corollary}

Next, let us illustrate how one can represent a RDDE of the form \eqref{eq:ol_ret} in \packageName{} and compute the associated characteristic roots in a given right half-plane. 
\begin{example}
\label{example:retarded}
Consider the following RDDE from \cite[Section~6.1]{Verheyden2008}
\begin{equation}
\label{eq:rdde1}
\dot{x}(t) = \begin{bmatrix}
-1 & 0 & 0 & 0 \\
0 & 1 & 0 & 0 \\
0 & 0 & -10 & -4 \\
0 & 0 & 4 & -10 
\end{bmatrix}x(t) + \begin{bmatrix}
3 & 3 & 3 & 3\\
0 & -1.5 & 0 & 0 \\
0 & 0 & 3 & -5 \\
0 & 5 & 5 & 5
\end{bmatrix}x(t-1).
\end{equation}
To represent this RDDE in \packageName{}, we will use the function \matlabfun{tds_create(A,hA)}. The first argument of this function (\verb|A|) should be a cell array containing the state matrices. The second argument (\verb|hA|) should be a vector with the state delays.
\begin{lstlisting}
A0 = [-1 0 0 0;0 1 0 0;0 0 -10 -4;0 0 4 -10];
A1 = [3 3 3 3;0 -1.5 0 0;0 0 3 -5;0 5 5 5];
tau = 1;
rdde = tds_create({A0,A1},[0 tau]); 
\end{lstlisting}
The output (\verb|rdde|) is an object of type \matlabfun{tds_ss_retarded} that stores the relevant information of the RDDE, such as the system matrices, the delay values, the dimensions of the state vector $x$ and the number of delays. We can inspect these properties using the \matlab{} Command Window.
\begin{lstlisting}
rdde = 

LTI Retarded Delay Differential Equation with properties:

	A: {[4x4 double]  [4x4 double]}
	hA: [0 1]
	mA: 2
	n: 4
\end{lstlisting}
Next, we will use the  function \matlabfun{tds_roots(tds,r,options)} to compute its characteristic roots in a given right half-plane. This function takes two mandatory arguments. The first argument is the \matlabfun{tds_ss_retarded}-object that we have just created, while the second argument is a real number that specifies the desired right half-plane, $\{z\in \C : \Re(z)\geq \verb|r| \}$. (Recall from \Cref{prop:finite_number} that such a right half-plane can only contain finitely many characteristic roots.) The third argument is optional and allows to specify additional options, some of which we will discuss later in this manual. If you want more information on the interface of this function or the available options you can use the \matlab{} \verb|help| command
\begin{lstlisting}
>> help tds_roots
\end{lstlisting}
or take a look at \Cref{sec:tds_roots}. We are now ready to compute the characteristic roots of \eqref{eq:rdde1} in the right half-plane $\{\lambda \in \C: \Re(\lambda)\geq -1.5\}$.
\begin{lstlisting}
cr = tds_roots(rdde,-1.5); 
\end{lstlisting}
After executing this command, the variable \verb|cr| contains the desired characteristic roots and the following text is printed to the \matlab{} output window
\begin{lstlisting}
Degree of the spectral discretisation is 19.
\end{lstlisting}
As explained below, the degree of the spectral discretization gives an indication of the quality of the approximation and the size of the region in which the characteristic roots are well approximated. To plot these characteristic roots, we use the function \verb|tds_eigenplot|. 
\begin{lstlisting}
tds_eigenplot(cr); 
xlim([-1.5 1])
\end{lstlisting}
\Cref{fig:EV_ret} shows the result. We observe that the null solution of \eqref{eq:rdde1} is unstable as three characteristic roots lie in the closed right half-plane.
\begin{figure}[!h]
    \centering
    \includegraphics[width=0.7\linewidth]{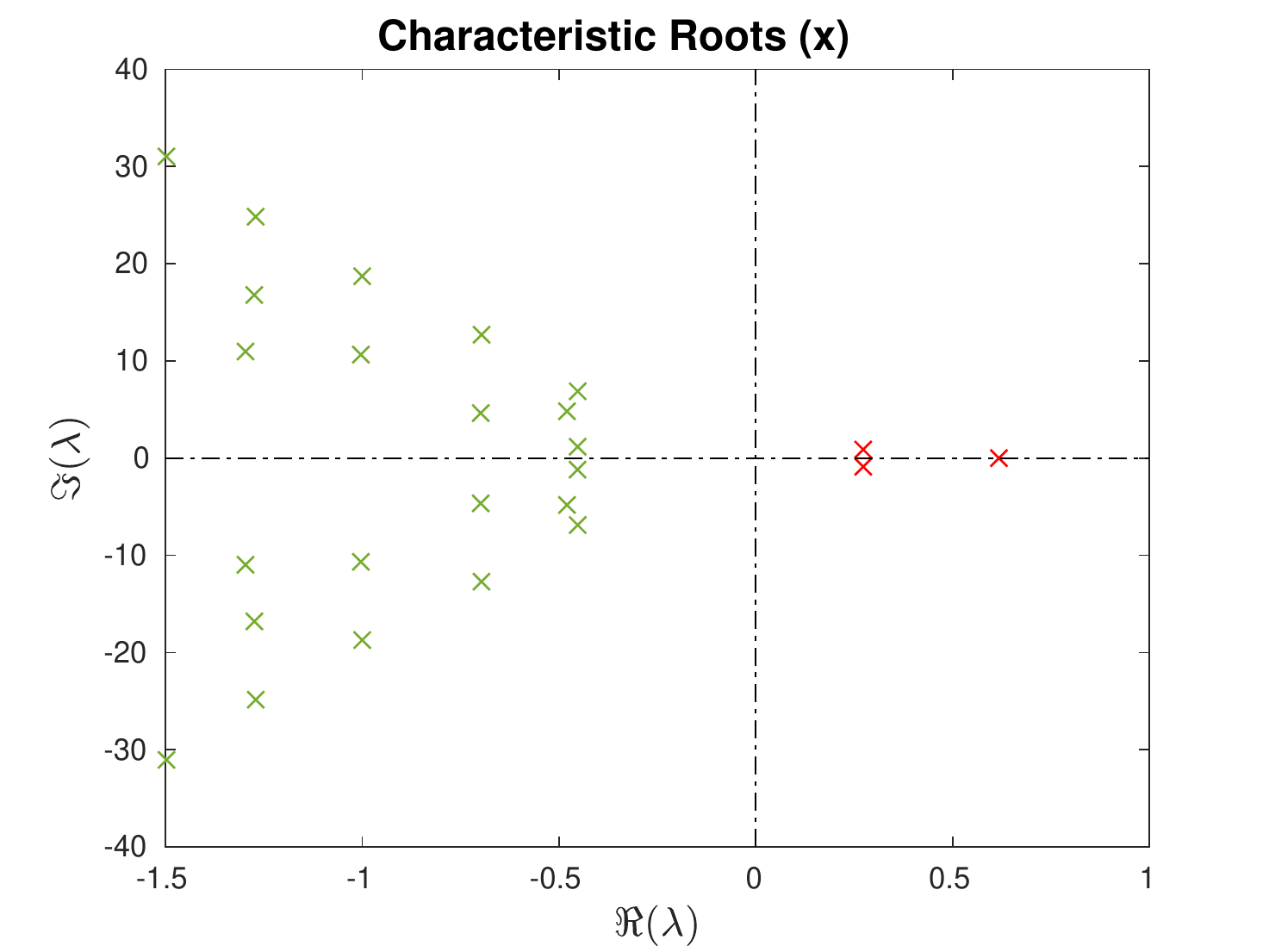}
    \caption{The characteristic roots of \eqref{eq:rdde1} in the region $[-1.5,1] \times \jmath [-40,40]$.}
    \label{fig:EV_ret}
\end{figure}

 Alternative, the function \matlabfun{tds_sa} can be used to directly compute the spectral abscissa.
 \begin{lstlisting}
>> tds_sa(rdde,-1.5)

ans =

0.6176
 \end{lstlisting}
 This function call will first compute all characteristic roots in the right half-plane $\{\lambda \in \C: \Re(\lambda)\geq -1.5\}$ and subsequently return the real part of the rightmost root(s).
\end{example}

To compute all characteristic roots of \eqref{eq:ol_ret} in a specified right half-plane, the function \matlabfun{tds_roots} implements the method from \cite{roots}. More specifically the following two-step approach is used. 
\begin{enumerate}
	\item Compute the eigenvalues of a generalized eigenvalue problem (GEP) of the form 
	\begin{equation}
	\label{eq:discrete_evp}
	\Sigma_N\,v  =  \lambda \Pi_N \,v
	\end{equation}
	with $\Sigma_N$ and $\Pi_N$ belonging to $\R^{(N+1)n\times (N+1)n}$. The matrices $\Sigma_N$ and $\Pi_N$ are constructed in such a way that eigenvalues of \eqref{eq:discrete_evp} form an approximation for the desired characteristic roots. More specifically, the GEP \eqref{eq:discrete_evp} corresponds to a spectral discretization of the infinitesimal generator of the solution operator underlying the delay differential equation. (For more details on spectral discretisation methods for computing the characteristic roots of time-delay systems, we refer the interested reader to \cite{breda2014}.) For a sufficiently high discretization degree (N), these eigenvalues form a good approximation for the characteristic roots of \eqref{eq:ol_ret} in the desired right half-plane. Note that the dimensions of the matrices $\Sigma_N$ and $\Pi_N$ depend on the parameter $N$, \ie the degree of the spectral discretisation. This $N$ is related to the quality of the spectral approximation: the larger $N$, the better individual characteristic roots are approximated and the larger the region in the complex plane in which the characteristic roots are well approximated. However, the larger $N$, the larger the size of the generalized eigenvalue problem and hence the longer the computation time.  An appealing feature of the method in \cite{roots} is that it provides an empirical method to automatically determine a good value for $N$ such that all characteristic roots in the considered right half-plane are sufficiently well captured while keeping the generalized eigenvalue problem as small as possible.
	\item  Improve these approximations using Newton's method.
\end{enumerate}
 For more information and an overview of the available options of \matlabfun{tds_roots}, we refer the interested reader to \Cref{sec:tds_roots}. Below we will discuss one important aspect of the \matlabfun{tds_roots}-function in more detail by means of an example. 

\begin{example}
	\label{example:warning}
	Consider again the differential equation from \eqref{eq:rdde1} but now we will compute the characteristic roots in the right half-plane $\{\lambda \in \C : \Re(\lambda) \geq -4.5 \}$.
	\begin{lstlisting}
cr = tds_roots(rdde,-4.5);
	\end{lstlisting}
	The function \matlabfun{tds_roots} now gives the following warning.
	\begin{lstlisting}[basicstyle=\ttfamily\color{red},language=,breakindent=0pt, breaklines]
Warning: (tds_roots): Size of the generalized EVP would exceed its maximum
value. Discretization around 0.000e+00 + 0*1j with N = 149 instead (size of
new eigenvalue problem: 600x600). As a consequence, not all characteristic
roots in the specified right half-plane might be found. To make sure that
all desired characteristic roots are found,either increase r (i.e., shift
the desired right half-plane to the right) or the option max_size_evp. For
more information consult the online documentation of this function.
	\end{lstlisting}
	\Cref{fig:rdde2} shows the computed characteristic roots. We see that some of the characteristic roots to the left of $-3.5$ are missing. This behavior can be understood as follows. As computing the eigenvalues of a large matrix (pencil) is time-consuming, \packageName{} will check that the dimensions of the GEP \eqref{eq:discrete_evp} does not become to large. More specifically, if $(N+1)n$ would exceed $\matlabfun{options.max_size_evp}$, the degree of the spectral discretisation is lowered such that for the new degree the inequality $(N+1)n \leq \matlabfun{options.max_size_evp}$ is satisfied. As a consequence, it is possible that certain characteristic roots in the specified right half-plane are no longer sufficiently well approximated by the eigenvalues of the GEP \eqref{eq:discrete_evp}. To resolve this warning, one can either move the chosen right half-plane to the right (which lowers the degree of the spectral discretization that is necessary to accurately capture all desired characteristic roots) or increase the option \matlabfun{max_size_evp} at the cost of an increase in the computation time. In most situations, the first approach is preferred as the rightmost eigenvalues have the largest effect on the asymptotic behavior of solutions of \eqref{eq:ol_ret}. To conclude this example, let us illustrate the second approach. To increase the maximal size of the generalized eigenvalue problem to 2000, we can use the function \matlabfun{tds_roots_options}.
	\begin{lstlisting}
options = tds_roots_options('max_size_evp',2000);
cr = tds_roots(rdde,-4.5,options);
	\end{lstlisting}
	In \Cref{fig:rdde3} we see that all characteristic roots in the desired right half-plane are now found. This comes however at a significant increase in the computation time. 
	\begin{figure}
		\centering
		\begin{subfigure}{0.49\linewidth}
			\includegraphics[width=\linewidth]{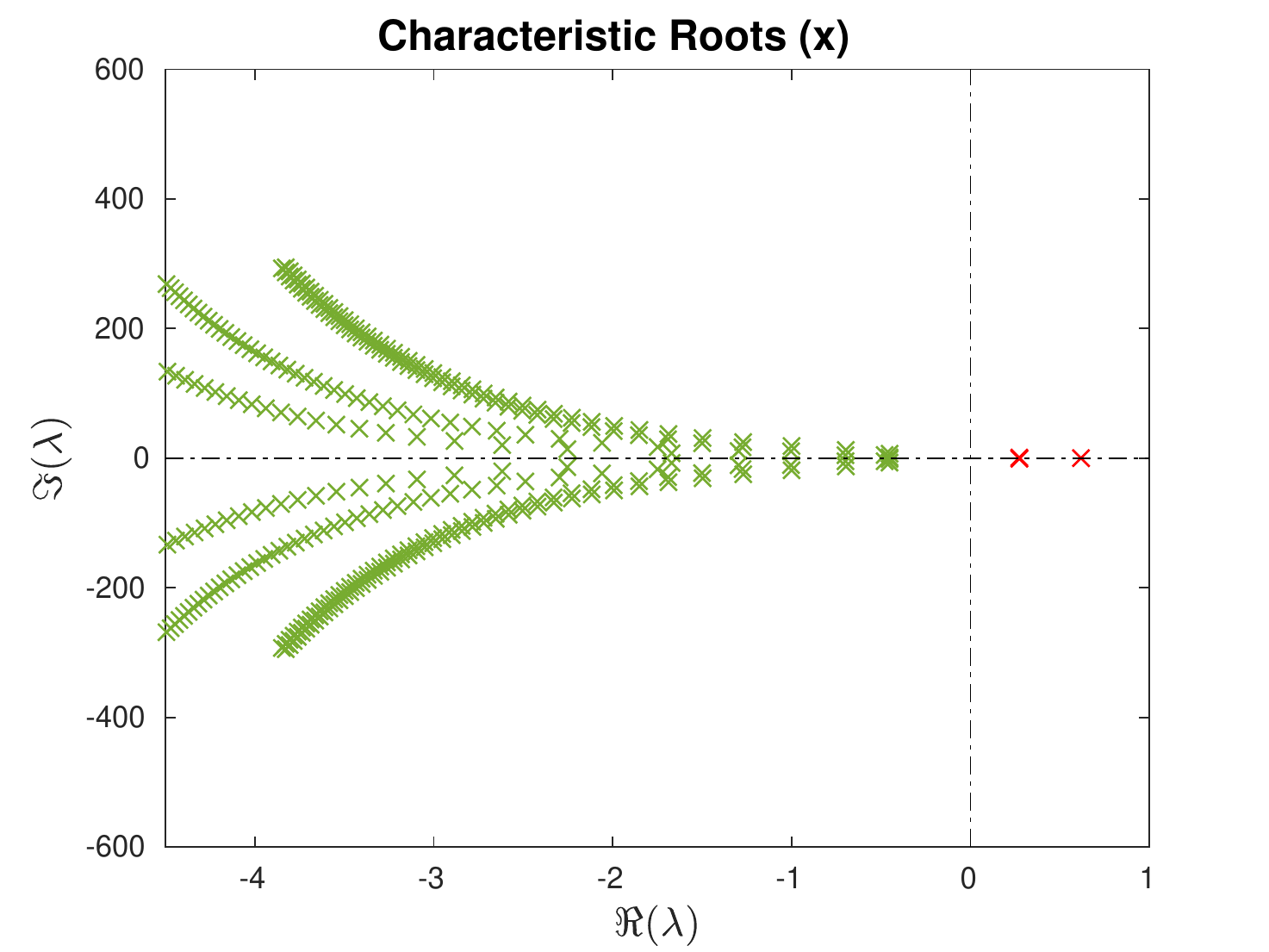}
			\caption{}
			\label{fig:rdde2}
		\end{subfigure}
		\begin{subfigure}{0.49\linewidth}
			\includegraphics[width=\linewidth]{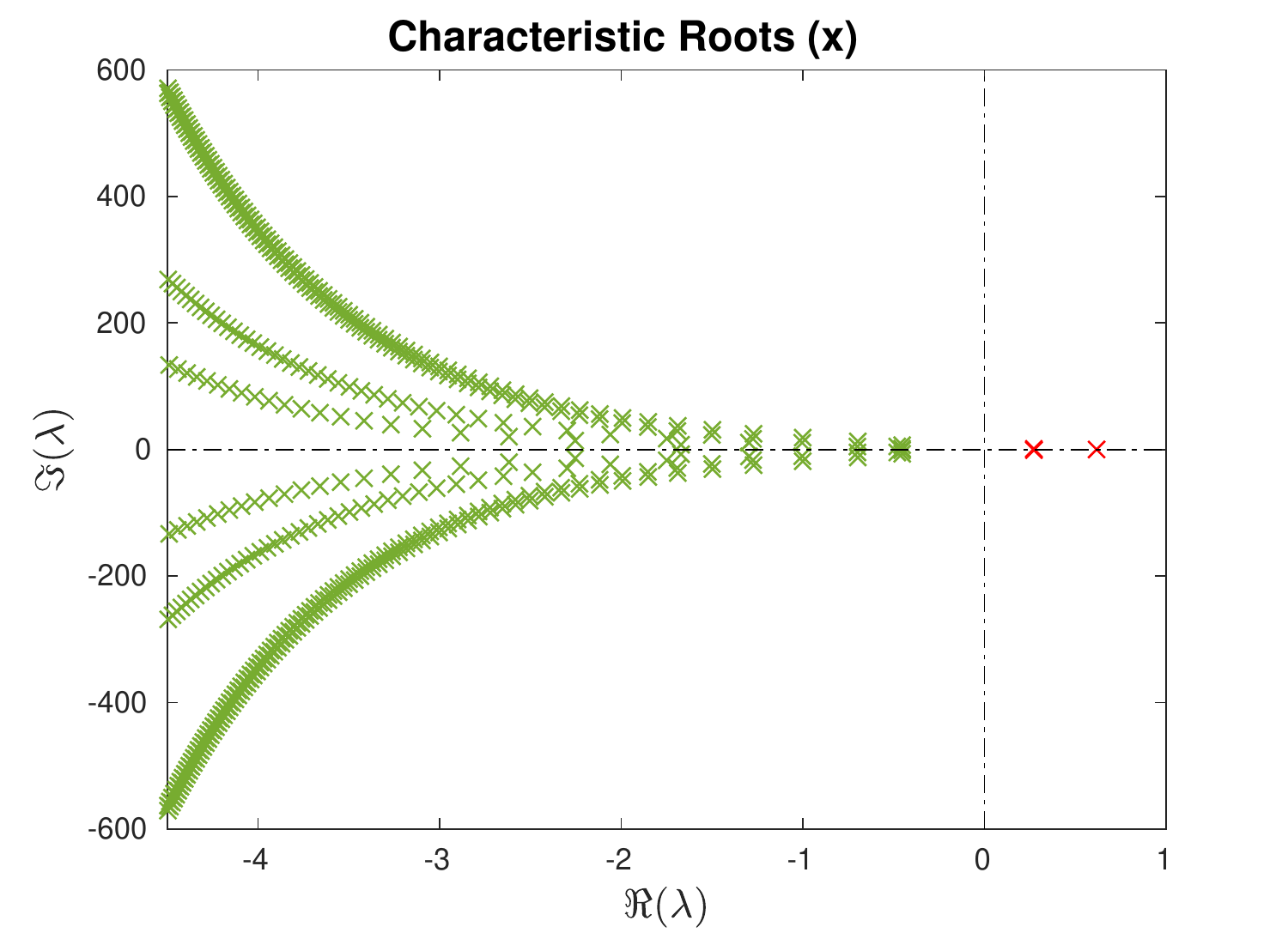}
			\caption{}
			\label{fig:rdde3}
		\end{subfigure}
		\caption{The characteristic roots of \eqref{eq:rdde1} in the right half-plane $\{\lambda \in \C : \Re(\lambda) \geq -4.5 \}$ obtained with the default options (left) and for \matlabfun{max_size_evp} increased to 2000 (right).}
		\label{fig:roots_warning}
	\end{figure}
\end{example}
\begin{remark}
Notice the scale difference between the real and imaginary axis in \Cref{fig:rdde3}. While the real part of the depicted characteristic roots is bounded to the interval $[-4.5,1]$, the imaginary part ranges from $-600$ to $600$. This observation can be related to the fact that the characteristic roots with large modulus  lie along a finite number of asymptotic curves which have an exponential trajectory \cite{bellman1963dde}. In light of the discussion above, it is therefore possible that the required number of discretization points rises (very) quickly and, as a consequence, the size of the generalized eigenvalue problem,  when $r$ is decreased.
\end{remark}

As mentioned before, the characteristic roots of \eqref{eq:ol_ret} corresponds to the zeros of the characteristic function at the left-hand side of \eqref{eq:RDEVP}. In some cases we are readily given this quasi-polynomial. In such a case we can use the function \matlabfun{tds_create_qp} to create a corresponding state-space representation of the form \eqref{eq:ol_ret}. Subsequently, we can again use \matlabfun{tds_roots} to compute the desired characteristic roots.
\begin{example}
	\label{example:cr_quasipolynomial}
Consider the following quasi-polynomial
\[
\lambda^2+\omega^2 - k\,e^{-\lambda 0.1}
\]
which corresponds to a second-order oscillator controlled by delayed position feedback. First we use \matlabfun{tds_create_qp} to create an equivalent state-space representation. 
\begin{lstlisting}
omega = 2; k = 3; tau = 0.1; 
qp = tds_create_qp([1 0 omega^2;0 0 -k],[0 tau]);
\end{lstlisting}
This function takes two input arguments. More specifically, the first argument should be a matrix defining the coefficients of the quasi-polynomial, while the second arguments gives the delays. For more details, we refer the interested reader to \Cref{subsec:quasi_polynomial}. Next, we inspect the resulting \matlabfun{tds_ss_retarded}-object in the \matlab{} command line.
\begin{lstlisting}
qp = 

LTI Retarded Delay Differential Equation with properties:

    A: {[2x2 double]  [2x2 double]}
   hA: [0 0.1000]
   mA: 2
    n: 2
\end{lstlisting}
We see that the obtained state-space representation is given by
\[
\begin{bmatrix}
\dot{x}_1(t) \\
\dot{x}_2(t)
\end{bmatrix} = \begin{bmatrix}
0 & 1 \\ -\omega^2 & 0
\end{bmatrix} \begin{bmatrix}
x_1(t) \\
x_2(t)
\end{bmatrix}
+
\begin{bmatrix}
0 & 0 \\ k & 0
\end{bmatrix}
\begin{bmatrix}
x_1(t-0.1) \\ x_2(t-0.1)
\end{bmatrix}.
\]
Finally, we use the function \matlabfun{tds_roots} to compute the characteristic roots with $\Re(\lambda)\geq -2$ and to check whether the considered system is exponentially stable. 
\begin{lstlisting}
z = tds_roots(qp,-2);
tds_eigenplot(z)
\end{lstlisting}
\end{example}
Next, we will consider neutral delay differential equations. We will see that the stability analysis of such systems is more involved than in the retarded case. For example, \Cref{prop:finite_number} is no longer true, meaning that a right half-plane might contain infinitely many characteristic roots. Furthermore, we will see that for NDDE with multiple delays stability might be sensitive to arbitrarily small perturbations on the delays.
\section{Delay differential equations of neutral type}
\label{sec:neutral_stability}
In this section we will consider LTI delay differential equations of neutral type with point-wise delays. Such systems can be written in the following standard form:
\begin{equation}
\label{eq:ol_neutral}
\dot{x}(t) = \sum_{k=1}^{m_A} A_k x(t-h_{A,k}) - \sum_{k=1}^{m_H} H_k \dot{x}(t-h_{H,k})
\end{equation}
in which $x(t) \in \C^{n}$  is the state variable at time $t$, the delays $h_{A,1},\dots,h_{A,m_A}$ are non-negative and the delays $h_{H,1},\dots,h_{H,m_H}$ are positive. The state matrices $A_1, \dots, A_{m_A}$ and $H_1,\dots, H_{m_H}$ belong to $\R^{n\times n}$. Compared to the RDDE in \eqref{eq:ol_ret}, the evolution of the state variable $x$ at time $t$ now also depends on its time-derivative at discrete instances in the past. 

As in the previous section, we can associate an initial value problem with \eqref{eq:ol_neutral} by considering function segments in the Banach space $X$. (Note that the maximal delay $\taum$ is now equal to $\max\{h_{A,1},\dots,h_{A,m_A},h_{H,1},\dots,h_{H,m_H}\}$.) Given such an initial function segment, it can be shown that \eqref{eq:ol_neutral} has a unique forward solution (assuming the initial function segment satisfies a certain differentiability condition). For more details we refer the interested reader to \cite[Chapter 1.2]{bookdelay} or \cite[Chapter 1.7]{hale1993}. 

Now, let us again focus on the exponential stability of (the null solution of) \eqref{eq:ol_neutral}. (The definition of this concept carries over from \Cref{def:exp_stability}.) As before, the associated characteristic equation can be obtained by plugging the sample solution $t\mapsto v e^{\lambda t}$ into the differential equation:
\begin{equation}
\label{eq:NDEVP}
\det\Bigg(\lambda \left(\mathrm{I}_n+\sum_{k=1}^{m_H} H_k e^{-\lambda h_{H,k}}\right) - \sum_{k=1}^{m_A} A_k e^{-\lambda h_{A,k}} \Bigg) = 0.
\end{equation}

The spectrum of \eqref{eq:ol_neutral} is thus given by:
\begin{multline*}
\Lambda(A_1,\dots,A_{m_A},H_{1},\dots,H_{m_H},\vec{\tau}) =\\ \textstyle \left\{\lambda\in\C: \det\left(\lambda \left(\mathrm{I}_n+\sum_{k=1}^{m_H} H_k e^{-\lambda h_{H,k}}\right) - \sum_{k=1}^{m_A} A_k e^{-\lambda h_{A,k}} \right) = 0 \right\}.
\end{multline*}

As illustrated below, \Cref{prop:finite_number} does not carry over from the retarded case, in the sense that a right half-plane might contain infinitely many charactersitic roots. 
\begin{example}
	\label{example:neutral1}
	Consider the following NDDE 
	\begin{equation}
	\label{eq:ex_neutral1}
	\dot{x}(t) 
	= \begin{bmatrix}
	-0.6 & -0.45 \\
	0.1 & -1.2
	\end{bmatrix} x(t) + \begin{bmatrix}
	-0.15 & 0.075\\
	0.225 & -0.75
	\end{bmatrix} x(t-1)
	-
	\begin{bmatrix}
	3 & -1.5 \\
	2.5 & -1
	\end{bmatrix}\dot{x}(t-1).
	\end{equation}
To represent this system, we use the function \verb|tds_create_neutral(H,hH,A,hA)|. Similarly as for \verb|tds_create|, the system matrices (\verb|H| and \verb|A|) should be passed as cell arrays, while the delays (\verb|hH| and \verb|hA|) should be passed as vectors. 
\begin{lstlisting}
H = [3 -1.5;2.5 -1];
A0 = [-0.6 -0.45;0.1 -1.2]; A1 = [-0.15 0.075;0.225 -0.75];
tau = 1; 
ndde = tds_create_neutral({H},[tau],{A0,A1},[0 tau]);
\end{lstlisting}
The output argument \matlabfun{ndde} is now a \matlabfun{tds_ss_neutral} object. We can again use the \matlab{} Command Window to inspect this object. 
\begin{lstlisting}
ndde = 

LTI Neutral Delay Differential Equation with properties:

    H: {[2x2 double]}
   hH: 1
   mH: 1
    A: {[2x2 double]  [2x2 double]}
   hA: [0 1]
   mA: 2
    n: 2
\end{lstlisting}
Next, we compute the characteristic roots of \eqref{example:neutral1} in the rectangular region $[-3,1]\times\jmath[-60,60]$ using the function \matlabfun{tds_roots}. As before, this function takes as first argument the \matlabfun{tds_ss}-object from which we want to compute the characteristic roots. The second argument is however no longer a scalar indicating the desired right half-plane but rather a vector of length 4 containing \matlabfun{[realmin realmax imagmin imagmax]}. 
\begin{lstlisting}
[crn] = tds_roots(ndde,[-3 1 -60 60]);
tds_eigenplot(crn);
xlim([-3 1])
ylim([-60,60])
\end{lstlisting}
\Cref{fig:neutral1_roots} shows the result. Notice that the spectrum now contains two vertical chains, \ie, sets of characteristic roots whose real parts remain bounded, yet whose imaginary parts tend to infinity.
\begin{figure}[!ht]
	\centering
	\includegraphics[width=0.6\linewidth]{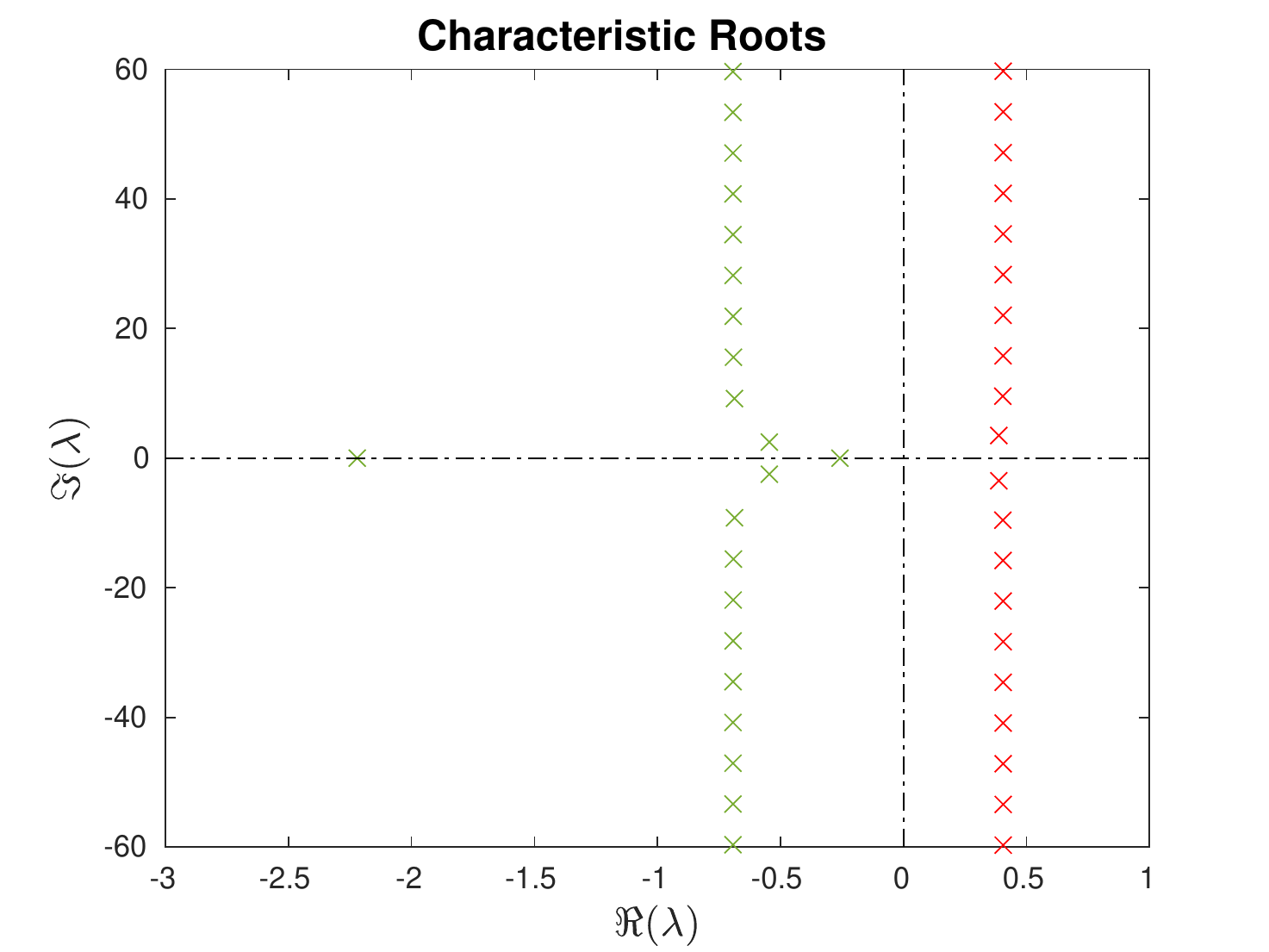}
	\caption{The characteristic roots of \eqref{eq:ex_neutral1} in the rectangular region $[-3,1]\times\jmath[-60,60]$.}
	\label{fig:neutral1_roots}
\end{figure}
\end{example}

As a consequence, the definition of the spectral abscissa must be modified as the maximum in \eqref{eq:sa_retarded} might not be attained.
\begin{definition}
\label{def:spectral_abscissa_neutral}
The spectral abscissa of \eqref{eq:ol_neutral} is equal to the supremum of the real part of its characteristic roots, or in other words,
\begin{equation*}
\sa(A_{1},\dots,A_{m_A},H_{1},\dots,H_{m_H},\vec{\tau}) =  \sup \big\{\Re(\lambda) : \lambda \in \Lambda(A_{1},\dots,A_{m_A},H_{1},\dots,H_{m_H},\vec{\tau}) \big\}.
\end{equation*}

\end{definition}
As for ODEs and RDDEs, the spectral abscissa can be used to asses the (global) exponential stability of \eqref{eq:ol_neutral}. More specifically, the null solution of \eqref{eq:ol_neutral} is exponentially stable if and only if its spectral abscissa is strictly negative. However, where for ODEs and RDDEs, the null solution is exponentially stable if and only if all characteristic roots lie in the open left half-plane, this last condition is not sufficient for NDDEs. There exist neutral time-delay systems for which all characteristic roots lie to the left of the imaginary axis, yet whose null solution is not exponentially stable \cite{verriest2006}. It can however be shown that a NDDE is exponentially stable if (and only if) its characteristic roots lie in the open left half-plane bounded away from the imaginary axis, \ie the spectrum can not contain a vertical chain of characteristic roots approaching the imaginary axis (from the left).

It should be clear that the stability analysis of a NDDE is more complex than for a RDDE. We will therefore first consider the single delay case. Afterwards, we will consider neutral systems with multiple delays and focus on the effect of infinitesimal delay perturbations.

\subsection{Single delay}
Let us first consider the case in which the NDDE contains only one unique delay. In this case \eqref{eq:ol_neutral}
reduces to
\begin{equation}
\label{eq:neutral_1delay}
\dot{x}(t) + H_1 \dot{x}(t-\tau) = A_0 x(t) + A_1 x(t-\tau)
\end{equation}
with $\tau$ a positive delay. As we will see below, the following associated delay difference equation
\begin{equation}
\label{eq:neutral_diff_1d}
x(t) + H_1 x(t-\tau) = 0.
\end{equation}
and the corresponding characteristic equation
\begin{equation*}
\label{eq:neutral_diff_evp_1d}
\det\left(\mathrm{I}_n + H_1 e^{-\lambda \tau}\right) = 0
\end{equation*}
play an imortant role in the stability analysis of \eqref{eq:neutral_1delay}. 
\begin{example}
\label{example:neutral1b}
Consider again the NDDE \eqref{eq:ex_neutral1}. The associated delay difference equation of the form \eqref{eq:neutral_diff_1d} is given by
\begin{equation}
\label{eq:difference1}
x(t) 
+
\begin{bmatrix}
3 & -1.5 \\
2.5 & -1
\end{bmatrix}x(t-1) = 0.
\end{equation}
We use the function \matlabfun{get_delay_difference_equation} to obtain a representation for this delay difference equation from the \matlabfun{ndde}-variable that we created before.
\begin{lstlisting}
delay_diff = get_delay_difference_equation(ndde);
\end{lstlisting}
The result is now a \matlabfun{tds_ss_ddae}-object (see below)
\begin{lstlisting}
delay_diff = 

LTI Delay Difference Equation with properties:

    E: [2x2 double]
    A: {[2x2 double]  [2x2 double]}
   hA: [0 1]
   mA: 2
    n: 2
\end{lstlisting}
with
\[
\texttt{E} = \begin{bmatrix}
0 & 0 \\ 0 & 0
\end{bmatrix}\text{, } \texttt{A{1}} = \begin{bmatrix}
1 & 0 \\
0 & 1
\end{bmatrix} \text{, and }  \texttt{A{2}} = \begin{bmatrix}
3 & -1.5 \\
2.5 & -1
\end{bmatrix}
\]
Note that the first term in \eqref{eq:difference1}, the term associated with $x(t)$, is explicitly represented in \matlabfun{delay_diff} via \verb|A{1}|. Next we analytically compute the characteristic roots of \eqref{eq:difference1}. To this end note that
\begin{align*}
\det\left(\mathrm{I}_{n} + H_1 e^{-\lambda} \right) &= \det\left(\mathrm{I}_n + Z V Z^{H} e^{-\lambda \tau}\right) \\
&= \det(Z) \det\left(\mathrm{I}_n + V e^{-\lambda \tau}\right) \det\left(Z^{H}\right),
\end{align*}
with $Z V Z^{H}$ the Schur decomposition of $H_1$, \ie $Z$ is an orthogonal matrix and $V$ is an upper-triangular matrix whose diagonal elements are given by the eigenvalues of $H_1$ (that, in this case, are equal to 1.5 and 0.5). The characteristic equations of \eqref{eq:difference1} can thus be reduced to 
\[
\left(1+1.5 e^{-\lambda h}\right)\left(1+0.5 e^{-\lambda h}\right) = 0.
\]
meaning that the corresponding characteristic roots are given by
\begin{equation}
\label{eq:difference_roots}
\frac{1}{\tau} \big(\ln{1.5}+\jmath (2l+1)\pi \big) \text{ and } \frac{1}{\tau} \big(\ln{0.5}+\jmath (2l+1)\pi \big) \text{ for }l\in \Z
\end{equation}
with $\ln 1.5 \approx 0.405$ and $\ln 0.5 \approx -0.693$.
Let us now verify this result by computing the characteristic roots of \eqref{eq:difference1} using \packageName{}.

\begin{lstlisting}
[crd] = tds_roots(delay_diff,[-3 1 -60 60]);
tds_eigenplot(crd);
xlim([-3 1])
ylim([-60,60])
\end{lstlisting}
\Cref{fig:neutral1_difference} shows that the computed characteristic roots indeed correspond to \eqref{eq:difference_roots}. Furthermore notice that the location of vertical chains in \Cref{fig:neutral1_roots} correspond to those in \Cref{fig:neutral1_difference}.   

\begin{figure}
    \centering
    \includegraphics[width=0.6\linewidth]{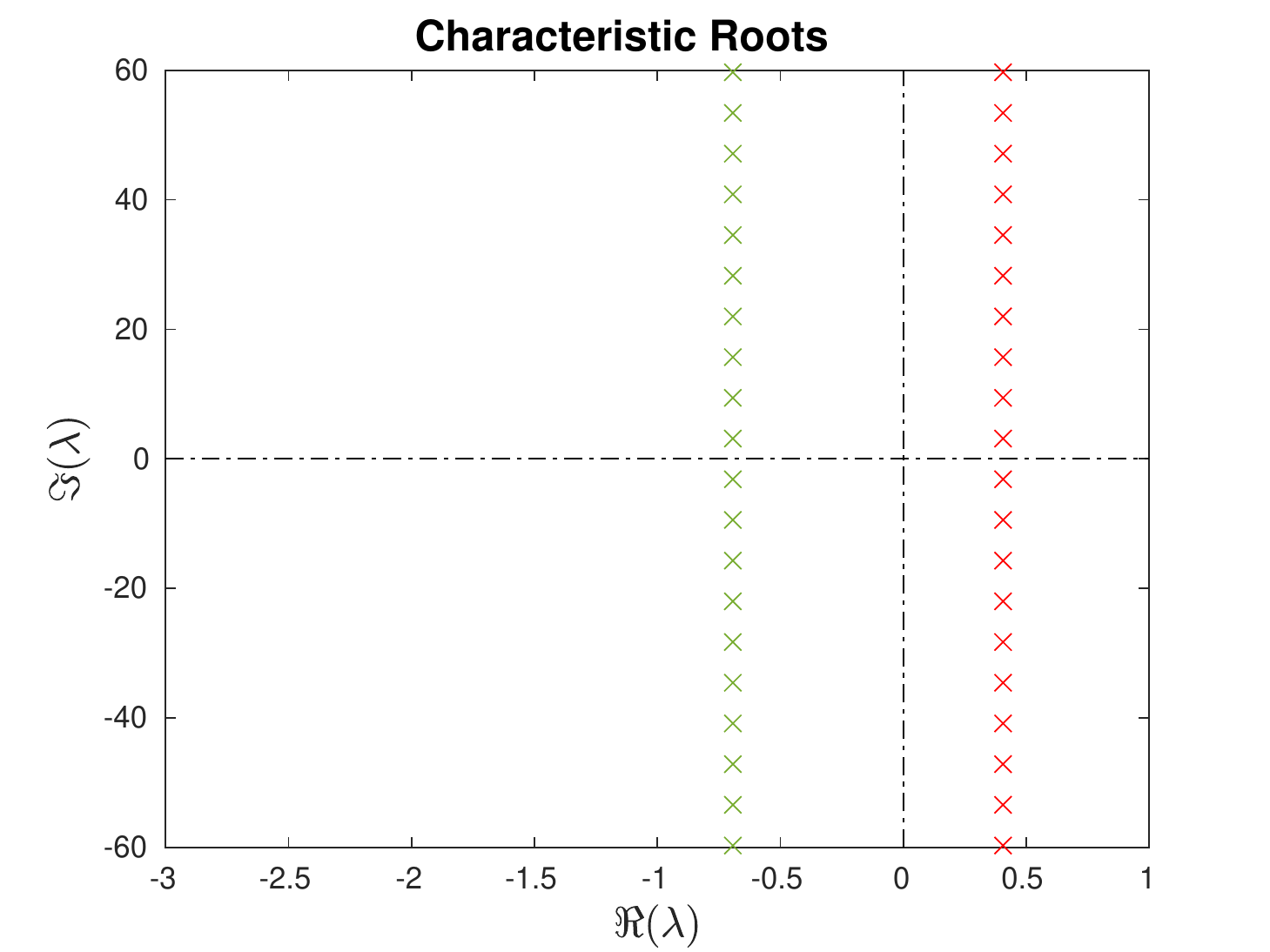}
    \caption{The characteristic roots of \eqref{eq:difference1} in the region $[-3,1]\times\jmath[-60,60]$.}
    \label{fig:neutral1_difference}
\end{figure}
\end{example}

This relation between the characteristic roots of a NDDE and the characteristic roots of its associated delay difference equation, given by
\begin{equation}
\label{eq:neutral_difference}
x(t) + \sum_{k=1}^{m_H} H_k x(t-h_{H,k}) = 0,
\end{equation}
can be formalized as in the following proposition.
\begin{proposition}
\label{prop:neutral_chain}
Consider the set
\begin{equation*}
\textstyle Z_D := \left\{\Re(\lambda) : \det\big(\mathrm{I}_n + \sum_{k=1}^{m_H} H_{k} e^{-\lambda h_{H,k}} \big) = 0 \right\},
\end{equation*}
then for any $\zeta\in \clos(Z_D)$ there exists a sequence of characteristic roots of the NDDE \eqref{eq:ol_neutral}, $\{\lambda_l\}_{l\geq1}$, that satisfies
\begin{equation}
\label{eq:ev_chain}
\lim_{l\rightarrow \infty} \Re(\lambda_l) = \zeta \text{ and } \lim_{l \rightarrow \infty} |\Im(\lambda_l)| = \infty.
\end{equation}
\end{proposition}
\begin{proof}
	See \cite[Proposition~1.28]{bookdelay}. Intuitively this result can be understood as follows. For $\lambda \neq 0$, the characteristic equation \eqref{eq:NDEVP} is equivalent with
	\[
	\det \left(\left(I + \sum_{k=1}^{m_H} H_k e^{-\lambda h_{H,k}} \right)- \dfrac{1}{\lambda} \left( \sum_{k=1}^{m_A} A_k e^{-\lambda h_{A,k}} \right)\right) = 0
	\]
	If $|\lambda| \gg \Re(\lambda)$, the equation above reduces to 
	\begin{equation}
	\label{eq:neutral_diff_evp}
	\det \left(\mathrm{I}_{n} + \sum\limits_{k=1}^{m_H} H_k e^{-\lambda h_{H,k}} \right) = 0.
	\end{equation} The characteristic roots of \eqref{eq:ol_neutral} with large modulus but small real part, must therefore lie close to the characteristic roots of \eqref{eq:neutral_diff_evp}. 
\end{proof}
In conclusion, in contrast to the RDDE case, the spectrum of a NDDE contains vertical chains whose location is determined by the spectrum of its associated delay difference equation. Next, we will focus on NDDE with multiple delays. We will see that for such systems the stability analysis is even more involved.
\subsection{Multiple delays}
\label{subsec:neutral_multiple_delays}
Next, we will consider NDDEs with multiple delays. In the following example we will see that the exponential stability of such systems might be fragile with respect to delay perturbations, in the sense that a NDDE with a strictly negative spectral abscissa can be destabilized by an arbitrarily small perturbation on the delays. 
\begin{example}
\label{example:neutral_2}
Consider the following NDDE
\begin{equation}
\label{eq:neutral_example2}
\dot{x}(t) = \frac{1}{4} x(t) - \frac{1}{3} x(t-\tau_1) + \frac{3}{4} \dot{x}(t-\tau_1) - \frac{1}{2} \dot{x}(t-\tau_2)
\end{equation}
with $\tau_1 = 1$ and $\tau_2 = 2$. The associated delay difference equation is by
\begin{equation}
\label{eq:difference_example2}
x(t)-\frac{3}{4}x(t-\tau_1) + \frac{1}{2}x(t-\tau_2) = 0.
\end{equation}
which (for $\tau_1 = 1$ and $\tau_2 = 2$) has roots at
\[
-\left(\ln\sqrt{2} \pm \jmath \left( \arctan\left(\sqrt{23}/3\right) + 2\pi l \right) \right) \text{ for } l\in \Z
\]
with $\ln\sqrt{2} \approx -0.34$ (for more details see \cite[Example 1.19]{bookdelay}\footnote{Note there is a $\sqrt{\cdot}$ missing in Eq. (1.43).}). Let us now use \packageName{} to compute the characteristic roots of \eqref{eq:neutral_example2}. As for the single delay case, we use the function \matlabfun{tds_create_neutral} to create a representation for \eqref{eq:neutral_example2}. Next, we use \matlabfun{tds_roots} to compute the characteristic roots in the rectangle $[-0.9,0.2]\times\jmath[-500,500]$. As we consider a large region in the complex plane, we increase the \matlabfun{max_size_evp} to 1500.
\begin{lstlisting}
ndde = tds_create_neutral({-3/4,1/2},[1 2],{1/4,-1/3},[0 1]);

cD = -log(sqrt(2));

options=tds_roots_options('max_size_evp',1500);
cr = tds_roots(ndde,[-0.9 0.2 -500 500],options);
tds_eigenplot(cr,'Markersize',4);
hold on
plot([cD cD],ylim,'b--')
xlim([-0.9 0.2])
ylim([-500,500])
\end{lstlisting}
\Cref{fig:neutral2_roots} shows the obtained characteristic roots which includes a sequence of the form \eqref{eq:ev_chain} with $\zeta=-\ln\sqrt{2}$. Because the spectral abscissa is strictly negative, we conclude that \eqref{eq:neutral_example2} is exponentially stable for $\tau_1 = 1$ and $\tau_2=2$. 
\begin{figure}[!h]
    \centering
    \includegraphics[width=0.6\linewidth]{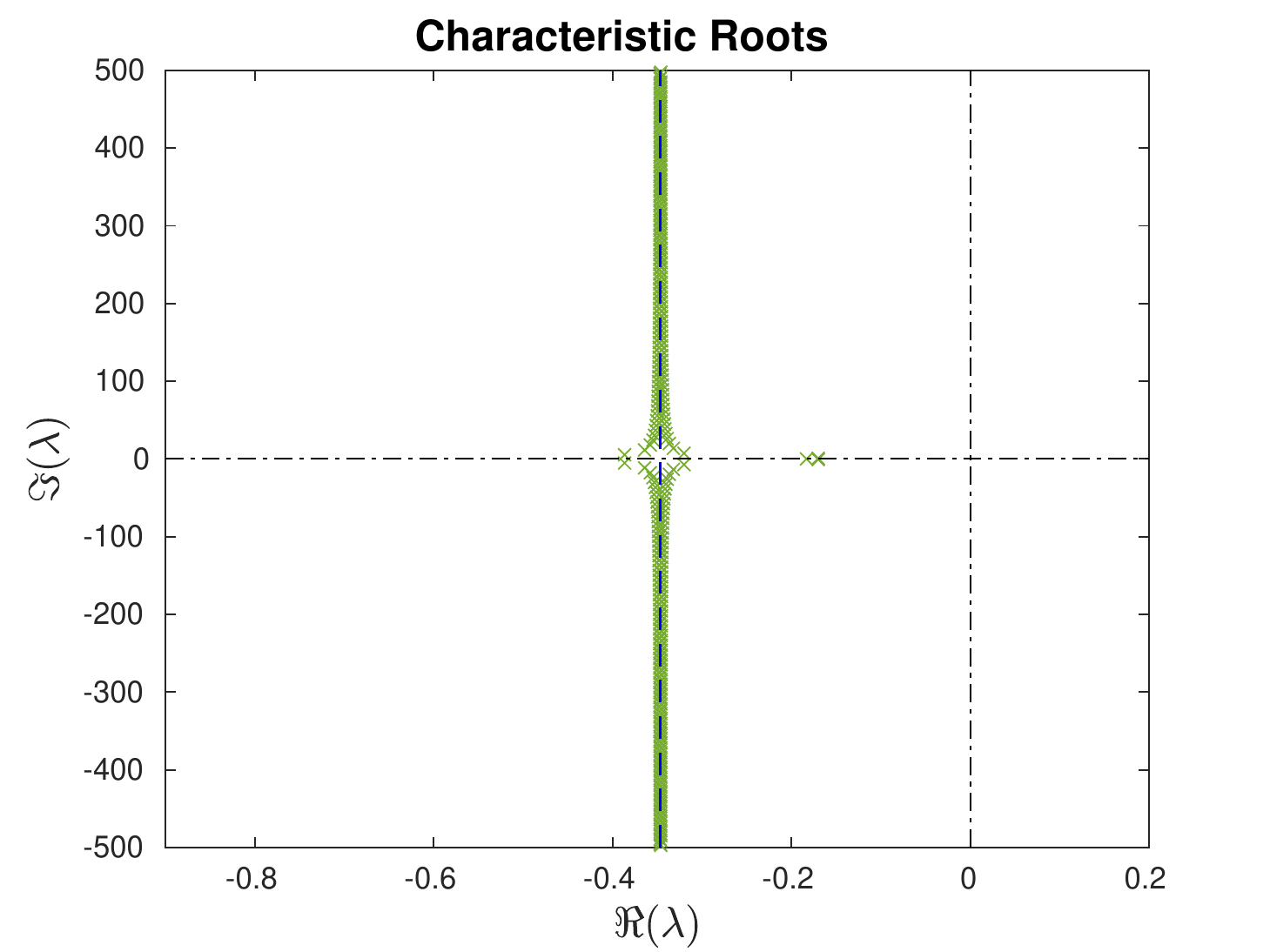}
    \caption{Characteristic roots of \eqref{eq:neutral_example2} for $\tau_1=1$ and $\tau_2 = 2$ in the region $[-0.9,0.2]\times\jmath[-500,500]$.}
    \label{fig:neutral2_roots}
\end{figure}
Let us now examine the effect of a small perturbation to the second delay, by taking \eg $\tau_2=2.05$. 
\begin{lstlisting}
ndde.hH(2) = 2.05; % Change second delay

cr2 = tds_roots(ndde,[-0.9 0.2 -500 500],options);
tds_eigenplot(cr2,'Markersize',4);
hold on
plot([cD cD],ylim,'b--')
xlim([-0.9 0.2])
ylim([-500,500])
\end{lstlisting}
Taking a look at \Cref{fig:neutral2_roots2}, we see that the spectrum has changed significantly. Furthermore, as the spectral abscissa is no longer negative, \eqref{eq:neutral_example2} is no longer exponentially stable. 
\begin{figure}[!h]
    \centering
    \includegraphics[width=0.6\linewidth]{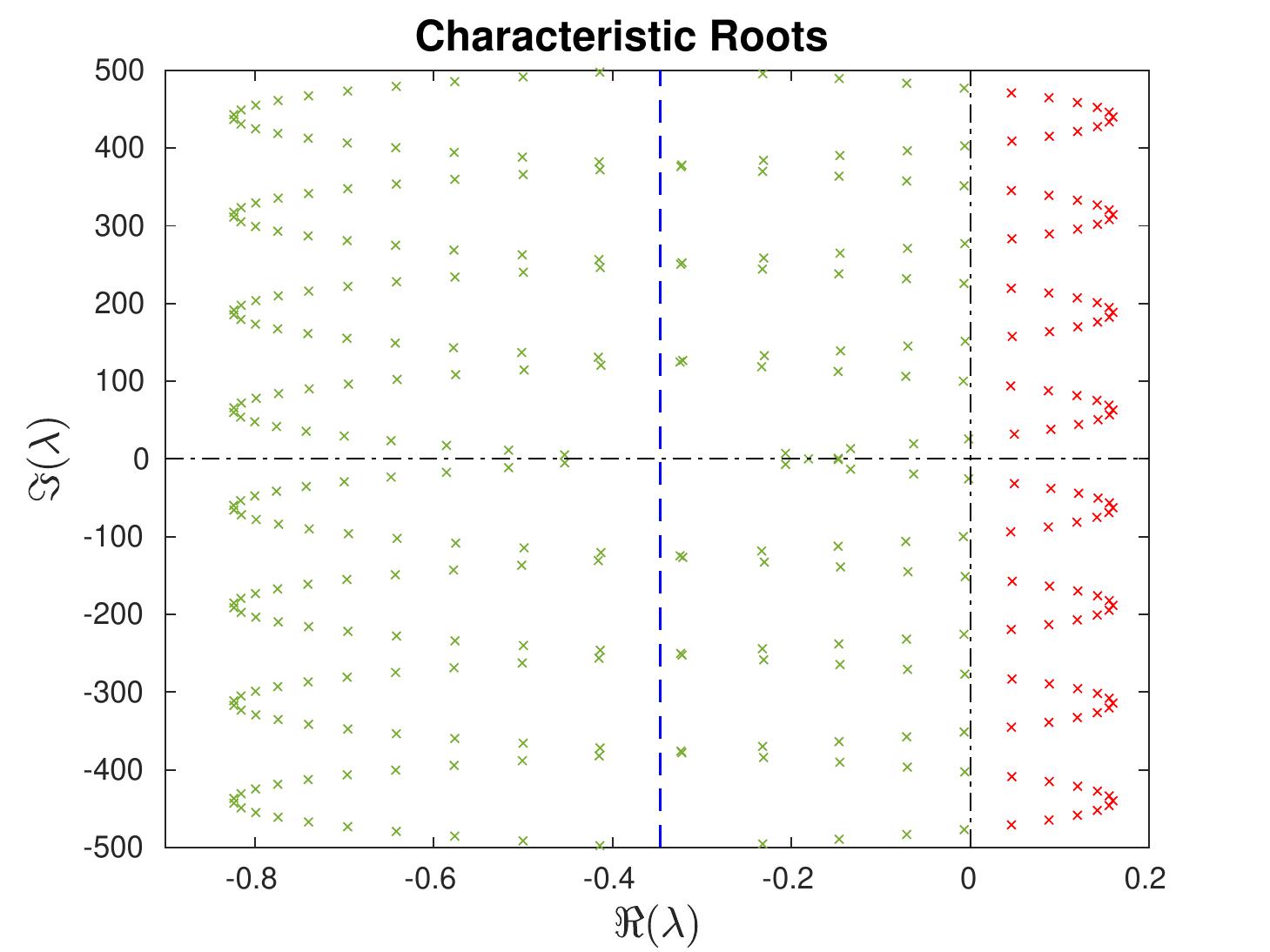}
    \caption{Characteristic roots of \eqref{eq:neutral_example2} for $\tau_1=1$ and $\tau_2 = 2.05$ in the region $[-0.9,0.2]\times\jmath[-500,500]$.}
    \label{fig:neutral2_roots2}
\end{figure}

Even if we further decrease the size of the delay perturbation, by taking \eg $\tau_2=2.005$, the system remains unstable as demonstrated in \Cref{fig:neutral2_roots3} which can be generated using the code below. 
\begin{lstlisting}
ndde.hH(2) = 2.005; % Change second delay

cr3 = tds_roots(ndde,[-0.9 0.2 -500 500],options);
tds_eigenplot(cr3,'Markersize',4);
hold on
plot([cD cD],ylim,'b--')
xlim([-0.9 0.2])
ylim([-500,500])
\end{lstlisting}
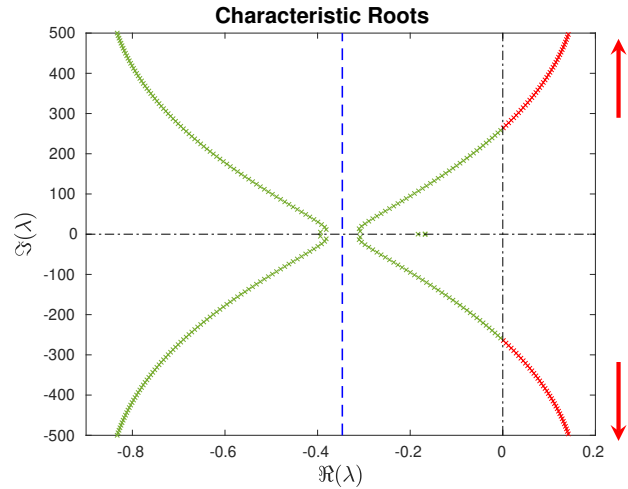
\begin{figure}[!h]
    \centering
    \begin{tikzpicture}[      
        every node/.style={anchor=south west,inner sep=0pt},
        x=1mm, y=1mm,
      ]   
     \node (fig1) at (0,0)
       { \includegraphics[width=0.6\linewidth]{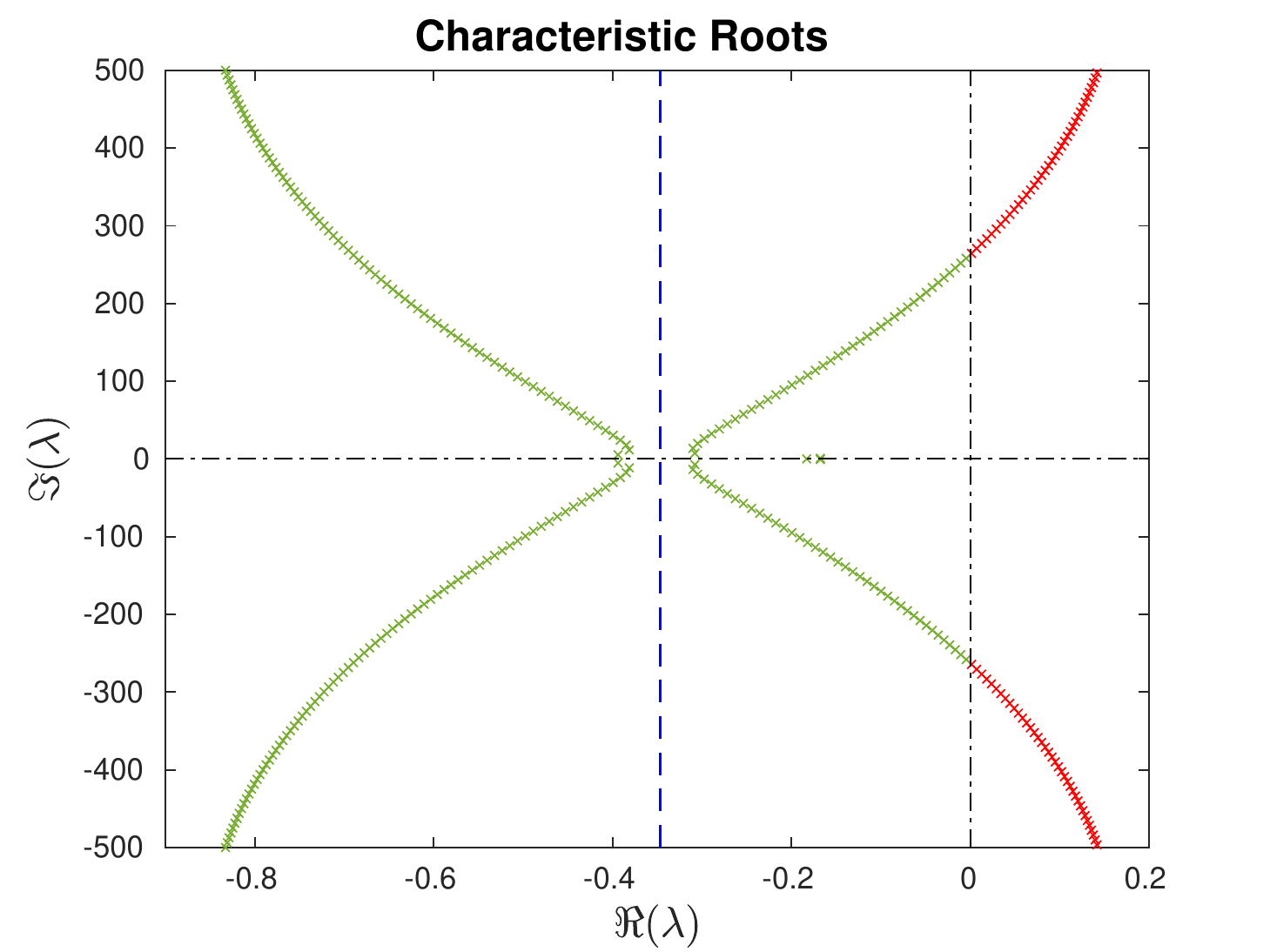}};
    \node[] (tt) at (82,60) {};
    \node[] (bt) at (82,49) {};
    \node[] (tb) at (82,17) {};
    \node[] (bb) at (82,6) {};
    \draw[-stealth,ultra thick,red] (bt) -- (tt);
    \draw[-stealth,ultra thick,red] (tb) -- (bb);
\end{tikzpicture}
    \caption{Characteristic roots of \eqref{eq:neutral_example2} for $\tau_1=1$ and $\tau_2 = 2.005$ in the region $[-0.9,0.2]\times\jmath[-500,500]$.}
    \label{fig:neutral2_roots3}
\end{figure}
Note that the destabilization mechanism that we observed above can be related to a similar phenomenon for the underlying delay difference equation (see \Cref{fig:neutral2_difference3}).
\pagebreak
\begin{lstlisting}
delay_diff = get_delay_difference_equation(ndde);

cr4 = tds_roots(delay_diff,[-0.9 0.2 -500 500],options);
tds_eigenplot(cr4,'Markersize',4);
hold on
plot([cD cD],ylim,'b--')
xlim([-0.9 0.2])
ylim([-500,500])
\end{lstlisting}
\begin{figure}[!h]
	\centering
	\includegraphics[width=0.6\linewidth]{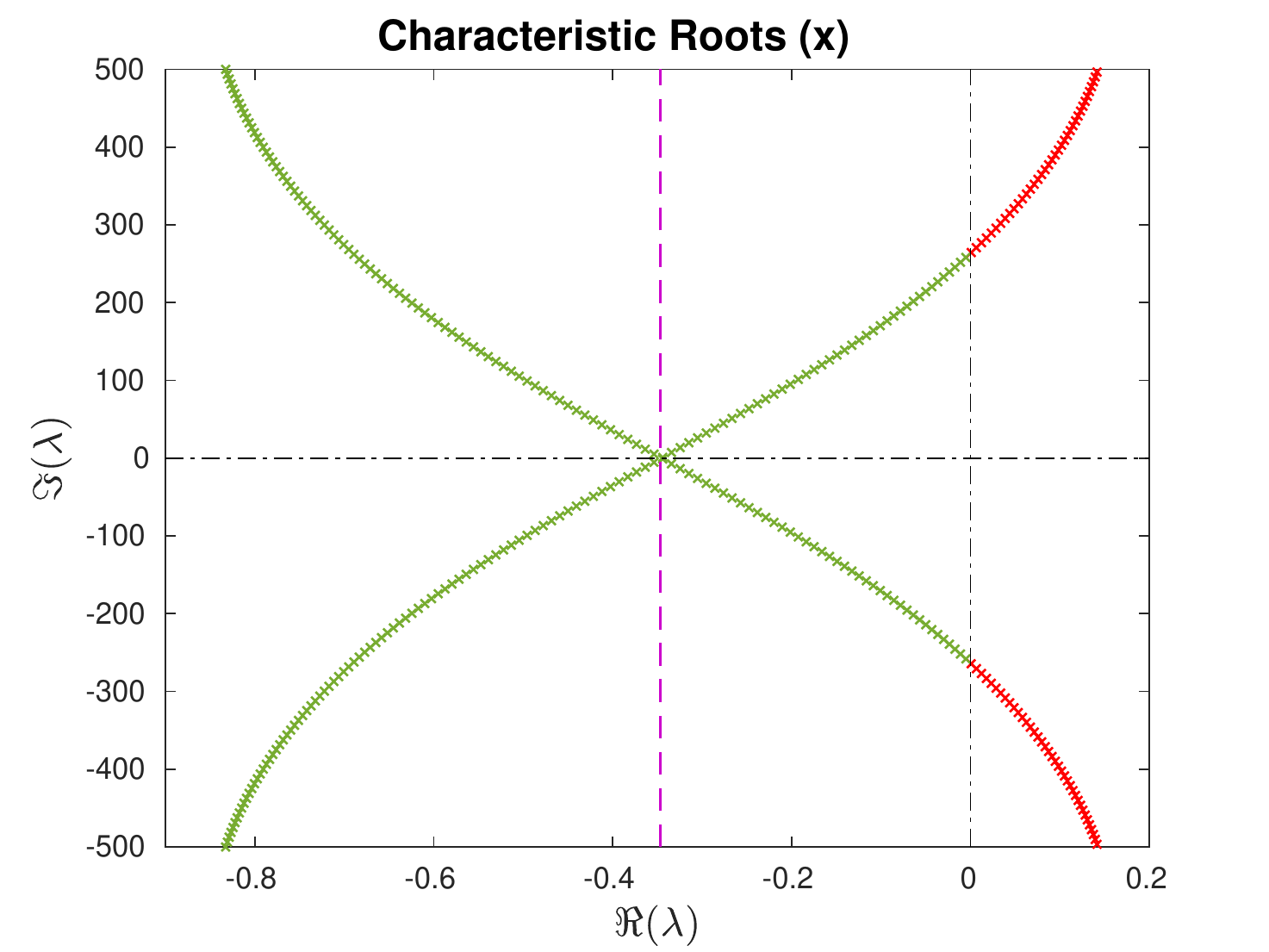}
	\caption{Characteristic roots of \eqref{eq:difference_example2} for $\tau_1=1$ and $\tau_2 = 2.005$ in the region $[-0.9,0.2]\times\jmath[-500,500]$.}
	\label{fig:neutral2_difference3}
\end{figure}
\end{example}
\begin{remark}
	\label{remark:sensitivity}
	It can be shown that the NDDE \eqref{eq:neutral_example2} is unstable for certain delay values arbitrary close to $\tau_2 = 2$. Note that this observation is not in contradiction with the continuous variation of the individual characteristic roots with respect to changes on both the entries of the system matrices and the delay values. If follows from \Cref{fig:neutral2_roots2,fig:neutral2_roots3} that the larger the frequency (\ie the absolute value of the imaginary part) of a characteristic root, the larger is its sensitivity (\ie its derivative) with respect to delay changes. In fact, this derivative grows unbounded along the characteristic root chain.  If the size of the delay perturbation is decreased, the characteristic roots that shift to the right half-plane must be more sensitive and thus have a larger imaginary part (in absolute value). If the size of the delay perturbation goes to zero, the imaginary part of the destabilizing roots must go to infinity (as indicated by the arrows in \Cref{fig:neutral2_roots3}).
\end{remark}

The exponential stability of a NDDE might thus be fragile with respect to the delays, in the sense that there might exist arbitrarily small destabilizing delay perturbations. Let us therefore consider the notion of \emph{\textbf{strong exponential stability}} as introduced in \cite{hale2002strong}. 
\begin{definition}
    \label{def:strong_stability}
    The null solution of \eqref{eq:ol_neutral} is strongly exponentially stable if it is exponentially stable and remains stable when subjected to sufficiently small variations on the delays. More formally, the null solution of \eqref{eq:ol_neutral} is strongly exponentially stable if there exists an $\epsilon>0$ such that exponentially stable is preserved for all delay values $\vec{\tau}_{\epsilon} \in \mathcal{B}(\vec{\tau},\epsilon) \cap \R_{+}^{m}$ in which $\mathcal{B}(\vec{\tau},\epsilon)$ is an open ball of radius $\epsilon>0$ centered at $\vec{\tau}$ and $m$ is the length of $\vec{\tau}$. In other words, strong exponential stability requires a non-zero delay margin. 
\end{definition}
To asses strong exponential stability, we will consider the strong spectral abscissa (for both \eqref{eq:ol_neutral} and \eqref{eq:neutral_difference}), which is defined as the smallest upper bound for the spectral abscissa that is insensitve to small delay changes. 
\begin{definition}[{\cite[Definition~1.31]{bookdelay}}]
\label{def:strong_spectral_abscissa}
The strong spectral abscissa of \eqref{eq:ol_neutral} is defined as follows
\begin{multline*}
C(A_{1},\dots,A_{m_A},H_{1},\dots,H_{m_H},\vec{\tau}) := \\ \lim_{\epsilon\searrow 0+} \sup \left\{\sa(A_{1},\dots,A_{m_A},H_{1},\dots,H_{m_H},\vec{\tau}_{\epsilon}): \text{ for }  \vec{\tau}_{\epsilon} \in \mathcal{B}(\vec{\tau},\epsilon) \cap \R_{+}^{m}  \right\}.
\end{multline*}
Similarly, the strong spectral abscissa of \eqref{eq:neutral_difference} is defined as
\begin{equation*}
C_{D}(H_{1},\dots,H_{m_H},\vec{\tau}) :=  \lim_{\epsilon\searrow 0+} \sup \left\{\sa_D(H_{1},\dots,H_{m_H},\vec{\tau}_{\epsilon}) \text{ for } \vec{\tau}_{\epsilon} \in \mathcal{B}(\vec{\tau},\epsilon) \cap \R_{+}^{m_H} \right\}.
\end{equation*}  
\end{definition}
In contrast to the nominal spectral abscissa, the strong spectral abscissa is continuous with respect to both the entries of the state matrices and the delays \cite[Theorem~1.44]{bookdelay}. The following proposition now gives a necessary and sufficient condition for strong exponential stability.
\begin{proposition}
\label{proposition:strong_stability_alt}
The null solution of \eqref{eq:ol_neutral} is strongly exponentially stable if and only if its strong spectral abscissa is strictly negative. The same holds for \eqref{eq:neutral_difference}.
\end{proposition}

Next, we present an alternative condition for strong exponential stability, based on \cite[Propisitons~1.30 and 1.43, and Theorem~1.32]{bookdelay}, which will be useful in \Cref{sec:stabilization}.
\begin{proposition}{\ }
\label{proposition:strong_stability}
\begin{itemize}
	\item The null solution of the delay difference equation \eqref{eq:neutral_difference} is strongly exponentially stable if and only if the inequality
	\[
	\gamma_0(H_1,\dots,H_{m_H},\vec{\tau}) := \max_{\vec{\theta} \in [0,2\pi)^{m_H}} \rho\left(\textstyle\sum\limits_{k=1}^{m_H} H_k e^{\jmath \theta_k}\right) < 1
	\]
	holds.
	\item  The null solution of the NDDE \eqref{eq:ol_neutral} is strongly exponentially stable if and only if
	\begin{itemize}
		\item the inequality $\gamma_0(H_1,\dots,H_{m_H},\vec{\tau})<1$ holds, and
		\item the spectral abscissa, $c(A_{1},\dots,A_{m_A},H_{1},\dots,H_{m_H},\vec{\tau})$ is strictly negative.
	\end{itemize}
\end{itemize}

\end{proposition}
\begin{remark}
 The value of $\gamma_0(H_1,\dots,H_{m_H},\vec{\tau})$ does not depend on the delays. As a consequence, the strong exponential stability of the delay difference equation \eqref{eq:neutral_difference} is independent of the delay values.
\end{remark}
\begin{remark} 
The proposition above implies that strong exponential stability of the underlying delay difference equation is a necessary condition for the exponential stability of a NDDE.
\end{remark}

Finally, we will briefly describes how the strong spectral abscissa of \eqref{eq:ol_neutral} and \eqref{eq:neutral_difference} can be computed. We start with a mathematical characterization of these quantities, amenable from a computation perspective, based on \cite[Theorem~1.32 and Proposition~1.43]{bookdelay}.
\begin{proposition}
\label{proposition:strong_spectral_abscissa}
The strong spectral abscissa of the delay difference equation \eqref{eq:neutral_difference} corresponds to the unique zero crossing of the monotonously decreasing function
\begin{equation}
	\label{eq:CD_zero_crossing}
    r \mapsto \gamma(r;H_1,\dots,H_{m_H},\vec{\tau}) - 1 
\end{equation}
with
\begin{equation}
    \label{eq:gamma_r}
    \gamma(r;H_1,\dots,H_{m_H},\vec{\tau}) := \max_{\vec{\theta} \in [0,2\pi)^{m_H}} \rho\left(\textstyle\sum\limits_{k=1}^{m_H} H_k e^{-r h_{H,k}} e^{\jmath \theta_k}\right).
\end{equation}
The strong spectral abscissa of the NDDE in \eqref{eq:ol_neutral} is given by
\begin{equation}
\begin{multlined}
\label{eq:CN_expression}
C\left(A_{1},\dots,A_{m_A},H_{1},\dots,H_{m_H},\vec{\tau}\right) = \\\qquad \qquad \qquad \max \big\{c(A_{1},\dots,A_{m_A},H_{1},\dots,H_{m_H},\vec{\tau}),\,C_D(H_{1},\dots,H_{m_H},\vec{\tau})\big\}.
\end{multlined}
\end{equation}
\end{proposition}
The strong spectral abscissa of a NDDE is thus equal to the maximum of its nominal spectral abscissa and the strong spectral abscissa of the underlying delay difference equation. Furthermore, it follows from \Cref{prop:neutral_chain}, that the spectrum of \eqref{eq:ol_neutral} contains a vertical root chain that approaches $C_D(H_1,\dots,H_{m_H},\vec{\tau})$ either for the nominal delay values or for a certain (arbitrarily small) delay perturbation. Next, let us consider the spectrum of \eqref{eq:ol_neutral} for $\Re(\lambda) > C_D(H_1,\dots,H_{m_H},\vec{\tau})$. The following result follows from \cite[Proposition~1.29]{bookdelay}.
\begin{proposition}
\label{prop:finite_number_roots}
For any $\epsilon>0$, the NDDE \eqref{eq:ol_neutral} has only a finite number of characteristic roots in the right half-plane
\[
\{\lambda \in \C : \Re(\lambda) \geq C_D(H_1,\dots,H_{m_H},\vec{\tau}) + \epsilon \}
\]
(even when infinitesimal delay perturbations are considered).
\end{proposition}
To compute the strong spectral abscissa of a NDDE one can thus use the following two-step approach. First, compute the strong spectral abscissa of the underlying delay difference equation by means of the characterization in \Cref{proposition:strong_spectral_abscissa}. Next, compute all characteristic roots of \eqref{eq:ol_neutral} in the right half-plane 
\[
\{\lambda \in \C : \Re(\lambda) \geq C_D(H_1,\dots,H_{m_H},\vec{\tau}) + \epsilon \}
\] for some sufficiently small $\epsilon>0$. In \packageName{}, this approach is implemented in the \matlabfun{tds_strong_sa}-function, as illustrated in the following example.

\begin{example}
\label{example:neutral_3}
Let us return to \Cref{example:neutral_2}. Recall that \eqref{eq:neutral_example2} was exponential stable for the nominal delays $\tau_1 = 1$ and $\tau_2 = 2$. However, for arbitrarily small perturbations on the second delay, the system became unstable. Let us therefore compute its strong spectral abscissa using \matlabfun{tds_strong_sa(tds,r)}. 
\begin{lstlisting}
>> C = tds_strong_sa(ndde,-0.2)

C =
    0.1614
\end{lstlisting}
As expected, the strong spectral abscissa is positive. Furthermore, the obtained value is close to the spectral abscissa in \Cref{fig:neutral2_roots2} for $\tau_1 = 1$ and $\tau_2 = 2.05$. Alternatively, we can compute $\gamma(0;H_{1},\dots,H_{m_H})$ using the function \matlabfun{tds_gamma_r(tds,r)}.
\begin{lstlisting}
>> gamma0 = tds_gamma_r(ndde,0)

gamma0 =

    1.2500
\end{lstlisting}
\end{example}
\noindent\textbf{Important:} If we want to compute the characteristic roots of a NDDE in a given right half-plane
\[
\{\lambda \in \C: \Re(\lambda) \geq r\}
\]
with $r$ is smaller than $C_D(H_1,\dots,H_{m_H},\vec{\tau})$, the following warning will be given.
\begin{lstlisting}[basicstyle=\ttfamily\color{orange},language=,breakindent=0pt, breaklines]
Warning (tds_roots): Case: gamma(r) >= 1 (i.e., CD>r).
Spectral discretization with N = 30 (lowered if maximum size of eigenvalue problem is exceeded).
\end{lstlisting}
The logic behind this warning is that if $r$ is smaller than $C_D(H_1,\dots,H_{m_H},\vec{\tau})$, the considered right half-plane contains a root chain of the form \eqref{eq:ev_chain} (or there exists an arbitrarily small perturbation on the delays for which it does). As such the function \matlabfun{tds_roots} can no longer determine the necessary degree for the spectral discretisation to capture all characteristic roots in the considered right half-plane. The function \matlabfun{tds_roots} will in that case use the default value $N = 30$.
\begin{example}
	\label{example:roots_NDDE_rhp}
Consider again the NDDE in \eqref{eq:neutral_example2}. Now we use the function \matlabfun{tds_roots} to compute the characteristic roots with real part larger than $-0.6$.
\begin{lstlisting}
[cr] = tds_roots(ndde,-0.6);
tds_eigenplot(cr)
\end{lstlisting} 
We get the aforementioned warning and observe that only a limited number of the characteristic roots of the vertical chain are found (see \Cref{fig:neutral2_roots4}). To capture more characteristic roots we manually increase the degree of the spectral discretization $N$ using the option \matlabfun{fix_N}. The result is shown in \Cref{fig:neutral2_roots5}. 
\begin{lstlisting}
options = tds_roots_options('fix_N',200);
[cr2] = tds_roots(ndde,-0.6,options);
tds_eigenplot(cr2)
\end{lstlisting}
\begin{figure}
	\centering
	\begin{subfigure}{0.49\linewidth}
		\includegraphics[width=\linewidth]{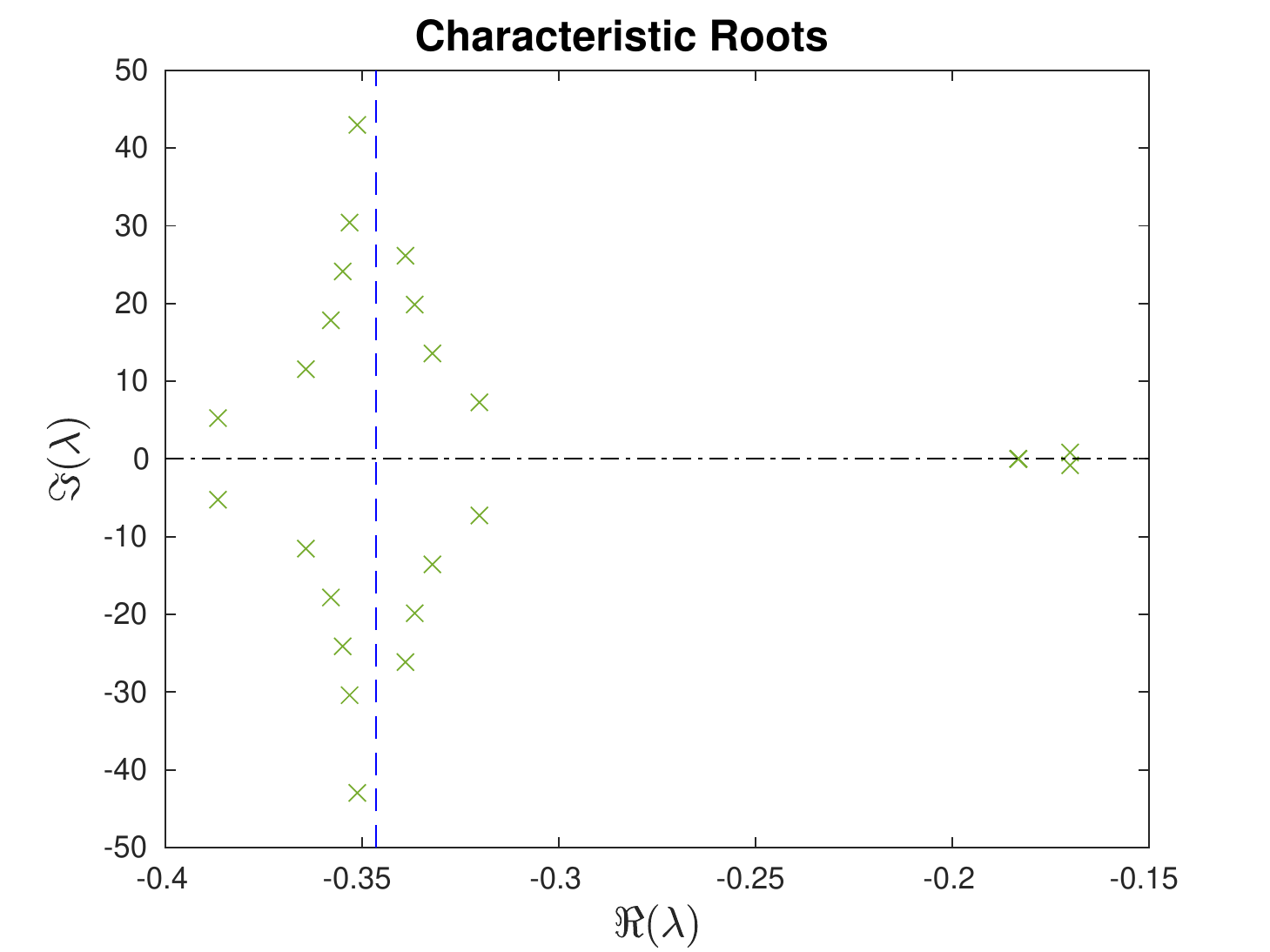}
		\caption{$N = 30$}
		\label{fig:neutral2_roots4}
	\end{subfigure}
	\begin{subfigure}{0.49\linewidth}
		\includegraphics[width=\linewidth]{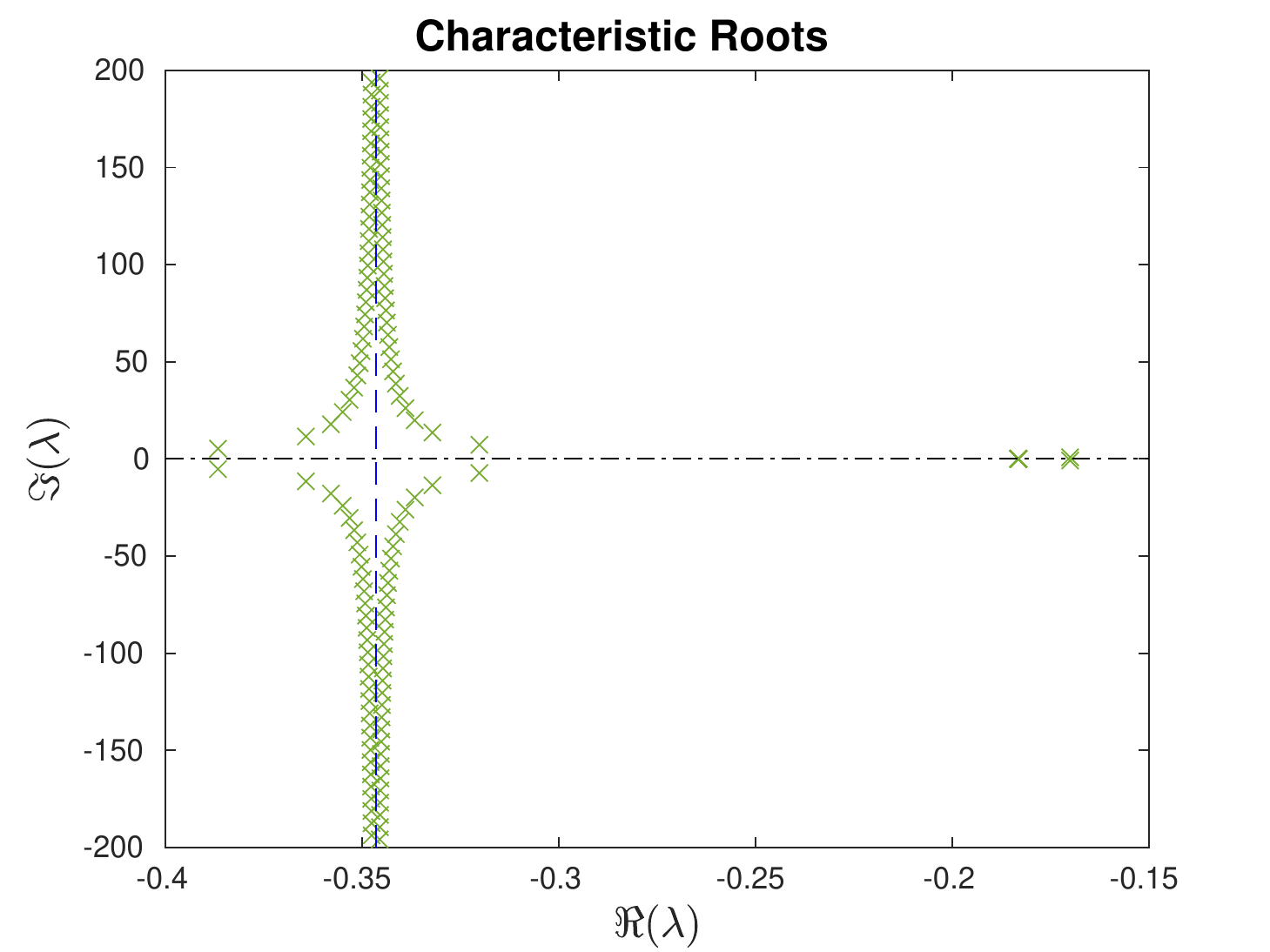}
		\caption{$N = 200$}
		\label{fig:neutral2_roots5}
	\end{subfigure}
	\caption{Characteristic roots of \eqref{eq:neutral_example2} for $\tau_1=1$ and $\tau_2 = 2$ in the right half-plane $\{\lambda\in\C : \Re(\lambda) \geq -0.6\}$ for several values of the degree of the spectral discretisation $N$.}
\end{figure}
\end{example}

In this section we saw that the stability analysis of a neutral delay differential is more involved than for retarded delay differential equations. Next we will examine a class of delay differential algebraic equations. We will see that this system class encompass time-delay systems of both retarded an neutral type.

\section{Delay differential algebraic equations}
\label{subsec:stability_DDAE}

In this paragraph we consider LTI delay differential algebraic equations with point-wise delays. Such differential equations take the following general form.
\begin{equation}
\label{eq:ol_ddae}
E \dot{x}(t) =  \sum_{k=1}^{m_A} A_k x(t-h_{A,k}) \text{ for } t\in [0,+\infty)
\end{equation}
with $x(t)\in\C^{n}$ the state variable at time $t$, $h_{A,1} = 0$ and the remaining state delays $h_{A,2},\dots,h_{A,m_A}$ non-negative. The matrices $A_1,\dots,A_m$ and $E$ belong to $\R^{n\times n}$. A DDAE differs from a RDDE in the fact that the matrix $E$ can be singular, \ie $\rank(E) \leq n$. As a consequence DDAEs can combines delay differential and delay algebraic equations. To see why this is true, let us introduce the orthogonal matrices  
\[\mathbf{U} = \begin{bmatrix}
U^{\bot} & U
\end{bmatrix} \text{ and } \mathbf{V}= \begin{bmatrix}
V^{\bot} & V
\end{bmatrix}\]
with $U\in\R^{n\times \nu}$ and $V\in\R^{n\times \nu}$ (in which $\nu := n-\rank(E)$) orthogonal matrices whose columns form a basis for, respectively, the left and right null space of $E$, and $U^{\bot}\in\R^{n\times(n-\nu)}$ and $V^{\bot}\in\R^{n\times(n-\nu)}$ orthogonal matrices whose columns span the orthogonal complement of the column space of $U$ and $V$, respectively. Premultiplying \eqref{eq:ol_ddae} with $\mathbf{U}$ and by introducing the change of variables
$
x = \mathbf{V} \begin{bmatrix}
x_1 \\
x_2
\end{bmatrix}
$
we obtain
\begin{equation}
\label{eq:ddae_transformed}
\begin{bmatrix}
E^{(11)} & 0 \\
0 & 0 
\end{bmatrix}
\begin{bmatrix}
\vphantom{E^{(11)}} \dot{x}_1(t) \\
\dot{x}_2(t)
\end{bmatrix} = 
\begin{bmatrix}
A_1^{(11)} & A_1^{(12)} \\
A_1^{(21)} & A_1^{(22)}
\end{bmatrix}
\begin{bmatrix}
\vphantom{A_0^{(12)}}x_1(t) \\
\vphantom{A_0^{(21)}}x_2(t)
\end{bmatrix}
+ \sum_{k=2}^{m_A}
\begin{bmatrix}
A_k^{(11)} & A_k^{(12)} \\
A_k^{(21)} & A_k^{(22)}
\end{bmatrix}
\begin{bmatrix}
x_1(t-h_{A,k}) \\
x_2(t-h_{A,k})
\end{bmatrix}
\end{equation}
with $E^{(11)}\in \R^{(n-\nu)\times (n-\nu)}$  invertible. The DDAE \eqref{eq:ol_ddae} can thus be transformed into a coupled system of $n-\nu$ delay differential equations and $\nu$ delay difference (delay algebraic) equations.

\begin{remark}
Delay differential equations of both retarded and neutral type fit in the DDAE framework. The RDDE \eqref{eq:ol_ret} can be written in form \eqref{eq:ol_ddae} by choosing $E=I_n$, while the NDDE
\[
\dot{x}(t) = A_0 x(t) + \sum_{k=1}^{m} A_k x(t-\tau_k) - \sum_{k=1}^{m} H_k \dot{x}(t-\tau_k)
\]
 is equivalent to the following DDAE:
\[
\begin{bmatrix}
0 & I_n \\
0 & 0
\end{bmatrix}
\begin{bmatrix}
\dot{x}(t) \\
\dot{\xi}(t) 
\end{bmatrix}
=
\begin{bmatrix}
A_1 & 0 \\
I_n & -I_n
\end{bmatrix}
\begin{bmatrix}
x(t) \\
\xi(t)
\end{bmatrix}
+ \sum_{k=1}^{m}
\begin{bmatrix}
A_k & 0 \\
H_k & 0
\end{bmatrix}
\begin{bmatrix}
x(t-\tau_k) \\
\xi(t-\tau_k)
\end{bmatrix},
\]
in which we introduced the auxiliary variable 
\[
\xi(t) =x(t) + \sum_{k=1}^{m} H_k x(t-\tau_k).
\] 
In \packageName{}, a NDDE can automatically be transformed into this DDAE form using the function \matlabfun{to_ddae}. To illustrate this, let us again consider the \matlabfun{tds_ss_neutral}-object \matlabfun{ndde} that we have created in \Cref{example:neutral_2}:
\begin{lstlisting}
>> ndde=tds_create_neutral({-3/4,1/2},[1 2],{1/4,-1/3},[0 1]);
>> ddae=ndde.to_ddae()

ddae = 

LTI Delay Differential Algebraic Equation with properties:

    E: [2x2 double]
    A: {[2x2 double]  [2x2 double]  [2x2 double]}
   hA: [0 1 2]
   mA: 3
    n: 2
\end{lstlisting} 
with
\[
\texttt{E} = \begin{bmatrix}
0 & 1 \\ 0 & 0
\end{bmatrix}\text{, } \texttt{A{1}} = \begin{bmatrix} 0.25 & 0 \\ 1 & -1 \end{bmatrix}\text{, } \texttt{A{2}} = \begin{bmatrix}
-1/3 & 0 \\ -0.75 & 0
\end{bmatrix}\text{, and } \texttt{A{2}} = \begin{bmatrix}
0 & 0 \\ 0.5 & 0
\end{bmatrix}.
\]
\end{remark}

Without putting restrictions on the system matrices, also a delay differential equation of advanced type (\cite[Section~3.3]{bellman1963dde}) can be transformed into a DDAE of form \eqref{eq:ol_ddae}. In the remainder of this manual we will therefore make the following assumption, which excludes DDAEs with advanced or impulsive solutions~\cite{fridman2002h}.
\begin{assumption}
\label{assumption:index_1}
The matrix $A_1^{(22)} = U^{T} A_1 V$ (as defined in \Cref{eq:ddae_transformed}) is invertible.
\end{assumption}

Before discussing the stability analysis of DDAEs, we again briefly mention the existence and uniqueness of solutions for the associated initial value problem (under \Cref{assumption:index_1}). Given any initial function segment $\phi\in X$ satisfying the condition
\[
 \sum_{k=1}^{m_A} \mathbf{U}^{T} A_k \phi(-h_{A,k}) = 0,
\]
there exists a unique forward solution of \eqref{eq:ol_ddae} on the interval $[0,t_0]$ for any $t_0>0$, see for example \cite{fridman2002stability}. Now we can move on to the stability analysis of its null solution. The characteristic equation associated with \eqref{eq:ol_ddae} is given by
\begin{equation}
\label{eq:DDAE_EVP}
\textstyle
\det\left(\lambda E - \sum\limits_{k=1}^{m_{A}} A_k e^{-\lambda h_{A,k}}\right) =0
\end{equation}
As the considered class of DDAEs encompasses NDDEs, it sshould not be supprising that certain properties from \Cref{sec:neutral_stability} carry over. For example, the exponential stability of \eqref{eq:ol_ddae} might be fragile with respect to the delays. When working with DDAEs it is thus better to focus on strong exponential stability (which assures that stability is preserved for sufficiently small perturbations on the delays). Next, let us therefore focus on the delay difference equation underlying \eqref{eq:ol_ddae}, which is now given by
\begin{equation}
\label{eq:ddae_difference}
A_1^{(22)} x_2(t) + \sum_{k=2}^{m_A} A_k^{(22)} x_2(t-h_{A,k}) = 0,
\end{equation}
with $A_1^{(22)}$, \dots, $A_{m_A}^{(22)}$ as defined in \eqref{eq:ddae_transformed} and the corresponding characteristic function
\begin{equation}
\label{eq:ddae_diff_cf}
\det\left(A_1^{(22)} + \sum_{k=2}^{m_A} A_k^{(22)} e^{-\lambda h_{A,k}}\right).
\end{equation}
Based on this characteristic equation, we can distinguish between two cases.
\begin{itemize}
    \item If \eqref{eq:ddae_diff_cf} does not depend on $\lambda$, the DDAE is \emph{essentially retarded}. In this case its spectrum has a similar outlook as that of a RDDE, meaning that it does not contain vertical chains of the form \eqref{eq:ev_chain}, the spectral abscissa can be defined as in \Cref{def:spectral_abscissa_ret} and is a continuous function of both the entries of the system matrices and the delays.
    \item If \eqref{eq:ddae_diff_cf} does depend on $\lambda$, the DDAE is \emph{essentially neutral}. In this case, its spectrum contains roots sequences of the form \eqref{eq:ev_chain} and exponential stability might be sensitivity to infinitesimal perturbations to the delays. In this case, strong exponential stability and the strong spectral abscissa can again be characterized in terms of \Cref{proposition:strong_stability,proposition:strong_stability_alt,proposition:strong_spectral_abscissa} with  the function $\gamma(r)$ now defined as
    \begin{equation}
    \label{eq:gamma_r_ddae}
        \gamma(r;A_{1}^{(22)},\dots,A_{m_A}^{(22)},\vec{\tau}) := \max_{\vec{\theta} \in [0,2\pi)^{m_{A}-1}} \rho\left(\sum_{k=2}^{m_{A}} \left(A_1^{(22)}\right)^{-1} A_k^{(22)} e^{-rh_{A,k}} e^{\jmath \theta_k}\right).
    \end{equation}
    
\end{itemize}

\noindent To conclude this section, let us now illustrate how one can deal with DDAEs in \packageName{}.

 \begin{example}
 \label{example:stability_analysis_DDAE}
 DDAEs typically arise in the context of interconnected systems as they avoid the explicit elimination of input and output channels which might be cumbersome or even impossible in the presence of delays \cite{gumussoy2011}. Below we will illustrate this approach using a simple example. For more examples, see \cite[Section~2]{gumussoy2011}. \\

We will consider the feedback interconnection of the following time-delay system
 \begin{equation*}
 \left\{
     \begin{array}{rcl}
        \dot{x}(t) & = & \underbrace{\begin{bmatrix}
        0.2 & 0.1 \\
        -0.5 & 1
        \end{bmatrix}}_{A_0}\,x(t) + \underbrace{\begin{bmatrix}
        0.5 & 0.3 \\
        0.1 & -0.1
        \end{bmatrix}}_{A_{1}}\, x(t-1) + \underbrace{\begin{bmatrix} 1 & 0 \\ 0 & 1 \end{bmatrix}}_{B_{1,1}} u(t-1) \\[25pt]
         y(t) & = & \underbrace{\begin{bmatrix}
         1 & 1
         \end{bmatrix}}_{C_{1,0}}x(t) +\underbrace{\begin{bmatrix}
         0.01 & 0.01
         \end{bmatrix}}_{D_{11,1}} \, u(t-1),
     \end{array}
 \right.
 \end{equation*}
 and the dynamic output feedback controller,
 \begin{equation*}
     \left\{
     \begin{array}{rcl}
        \dot{x}_c(t) & = & \underbrace{-3.48}_{A_c}\,x_c(t) + \underbrace{3.1}_{B_c} \, y(t) \\[18pt]
         u(t) & = & \underbrace{\begin{bmatrix}
         1.79 \\ -0.09
         \end{bmatrix}}_{C_c}\,x_c(t) + \underbrace{\begin{bmatrix}
         -1.86 \\ -1.4
         \end{bmatrix}}_{D_c} y(t),
     \end{array}
 \right.
 \end{equation*}
 Explicitly eliminating the input and output channels $u$ and $y$ is not trivial due to the delayed direct feed-through term. However, by introducing the auxiliary variables $\zeta_y$ and $\zeta_u$ the closed-loop system can represented in terms of the following DDAE
 \begin{equation*}
     \label{eq:ddae_example}
     \setlength{\arraycolsep}{4pt}
 \begin{bmatrix}
 \mathrm{I}_2 & 0 & 0 & 0 \\
 0 & 0 & 0 & 0 \\
 0 & 0 & 1 & 0 \\
 0 & 0 & 0 & 0
 \end{bmatrix}
 \begin{bmatrix}
 \dot{x}(t) \\
 \dot{\zeta}_y(t) \\
 \dot{x}_c(t) \\
 \dot{\zeta}_u(t)
 \end{bmatrix} = 
\begin{bmatrix}
 A_0 & 0 & 0 & 0 \\
 C_{1,0} & -1 & 0 & 0 \\
 0 & B_c & A_c & 0 \\
 0 &D_c & C_c & -\mathrm{I}_2
 \end{bmatrix}
 \begin{bmatrix}
 x(t) \\
 \zeta_y(t) \\
 x_c(t) \\
 \zeta_u(t)
 \end{bmatrix}
 +
\begin{bmatrix}
 A_1 & 0 & 0 & B_{1,1} \\
 0 & 0 & 0 & D_{11,1} \\
 0 & 0 & 0 & 0 \\
 0 & 0 & 0 & 0
 \end{bmatrix}
 \begin{bmatrix}
 x(t-1) \\
 \zeta_y(t-1) \\
 x_c(t-1) \\
 \zeta_u(t-1)
 \end{bmatrix}.
 \end{equation*}
To create a representation for this system we use the  \matlabfun{tds_create_ddae(E,A,hA)}-function. The first argument of this function should be the $E$ matrix, the second argument should be a cell array containing the matrices $A_1$, \dots, $A_{m_A}$, and the final argument should be an array with the corresponding delays $h_{A,1}$, \dots, $h_{A,m_A}$. 
 \begin{lstlisting}
A0 = [0.2 0.1;-0.5 1]; A1 = [0.5 0.3;0.1 -0.1];
B = eye(2); C = [1 1]; D = [0.01 0.01];
tau = 1;

Ac = -3.48; Bc = 3.1; Cc = [1.79;-0.09]; Dc = [-1.86;-1.4];

n = size(A0,1);p = size(B,2);q = size(C,1);nc = size(Ac,1);

E_CL = zeros(n+q+nc+p);
E_CL(1:n,1:n) = eye(n);
E_CL(n+q+(1:nc),n+q+(1:nc)) = 1;

A_CL_0 = zeros(n+q+nc+p);
A_CL_0(1:n,1:n) = A0; A_CL_0(n+(1:q),1:n) = C; 
A_CL_0(n+(1:q),n+(1:q)) = -eye(q);
A_CL_0(n+q+(1:nc),n+(1:q)) = Bc;
A_CL_0(n+q+(1:nc),n+q+(1:nc)) = Ac;
A_CL_0(n+q+nc+(1:p),n+(1:q)) = Dc;
A_CL_0(n+q+nc+(1:p),n+q+(1:nc)) = Cc;
A_CL_0(n+q+nc+(1:p),n+q+nc+(1:p)) = -eye(p);

A_CL_1 = zeros(n+q+nc+p);
A_CL_1(1:n,1:n) = A1;
A_CL_1(1:n,n+q+nc+(1:p)) = B;
A_CL_1(n+(1:q),n+q+nc+(1:p)) = D;

CL = tds_create_ddae(E_CL,{A_CL_0,A_CL_1},[0 tau]);
\end{lstlisting}
We can now use the tools from the previous section to analyze (strong) exponential stability. Let us start by taking a look at the underlying delay difference equation:
\begin{lstlisting}
>> [delay_diff] = get_delay_difference_equation(CL)

delay_diff = 

Delay Difference Equation with properties:
    E: [3x3 double] 
    A: {[3x3 double]  [3x3 double]}
   hA: [0 1]
   mA: 2
    n: 3
\end{lstlisting}
with
\[
\texttt{E} =
\begin{bmatrix}
0 & 0 & 0 \\
0 & 0 & 0 \\
0 & 0 & 0
\end{bmatrix} \text{, }
\texttt{A\{1\}} =
\begin{bmatrix}
-1 & 0 & 0 \\
-1.86 & -1 & 0 \\
-1.4 & 0 & -1
\end{bmatrix} \text{ and }
\texttt{A\{2\}} = \begin{bmatrix}
0 & 0.01 & 0.01 \\
0 & 0 & 0 \\
0 & 0 & 0
\end{bmatrix}.
\]
We conclude that the considered DDAE is essentially neutral. To compute the strong spectral abscissa of this delay difference equation, we use the \verb|tds_CD|-function:
\begin{lstlisting}
>> CD = tds_CD(delay_diff)

CD =

-3.4234
\end{lstlisting}
It follows from \Cref{prop:neutral_chain} that the spectrum contains a vertical chain of roots approaching $\Re(\lambda) = -3.4234$, see \Cref{fig:ddae_roots_rect} which was generated using the following code.
\begin{lstlisting}
>> options = tds_roots_options('max_size_evp',2000);
>> l_rect = tds_roots(CL,[-4 0.5 -500 500],options);
>> tds_eigenplot(l_rect);
>> xlim([-4,0.5])
>> ylim([-500 500])
>> hold on
>> plot([CD CD],ylim,'b--')
\end{lstlisting}
On the other hand, \Cref{prop:finite_number_roots} implies that the right half-plane $\{\lambda \in \C: \Re(\lambda) \geq -3 \}$ only contains finitely many characteristic roots, see \Cref{fig:ddae_roots_rhp} which was generated using the following code.
\begin{lstlisting}
>> l_rhp = tds_roots(CL,-3)
>> tds_eigenplot(l_rhp);
>> xlim([-3,0.5])
\end{lstlisting}
 Next, we use the function \matlabfun{tds_strong_sa} to compute the strong spectral abscissa.
\begin{lstlisting}
>> tds_strong_sa(CL,-4)

ans =

   -0.2845
\end{lstlisting}
We conclude that the closed-loop system is strongly exponentially stable.
\begin{figure}[!ht]
    \centering
    \begin{subfigure}{0.49\linewidth}
   	    \includegraphics[width=\linewidth]{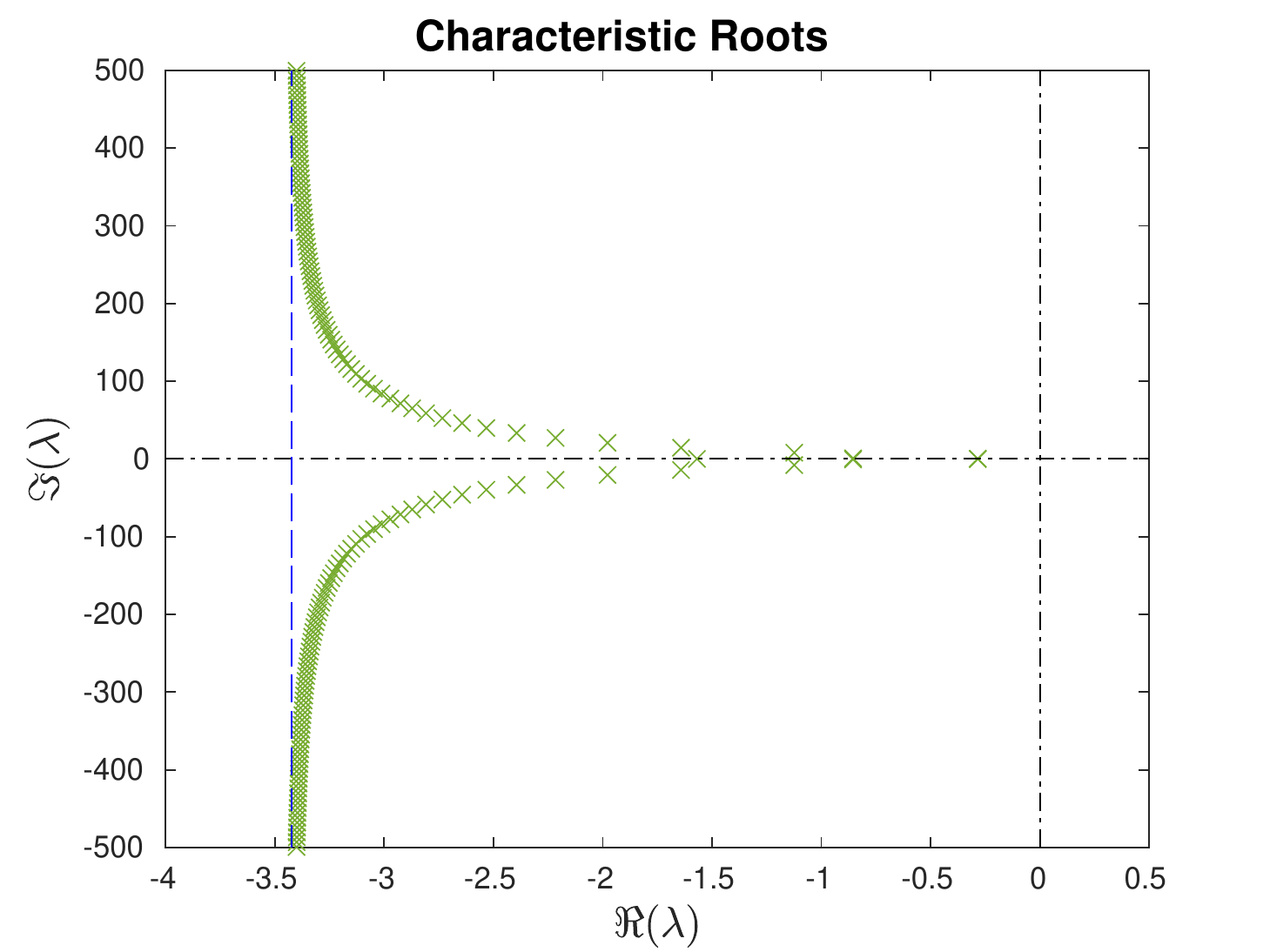}
		\caption{Rectangle $[-4, 0.5] \times \jmath [-500, 500]$}
		\label{fig:ddae_roots_rect}
    \end{subfigure}
    \begin{subfigure}{0.49\linewidth}
   	    \includegraphics[width=\linewidth]{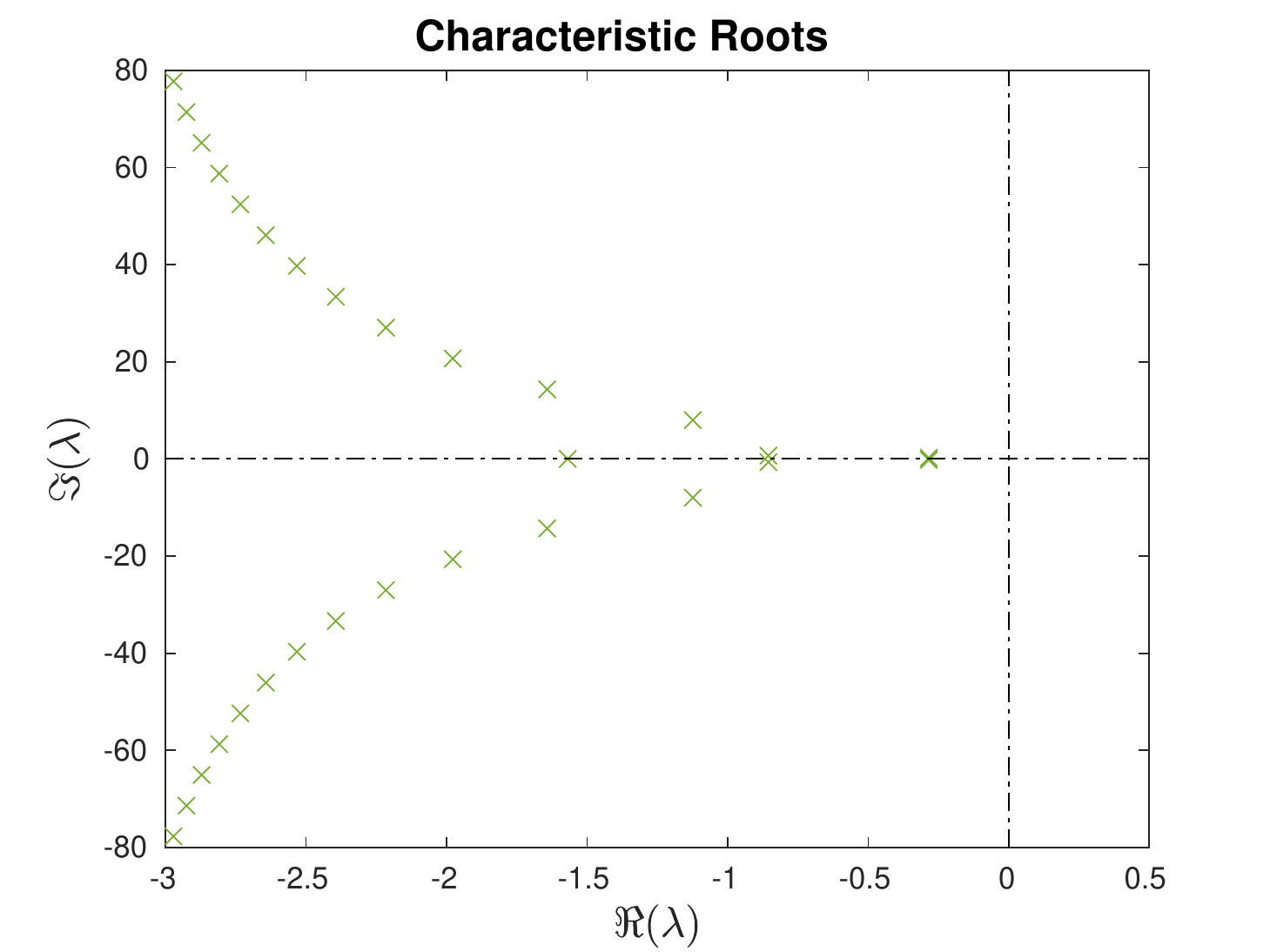}
   	    \caption{Right half-plane $\left\{\lambda\in\C:\Re(\lambda)\geq -3\right\}$}
   	    \label{fig:ddae_roots_rhp}
    \end{subfigure}
    \caption{The characteristic roots of the closed-loop system.}
    \label{fig:ddae_roots}
\end{figure}
 \end{example}  
\begin{remark}
\label{remark:createCL}
Instead of manually creating the system matrices of the closed-loop system, we can uses the function \matlabfun{tds_create_cl(P,K)} which takes two input arguments: a representation for the open-loop plant (\matlabfun{P}) and one for the controller (\matlabfun{K}). Let us demonstrate this functionality for the example above.
\end{remark}
\begin{lstlisting}
P = tds_create({A0,A1},[0,1],{B},[1],{C},[0],{D},[1]);
K = ss(Ac,Bc,Cc,Dc);

CL2 = tds_create_cl(P,K); 
\end{lstlisting}
If we inspect the resulting object using the Matlab Command Window, we see that it indeed corresponds to the DDAE system we derived before.

In this section and the sections before, it was assumed that a state-space representation for the considered systems was available. However, in some cases only a frequency-domain representation for the system, such as its characteristic polynomial, is available. In such a situation, one can use the function \matlabfun{tds_create_qp} to obtain a corresponding state-space representation, as explained in the next section.

\section{Computing the roots of a quasi-polynomial}
\label{subsec:quasi_polynomial}
The software package \packageName{} can also handle a system description in the frequency domain, as it contains functionality to automatically convert such a description to a \matlabfun{tds_ss}-object. As an illustration we will compute the roots of a quasi-polynomial using the functions \matlabfun{tds_create_qp} and \matlabfun{tds_roots}. The function \matlabfun{tds_create_qp} will create a \matlabfun{tds_ss}-object that represents a retarded or neutral DDE whose characteristic roots corresponds to the roots of the quasi-polynomial. Subsequently, the function \matlabfun{tds_roots} can be used to compute the roots of the quasi-polynomial in a desired region.

\begin{example}
	\label{example:roots_qp2}
Let the quasi-polynomial be given by
\[
\lambda^2 -2 \lambda + 1 + (-2\lambda+2) e^{-\lambda} + e^{-2\lambda}.
\]
The function \matlabfun{tds_create_qp(P,D)} allows to create a \matlabfun{tds_ss}-representation for quasi-polynomials of the form
\begin{equation}
\label{eq:quasi_polynomial}
\begin{array}{c}
\left(p_{1,1} \lambda^{n} + \dots + p_{1,n} \lambda + p_{1,n+1}\right) e^{-\lambda \tau_1}     + \\
\left(p_{2,1} \lambda^{n} + \dots + p_{2,n} \lambda + p_{2,n+1}\right) e^{-\lambda \tau_2}     + \\
\vdots \\       
\left(p_{m,1} \lambda^{n} + \dots + p_{m,n} \lambda + p_{m,n+1}\right) e^{-\lambda \tau_m}
\end{array}
\end{equation}
with $\tau_1,\dots,\tau_m$ non-negative delay values and $p_{k,i} \in \R$ for $k=1,\dots,m$ and $i=1,\dots,n+1$ the coefficients. To avoid advanced systems, there must be at least one $k$ such that $\tau_k = 0$ and the corresponding leading coefficient, $p_{k,1}$ should be non-zero. The first input argument of \matlabfun{tds_create_qp}, \matlabfun{P}, should be a matrix containing the coefficients of the quasi-polynomal with the element on the $k$\textsuperscript{th} row and $i$\textsuperscript{th} column equal to $p_{k,i}$. The second argument should be a vector containing the delays. In this case we thus have.
\begin{lstlisting}
qp = tds_create_qp([1 -2 1;0 -2 2;0 0 1],[0 1 2])
\end{lstlisting}
The return argument \matlabfun{qp} is a \matlabfun{tds_ss_redarded}-object whose characteristic function corresponds to the given quasi-polynomial.
\begin{lstlisting}
qp = 

Retarded Delay Differential Equation with properties:

    A: {1x3 cell}
   hA: [0 1 2]
   mA: 3
    n: 2
\end{lstlisting}
with 
\begin{equation}
\label{eq:matrices_qp}
\texttt{A\{1\}} = \begin{bmatrix}
0  &  1 \\
-1  &  2
\end{bmatrix}\text{, }
\texttt{A\{2\}} = 
\begin{bmatrix} 
0 & 0 \\ 
-2 & 2
\end{bmatrix}
\text{, and }
\texttt{A\{3\}} = 
\begin{bmatrix}
0 & 0 \\
-1 & 0
\end{bmatrix}.
\end{equation}
Next, we use the function \matlabfun{tds_roots} to compute the roots of the quasi-polynomial in the region of interest. Below we compute all roots in the right half-plane $\{z\in \C: \Re(z) \geq -4 \}$. As the quasi-polynomial corresponds to the characteristic function of a RDDE, \packageName{} can automatically determine the required degree for the spectral discretisation.
\begin{lstlisting}
cr = tds_roots(qp,-4);
\end{lstlisting}
\end{example}

\noindent \textbf{Important:} Explicitly constructing the characteristic quasi-polynomial of a DDE should generally be avoided in case the original system description is already in state-space form. Firstly, it typically reduces the accuracy of the computed characteristic roots. Secondly, it might introduce additional delays in the \matlabfun{tds_ss}-representation that were not present in the original DDE. For example, consider the following RDDE,
\begin{equation}
\label{eq:qp_compact}
\dot{x}(t) = \begin{bmatrix}
1 & 0 \\
0 & 1
\end{bmatrix} x(t) + \begin{bmatrix}
1 & 0 \\
0 & 1
\end{bmatrix} x(t-1).
\end{equation}
The corresponding characteristic function is equal to
\[
\lambda^2 -2 \lambda + 1 + (-2\lambda+2) e^{-\lambda} + e^{-2\lambda}.
\]
Note that this characteristic function corresponds to the quasi-polynomial considered in \Cref{example:roots_qp2}. However, the \matlabfun{tds_ss_retarded}-object in \eqref{eq:matrices_qp}, obtained using \matlabfun{tds_create_qp}, differs significantly from \eqref{eq:qp_compact}. For example, where the original DDE has only one delay, while \matlabfun{qp} has two. This is because the function \matlabfun{tds_create_qp} can not determine whether the term $e^{-2\lambda}$ in the quasi-polynomial originates from a separate delay or from the product of two $e^{-\lambda}$ terms. This difference in representation has an effect on the computations inside the \matlabfun{tds_roots}-function. For example, to compute all characteristic roots of \matlabfun{qp} in the right half-plane $\{z\in \C: \Re(z) \geq -4 \}$, a spectral discretisation of degree 169 is needed. On the other hand, if we manually create a \matlabfun{tds_ss_retarded}-representation for the RDDE \eqref{eq:qp_compact},
\begin{lstlisting}
tds = tds_create({eye(2),eye(2)},[0 1]);
\end{lstlisting}
and call \matlabfun{tds_roots} with the resulting object as an argument, then a spectral discretisation of degree 34 suffices. 
\section{Examples}

\subsection{Computing transmission zeros of a SISO system}
\label{subsec:tzeros}
In this subsection we will demonstrate how one can compute the transmission zeros of a single-input single-output (SISO) time-delay system in \packageName{}. Below we illustrate the procedure for a retarded delay system, but the procedure for a neutral or a delay-descriptor (\ie{} a system governed by a delay differential algebraic equation) system is similar. Let the system be given by
\begin{equation}
\label{eq:siso_system}
\left\{
\begin{array}{rcl}
   \dot{x}(t) &= & \displaystyle\sum\limits_{k=1}^{m} A_{k}\,x(t-\tau_k) + \sum\limits_{k=1}^{m} b_{1,k}\,u(t-\tau_k) \\[18pt]
    y(t) &=& \displaystyle\sum\limits_{k=1}^{m}c_{1,k}\,x(t-\tau_{k})+ \sum\limits_{k=1}^{m} d_{11,k}\,u(t-\tau_k),
\end{array}\right.
\end{equation}
with $x(t)\in\R^{n}$ the state variable, $u(t)\in\R$ the input and $y(t)\in\R$ the output, the delays $\tau_1,\dots,\tau_m$ non-negative real numbers, and the matrices are real-valued and of appropriate dimension. The corresponding transfer function is given by
\begin{equation*}
T(s) = \left(\sum\limits_{k=1}^{m}c_{1,k} e^{-s \tau_{k}}\right) \left(sI - \sum\limits_{k=1}^{m} A_{k} e^{-s \tau_{k}}\right)^{-1} \left(\sum\limits_{k=1}^{m} b_{1,k} e^{-s\tau_{k}}\right)  +  \sum\limits_{k=1}^{m} d_{11,k} e^{-s\tau_{k}}.  
\end{equation*}
For a SISO system, a transmission zero corresponds to a point $z_{0}$ in the complex plane for which the transfer function becomes zero, \ie{} 
$
T(z_0) = 0
$. Below we will see that finding the complex numbers $z_0$ for which the transfer function becomes zero, can be reformulated as finding the characteristic roots of a DDAE. More specifically, $T(z_0) = 0$ if (and only if) there exists a non-zero $v\in\C$ and a $w\in\C^{n}$ such that
\begin{equation}
    \label{eq:evp_tzero}
    \left(
\begin{bmatrix}
z_0 I_n & 0 \\
0 & 0
\end{bmatrix} - \sum_{k=1}^{m} \begin{bmatrix}
A_k & b_{1,k} \\
c_{1,k} & d_{11,k}
\end{bmatrix} e^{-z_0 \tau_k}\right)\begin{bmatrix}
w \\
v
\end{bmatrix} =
\begin{bmatrix}
0 \\ 0
\end{bmatrix}.
\end{equation}
The transmission zeros of \eqref{eq:siso_system} thus correspond to (a subset of) the characteristic roots of a DDAE. However, although system \eqref{eq:siso_system} is retarded, the DDAE associated with \eqref{eq:evp_tzero} might not satisfy \Cref{assumption:index_1}, meaning that the underlying DDAE may be neutral or even advanced (see also \cite[Section~6]{tzero}).
\begin{example}
\label{example:computing_transmission_zeros}
Consider the following time-delay system
\begin{equation}
\label{eq:sys_tzeros}
\left\{
\begin{array}{rcl}
    \dot{x}(t) &=& \begin{bmatrix}
    1 & 0 & 1 \\
    0 & 0 & 0 \\
    0 & 1 & 0
    \end{bmatrix} x(t) + \begin{bmatrix}
    0 & 1 & 0 \\
    0 & 0 & 0 \\
    0 & 0 & 0
    \end{bmatrix} x(t-1) + \begin{bmatrix} 
    0 \\
    1 \\
    0
    \end{bmatrix} u(t) \\[20pt]
    y(t) & = & \begin{bmatrix}
    1 & 0 & 0
    \end{bmatrix} x(t),
\end{array}
\right.
\end{equation}
whose corresponding transfer function is given by 
\[
T(s) = \frac{se^{-s}+1}{s^{2}(s-1)}.
\]
The (finite) transmission zeros of this system are the roots of the (advanced) quasi-polynomial
\[
se^{-s}+1.
\]
To compute the transmission zeros of \eqref{eq:sys_tzeros}, we start with representing the system in \packageName{} using \matlabfun{tds_create}.
\begin{lstlisting}
A0 = [1 0 1;0 0 0;0 1 0];
A1 = [0 1 0;0 0 0;0 0 0];

B = [0;1;0];
C = [1 0 0];
tds = tds_create({A0,A1},[0 1],{B},{C});
\end{lstlisting}
Note that \matlabfun{tds_create} now takes two additional cell arrays as argument, representing the input and output matrices. Next we use \matlabfun{tds_tzeros} to compute the transmission zeros in the region $[-4,4]\times \jmath[-50,50]$.
\begin{lstlisting}
tzeros = tds_tzeros(tds,[-4 4 -50 50]);

plot(real(tzeros),imag(tzeros),'o')
xlim([-4 4])
ylim([-50 50])
\end{lstlisting}
\Cref{fig:tzeros} shows the transmission zeros in the desired region. Note that \matlabfun{tds_tzeros} internally constructs a DDAE whose characteristic roots correspond to the transmission zeros of \eqref{eq:sys_tzeros}.
\begin{figure}
  \centering
  \includegraphics[width=0.6\linewidth]{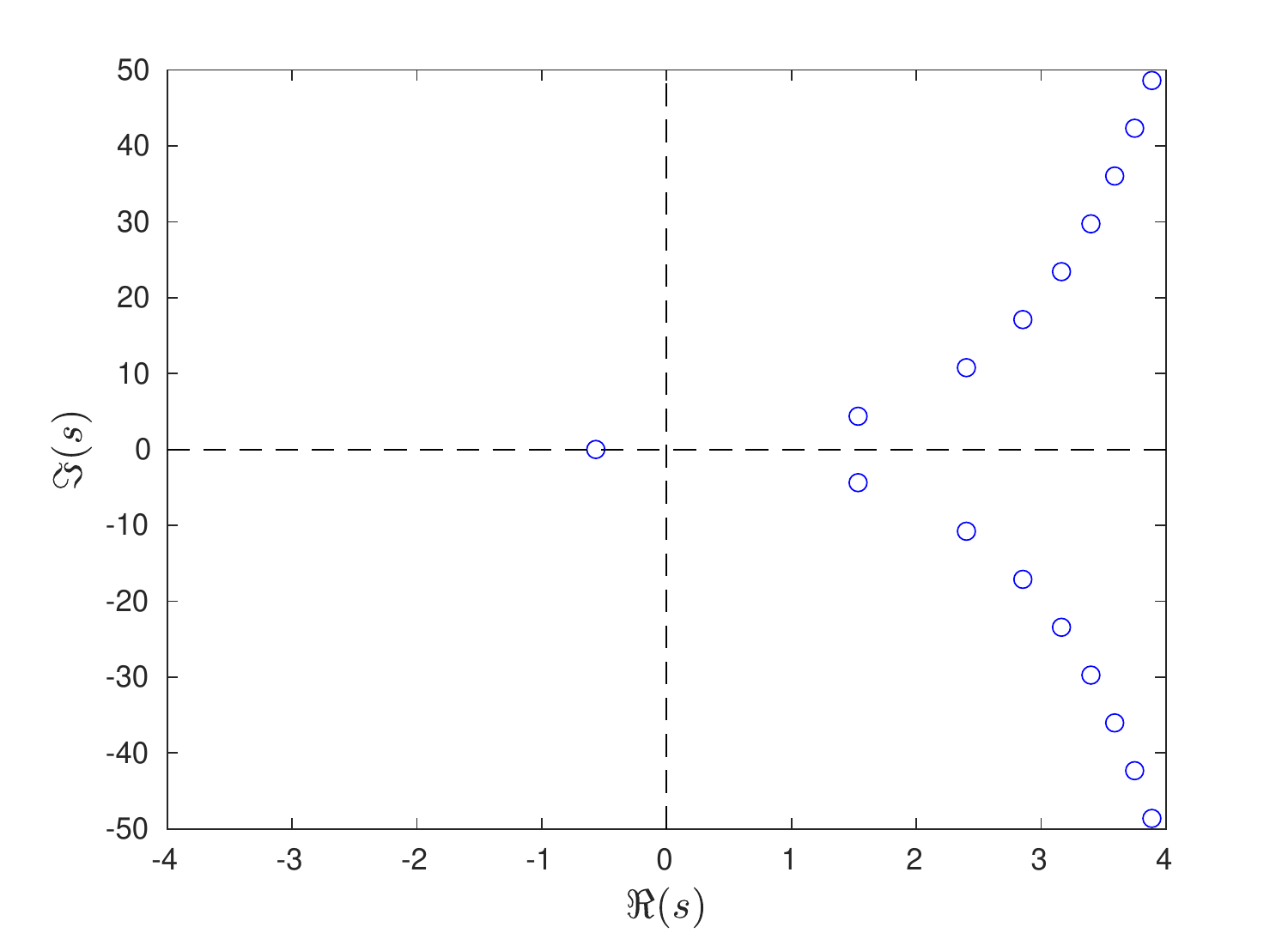}
  \caption{The transmission zeros of \eqref{eq:sys_tzeros} in the rectangular region $[-4,4] \times \jmath [-50,50]$.}
  \label{fig:tzeros}
\end{figure}

\end{example}
\subsection{Stability analysis of a Smith predictor with delay mismatch}
The Smith predictor is a popular scheme for the control of stable systems with a significant feedback delay \cite{smith_predictor}. The idea behind the Smith predictor is to use a controller structure for which the delay is eliminated from the design problem. More specifically, let the plant be described by the transfer function 
\[
H(s) = H_0(s) e^{-s\tau}
\]
with $H_0(s)$ a (stable) rational function and $\tau$ a feedback delay, then the Smith predictor is given by
\[
C'(s) = \frac{C(s)}{1+C(s)H_0(s)-C(s)H_0(s)e^{-s\tau}}.
\]
By taking the feedback interconnection of $H(s)$ and $C'(s)$ we obtain the following expression for the closed-loop transfer function:
\[
H_{cl}(s) = \frac{C(s)H_0(s)}{1+C(s)H_0(s)}e^{-s\tau},
\]
from which it is clear that closed-loop stability can be examined in terms of finitely many characteristic roots. However, for the closed-loop system to be stable in practice, it is important that an accurate estimate for the delay value $\tau$ is available. In the following example we will therefore examine the effect of a delay mismatch on the closed-loop stability. To this end, we note that given the transfer functions
\[
H_0(s) = \frac{B_1(s)}{A_1(s)} \text{ and } C(s) = \frac{B_2(\lambda)}{A_2(s)},
\]
the ``real'' closed-loop system is stable if (and only if) all roots of the following quasi-polynomial lie in the left open half-plane (bounded away from the imaginary axis):
\begin{equation}
\label{eq:smith_delay_mismatch}
F(\lambda) := A_1(\lambda)A_2(\lambda)+B_1(\lambda)B_2(\lambda) + B_1(\lambda)B_2(\lambda)(1-e^{-\lambda\delta})e^{-\lambda\tau},
\end{equation}
with $\delta$ the delay mismatch, see for example \cite[Eq.~(5)]{smith_delay_mismatch}. Below we will apply this result to examine the delay mismatch problem from \cite[Section~4]{smith_delay_mismatch} using \packageName{}.
\begin{example}
	\label{example:smith_predictor}
	In this example, we will investigate the delay mismatch problem for
	\[
	H_0(s) = \frac{1}{s+1} \text{ and } C(s) = \frac{s}{2} + 2.
	\]
	More specifically, we want to determine the regions in the $(\tau,\delta)$-parameter space for which the closed-loop system is stable. To this end, note that for this example the quasi-polynomial \eqref{eq:smith_delay_mismatch} reduces to:
	\begin{equation}
	\label{eq:ex_smith_qp}
	\frac{3}{2}\lambda+3+\left(\frac{\lambda}{2} + 2\right)e^{-\lambda \tau} + \left(\frac{-\lambda}{2}-2\right)e^{-\lambda(\tau+\delta)}.
	\end{equation}
	To examine the stability of this quasi-polynomial in \packageName{}, we start by creating a corresponding state-space representation using the function \matlabfun{tds_create_qp}.
	\begin{lstlisting}
tau = 1; delta = 0.5;
qp = tds_create_qp([1.5 3;0.5 2;-0.5 -2],[0 tau tau+delta])
\end{lstlisting}
	Next we perform a grid search over the $(\tau,\delta)$-parameter space and compute the correspoding strong spectral abscissa. The stability boundaries can now be found by taking the zero-level contour lines. \Cref{fig:ex_smith_pred} shows the result (the computations may take quite some time). 
\begin{lstlisting}
tau_grid = linspace(0,8,201);
delta_grid = linspace(-8,10,451);

options = struct;
options.roots = tds_roots_options('fix_N',20,'quiet',true);

Z = zeros(length(delta_grid),length(tau_grid));
for i2 = 1:length(tau_grid)
    tau = tau_grid(i2);
    for i1=1:length(delta_grid)
        delta = delta_grid(i1);
        if tau+delta<0
        %case: "real" delay can not be negative
            Z(i1,i2) = -inf;
        else
            qp.hH(1) = tau; qp.hH(2) = tau+delta;
            qp.hA(2) = tau; qp.hA(3) = tau+delta;
            Z(i1,i2) = tds_strong_sa(qp,-0.5,options);
        end
    end
end

[X,Y] = meshgrid(tau_grid,delta_grid);
contour(X,Y,Z,[0 0],'b')
hold on
plot(xlim,[0 0],'k--')
plot([0 8],[0 -8],'k-.')
xlabel('\tau')
ylabel('\delta')
	\end{lstlisting}

	\begin{figure}
		\centering
		    \begin{tikzpicture}[      
		every node/.style={anchor=south west,inner sep=0pt},
		x=1mm, y=1mm,
		]   
		\node (fig1) at (0,0)
		{ \includegraphics[width=0.6\linewidth]{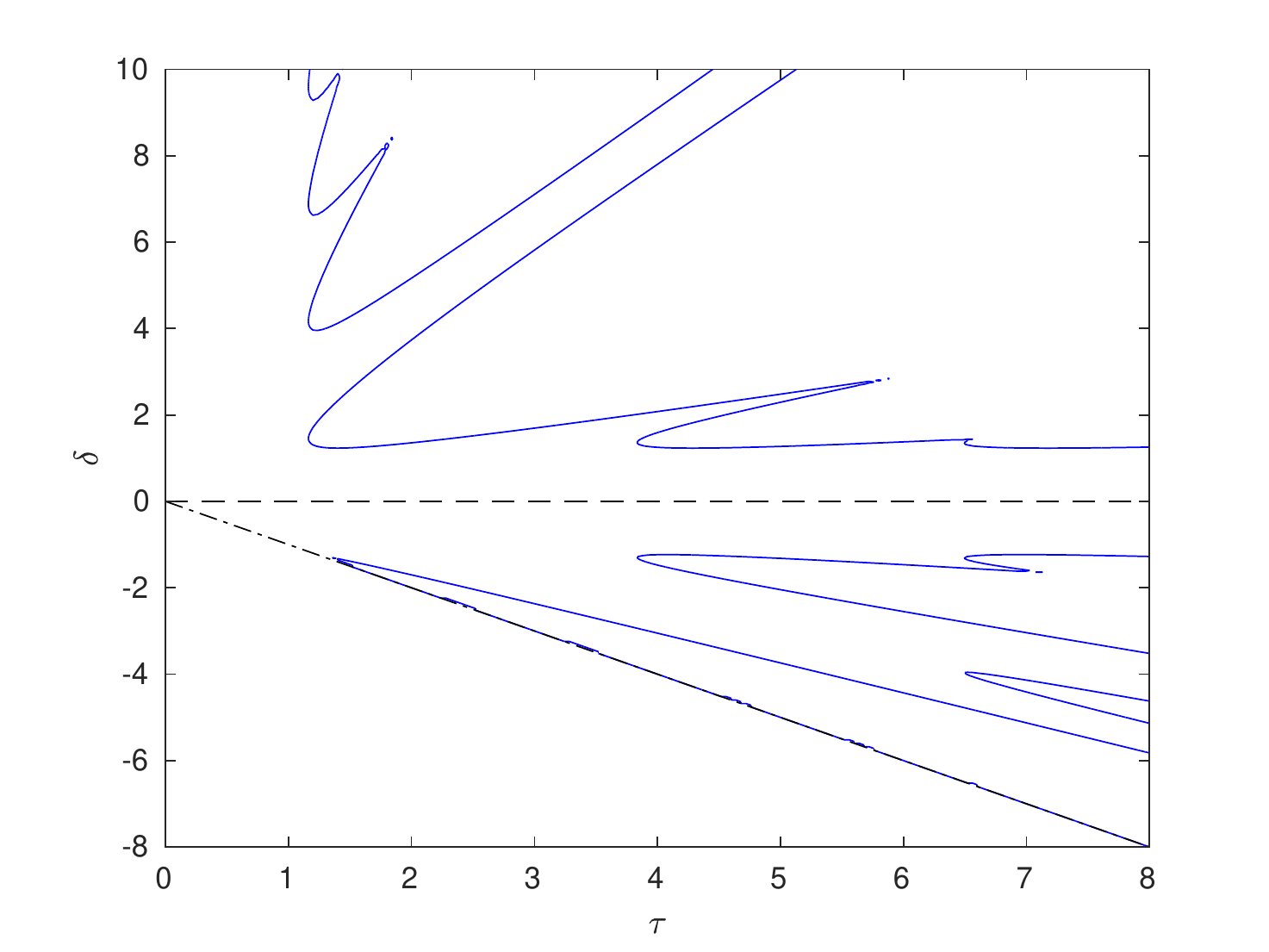}};
		\node[] (text) at (14,30) {\textbf{stable}};
		\node[] (text) at (53,50) {\textbf{unstable}};
		\end{tikzpicture}
		\caption{Stability boundaries of \eqref{eq:ex_smith_qp} in function of $\tau$ and $\delta$.}
		\label{fig:ex_smith_pred}
	\end{figure}
\end{example}
\chapter{Stabilization using a direct optimization approach}
\label{sec:stabilization}
In this section we will discuss how to design (strongly) stabilizing output feedback controllers for linear time-delay systems using \packageName{}.
More specifically, we will consider LTI time-delay systems that admit the following  state-space representation:
\begin{equation}
\label{eq:state_space_system}
P \leftrightarrow
\left\{
\setlength{\arraycolsep}{1pt}
\begin{array}{rclcl}
     E\dot{x}(t) &=& \displaystyle \sum\limits_{k=1}^{m_A} A_k\, x(t-h_{A,k}) &+&   \displaystyle\sum\limits_{k=1}^{m_{B_1}} B_{1,k}\, u(t-h_{B_1,k}) - \displaystyle\sum\limits_{k=1}^{m_H} H_k\, \dot{x}(t-h_{H,k}) \\[13pt]
     y(t)& =&  \displaystyle\sum\limits_{k=1}^{m_{C_1}} C_{1,k} \, x(t-h_{C_1,k}) &+& \displaystyle\sum\limits_{k=1}^{m_{D_{11}}} D_{11,k} \, u(t-h_{D_{11},k})
\end{array}
\right.
\end{equation}
with $x\in \R^{n}$ the state variable, $u \in \R^{p_1}$ the control input (representing the \emph{``actuators''} that steer the system) and  $y \in \R^{q_1}$ the control output (representing the available measurements). The delays $h_{A,1},\dots,h_{A,m_A}$, $h_{B_1,1},\dots,h_{B_1,m_{B_1}}$, $h_{C_1,1},\dots,h_{C_1,m_{C_1}}$, and $h_{D_{11},1},\dots,h_{D_{11},m_{D_{11}}}$ should be non-negative real numbers, while the delays $h_{H,1},\dots,h_{H,m_H}$ should be positive. The matrix $E$ belongs to $\R^{n\times n}$ and is allowed to be singular, while the remaining matrices are also real-valued and have appropriate dimension.

To represent such state-space models in \packageName{}, additional arguments have to be passed to the functions \matlabfun{tds_create}, \matlabfun{tds_create_neutral} and \matlabfun{tds_create_ddae}:
\begin{itemize}
    \item \matlabfun{tds_create(A,hA,B1,hB1,C1,hC1,D11,hD11)} allows to create retarded state-space models, \ie $E = I_n$ and $m_H=0$.
    \item \matlabfun{tds_create_neutral(H,hH,A,hA,B1,hB1,C1,hC1,D11,hD11)} allows to create neutral state-space models, \ie $E=I_n$ and $m_H>0$.
    \item \matlabfun{tds_create_ddae(E,A,hA,B1,hB1,C1,hC1,D11,hD11)} allows to create delay-descriptor systems, \ie $E$ is not necessarily identity and potentially singular, and $m_H = 0$. 
\end{itemize} The arguments \texttt{H}, \texttt{A}, \texttt{B1}, \texttt{C1}, and \texttt{D11} should be cell arrays containing the system matrices, while the arguments \texttt{hH}, \texttt{hA}, \texttt{hB1}, \texttt{hC1}, and \texttt{hD11} should be arrays containing the corresponding delay values. The argument \texttt{E} of \matlabfun{tds_create_ddae} should be a (real-valued) matrix. Finally, in all function calls above, the arguments \texttt{D11} and \texttt{hD11} are optional and can be omitted if no direct feedthrough term is present. \\
\textbf{Note.} In \Cref{subsec:transfer_function} we will show how a state-space representation of the form \eqref{eq:state_space_system} can be obtained  for a SISO \emph{transfer function} using the function \matlabfun{tds_create_tf}.\\

As mentioned in the introduction, \packageName{} allows to design dynamic output feedback controllers with a user-chosen order, denoted below by $n_c$, \ie 
\begin{equation}
\label{eq:dynamic_output_feedback_controller}
   C \leftrightarrow \left\{
    \begin{array}{rcl}
         \dot{x}_c(t) & = &  A_{c} \, x_c(t)+ B_{c} \, y(t)\\[5pt]
         u (t)  & = & C_{c} \, x_c(t) + D_{c} \, y(t)
    \end{array}
    \right.
\end{equation}
with $x_c \in \R^{n_c}$ the state variable of the controller and $A_c$, $B_c$, $C_c$, and $D_c$ real-valued matrices with appropriate dimension. Note that if the controller order $n_c$ is set to $0$, \eqref{eq:dynamic_output_feedback_controller} reduces to the static output feedback law
\begin{equation}
    \label{eq:static_output_feedback_controller}
    C \leftrightarrow u(t) = D_c \, y(t).
\end{equation}

At first sight, the controller structure in \eqref{eq:dynamic_output_feedback_controller} might seem restricting. However, as \eqref{eq:state_space_system} allows for delays in the input, output and direct feedthrough terms and a singular matrix $E$, one can design a broad class of controllers by introducing the necessary auxiliary variables. Furthermore, as the design algorithms in \packageName{} allow to fix the value of certain entries of the controller matrices, it is also possible to design structured controllers, such as PID or decentralized controllers. Below we give an overview of the controller classes that will be considered in this section:
\begin{itemize}
	\item unstructured static and dynamic output feedback controllers in \Cref{example:stabilization_retarded_2,example:stab_neutral},
	\item Pyragas-type feedback controllers in \Cref{ex:pyragas_feedback},
	\item delayed feedback controllers in \Cref{ex:delayed_feedback},
	\item acceleration-based controllers in \Cref{example:acceleration_feedback},
	\item static and dynamic output feedback controllers with fixed entries in the controller matrices in \Cref{example:fixed_entries},
	\item PID controllers in \Cref{example:pid_control},
	\item and decentralized controllers in \Cref{ex:decentralized}.
\end{itemize}
To design such controllers, the functions \matlabfun{tds_stabopt_dynamic} and \matlabfun{tds_stabopt_static} should be used. Both functions take as first argument a \matlabfun{tds_ss}-object that represents the open-loop plant (created using one of the three functions listed above). The function \linebreak \matlabfun{tds_stabopt_dynamic} also takes a second mandatory argument specifying the desired controller order. The remaining four arguments of these two function are optional and should be passed in key-value pairs:
\begin{itemize}
	\item \verb|'options'| allows to specify additional options. The \verb|'options'|-argument should be a \verb|struct| with the necessary fields. A valid \verb|'options'|-structure can be created using the \verb|tds_stabopt_options|-function  (for more details see \Cref{sec:tds_stabopt} or use \verb|help tds_stabopt_options| in the \matlab{} Command Window).
	\item \verb|'initial'|  allows to specify the initial values for the optimization variables. 
	\item \verb|'mask'| allows to specify which entries of the matrices in \eqref{eq:dynamic_output_feedback_controller} or \eqref{eq:static_output_feedback_controller} that should remain fixed to their basis value (specified using the \verb|'basis'| argument), \ie the entries in the controller that should not be touched by the optimization algorithm.
	\item \verb|'basis'| allows to specify the value of the fixed entries of the controller matrices  (indicated using the \verb|'mask'| argument). 
\end{itemize}

Below we will give a high-level description of these controller design algorithms. To this end, we start with a description of the closed-loop system in the form of a DDAE (see \Cref{sec:tds_stabopt}):
\begin{equation}
\label{eq:cl}
E_{cl}\,\dot{x}_{cl}(t) = \sum_{k=1}^{m_{cl}} A_{cl,k}(\mathbf{p})\,x_{cl}(t-\tau_{cl,k})    
\end{equation}
in which the delays $\tau_{cl,1},\dots,\tau_{cl,m_{cl}}$ are non-negative and the entries of the matrices \linebreak $A_{cl,1},\dots,A_{cl,m_{cl}}$ are affine functions of the controller parameters $\mathbf{p}$, \ie the entries of the controller matrices that are not fixed. Note that as a DDAE generally has infinitely many characteristic roots and the number of controller parameters in $\mathbf{p}$ in contrast is finite, it is not possible to fully control the dynamics of the closed-loop system. As a consequence, the considered controller design problem is difficult, exhibiting hard limitations on stabilizability, performance, and robustness. As mentioned before, the controller design routines of \packageName{} therefore employ a direct optimization approach, \ie \matlabfun{tds_stabopt_dynamic} and \matlabfun{tds_stabopt_static} will derive suitable controller parameters by (numerically) minimizing the strong (spectral) abscissa of \eqref{eq:cl} with respect to the control parameters. As a strictly negative (strong) spectral abscissa is a necessary and sufficient condition for (strong) stability, the presented design algorithms are not conservative in the sense that a stabilizing controller can be found whenever it exists. Yet, minimizing the spectral abscissa with respect to the controller parameters is a nonsmooth and nonconvex optimization problem. \packageName{} therefore relies on HANSO v3.0 \cite{HANSO,HANSO2}, a solver for nonsmooth and nonconvex optimization, to tackle these optimization problems. Furthermore, to avoid ending up at a bad local optimum, the optimization procedure is restarted from several initial points.\medskip

\noindent \textbf{Important.} When running the examples below on your own machine, you may obtain different controller parameters, depending on the \matlab{} version and the hardware you are using. This can be explained by the fact that for the considered objective function, the first phase in HANSO (an adaptation of BFGS aimed at finding controller parameters in the neighborhood of a local optimum) may be very sensitive to small perturbations on certain hyper-parameters (even for the same initial values), due to
\begin{itemize}
	\item the nonconvexity of the objective function, implying the existence of multiple local optima (potentially close in terms of spectral abscissa),
	\item the nonsmoothness of the objective function at local minima, implying high sensitivity in their neighborhood, and
	\item the fact that the adopted parametrization of a dynamic controller in terms of the matrices of its state-space representation is only unique up to a similarity transformation.
\end{itemize}
When optimizing performance measures, as in the next section, this phenomenon is much less pronounced.

\medskip

The remainder of this section is structured as follows. First we will discuss the situation in which the DDAE describing the closed-loop system is essentially retarded. Next we will discuss the case in which \eqref{eq:cl} is essentially neutral. Note however that the user does not need to specify the type of the resulting closed-loop system, as this can \emph{automatically} be determined by \packageName{}. Thereafter, we will show that \packageName{} can even be useful for the (robust) controller design of undelayed systems. More precisely, we will see that certain stable ``delay-free'' closed-loop systems may have a delay margin of zero, meaning that they can be destabilized by an arbitrary small feedback delay. \packageName{} in contrast will automatically assume an infinitesimal feedback delay during the design process, thereby avoiding the aforementioned fragility problem. Next, we will demonstrate how \packageName{} can be used to design various structured controllers. For more technical details on the presented design algorithms, we refer to \Cref{sec:tds_stabopt}. \\

\noindent \textbf{Important:} To run the examples described below, make sure that HANSO v3.0 is present in your \matlab{} search path. The code of HANSO v3.0 is included in the folder \matlabfun{hanso3_0}.

\section{Essentially retarded closed-loop systems}
For simplicity we will first consider the case in which the DDAE \eqref{eq:cl} is essentially retarded. For example, when the open-loop system in \eqref{eq:state_space_system} is retarded and does not have any direct feedthrough term. 
\begin{example}
\label{example:stabilization_retarded_2}
Consider the following five dimensional time-delay system that models the heating system described in \cite[Section~5]{vyhlidal2003analysis}
\begin{equation}
\label{eq:heating_system}
    		\left\{
		\begin{array}{lcl}
		\dot{x}_h(t) & = & \frac{1}{T_h} \left( -x_h(t-\eta_h)+K_b x_a(t-\tau_b) + K_u u(t-\tau_u)\right)\\[5pt]
		\dot{x}_a(t) &=& \frac{1}{T_a}\left(-x_a(t)+x_c(t-\tau_e) +K_a\big(x_h(t) -\frac{1+q}{2}x_a(t)-\frac{1-q}{2}x_c(t-\tau_e)\big)\right) \\[5pt]
		 \dot{x}_d(t) &=& \frac{1}{T_d} \left(-x_d(t) + K_d x_a(t-\tau_d)\right)\\[3pt]
		\dot{x}_c(t) &=& \frac{1}{T_c} \left(-x_c(t-\eta_c)+K_c x_d(t-\tau_c)\right)\\[3pt]
		\dot{x}_I(t) &=& x_{c,set}-x_c(t),
		\end{array}
		\right.
\end{equation}
in which the last state can be interpreted as an integrator, used to eliminate the steady-state offset from the target value. As we are interested in the stability of the closed-loop system around the equilibrium point $x_c(t) \equiv x_{c,set}$, we can, without loss of generality, assume that $x_{c,set} = 0$. Furthermore, it is assumed that all states are available for measurement, \ie $C_1=\mathrm{I}_5$. We start with representing this system in \packageName{} using the values for the constants given in the reference above.
\begin{lstlisting}
Th = 14;  Ta = 3;  Td = 3;   Tc = 25;
Kb = 0.24; Ka = 1;  Kd = 0.94; Kc = 0.81; Ku = 0.39;
nh = 6.5; tb = 40; te = 13; td = 18;  
tc = 2.8; nc = 9.2; tu = 13.2;

dA = [0 nh tb te td tc nc]; dB = tu;

A0 = zeros(5,5);
A0(2,1) = Ka/Ta; A0(2,2) = (-Ka-1)/Ta;
A0(3,3) = -1/Td; A0(5,4) = -1;
A1 = zeros(5,5); A1(1,1) = -1/Th;
A2 = zeros(5,5); A2(1,2) = Kb/Th;
A3 = zeros(5,5); A3(2,4) = 1/Ta;
A4 = zeros(5,5); A4(3,2) = Kd/Td;
A5 = zeros(5,5); A5(4,3) = Kc/Tc;
A6 = zeros(5,5); A6(4,4) = -1/Tc;

B = [Ku/Th; 0; 0; 0; 0];

plant=tds_create({A0,A1,A2,A3,A4,A5,A6},dA,{B},dB,{eye(5)},0);
\end{lstlisting}
We can again use the \matlab{} Command Window to inspect the resulting object.
\begin{lstlisting}
plant = 

LTI Retarded Time-Delay State-Space System with properties:

    A: {1x7 cell}
   hA: [0 6.5000 40 13 18 2.8000 9.2000]
   mA: 7
   B1: {[5x1 double]}
  hB1: 13.2000
  mB1: 1
   C1: {[5x5 double]}
  hC1: 0
  mC1: 1
    n: 5
   p1: 1
   q1: 5
\end{lstlisting}
Next, we use \matlabfun{tds_stabopt_static} to design a static output feedback controller of the form \eqref{eq:static_output_feedback_controller}.  As mentioned before, this function takes one mandatory argument, namely the open-loop state-space system. In this case we also pass the additional arguments \verb|options| (with \matlabfun{nstart} equal to 1) and \verb|initial| to specify the initial values for the feedback matrix. Here we will run the optimization routine starting from the initial controller 
\[
D_c = \begin{bmatrix}
    0 & 0 & 0 & 0 & 0 
\end{bmatrix}
\] to assure reproducibility. However, recall that the considered optimization problem is typically non-convex which means that different initial values for the optimization variables might result in a different local optimum. The default behavior of \matlabfun{tds_stabopt_static} is therefore to run the optimization procedure starting from five randomly sampled initial controllers and then return the best result.
\begin{lstlisting}
initial = zeros(1,5);
opt = tds_stabopt_options('nstart',1);
[Dc,cl] = tds_stabopt_static(plant,'options',opt,'initial',initial);
\end{lstlisting}
The default behavior of \matlabfun{tds_stabopt_static} is to print the result of every successful iteration of the optimization procedure to the \matlab{} Command Window.  In this case BFGS converges in 146 steps to a local optimum. The function \matlabfun{tds_stabopt_static} has two important return arguments: the resulting feedback matrix and a \matlabfun{tds_ss}-object that represents the closed-loop system. In this case, the obtained feedback matrix (rounded to four decimal places) is given by
\[
Dc =
\begin{bmatrix}
 -0.1659  & -0.2968  & -0.3612  & -0.3629 &   0.0168
\end{bmatrix}
\]
and the spectral abscissa of the corresponding closed-loop system is equal to $-7.961\times10^{-3}$.

Recall that due to the nonconvex, nonsmooth nature of the optimization problem, the obtained controller may depend on the initial values of the optimization process. To demonstrate this, let us rerun \matlabfun{tds_stabopt_static} starting from the initial controller 
\[
D_c =
\begin{bmatrix}
-1 & -1 & -1 & -1 & 1
\end{bmatrix}.
\]
\begin{lstlisting}
initial2 = [-1 -1 -1 -1 1];
[Dc2,cl2] = tds_stabopt_static(plant,'options',opt,'initial',initial2);
\end{lstlisting}
We now obtain the controller 
\[
D_c = \begin{bmatrix}
-0.35936 & -1.2544 & -3.1696 & -3.9919 & 0.14344
\end{bmatrix}
\]
and the corresponding closed-loop spectral abscissa is equal to $-6.0982\times10^{-2}$. This also means that for certain initial controllers the optimization routine might end up in a local optimum for which the spectral abscissa is positive.
\begin{lstlisting}
initial3 = [10 10 10 10 13];
[Dc3,cl3] = tds_stabopt_static(plant,'options',opt,'initial',initial3);

>> tds_strong_sa(cl3,-0.1)
   1.117611e-02
\end{lstlisting}
as indicated by the following warning
\begin{lstlisting}[language=,basicstyle=\ttfamily\color{orange}]
Warning: Resulting controller is not stabilizing. 
\end{lstlisting}
Note that this warning does not imply that there does not exist a stabilizing controller, it only indicates that starting from the chosen initial values, the optimization procedure ended up in local optima that does not correspond with a strictly negative spectral abscissa. If you encounter this warning, try to increase the number of initial values using the option \verb|'nstart'| and/or the optional argument \verb|'initial'|. If after sampling sufficiently many initial values, still no stabilizing static output feedback controller is found, this might be an indication that there does not exist such a controller and one might consider using a dynamic feedback controller.

Similarly as for static feedback controllers, we can design a dynamic output feedback controller in \packageName{} using \verb|tds_stabopt_dynamic|. This function takes two mandatory arguments: the open-loop system and the desired controller order $n_c$. Let us illustrate the usage of this function for \eqref{eq:heating_system} and $n_c$ equal to 1. As initial controller matrices we use 
\[
A_c=-1\text{, }B_c = \begin{bmatrix}
1 &1 &1 & 1 & 1
\end{bmatrix}\text{, }C_c = 1,
\]
and $D_c$ the static feedback matrix obtained before. 
\begin{lstlisting}
nc = 1;
init_dyn = ss(-1,ones(1,5),1,Dc);
[K_dyn,cl4]=tds_stabopt_dynamic(plant,nc,'options',opt,'initial',init_dyn);
\end{lstlisting}
The resulting controller \verb|K_dyn| is now a \verb|ss|-object from the \matlab{} Control System Toolbox and we can use the \matlab{} Command Window to inspect the values of the matrices in \eqref{eq:dynamic_output_feedback_controller}.
\begin{lstlisting}
>> Ac = K_dyn.A
   -0.6256
>> Bc = K_dyn.B
    6.4688   -5.5782  -20.9171   12.4451    0.4695
>> Cc = K_dyn.C
    0.0442
>> Dc = K_dyn.D 
   -0.9265   -1.5847   -1.8875   -5.5270    0.1415
\end{lstlisting}
The closed-loop spectral abscissa is now decreased to 
\begin{lstlisting}
>> tds_strong_sa(cl4,-0.1)
  -7.188430e-02
\end{lstlisting}
which is to be expected, as the number of free control parameters has increased. 

In the examples above we only used the nonsmooth, nonconvex BFGS-solver of HANSO. The result after BFGS can in some cases be further improved by using the gradient sampling algorithm. Because gradient sampling is computationally expensive, it is disabled by default in \packageName{}. However, by setting the option \verb|'gradient_sampling'| to true, gradient sampling can be enabled. Let us now retake the design of the static output feedback controller.
\begin{lstlisting}
opt2 = tds_stabopt_options('nstart',1,'gradient_sampling',true);
[Dc5,cl5] = tds_stabopt_static(plant,'options',opt2,'initial',initial);
\end{lstlisting}
This results in the feedback matrix
\begin{equation}
\label{eq:gradsamp}
D_c = \begin{bmatrix}
 -0.4838 & -2.3657 & -3.8564 & -4.9999 & 0.2037   
\end{bmatrix}
\end{equation}
and the spectral abscissa of the closed-loop system is equal to $-6.088\times10^{-2}$. \Cref{fig:gradsamp} shows the rightmost characteristic roots of the resulting closed-loop system. Notice that the optimum corresponds to 4 pairs of complex conjugate characteristic roots with the same real part.

\begin{figure}
    \centering
    \includegraphics[width=0.6\linewidth]{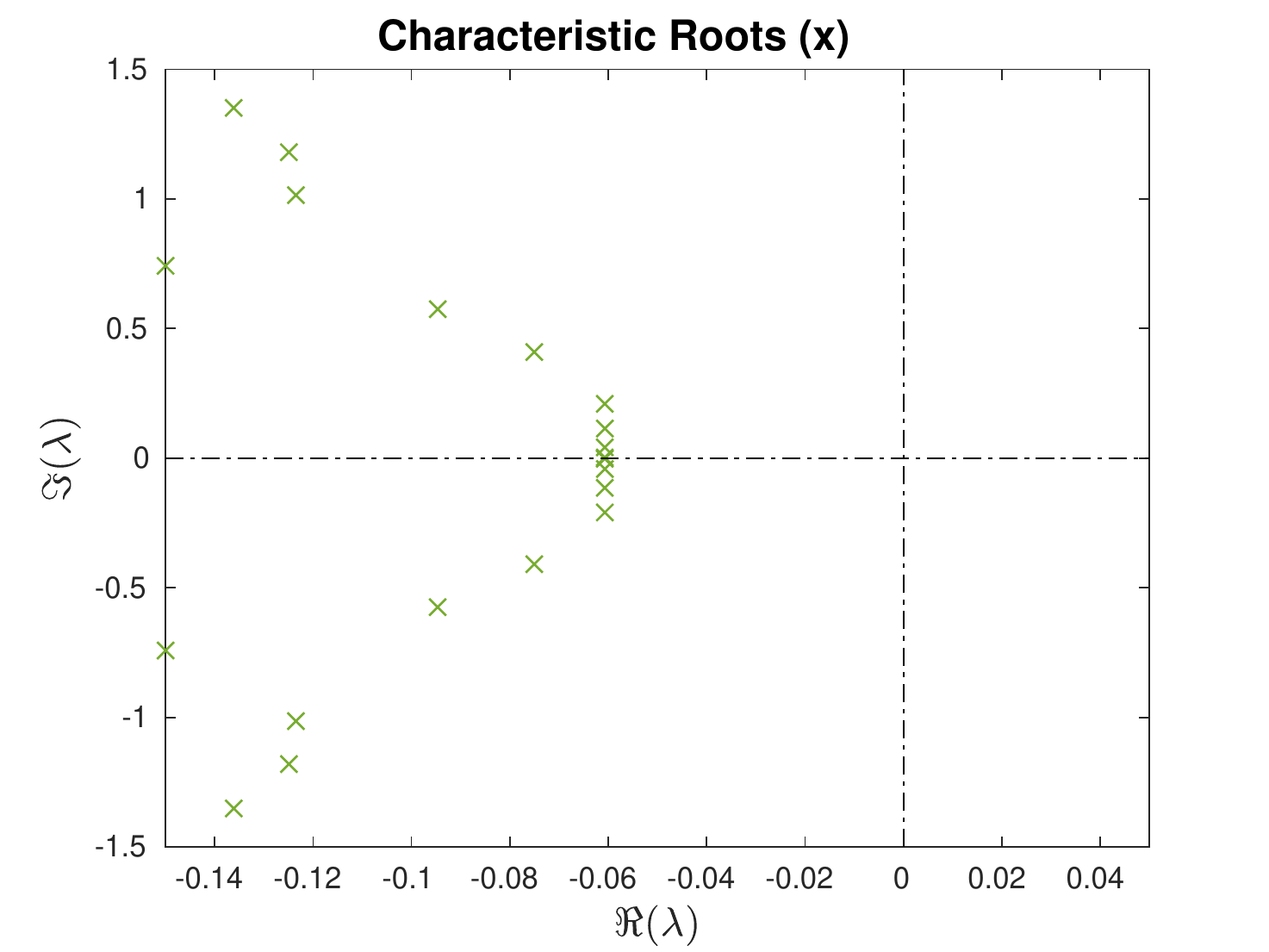}
    \caption{Rightmost characteristic roots of the feedback interconnection of \eqref{eq:heating_system} and \eqref{eq:gradsamp}.}
    \label{fig:gradsamp}
\end{figure}

\textbf{Important:} The result obtained by running the gradient sampling algorithm may differ over multiple runs of the optimization procedure, even if the same initial controller is used. This is due to the random sampling inside the gradient sampling algorithm. For sake of reproducibility, we will therefore refrain from using gradient sampling in the remainder of this manual.
\end{example}
\section{Essentially neutral closed-loop systems}
\label{subsec:neutral_stabilization}
In this section we will consider the case in which the DDAE \eqref{eq:cl} is essentially neutral. From \Cref{sec:neutral_stability,subsec:stability_DDAE} we know that for such systems exponential stability might be sensitive to infinitesimal delay perturbations. \packageName{} will therefore design controllers that preserve exponential stability in the presence of sufficiently small delay perturbations, \ie it will design controllers for which the closed-loop system is strongly exponentially stable.

Recall that in \Cref{sec:neutral_stability} we encountered two necessary and sufficient conditions for strong exponential stability, namely,
\begin{itemize}
    \item \textbf{\Cref{proposition:strong_stability_alt}}: the \emph{strong} spectral abscissa of the system is strictly negative, or,
    \item \textbf{\Cref{proposition:strong_stability}}: the spectral abscissa of the system is strictly negative and the inequality $\gamma(0)<1$ holds with $\gamma(r)$ as defined in \eqref{eq:gamma_r_ddae}.
\end{itemize}
To design a strongly stabilizing controller one thus could consider two approaches.
\begin{itemize}
    \item \textbf{Approach 1:} minimize the \emph{strong} spectral abscissa of the closed-loop systems.
    \item \textbf{Approach 2:} minimize the spectral abscissa of the closed-loop system under the constraint that $\gamma(0;A_{cl,1}^{(22)},\dots,A_{cl,m_{cl}}^{(22)},\vec{\tau}_{cl})$ is strictly smaller than 1.
\end{itemize}
The advantage of the first approach is that it guarantees an exponential decay rate for the trajectories of the closed-loop system, even under sufficiently small (delay) perturbations. However, this approach typically has a higher computational cost, especially when the delay difference equation associated with \eqref{eq:cl} contains more than two terms. Indeed, recall from \Cref{proposition:strong_spectral_abscissa} that computing the strong spectral abscissa of the underlying delay difference equation requires multiple evaluations of the function $r\mapsto \gamma(r)$, which itself requires solving a nonconvex optimization problem. When using the second approach only one evaluation of $\gamma(r)$ is requiered in each step of the optimization routine. Yet, when using the second approach, there is no guarantee for the exponential decay rate of trajectories of the closed-loop system that is robust against arbitrarily small delay perturbations. It is only guaranteed that the strong spectral abscissa is smaller than 0. 

Because the user does not need to indicate whether the closed-loop system is essentially retarded or essentially neutral, the functions \matlabfun{tds_stabopt_static} and \matlabfun{tds_stabopt_dynamic}  automatically selects one of the two approaches mentioned based on the number of terms in the associated delay difference equation in case of an essentially neutral closed-loop system. In the following example, we will here manually select which approach is applied using the option \matlabfun{'method'}:
\begin{itemize}
	\item \verb|'CD'| will apply \textbf{Approach 1}, and
	\item \verb|'barrier'|\footnote{For more details about the used nomenclature, see \Cref{sec:tds_stabopt}.} will apply \textbf{Approach 2},
\end{itemize}
to illustrate the difference between the two approaches.For more information on the implementation of these two approaches, we refer the interested reader to \Cref{sec:tds_stabopt}.
\begin{example}
\label{example:stab_neutral}
Consider the following time-delay system from \cite[Section~7.4.3]{bookdelay}
\begin{equation}
\label{eq:example_stabopt_neutral}
\left\{
\begin{array}{rcl}
\dot{x}(t) &=& \begin{bmatrix}
-0.08 & -0.03 & 0.2 \\
0.2 & -0.04 & -0.005 \\
-0.06 & 0.2 & -0.07
\end{bmatrix} x(t) + \begin{bmatrix}
-0.1 \\
-0.2 \\
\phantom{-}0.1
\end{bmatrix} u(t-5), \\[20px]
y(t) &=& x(t) + \begin{bmatrix}
3 \\ 4 \\ 1
\end{bmatrix}
u(t-2.5) +
\begin{bmatrix}
\phantom{-}0.4\\ -0.4 \\-0.4
\end{bmatrix} u(t-5).
\end{array}
\right.
\end{equation}
To represent this system we can use the following code.
\begin{lstlisting}
A = [-0.08 -0.03 0.2;0.2 -0.04 -0.005;-0.06 0.2 -0.07]; 
B = [-0.1;-0.2;0.1]; C = eye(size(A,1));
D1 = [3;4;1]; D2 = [0.4;-0.4;-0.4];
tau1 = 2.5; tau2 = 5;
sys = tds_create({A},0,{B},tau2,{C},0,{D1,D2},[tau1 tau2]);
\end{lstlisting}
The (strong) spectral abscissa of the open-loop system  --notice that the plant itself is described by an ODE-- is equal to 0.10806.

Next, we will design a strongly stabilizing static controller using Approach 1 starting from a zero feedback gain matrix.
\begin{lstlisting}
initial = zeros(1,3);
options1 = tds_stabopt_options('method','CD','nstart',1);
[Dc1,cl1] = tds_stabopt_static(sys,'options',options1,'initial',initial);
\end{lstlisting}
The resulting feedback matrix (rounded to four decimal places) is equal to
\[
D_c = \begin{bmatrix}
0.0409 &   0.0612 &   0.3837
\end{bmatrix}
\]
and \Cref{tab:stabilization_neutral} gives the quantity $\gamma(0;A_{cl,1}^{(22)},\dots,A_{cl,m_{cl}}^{(22)},\vec{\tau}_{cl})$, the spectral abscissa ($c$), the strong spectral abscissa of the associated delay difference equation ($C_D$) and the strong spectral abscissa ($C$). As this last quantity is strictly negative, we conclude that the closed-loop system is exponentially stable and stability is preserved when considering (sufficiently small) delay perturbations. We also notice that the obtained (local) minimum of $C$ is characterized by the (almost) equality of $c$ and $C_D$. 

Next, we try the second approach:
\begin{lstlisting}
options2 = tds_stabopt_options('method','barrier','nstart',1);
[Dc2,cl2] = tds_stabopt_static(sys,'options',options2,'initial',initial);
\end{lstlisting}
Starting from the zero initial feedback matrix we now do not obtain a strongly stabilizing controller, as indicated by the warning we get. We therefore rerun the optimization procedure starting from 5 random initial feedback matrices:
\begin{lstlisting}
% Create a cell containing five random 1x3-matrices
initial = mat2cell(randn(1,3*5),1,3*ones(1,5)) 
options2 = tds_stabopt_options('method','barrier','nstart',5);
[Dc2,cl2] = tds_stabopt_static(sys,'options',options2,'initial',initial);
\end{lstlisting}
The minimizing feedback matrix is now equal to
\[
D_c = \begin{bmatrix}
0.0253 & 0.1066 & 0.3188
\end{bmatrix}
\]
and \Cref{tab:stabilization_neutral} gives again the value of some quantities related to the stability of the resulting closed-loop system. Compared to Approach 1 the spectral abscissa ($c$) is decreased to $-0.0344$. This comes however at the cost of an increase in the strong spectral abscissa ($C$). This was to be expected as the strong spectral abscissa is now not explicitly taken into account during the optimization process as the constraint $\gamma(0;A_{cl,1}^{(22)},\dots,A_{cl,m_{cl}}^{(22)},\vec{\tau}_{cl})<1$ only guarantees $C_D(A_{cl,1}^{(22)},\dots,A_{cl,m_{cl}}^{(22)},\vec{\tau}_{cl})<0$.
\begin{table}[!ht]
    \centering
	\renewcommand{\arraystretch}{1.2}
    \begin{tabular}{ccc}
    	\toprule
    & Approach 1 & Approach 2 \\ \midrule
        $\gamma(0)$ &  0.9554 & 0.9905 \\
        $c$ & -0.0309 & -0.0344 \\
        $C_D$ & -0.0309 & -0.0066 \\
        $C$ & -0.0309 & -0.0066 \\ \bottomrule
    \end{tabular}
    \caption{The quantity $\gamma(0;A_{cl,1}^{(22)},\dots,A_{cl,m_{cl}}^{(22)},\vec{\tau}_{cl})$, the spectral abscissa ($c$), the strong spectral abscissa of the associated delay difference equation ($C_D$) and the strong spectral abscissa ($C$) of the closed-loop systems obtained by designing a static output feedback controller using Approaches~1~and~2.}
    \label{tab:stabilization_neutral}
\end{table}
\end{example}

Next we will consider a synthesis problem that is related to the design of strongly stabilizing controllers for time-delays systems. More specifically, we will consider the feedback interconnection of a delay-free system and a static output feedback controller. Although the nominal closed-loop system is exponentially stable, it has a delay margin of zero, meaning that it can is destabilized by introducing an arbitrary small feedback delay. First we will illustrate this phenom by analyzing the spectrum of the nominal closed-loop system and the effect of introducing a feedback delay. Next, we will see how  \packageName{} avoids this fragility problem when designing a control law. 
\begin{example}
	\label{example:robustness_feedback_delay}
Consider the following delay-free system
\begin{equation}
\label{eq:fragility}
\left\{
\begin{array}{rcl}
\dot{x}(t) &=& \underbrace{\begin{bmatrix}
1.25 & -0.8 & -0.95 \\
0.175 & -0.4 & -0.125 \\
-1.15 & -0.4 & 0.65
\end{bmatrix}}_A x(t) + \underbrace{\begin{bmatrix}
2\\
0 \\
-2
\end{bmatrix}}_B u(t) \\[30px]
y(t) &=& \underbrace{\begin{bmatrix}
-7 & 25 & -11
\end{bmatrix}}_{C} x(t) + \underbrace{1}_D u(t).
\end{array}
\right.
\end{equation}
and the stabilizing static output feedback law:
\begin{equation}
\label{eq:fragility_feedback}
u(t) = \underbrace{-5}_{D_c} y(t).
\end{equation}
The corresponding characteristic roots are $-1.3622$ and $ -1.9023  \pm \jmath 0.2871$, which confirms that the closed-loop system is indeed exponentially stable. However, let us now introduce a small delay $\tau$ in the feedback loop, \ie $u(t) = D_c y(t-\tau)$  with $\tau>0$. The dynamics of the resulting closed-loop system can be expressed in terms of the following DDAE: 
\[
\begin{bmatrix}
\mathrm{I}_3 & 0 &0 \\
0 & 0 & 0 \\
0 & 0 & 0
\end{bmatrix} \begin{bmatrix}
\dot{x} \\
\dot{y} \\
\dot{u}
\end{bmatrix}(t)
 = 
 \begin{bmatrix}
 A & 0&B \\
 C & -1 & D \\
 0 & 0 & -1
 \end{bmatrix}\begin{bmatrix}
x \\
y \\
u
\end{bmatrix}(t) +
\begin{bmatrix}
 0 & 0&0 \\
 0 & 0 & 0 \\
 0 & D_c & 0
 \end{bmatrix}\begin{bmatrix}
x \\
y \\
u
\end{bmatrix}(t-\tau).
\]
The associated delay difference equation is given by
\begin{equation}
\label{eq:fragility_deldiff}
\begin{bmatrix}
-1 & D \\
0 & -1
\end{bmatrix} \begin{bmatrix}
y \\
u
\end{bmatrix}(t) + 
\begin{bmatrix}
0 & 0 \\
D_c & 0 \\
\end{bmatrix} \begin{bmatrix}
y\\
u
\end{bmatrix}(t-\tau) = 0
\end{equation}
and the spectrum of \eqref{eq:fragility_deldiff} consists of a chain of characteristic roots at
\begin{equation}
    \label{eq:fragility_chain}
\frac{\ln\lvert DD_c \rvert +  \jmath (\angle(DD_c) + 2\pi l)}{\tau} = \frac{\ln 5+\jmath (2l+1)\pi l}{\tau} \text{ for } l \in \Z,
\end{equation}
with $\ln 5 = 1.6094$, meaning that $C_D>0$ for \emph{any} $\tau>0$. We conclude from the theory of the previous section, that the closed-loop system with feedback delay is not asymptotically stable, no matter how small the feedback delay is made. The closed-loop interconnection of \eqref{eq:fragility} and \eqref{eq:fragility_feedback} is said to have a zero-stability margin. \packageName{} avoids this fragility problem by assuming an infinitesimal feedback delay into account during the design process. It now follows from the discussion at the beginning of this chapter that a natural approach to design a non-fragile controller for \eqref{eq:fragility} consists of minimizing the spectral abscissa of the corresponding closed-loop system under the constraint
\[
\gamma(0) = \max_{\theta \in [0,2\pi) } \rho\left(\begin{bmatrix}
-1 & D \\
0 & -1
\end{bmatrix}^{-1}\begin{bmatrix}
0 & 0 \\
D_c & 0 \\
\end{bmatrix} e^{-\jmath\theta}\right) = \lvert DD_c \rvert < 1.
\]
To solve this optimization problem, we can use \matlabfun{tds_stabopt_static} with Approach 2, as \packageName{} always assumes an infinitesimal feedback delay during the design process.
\begin{lstlisting}
A = [1.25 -0.8 -0.95;0.175 -0.4 -0.125;-1.15 -0.4 0.65];
B = [2; 0; -2]; C = [-7 25 -11]; D = 1;   

sys = tds_create({A},[0],{B},[0],{C},[0],{D},[0]);
options = tds_stabopt_options('method','barrier','nstart',1);
[Dc,cl]=tds_stabopt_static(sys,'options',options,'initial',0);
\end{lstlisting}
We now find the feedback gain $D_{c}= -0.9979$, for which the (zero-delay) spectral abscissa is equal to $-0.8279$. Notice that in contrast to the previous example, stability is preserved for a sufficiently small feedback delay as $\gamma(0)=0.9979<1$. For example, for $\tau=0.01$ the strong spectral abscissa is equal to $-0.2022$, as becomes clear by executing the following code.
\begin{lstlisting}
cl.hA(2) = 1e-2;
options = tds_roots_options('fix_N',40);
cr = tds_roots(cl,-1,options);

tds_eigenplot(cr)
hold on
sa = tds_strong_sa(cl,-2)
xlim([-1 0.5])
ylim([-3000 3000])
plot([sa sa],ylim,'b--')
\end{lstlisting}
\end{example}

\section{Structured controllers}
\label{subsec:structured_controllers}
In this section we will discuss how various controller structures can be designed using the \packageName{} framework. We start with designing a Pyragas-type feedback controller. We then design a delayed output feedback controller for an inverted pendulum set-up. Next, we consider an example with acceleration feedback. Subsequently, we will discuss how one can impose structure on the matrices of \eqref{eq:dynamic_output_feedback_controller}, \ie only certain entries of the controller matrices are optimized while the other entries remain fixed to a given value. We will then apply this technique to design PID and decentralized controllers using \packageName{}.

\subsection{Pyragas-type feedback controllers}
In this example we will design a Pyragas-type feedback controller \cite{Pyragas} for a delay-free system. Such controllers have the following standard form
\begin{equation}
\label{eq:pyragas_feedback}
    u(t) = D_c \big(x(t)-x(t-\tau)\big)
\end{equation}
and are frequently used for the stabilization of (unstable) periodic orbits of chaotic systems as they preserve the shape of the cycle when $\tau$ is chosen equal to the period of the orbit. 
\begin{example}
\label{ex:pyragas_feedback}
Consider the following delay-free system
\begin{equation}
\label{eq:example_pyragas}
\dot{x}(t) = \begin{bmatrix}
-3.5 & -6.5 \\
\phantom{-}4.5 & \phantom{-}5.5 
\end{bmatrix} x(t) + \begin{bmatrix}
\phantom{-}1 \\
-1
\end{bmatrix} u(t).
\end{equation}
It follows from \cite[Figure~8.9]{bookdelay} that for $\tau=1$ this system is stabilizable using Pyragas-type state feedback. Let us now design such a stabilizing controller using \packageName{}. To this end, we introduce the following measured output
\[
y(t) = x(t) - x(t-\tau),
\]
which allows to reformulate \eqref{eq:pyragas_feedback} as a (static) output feedback controller. Next we can use the same functionality as before to design our controller.
\begin{lstlisting}
tau = 1;
A = {[-3.5 -6.5; 4.5 5.5]};
B = {[1;-1]};
% introduce the output y(t) = x(t) - x(t-tau)
C = {eye(2),-eye(2)}; hC = [0 tau];
% Create a tds_ss-object representing the open-loop system
sys = tds_create(A,0,B,0,C,hC);
% Design a static output feedback controller for sys
o = tds_stabopt_options('nstart',1);
[Dc,CL]= tds_stabopt_static(sys,'options',o,'initial',{[0 0]});
% Analyse the resuling closed-loop system
C = tds_strong_sa(CL,-2)
l = tds_roots(CL,-4.5);
tds_eigenplot(l)
\end{lstlisting}
The resulting feedback gain is equal to
\[
D_c = \begin{bmatrix}
-0.5917 & 0.5347
\end{bmatrix}
\]
and the (strong) spectral abscissa of the closed-loop system is equal to -0.5234. Notice that for the resulting feedback law the spectral abscissa  corresponds to a pair of complex conjugate roots with multiplicity 2.
\end{example}
\subsection{Delayed feedback control}
In the following example we will illustrate how to design a delayed feedback controller in \packageName{} by extending the output variable in the state-space model for the plant. Delayed feedback control has several applications. Firstly, \cite{Olgac1994DelayedResonator} showed that purposely adding a delay into the control loop can be beneficial for stabilizing oscillations. Secondly, in \cite{ramirez2015} the Proportional-Integral-Retarded (PIR) controller was proposed as an alternative for the Proportional-Integral-Derivative (PID) controller. Such PIR controllers seek to address the sensitivity to high frequency noise of which PID controllers suffer, by replacing the derivative of the output signal by a delayed version of the output signal. 

\begin{example}
\label{ex:delayed_feedback}
Consider the set-up depicted in \Cref{fig:mass_spring_pendulum} that consists of a cart that balances an inverted pendulum and that is connected to two walls using two identical springs. Assuming that the angular displacement $\theta$ remains sufficiently small, the dynamics of this system can be approximated by the following LTI model 
\begin{subequations}
\label{eq:example_pr}
 \begin{empheq}[left={\empheqlbrace\,}]{align}
\label{eq:state}
\begin{bmatrix}
\vphantom{\tfrac{(m+M)g}{Ml}}\dot{x}_1(t) \\
\vphantom{\tfrac{(m+M)g}{Ml}}\dot{x}_2(t) \\
\vphantom{\tfrac{(m+M)g}{Ml}}\dot{x}_3(t) \\
\vphantom{\tfrac{(m+M)g}{Ml}} \dot{x}_4(t)
\end{bmatrix}=&
\underbrace{
\begin{bmatrix}
\vphantom{\tfrac{(m+M)g}{Ml}} 0 & 1 & 0 & 0 \\
\vphantom{\tfrac{(m+M)g}{Ml}} \tfrac{-2k}{M} & 0 & \dfrac{-mg}{M} & 0\\
\vphantom{\tfrac{(m+M)g}{Ml}}0 & 0 & 0 & 1 \\
\tfrac{2k}{Ml} & 0 & \tfrac{(m+M)g}{Ml} & 0
\end{bmatrix}}_{A} \begin{bmatrix}
\vphantom{\tfrac{(m+M)g}{Ml}} x_1(t) \\
\vphantom{\tfrac{(m+M)g}{Ml}} x_2(t) \\
\vphantom{\tfrac{(m+M)g}{Ml}} x_3(t) \\
\vphantom{\tfrac{(m+M)g}{Ml}} x_4(t)
\end{bmatrix} + \underbrace{\begin{bmatrix}
\vphantom{\tfrac{(m+M)g}{Ml}} 0 \\
\vphantom{\tfrac{(m+M)g}{Ml}} \tfrac{1}{M} \\
\vphantom{\tfrac{(m+M)g}{Ml}} 0 \\
\vphantom{\tfrac{(m+M)g}{Ml}} \tfrac{-1}{Ml}
\end{bmatrix}}_{B} u(t)  \\
\label{eq:position_output}
y(t) =& \underbrace{\begin{bmatrix}
1 & 0 & 0 & 0 \\
0 & 0 & 1 & 0
\end{bmatrix}}_{C}
x(t),
\end{empheq}
\end{subequations}
with $k>0$ the spring constant of the two springs, $m>0$ the mass of the pendulum's bob, $M>0$ the mass of the cart,  $x_1=p$ the horizontal displacement of the cart with respect to its equilibrium position, $x_2=\dot{p}$ its horizontal velocity, $x_3=\theta$ the angular displacement of the pendulum, $x_4=\dot{\theta}$ the angular velocity and $u$ a controlled force acting on the cart. As measured output, the position of the cart and the angular displacement of the pendulum are available.

To stabilize the pendulum in the upright position, \ie $x\equiv0$, we will first try a (delay-free) static output feedback law, \ie 
\begin{equation}
\label{eq:static_feedback}
u(t) = \underbrace{\begin{bmatrix} d_1 & d_2 \end{bmatrix}}_{D_c} y(t).
\end{equation}
Plugging \eqref{eq:position_output} and \eqref{eq:static_feedback} into \eqref{eq:state} we find that the characteristic polynomial of the closed-loop system is given by
\[
\lambda^{4}+\lambda^{2} \left(\frac{2k-d_1}{M}-\frac{(M+m)g-d_2}{Ml}\right)-\frac{(2k-d_1)g}{Ml}.
\]
As this polynomial is bi-quadratic, it has at least two roots in the closed right half-plane, independent of the value of the control parameters $d_1$ and $d_2$, meaning taht the considered system can thus not be stabilized using (undelayed) position-angle feedback. A solution could be to resort to full state-feedback. This would however require that the velocities $\dot{p}$ and $\dot{\theta}$ are known. As we assumed that these quantities can not be measured directly, this would require a finite difference approximation, \eg $\dot{p}(t) \approx \frac{p(t)-p(t-\tau)}{\tau}$ with $\tau>0$ a sufficiently small delay. Such an approximation is however sensitive to high frequency measurement noise. We therefore opt for an alternative approach: we will consider the following proportional-delayed proportional feedback law
\begin{equation}
\label{eq:pr_feedback}
u(t) = D_{c,0} y(t) + D_{c,1} y(t-\tau).
\end{equation}
At first sight, this control law does not fit either \eqref{eq:dynamic_output_feedback_controller} or \eqref{eq:static_output_feedback_controller}. However, by defining the extended output
\begin{equation}
\label{eq:extended_position_output}
y_{e}(t) = \begin{bmatrix}
C \\
0
\end{bmatrix} x(t) +  \begin{bmatrix}
0 \\
C
\end{bmatrix} x(t-\tau), 
\end{equation}
feedback law \eqref{eq:pr_feedback} corresponds to the static output feedback law
\[
u(t) = \begin{bmatrix}
    D^{c}_0 & D^{c}_1
\end{bmatrix} y_e(t).
\]
Using this observation, we can now use the following code to design such a controller and verify the stability properties of resulting closed-loop system for the system parameters $k =$ \SI[inter-unit-product =\cdot]{1000}{\newton\per\metre}, $l =$ \SI{0.4}{\metre}, $g =$ \SI[inter-unit-product =\cdot]{9.8}{\meter\per\second\squared}, $m = $ \SI{0.1}{\kilogram}, $M =$ \SI{1}{\kilogram},  and $\tau =$ \SI{0.1}{\second}:
\begin{lstlisting}
% System constants
k = 1000; l = 0.4; g = 9.8; m = 0.1; M = 1; tau = 0.1;
% System matrices
A = [   0     1       0       0;
    -2*k/M    0   -m*g/M      0;
        0     0       0       1;
    2*k/(M*l) 0 (m+M)*g/(M*l) 0];
B = [0;1/M;0;-1/(M*l)];
C = [1 0 0 0;
     0 0 1 0];
% Define the extended output matrices
Ce0 = [C;zeros(size(C))]; Ce1 = [zeros(size(C));C];
% Use tds_create to represent the extended plant 
P = tds_create({A},[0],{B},[0],{Ce0, Ce1},[0 tau]);
%% Design a stabilizing controller
o = tds_stabopt_options('nstart',1);
[Dc,cl]=tds_stabopt_static(P,'options',o,'initial',zeros(1,4));

%% Examine the resulting closed-loop system
cr = tds_roots(cl,-70);
tds_eigenplot(cr);
SA = tds_strong_sa(cl,-2);
\end{lstlisting}
The resulting feedback matrices are given by
\begin{equation}
\label{eq:pr_matrix}
D_c =  10^{3} \times \begin{bmatrix}
\smash[b]{\underbrace{\begin{matrix} 8.2456 & 4.9617
\end{matrix}}_{D_0^c}} & \smash[b]{\underbrace{\begin{matrix} -5.8168 &  -2.5176
\end{matrix}}_{D_1^c}}
\end{bmatrix}
\vspace{10pt}
\end{equation}
and the strong spectral abscissa of the closed-loop system is equal to $-1.4932$.

\begin{figure}[]
\centering
	\begin{tikzpicture}
		\draw (0,3) -- (0,0) -- (5,0) -- (5,3);
		\fill[LineSpace=9pt, pattern=my north west lines,pattern color=gray] (0,3) -- (0,0) -- (5,0) -- (5,3) -- (5.5,3) -- (5.5,-0.5) -- (-0.5,-0.5) -- (-0.5,3) -- cycle;
		
		\draw[decoration={aspect=0.4, segment length=2mm, amplitude=2mm,coil},decorate] (0.15,0.9) -- (1.5,0.9); 
		\node (k1) at (0.75,0.45) {$k$};

		\draw (2,0.3) circle (0.3);
		\draw (3,0.3) circle (0.3);
		\draw (1.5,0.6) rectangle (3.5,1.2);
		\node (M) at (2.5,0.9) {$M$};
		\filldraw (3.25,3) circle (0.15);
		\draw (2.5,1.2) -- (3.25,3);
		\node (m) at (3.35,3.35) {$m$};
		\node [rotate=67.38] (l) at (3.15,2.1) {$l$};
		\draw[dashed] (2.5,1.2) -- (2.5,2.5);
		\draw [domain=67.38:90] plot ({2.5+0.7*cos(\x)}, {1.2+0.7*sin(\x)});
		\node (theta) at (2.72,2.3) {$\theta$};
		\draw (2.5,-0.7) -- (2.5,-1.1);
		\draw [-{>[scale=2]}] (2.5,-0.9) -- (3.5,-0.9);
		\node (x) at (3,-1.3) {$p$};
		\draw (3.5,1.2) -- (3.5,1.6);
		\draw[->] (3.5,1.4) -- (4.3,1.4);
		\node (F) at (4.2,1.65) {$u$};
		
		\draw[decoration={aspect=0.4, segment length=2mm, amplitude=2mm,coil},decorate] (3.65,0.9) -- (5,0.9); 
		\draw (3.5,0.9) -- (3.65,0.9);
		\node (k2) at (4.25,0.45) {$k$};
		\end{tikzpicture}
	\caption{Inverted pendulum set--up in which the cart is connected to two walls using two identical springs.}
	\label{fig:mass_spring_pendulum}
	\end{figure}
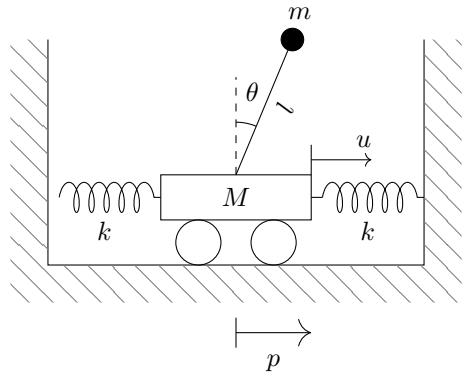

\end{example}

\subsection{Acceleration feedback}
\label{subsec:acceleration_feedback}
As accelerometers are readily-available and cheap \cite{vyhlidalAcceleration}, they are a popular choice for vibration suppression. In the next example we will therefore design an acceleration feedback controller for a mass-spring system.
\begin{example}
\label{example:acceleration_feedback}
Consider the following mass-spring system.
\begin{subequations}
\label{eq:mass_spring}
 \begin{empheq}[left={\empheqlbrace\,}]{align}
     \dot{x}     
     (t) = & \begin{bmatrix}
     0 & 1 \\
     \frac{-k}{m} & 0
     \end{bmatrix} x(t) + \begin{bmatrix}
     0 \\
     \frac{1}{m}
     \end{bmatrix}u(t-\tau) \label{eq:mass_spring_1} \\[7px]
     y(t) = & \begin{bmatrix}
     0 & 1
     \end{bmatrix} \dot{x}(t), \label{eq:mass_spring_2}
\end{empheq}
\end{subequations}
with $k =$ \SI[inter-unit-product =\cdot]{1000}{\newton\per\metre} and $m =$ \SI{1}{\kilogram}. Notice that the measured output now depends on $\dot{x}(t)$ meaning that  \eqref{eq:mass_spring} does not fit in the framework of \eqref{eq:state_space_system}. A first idea is to eliminate $\dot{x}(t)$ in \eqref{eq:mass_spring_2} using \eqref{eq:mass_spring_1} which introduces a direct feedthrough term
\[
y(t) = \begin{bmatrix}
0 & 1
\end{bmatrix} \left(A x(t) + B u(t-\tau)\right) = \begin{bmatrix}
-\frac{k}{m} & 0
\end{bmatrix} x(t) +
\frac{u(t-\tau)}{m}.
\]
Alternative we can use a DDAE reformulation by introducing the auxiliary variable $a$
\begin{equation}
\label{eq:ddae_reformulation_acceleration}
\left\{
\begin{array}{rcl}
\underbrace{
\begin{bmatrix}
1 & 0 & 0 \\
0 & 1 & 0 \\
0 & 1& 0  
\end{bmatrix}
}_{\texttt{Ea}}
\left[
\begin{array}{c}
\multirow{2}{*}{$\dot{x}(t)$}\\
\\
\dot{a}(t)
\end{array}\right] &=&
\underbrace{
\begin{bmatrix}
0 & 1 & 0 \\
\frac{-k}{m} & 0 & 0 \\
0 & 0& 1  
\end{bmatrix}
}_{\texttt{Aa}}
\begin{bmatrix}
\multirow{2}{*}{$x(t)$}\\
\\
a(t)
\end{bmatrix}
+ \underbrace{\begin{bmatrix}
0 \\ \frac{1}{m} \\ 0
\end{bmatrix}}_{\texttt{Ba}}u(t-\tau) \\[15pt]
y(t) &=& \underbrace{\begin{bmatrix}
0 & 0 & 1
\end{bmatrix}}_{\texttt{Ca}}
\begin{bmatrix}
\multirow{2}{*}{$x(t)$}\\
\\
a(t)
\end{bmatrix},
\end{array}
\right.
\end{equation}
in which \verb|Ea|, \verb|Aa|, \verb|Ba|, and \verb|Ca| correspond to the matrices in the code below. Here we will opt for the second approach in order to demonstrate the usage of the function \verb|tds_create_ddae| for systems with in- and outputs. The acceleration feedback controller now reduces to the static output feedback law $u(t) = D_c y(t)$. As the closed-loop system is not stabilizable for $\tau=0$, we artificially introduce a feedback delay $\tau =$ \SI{0.1}{\second}. Note that for $\tau\neq 0$ the DDAE describing the resulting closed-loop system is essentially neutral. To guarantee that the closed loop is \emph{strongly} exponentially stable, we will use \textbf{Approach 1}.
\begin{lstlisting}
% System constants
k = 100; m = 1; tau = 0.1;
% Define transformed system matrices
Ea = [1 0 0;0 1 0;0 1 0];
Aa = [0 1 0;-k/m 0 0;0 0 1];
Ba = [0; 1/m; 0];
Ca = [0 0 1];
% Use tds_create_ddae to represent the DDAE (3.20)
P = tds_create_ddae(Ea,{Aa},0,{Ba},tau,{Ca},0);
% Design a stabilizing controller
o=tds_stabopt_options('nstart',1,'method','CD');
[Dc,cl] = tds_stabopt_static(P,'options',o,'initial',0);
% Examine the resulting closed-loop system
CDDAE = tds_strong_sa(cl,-4);
\end{lstlisting}
The resulting feedback matrix is given by 
\[
Dc = -0.7886
\]
and the strong spectral abscissa of the closed loop is equal to $-1.3061$. In this case, the optimimum corresponds to two pairs of complex conjugate rightmost characteristic roots and $c>C_D$. We conclude that the resulting closed-loop system is strongly stable. 
\end{example}
\subsection{Imposing structure on the controller matrices}
In some cases, only certain entries of the controller matrices in \eqref{eq:dynamic_output_feedback_controller} are allowed to be optimized, while the other entries must remain fixed to a given value. As we will see below, by imposing such structure on the controller matrices, PID and decentralized controllers can be designed using \packageName{}. To this end, we now illustrate the role of the optional arguments \verb|'mask'| and \verb|'basis'| in more detail. The \verb|'mask'| argument is used to indicate which entries  are allowed to be optimized, while the \verb|'basis'| argument allows to specify the default values for the entries that must remain fixed. We will first demonstrate how to use these arguments on an artificial example, before moving on to the more realistic applications of PID and decentralized control.
\begin{example}
\label{example:fixed_entries}
Consider the following time delay system with two input and two outputs:
\begin{equation}
\label{eq:fixed_entries_sys}
\left\{
\setlength{\arraycolsep}{2pt}
\begin{array}{rcl}
\dot{x}(t) &=& \underbrace{\begin{bmatrix}
-3.5 & -2 & 3 & -1.5\\
3 & -6.5 & 3 & -1 \\
4 & -2 & -7 & 2.5 \\
-2.5 & 1.5 & -3.5 & 1 
\end{bmatrix}}_{A_0} x(t) + \underbrace{\begin{bmatrix}
0.25 & 0.5 & -0.25 & 0.25\\
0.5 & -0.5 & 1 & -0.5 \\
1 & -1 & 0.5 & -0.15 \\
-0.15 & -0.15 & 0.1 & 0.2 
\end{bmatrix}}_{A_1} x(t-1)+ \phantom{a} \\[12mm] && \underbrace{\begin{bmatrix}
-1 & 2 \\
-2 & 1\\
3 & -2.5 \\
-4 & 0
\end{bmatrix}}_{B} u(t-5) \\[12mm]
y(t) &=& \underbrace{\begin{bmatrix}
0 & 1 & 1 & 0 \\
1& 0 &0 & 1
\end{bmatrix}}_{C}x(t).
\end{array}
\right.
\end{equation}
We want to stabilize this system using a dynamic output feedback controller of order two (\ie $n_c=2$) but with the restriction that the matrices $B_c$ and $D_c$ are diagonal. We again start with creating a representation for the open-loop system using \verb|tds_create|.
\begin{lstlisting}
A0=[-3.5 -2 3 -1.5;3 -6.5 3 -1;4 -2 -7  2.5;-2.5 1.5 -3.5  1];
A1=[-0.25 0.5 -0.25 0.25;0.5 -0.5 1 -0.5;
    1 -1 0.5 -0.15;-0.15 -0.15 0.1 0.2];
B = [-1 2;-2 1;3 -2.5;-4  0];
C = [0 1 1 0;1 0 0 1];

sys = tds_create({A0,A1},[0 1],{B},5,{C},0);
\end{lstlisting}
Next, we create the \verb|mask| for our controller, \ie an \verb|ss|-object with entries \verb|1| and \verb|0| to indicate which entries are allowed to be optimized and which should remain fixed.
\begin{lstlisting}
mask = ss(ones(2),eye(2),ones(2),eye(2));
\end{lstlisting}
To specify the basis values for the fixed entries, we create a \verb|ss|-object \verb|basis|, whose elements dictate the values of the fixed entries.
\begin{lstlisting}
basis = ss(zeros(2),zeros(2),zeros(2),zeros(2));
\end{lstlisting}
Next we run the optimization procedure starting from the initial controller
\[
A_c = -I_2\text{ and }
B_c = C_c = D_c = I_2.
\]
\begin{lstlisting}
initial = ss(-eye(2),eye(2),eye(2),eye(2));
options_stabopt = tds_stabopt_options('nstart',1);
[K,cl] = tds_stabopt_dynamic(sys,2,'options',options_stabopt,...
                      'initial',initial,'basis',basis,'mask',mask);
\end{lstlisting}
Inspecting the resulting controller using the \matlab{} Command Window, we see that the resulting controller indeed satisfies the specified pattern.
\begin{align*}
	& A_c = \begin{bmatrix}
	-9.2431 & 11.7967 \\
	0.8572 & -3.1225
	\end{bmatrix}\text{, } B_c = \begin{bmatrix}
	-2.5544 & 0 \\
	0 & 0.2507
	\end{bmatrix}\text{, } \\[4mm]
	& C_c = \begin{bmatrix}
	0.8863 & -2.3617 \\
	-15.8203 & 12.1480
	\end{bmatrix}\text{ and } D_c = \begin{bmatrix}
	0.8420 & 0 \\
	0 & 0.0265
	\end{bmatrix}.
\end{align*}
Analyzing the resulting closed-loop system, we see that it is (strongly) exponential stable as its (strong) spectral abscissa is equal to -0.1786. 
\end{example}

\subsection{PID controller}
Proportional-Integral-Derivative (PID) controllers are frequently used in industry to stabilize a set-point. The feedback signal of a PID controller consists of an affine combination of a term proportional to the difference between the output and the desired set point, a term proportional to the accumulation of the error over time (\ie its integral), and a term proportional to the rate of change of this error (\ie its time-derivative):
\begin{equation}
\label{eq:PID}
u(t) = K_p \, y(t)  + K_i \int_{0}^{t} y(s) \, \dd s + K_d \, \dot{y}(t).
\end{equation}
The derivative component improves the responsiveness of the controller as it allows to \mbox{anticipate} the future value of the error. The integrator on the other hand allows to eliminate a steady-state error for a non-zero reference signal. In the following example we will demonstrate how the design of a PID controller can be transformed into the design of a dynamic output feedback controller of the form \eqref{eq:dynamic_output_feedback_controller} by augmenting the state and output variables and by imposing a certain structure on some of the controller matrices.
\begin{example}
\label{example:pid_control}
Consider the following time-delay system:
\begin{equation}
    \label{eq:example_PID}
    \left\{
\begin{array}{rcl}
     \dot{x}(t)  & = & \underbrace{\begin{bmatrix}
         -3 & \phantom{-}2 & \phantom{-}6 \\ -3 & -4 & \phantom{-}1 \\ \phantom{-}1 & -1 & -2
     \end{bmatrix}}_{A_0} x(t) + \underbrace{\begin{bmatrix}
         \phantom{-}0.5 & -0.7 & -0.8 \\ -0.3 & \phantom{-}0.6 & \phantom{-}0.2 \\ -0.2 & \phantom{-}0.6 & -0.4
     \end{bmatrix}}_{A_1} x(t-1) + \underbrace{\begin{bmatrix}
         1 & \phantom{-}0 \\ 1 & \phantom{-}1 \\ 1 & -1
     \end{bmatrix}}_{B} u(t-0.2)  \\[10mm]
     y(t) & = & \underbrace{\begin{bmatrix}
         1 & 1 & 1
     \end{bmatrix}}_{C} x(t).
    \end{array}
\right.
\end{equation}
To express a controller of the form \eqref{eq:PID} as a dynamic output feedback controller, we first need to extend the state-space model \eqref{eq:example_PID}. As in \Cref{example:acceleration_feedback}, we will use a DDAE reformulation of the plant. More specifically, we will introduce the auxiliary variable $\zeta$, representing the output derivative, and the novel output $y_e$
\begin{equation}
\label{eq:PID_extended_ss}
\left\{
\begin{array}{rcl}
     \underbrace{\begin{bmatrix} 
     I & 0 \\
     C & 0
     \end{bmatrix}}_{\texttt{Et}} \begin{bmatrix}
         \dot{x} \\
         \dot{\zeta}
     \end{bmatrix} (t) & = & \underbrace{\begin{bmatrix}
         A_0 & 0 \\
         0 & I
     \end{bmatrix}}_{\texttt{A0t}} \begin{bmatrix}
         x \\
         \zeta
     \end{bmatrix}(t) + \underbrace{\begin{bmatrix}
         A_1 & 0 \\ 0 & 0
     \end{bmatrix}}_{\texttt{A1t}} \begin{bmatrix}
         x \\
         \zeta
     \end{bmatrix}(t-\tau_1) + \underbrace{\begin{bmatrix}
         B \\ 0
     \end{bmatrix}}_{\texttt{Bt}} u(t-\tau_2)  \\[10mm]
     y_e(t) & = & \underbrace{\begin{bmatrix}
         C & 0 \\ 0 & I
     \end{bmatrix}}_{\texttt{Ct}} \begin{bmatrix}
         x \\
         \zeta
     \end{bmatrix}(t), 
\end{array}
\right.
\end{equation}
in which \verb|Et|, \verb|A0t|, \verb|A1t|, \verb|Bt|, and \verb|Ct| correspond to the matrices in the code below. A PID controller now corresponds to a structured dynamic output feedback controller:
\[
\left\{
\begin{array}{rcl}
\dot{x}_c(t) &=& 0\, x_c(t) + \begin{bmatrix}
    I & 0
\end{bmatrix} y_e(t) \\[6pt]
u(t) &=& K_I x_c(t) + \begin{bmatrix}
    K_P & K_D
\end{bmatrix} y_e(t),
\end{array}
\right.
\]
\ie the matrix $A_c$ should remain fixed to the zero matrix and the matrix $B_c$ should always equal $\begin{bmatrix}
    I & 0
\end{bmatrix}$. The entries of the matrices $C_c$ and $D_c$ on the other hand can be freely tuned. To design such a controller in \packageName{} we can use the following code.
\begin{lstlisting}
% System matrices
A0 = [-3 2 6;-3 -4 1;1 -1 -2];
A1 = [0.5 -0.7 -0.8;-0.3 0.6 0.2;-0.2 0.6 -0.4];
B = [1 0;1 1;1 -1]; C = [1 1 1];
tau1 = 1; tau2 = 0.2;
%Define transformed system matrices
n = size(A0,1); p = size(B,2); q = size(C,1);
Et = [eye(n) zeros(n,q);C zeros(q)];
A0t = [A0 zeros(n,q);zeros(q,n) eye(q)];
A1t = [A1 zeros(n,q); zeros(q,n+q)];
Bt = [B; zeros(q,p)]; Ct = [C zeros(q);zeros(q,n) eye(q)];
% Use tds_create_ddae to create a representation for (3.24)
tds = tds_create_ddae(Et,{A0t A1t},[0 tau1],{Bt},tau2,{Ct});
% Define the mask and basis arguments
mask=ss(zeros(q),zeros(q,2*q),ones(p,q),ones(p,2*q));
basis=ss(zeros(q),[eye(q) zeros(q)],zeros(p,q),zeros(p,2*q));
% Design a structured controller
initial = ss(zeros(q),[eye(q) zeros(q)],ones(p,q),ones(p,2*q));
options_stabopt = tds_stabopt_options('nstart',1);
nc = q;
[K,cl] = tds_stabopt_dynamic(tds,nc,'mask',mask,'basis',basis,...
                      'initial',initial,'options',options_stabopt);
\end{lstlisting}
The resulting controller \verb|K| is an \verb|ss|-object with 
\begin{lstlisting}
K.A =

     0

K.B =

     1     0

K.C =

   -5.1159
   -2.5347

K.D =

   -2.3480   -0.2611
    1.6036   -0.2695
\end{lstlisting}
meaning that
\begin{equation}
\label{eq:PID_controller_stab}
K_P = \begin{bmatrix}
    -2.3480 \\ \phantom{-}1.6036
\end{bmatrix}, \, K_I = \begin{bmatrix}
    -5.1159 \\ -2.5347
\end{bmatrix},  \text{ and } K_D = \begin{bmatrix}
    -0.2611 \\ -0.2695
\end{bmatrix}.
\end{equation}
Using \verb|tds_strong_sa|, we find that the strong spectral abscissa of the corresponding closed loop is equal to -1.1563 meaning that the closed-loop system is strongly exponentially stable.

\end{example}

\subsection{Networked systems and decentralized controllers}
In the following example we will consider a controller design problem for a networked system, \ie a system that consists of multiple interacting subsystem. In such a context, centralized control, in which one global controller steers all subsystems --- see \Cref{fig:centralized_control}, is often impractical or even infeasible, due to  the high communication and processing requirements. Therefore, decentralized control, in which each subsystem is controlled by using only local information --- see \Cref{fig:decentralized_control}, is often preferred as it lowers the communication requirements. Furthermore, the local controllers can process the available information in parallel. However, as each controller can only use local information, this comes at the cost of a decrease in performance. Distributed and overlapping controllers form a middle ground between centralized and decentralized control. Each subsystem is still controlled by a local controller, but these local controllers can now exchange information with neighboring controllers or neighboring subsystems, see \Cref{fig:distributed_control}. This allows for a trade-off between communication and processing requirements on the one hand, and performance on the other hand.
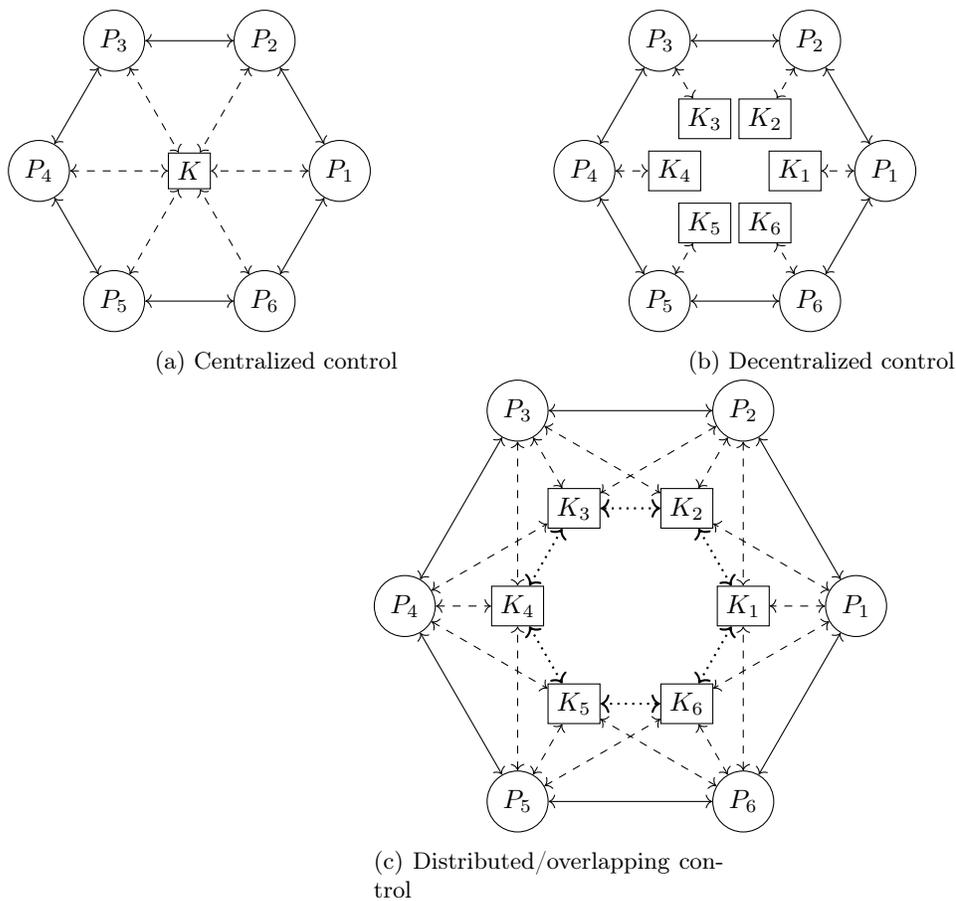
\begin{figure}[t]
    \centering
    \begin{subfigure}{0.49\linewidth}
    \begin{tikzpicture}
      \node[draw,circle] (n1) at (2,0) {$P_1$};
      \node[draw,circle] (n2) at (1,1.732) {$P_2$};
      \node[draw,circle] (n3) at (-1,1.732) {$P_3$};
      \node[draw,circle] (n4) at (-2,0) {$P_4$};
      \node[draw,circle] (n5) at (-1,-1.732) {$P_5$};
      \node[draw,circle] (n6) at (1,-1.732) {$P_6$};
      
      \node[draw,rectangle] (K) at (0,0) {$K$};
      
      \draw[<->] (n1) -- (n2);
      \draw[<->] (n2) -- (n3);
      \draw[<->] (n3) -- (n4);
      \draw[<->] (n4) -- (n5);
      \draw[<->] (n5) -- (n6);
      \draw[<->] (n6) -- (n1);
      
      \draw[<->,dashed] (n1) -- (K);
      \draw[<->,dashed] (n2) -- (K);
      \draw[<->,dashed] (n3) -- (K);
      \draw[<->,dashed] (n4) -- (K);
      \draw[<->,dashed] (n5) -- (K);
      \draw[<->,dashed] (n6) -- (K);
    \end{tikzpicture} 
        \caption{Centralized control}
    \label{fig:centralized_control}
    \end{subfigure}
     \begin{subfigure}{0.49\linewidth}
    \begin{tikzpicture}
      \node[draw,circle] (n1) at (2,0) {$P_1$};
      \node[draw,circle] (n2) at (1,1.732) {$P_2$};
      \node[draw,circle] (n3) at (-1,1.732) {$P_3$};
      \node[draw,circle] (n4) at (-2,0) {$P_4$};
      \node[draw,circle] (n5) at (-1,-1.732) {$P_5$};
      \node[draw,circle] (n6) at (1,-1.732) {$P_6$};
      
      \node[draw,rectangle] (K1) at (0.8,0) {$K_1$};
      \node[draw,rectangle] (K2) at (0.4,0.693) {$K_2$};
      \node[draw,rectangle] (K3) at (-0.4,0.693) {$K_3$};
      \node[draw,rectangle] (K4) at (-0.8,0) {$K_4$};
      \node[draw,rectangle] (K5) at (-0.4,-0.693) {$K_5$};
      \node[draw,rectangle] (K6) at (0.4,-0.693) {$K_6$};
      
      \draw[<->] (n1) -- (n2);
      \draw[<->] (n2) -- (n3);
      \draw[<->] (n3) -- (n4);
      \draw[<->] (n4) -- (n5);
      \draw[<->] (n5) -- (n6);
      \draw[<->] (n6) -- (n1);
      
      \draw[<->,dashed] (n1) -- (K1);
      \draw[<->,dashed] (n2) -- (K2);
      \draw[<->,dashed] (n3) -- (K3);
      \draw[<->,dashed] (n4) -- (K4);
      \draw[<->,dashed] (n5) -- (K5);
      \draw[<->,dashed] (n6) -- (K6);
    \end{tikzpicture} 
        \caption{Decentralized control}
    \label{fig:decentralized_control}
    \end{subfigure}
     \begin{subfigure}{0.32\linewidth}
    \begin{tikzpicture}
      \node[draw,circle] (n1) at (3,0) {$P_1$};
      \node[draw,circle] (n2) at (1.5,2.6) {$P_2$};
      \node[draw,circle] (n3) at (-1.5,2.6) {$P_3$};
      \node[draw,circle] (n4) at (-3,0) {$P_4$};
      \node[draw,circle] (n5) at (-1.5,-2.6) {$P_5$};
      \node[draw,circle] (n6) at (1.5,-2.6) {$P_6$};
      
      \node[draw,rectangle] (K1) at (1.5,0) {$K_1$};
      \node[draw,rectangle] (K2) at (0.75,1.3) {$K_2$};
      \node[draw,rectangle] (K3) at (-0.75,1.3) {$K_3$};
      \node[draw,rectangle] (K4) at (-1.5,0) {$K_4$};
      \node[draw,rectangle] (K5) at (-0.75,-1.3) {$K_5$};
      \node[draw,rectangle] (K6) at (0.75,-1.3) {$K_6$};
      
      \draw[<->] (n1) -- (n2);
      \draw[<->] (n2) -- (n3);
      \draw[<->] (n3) -- (n4);
      \draw[<->] (n4) -- (n5);
      \draw[<->] (n5) -- (n6);
      \draw[<->] (n6) -- (n1);
      
      \draw[<->,dotted,thick] (K1) -- (K2);
      \draw[<->,dotted,thick] (K2) -- (K3);
      \draw[<->,dotted,thick] (K3) -- (K4);
      \draw[<->,dotted,thick] (K4) -- (K5);
      \draw[<->,dotted,thick] (K5) -- (K6);
      \draw[<->,dotted,thick] (K6) -- (K1);
      
      \draw[<->,dashed] (n2) -- (K1);
      \draw[<->,dashed] (n6) -- (K1);
      \draw[<->,dashed] (n1) -- (K2);
      \draw[<->,dashed] (n3) -- (K2);
      \draw[<->,dashed] (n2) -- (K3);
      \draw[<->,dashed] (n4) -- (K3);
      \draw[<->,dashed] (n3) -- (K4);
      \draw[<->,dashed] (n5) -- (K4);
      \draw[<->,dashed] (n4) -- (K5);
      \draw[<->,dashed] (n6) -- (K5);
      \draw[<->,dashed] (n5) -- (K6);
      \draw[<->,dashed] (n1) -- (K6);
      
      \draw[<->,dashed] (n1) -- (K1);
      \draw[<->,dashed] (n2) -- (K2);
      \draw[<->,dashed] (n3) -- (K3);
      \draw[<->,dashed] (n4) -- (K4);
      \draw[<->,dashed] (n5) -- (K5);
      \draw[<->,dashed] (n6) -- (K6);
    \end{tikzpicture} 
        \caption{Distributed/overlapping control}

    \label{fig:distributed_control}
    \end{subfigure}
    
    \caption{Different approaches for controlling networked systems.}
    \label{fig:networked_system}
\end{figure}

\begin{example}[{\cite[Section~5.1]{dileep2020}}]
\label{ex:decentralized}
Consider a networked system that consists of a ring of $N$ interconnected subsystems (similar as in \Cref{fig:networked_system}). The dynamics for the $i$\textsuperscript{th} subsystem  are given by
\begin{equation}
\left\{
\begin{array}{rcl}
    \dot{x}_i(t) &=& \begin{bmatrix}
        0 & 0.5 \\ 0.5 & -3
    \end{bmatrix} x_i(t) + \begin{bmatrix}
        0 \\ 5
    \end{bmatrix} u_i(t-0.1) + 0.5 x_{l,i}(t-0.3) + 0.5 x_{r,i}(t-0.3)\\[15pt]
    y_i(t) &=& \begin{bmatrix}
        1 & 1 \\ 0 & 1
    \end{bmatrix} x_i(t),
    \end{array}
    \right.
\end{equation}
with $x_{l,i}(t)$ the state of its left neighbor and $x_{r,i}(t)$ the state of its right neighbor. To create a representation for the overall system we use the \verb|kron| function (which returns  the Kronecker tensor product of two matrices).
\pagebreak
\begin{lstlisting}
% Number of subsystems
N = 5; 
% Delays and system matrices of the individual subsystems
tau1 = 0.1; tau2 = 0.3;
Ai = [0 0.5;0.5 -3]; Adi = 0.5*eye(2);
Bi = [0;5]; Ci = [1 1;0 1];
% Adjacency matrix of the network
PN = diag(ones(1,N-1),-1) + diag(ones(1,N-1),1);
PN(1,N) = 1; PN(N,1) = 1;
% System matrices of the overall system
A0 = kron(eye(N),Ai); A1 = kron(PN,Adi);
B = kron(eye(N),Bi); C = kron(eye(N),Ci);
% Create open-loop system
P = tds_create({A0,A1},[0 tau2],{B},tau1,{C},0);
\end{lstlisting}
Next we will compare different controller structures for this system. We start with designing a centralized static output feedback controller. As initial controller, we use the values given in \cite[Section~5.1]{dileep2020}.
\begin{lstlisting}
%% STEP 2: Centralized controller
o = tds_stabopt_options('nstart',1);
Kinit = blkdiag([-0.8045 0.6966],[0.8351 -0.2437],...
          [-1.1480 0.1049],[0.7223 2.5855],[0.2157 -1.1658]);
[K1,cl1] = tds_stabopt_static(P,'options',o,'initial',Kinit);
\end{lstlisting}
We find the densely filled feedback matrix
\[
\setlength{\arraycolsep}{3pt}
K = \begin{bmatrix}
 -4.010 & \phantom{-}4.770 & -6.227 & \phantom{-}5.330 & -1.081 & 0.276 & 2.486 & -0.993 & -11.033 & \phantom{-}8.613 \\
 -0.001  &  \phantom{-}1.657 & -11.586 &  \phantom{-}9.225 & -6.228 & \phantom{-}6.175 & \phantom{-}1.599 & -0.682 & -5.337 & \phantom{-}3.453 \\
  \phantom{-}5.321 & -3.217 & -6.622 & \phantom{-}5.387 & -7.855 & \phantom{-}6.30 & \phantom{-}0.323 & \phantom{-}0.974 & -7.508 & \phantom{-}5.130
 \\
   \phantom{-}0.166 & \phantom{-}0.532  & -0.160 & \phantom{-}0.085 & -0.478 & -0.055 & -8.242 & \phantom{-}7.473 & -5.090 & \phantom{-}4.458\\
   -1.129 & \phantom{-}1.674 & -0.242 & \phantom{-}0.391 & \phantom{-}3.320 & -4.119 & -2.007 & \phantom{-}2.357 & -9.736 & \phantom{-}8.144
   
\end{bmatrix}
\]
and the corresponding (strong) spectral abscissa is equal to -5.3382. Next, we design a set of decentralized controllers by imposing a block diagonal structure on the feedback matrix.
\begin{lstlisting}
mask = kron(eye(N),[1 1]);
basis = zeros(N,N*2);
[K2,cl2] = tds_stabopt_static(P,'options',o,'initial',Kinit,...
         'mask',mask,'basis',basis);
\end{lstlisting}
We obtain the feedback matrices
\begin{equation*}
    \begin{array}{lll}
        K_1 = \begin{bmatrix}
           -9.7969 & 9.0638
        \end{bmatrix}, & K_2 = \begin{bmatrix}
            -17.0255 & 15.4809 
        \end{bmatrix}, &
                 K_3 = \begin{bmatrix}
            -7.1284  & 6.0666
        \end{bmatrix}, \\ [7px] K_4 = \begin{bmatrix}
           -8.7850 & 7.7052  
        \end{bmatrix}, &
                K_5 = \begin{bmatrix}
            -15.2369 & 13.5828
        \end{bmatrix},
    \end{array}
\end{equation*}
and the (strong) spectral abscissa of the corresponding closed-loop system is equal to -4.1872. As we have less free parameters in our feedback matrix, it is not surprising that the damping rate decreases. The resulting controller is however easier to implement. Finally, we design an overlapping controller in which each local controller also receives the measurements of the neighboring systems.
\begin{lstlisting}
mask = kron(eye(N),[1 1]) + kron(PN,[1 1]);
basis = zeros(N,N*2);
[K3,cl3] = tds_stabopt_static(P,'options',o,'initial',Kinit,...
	'mask',mask,'basis',basis);
\end{lstlisting}
We obtain the feedback matrix
\begin{equation*}
    \begin{array}{lll}
        K = \begin{bmatrix}
             -6.511 & 5.532 & -2.394 & 1.993 & 0 & 0 & 0 & 0 & -7.581 & 7.459 \\
          -1.459 & 1.377 & -7.959 & 6.919 & -0.892 & 0.899 & 0 & 0 & 0 & 0 \\
          0 & 0 & -10.15 & 10.12 & -11.63 & 10.28 & -4.757  & 4.991 & 0 & 0 \\
          0 & 0 & 0 & 0 & -4.103 & 3.760 & -8.242 & 7.408 & -5.409 & 5.102 \\
          -2.404 & 2.791 &  0 & 0 & 0 & 0 & 1.354 & -1.041 & -9.455 & 7.691
        \end{bmatrix}
    \end{array}
\end{equation*}
and the corresponding strong spectral abscissa is equal to -4.7371.
\end{example}
\section{Obtaining a state-space representation from a SISO transfer function}
\label{subsec:transfer_function}
In this subsection we will briefly discuss how one can create a \matlabfun{tds_ss}-representation for a time-delay system described by a SISO transfer function. More specifically, the function \matlabfun{tds_create_tf} allows to represent transfer functions of the following form
\[
T(s) = \frac{P(s)}{Q(s)} + D(s)
\]
with
\[
\textstyle
P(s) = \sum\limits_{k=1}^{m}\left(\sum\limits_{i=1}^{n_P} p_{k,i}\lambda^{n_P-(i-1)}\right) e^{-s\tau_{k}}\text{, }
Q(s) = \sum\limits_{k=1}^{m}\left(\sum\limits_{i=1}^{n_Q} q_{k,i}\lambda^{n_Q-(i-1)}\right) e^{-s\tau_{k}},
\]
and
\[
\textstyle
D(s) = \sum\limits_{k=1}^{m} d_{k}  e^{-\lambda \tau_{k}}
\]
in which $p_{k,i}$, $q_{k,i}$, and $d_{k}$ are real-valued coefficients and $\tau_{1},\dots,\tau_{m}$ are  (non-negative) delay values. To avoid advanced systems, at least one delay must be equal to zero and the corresponding leading coefficient of $Q(s)$ must be non-zero, \ie let $\tau_k=0$ then $q_{k,1}\neq 0$. Furthermore, the transfer function must be proper, \ie $n_P\leq n_Q$. The quasi-polynomial $D(s)$ allows to specify direct feedthrough terms if $\frac{P(s)}{Q(s)}$ is strictly proper, \ie $n_P<n_D$.

The function \matlabfun{tds_create_tf} either takes three or four arguments:
\begin{itemize}
    \item \verb|tds_create_tf(P,Q,TAU)| with \verb|P| an $m \times (n_P+1)$-dimensional real-valued matrix containing the coefficients $p_{k,i}$, \verb|Q| an $m\times (n_Q+1)$-dimensional real-valued matrix containing the coefficients $q_{k,i}$, and \verb|TAU| an $m$-dimensional vector containing the delays;
    \item \verb|tds_create_tf(P,Q,D,TAU)| with \verb|P| an $m \times (n_P+1)$-dimensional real-valued matrix containing the coefficients $p_{k,i}$, \verb|Q| an $m\times (n_Q+1)$-dimensional real-valued matrix containing the coefficients $q_{k,i}$, \verb|D| an $m$-dimensional real-valued vector containing the coefficients $d_k$ and \verb|TAU| an $m$-dimensional vector containing the delays.
\end{itemize}
The return argument is a \verb|tds_ss|-object that represents, depending on the transfer function, a retarded, neutral or delay-descriptor system. This object can subsequently be used as an input for the controller design functions mentioned above. Let us now illustrate this function using two examples.
\begin{example}
	\label{example:SISO_transfer_function}
As a first example we will consider the following strictly proper transfer function, corresponding to a retarded time-delay system
\[
T(s) = \frac{(2s-1)+3se^{-s}}{(4s^2-2s+1) + (3s-1)e^{-s}}.
\]
Extracting the coefficients of the quasi-polynomials and the delay values, we get
\begin{lstlisting}
P = [2 -1;3 0]; Q = [4 -2 1;0 3 -1]; TAU = [0 1];
sys = tds_create_tf(P,Q,TAU)
\end{lstlisting}
with \matlabfun{sys} a \matlabfun{tds_ss_retarded}-object with inputs and outputs:
\begin{lstlisting}
sys = 

LTI Retarded Time-Delay State-Space System with properties:

    A: {[2x2 double]  [2x2 double]}
   hA: [0 1]
   B1: {[2x1 double]  [2x1 double]}
  hB1: [0 1]
   C1: {[0 1]}
  hC1: 0
\end{lstlisting}
in which some parts of the output were omitted for notational convenience.

Similiarly we can create a representation for the following neutral system with explicit direct feedthrough terms
\[
T(s) = \frac{(6s-2)e^{-s}+se^{-2s}+(s-1)}{(5s^2-3s-2)e^{-2s}+(-2s^2+3s-2)}-2e^{-s}+3 e^{-2s} + 1
\]
using the code
\begin{lstlisting}
P = [6 -2;1 0;1 -1]; Q = [0 0 0;5 -3 -2;-2 3 -2];
D = [-2; 3; 1]; TAU = [1 2 0];
sys = tds_create_tf(P,Q,D,TAU);
\end{lstlisting}
which returns a \matlabfun{tds_ss_neutral}-object
\begin{lstlisting}
sys = 

LTI Neutral Time-Delay State-Space System with properties:

    H: {[2x2 double]}
   hH: 2
    A: {[2x2 double]  [2x2 double]}
   hA: [0 2]
   B1: {[2x1 double]  [2x1 double]  [2x1 double]}
  hB1: [0 1 2]
   C1: {[0 1]}
  hC1: 0
  D11: {[1]  [-2]  [3]}
 hD11: [0 1 2]
\end{lstlisting}
\end{example}

\chapter{Performance criteria and robust control}
\label{chapter:performance_robust_control}

In this chapter, we will discuss how the \hinfnorm{}, the pseudospectral abscissa and the distance to instability of a LTI time-delay system can be computed within \packageName{}. Besides, we will show how a static or dynamic output feedback controller that (locally) minimizes the closed-loop \hinfnorm{} can be designed. First, \Cref{subsec:compute_hinfnorm} will introduce the \hinfnorm{} for time-delay systems. As was the case for the spectral abscissa of a neutral DDE, the \hinfnorm{} of a time-delay system that is not strictly proper, might suffer from a fragility problem, in the sense that the \hinfnorm{} might be discontinuous with respect to the delay values. \packageName{} will therefore consider the strong \hinfnorm{} (introduced in \cite[Definition~4.4]{gumussoy2011}), which is a continuous function of both the entries of the system matrices and the delays. Besides we also will discuss how the \hinfnorm{} can be used for examining the robust stability of feedback systems in the presence of uncertainties and for characterizing input-to-output performance. Next, \Cref{subsec:optimize_hinfnorm} will consider several $\hinf{}$-optimization problems arising in model order reduction and the robust control framework. For more details on robust control theory we refer the interested reader to \cite{zhou1996robust} or the derived textbook \cite{zhou1998essentials}, which discuss various aspects of robust control theory focusing on multiple-input-multiple-output systems without delays. For specific results related to delay systems, we refer to \cite{ozbay2018frequency} and \cite{zhong2006robust}. Finally, \Cref{sec:robust_stability} introduces the pseudospectral abscissa and the distance to instability which allow to examine the robust stability of a dynamical system in the presence of structured uncertainties \cite{hinrichsen2006}.
\label{sec:performance}
\section{Computing the H-infinity norm}
\label{subsec:compute_hinfnorm}
In this subsection we will discuss how the \hinfnorm{} of a (stable) LTI time-delay system can be computed in \packageName{}. More specifically, we will consider systems of the following form
\begin{equation}
\label{eq:open_loop_hinfnorm}
\left\{
\begin{array}{rcl}
    \displaystyle E\dot{x}(t) &=& \displaystyle \sum\limits_{k=1}^{m_A} A_k\, x(t-h_{A,k}) + \sum_{k=1}^{m_{B_2}} B_{2,k}\, w(t-h_{B_2,k}) - \sum_{k=1}^{m_H} H_k \dot{x}(t-h_{H,k}) \\[13pt]
     
     z(t) &=& \displaystyle  \sum\limits_{k=1}^{m_{C_2}} C_{2,k}\, x(t-h_{C_2,k}) + \sum_{k=1}^{m_{D_{22}}} D_{22,k}\, w(t-h_{D,k})
\end{array}
\right.
\end{equation}
with $x\in\R^{n}$ the state variable, $w\in \R^{p_2}$ the performance input (representing \eg external disturbances) and  $z\in\R^{q_2}$ the performance output (representing \eg error signals). As in \eqref{eq:state_space_system}, all matrices are real-valued and of appropriate dimensions; all delays are non-negative (except for $h_{H,1}$,\dots,$h_{H,H}$ which should be positive); and the matrix $E$ can be singular.
\\

\textbf{Note:} In contrast to \eqref{eq:state_space_system} the
input and output channels are now denoted by $w$ and $z$ instead of $u$ and $y$. In  \eqref{eq:state_space_system} the input and output channels ($u$ and $y$) were used for feedback control, while here the input and output channels ($w$ and $z$)  are used to quantify the input-to-output behavior of the system. As these input and output channel serve a difference purpose, we will use a different symbol to denote them. In the next section (\Cref{subsec:optimize_hinfnorm}), where we will design output feedback controllers that (locally) minimize the \hinfnorm{}, we will need both feedback channels, denoted by $u$ and $y$, and performance channels, denoted $w$ and $z$. Accordingly, we use $B_{2,k}$, $C_{2,k}$, and $D_{22,k}$ instead of $B_{1,k}$, $C_{1,k}$, and $D_{11,k}$ to denote the input, output and direct feedthrough matrices. This change in notation has also consequences on the syntax that needs to be used when constructing a \verb|tds_ss|-object for the system, as spelled out below. \\

We can again use the functions \matlabfun{tds_create}, \matlabfun{tds_create_neutral} and \matlabfun{tds_create_ddae} to represent \eqref{eq:open_loop_hinfnorm} in \packageName{}. However, we now need to explicitly indicate that we want to specify the performance input, output and direct feedthrough terms.
\begin{itemize}
    \item \matlabfun{tds_create(A,hA,'B2',B2,hB2,'C2',C2,hC2,'D22',D22,hD22)} allows to create retarded state-space models, \ie $E = I_n$ and $m_{H} = 0$.
    \item \matlabfun{tds_create_neutral(H,hH,A,hA,'B2',B2,hB2,'C2',C2,hC2,'D22',D22,hD22)} allows to create neutral state-space models, \ie{} $E = I_n$ and $m_{H} > 0$.
    \item \matlabfun{tds_create_ddae(E,A,hA,'B2',B2,hB2,'C2',C2,hC2,'D22',D22,hD22)} allows to create delay descriptor systems, \ie $E$ not necessarily identity and potentially singular, and $m_H = 0$. \textbf{Important:} Recall that all DDAEs and delay descriptor systems considered in \packageName{} should satisfy \Cref{assumption:index_1}.
\end{itemize}
Notice that we now have to use the flags \verb|'B2'|, \verb|'C2'|, and \verb|'D22'| to indicate that we want to set the fields related to the inputs $w$ and outputs $z$.  In all function calls above, the arguments \verb|D22| and \verb|hD22| are again optional. Further, the arguments \verb|H|, \verb|A|, \verb|B2|, \verb|C2|, and \verb|D22| should be cell arrays containing the system matrices. The arguments \verb|hH|, \verb|hA|, \verb|hB2|, \verb|hC2|, and \verb|hD22| should be arrays containing the delay values. Finally, the argument \verb|E| of \matlabfun{tds_create_ddae} should be a (real-valued) matrix.

For the \hinfnorm{} computation, \packageName{} will internally transform all representable systems to the following standard form 
\begin{equation}
\label{eq:ol_hinfnorm2}
\left\{
\begin{array}{rcl}
    \displaystyle E_{s}\dot{x}_{s}(t) &=& A_{s,0}\, x_{s}(t) + \displaystyle \sum\limits_{k=1}^{m} A_{s,k}\, x_{s}(t-\tau_{k}) + B_{s} w(t)\\[13pt]
    z(t) &=& \displaystyle  C_{s} x_{s}(t).
\end{array}
\right.
\end{equation}
 with $\vec{\tau} = [\tau_1,\dots,\tau_m]$ the union of all delays and $x_{s}$ the augmented state variable to which additional states are added to eliminate the input and output delays, the direct feedthrough terms and the delayed state-derivatives (\ie the ``neutral'' terms). For ease of notation, we will use this standard form in the following discussion, but note that the software can deal with all systems that can be created using \matlabfun{tds_create}, \matlabfun{tds_create_neutral} and \matlabfun{tds_create_ddae}.
 
Let us start with defining the \hinfnorm{} in terms of the standard form \eqref{eq:ol_hinfnorm2}. To this end, let us first introduce the corresponding transfer function from input $w$ to output $z$, which is given by
\begin{equation}
\label{eq:transferfunction}
T_{zw}(s;\vec{\tau}) = C_{s} \left(s E_{s} - A_{s,0} - \sum\limits_{k=1}^{m} A_{s,k}\, e^{-s \tau_{k}} \right)^{-1} B_{s}.
\end{equation}
This transfer function describes the input-to-output behavior of the system for a zero initial state. Furthermore, if the null solution of the following DDAE
\[
\displaystyle E_{s}\dot{x}_{s}(t) = A_{s,0}\,x_{s}(t) + \displaystyle \sum\limits_{k=1}^{m} A_{s,k}\, x_{s}(t-\tau_{k})
\]
is exponentially stable, the \hinfnorm{} of \eqref{eq:ol_hinfnorm2} is defined as:
\begin{equation}
\label{eq:hinfnorm2}
\|T_{zw}(s;\vec{\tau})\|_{\hinf} := \sup_{\substack{s\in\C\\\Re(s)> 0}} \|T_{zw}(s;\vec{\tau})\|_2 = \sup_{\omega\in\R^{+}} \|T_{zw}(\jmath\omega;\vec{\tau})\|_2, 
\end{equation}
in which the last equality follows from an extension of the maximum modulus principle. The \hinfnorm{} of a stable system is thus equal to the maximal value of the frequency response (\ie the transfer function evaluated on the imaginary axis) measured in spectral norm. On the other hand, we will say that the \hinfnorm{} of an unstable system is equal to $\infty$. \\

\noindent The \hinfnorm{} has several applications in (robust) control theory. 
\begin{itemize}
    \item Firstly, the \hinfnorm{} can be used to measure the performance of a dynamical system. More precisely, it quantifies the `worst-case' suppression (or amplification) from the input signal $w$ to the output $z$ (assuming zero initial conditions), in the sense that
    \[
    \|T_{zw}(s;\vec{\tau})\|_{\hinf} = \limsup_{w\in L_2} \frac{\|z\|_{L_2}}{\|w\|_{L_2}},
    \]
    with $L_2$ the (Hilbert) space of square-integratable functions on the interval $[0,+\infty)$ and $\|\cdot\|_{L_2}$ the associated norm, \ie $f\in L_2$ if $\|f\|_{L_2} := \left(\int_{0}^{+\infty} \|f(t)\|^2_2\, \dd t\right)^{1/2} <\infty$. For more details see \cite[Chapter 4]{zhou1996robust}. By appropriately choosing the performance inputs and outputs, the \hinfnorm{} can thus be used to quantify the reference tracking, the disturbance rejection, and the noise suppression capabilities of the system.  We will further examine this interpretation in \Cref{ex:measurement_noise,example:mixed_sensitivity}.
    
    \textbf{Note.} By viewing the system \eqref{eq:ol_hinfnorm2} as a (linear) operator mapping input signals in $L_2$ to output signals in $L_2$, the \hinfnorm{} can thus be interpreted as an induced norm. 
    \item Secondly, the \hinfnorm{} allows to quantify the effect of uncertainties on the stability of a dynamic system. More specifically, the \hinfnorm{} is equal to the reciprocal of the structured complex distance to instability of an exponentially stable system, \ie 
        \begin{multline*}
    \|T_{zw}\|_{\hinf} = \inf \{\|\Delta \|_2 : \Delta \in \C^{p_2 \times q_2} \text{ such that } \\ 
    E_{s}\dot{x}_{s}(t) = \sum\limits_{k=0}^{m} A_{s,k}\, x_{s}(t-\tau_{k}) + B_{s} \Delta C_{s}\, x_{s}(t)
    \text{ is not exponentially stable }\}^{-1},
    \end{multline*}
    with $1/0 = +\infty$, to be consistent with the definition of the \hinfnorm{} given above. 
    
    The \hinfnorm{} is thus a measure for the robust stability of the system. Note however that the \hinfnorm{} is only a rough estimate for the distance to instability of the system in a practical sense  as it only considers one complex (matrix-valued) uncertainty. In \Cref{sec:robust_stability}, we will see how the structured distance to instability with respect to real-valued uncertainties acting on both the system matrices and the delays can be computed in \packageName{}, which is a better measure for robust stability in practice. 
    
    \item Thirdly, the \hinfnorm{} can be used as an error measure for a reduced-order approximation of a dynamical system \cite[Chapter 7]{zhou1996robust}. For example, consider the following exponentially stable time-delay system
    \[
    \left\{
    \begin{array}{rcl}
         \dot{x}(t)& = &A_0\, x(t) + \sum\limits_{k=1}^{m} A_k\, x(t-\tau_k) + B_2\, w(t)\\
         z(t) & = & C_2\, x(t)
    \end{array}
    \right.
    \]
	and the following delay-free approximation
    \begin{equation}
    \label{eq:MOR_approx}
    \left\{
    \begin{array}{rcl}
    \dot{x}_{a}(t) &=& A_{a}\, x_{a}(t) + B_{a}\, w(t) \\
    z(t) &=& C_{a}\, x_{a}(t)
    \end{array}
    \right.
    \end{equation}
   	The closeness of these two models can then be quantified by
    \begin{equation}
    \label{eq:MOR_mismatch}
    \| T(s) - T_{\mathrm{approx.}}(s) \|_{\hinf},
    \end{equation}
    with $T(s)$ and $T_{\mathrm{approx.}}(s)$ the transfer function of the original system and the approximation, respectively. Note that \eqref{eq:MOR_mismatch} corresponds to the maximal mismatch of the original frequency response and its approximation in spectral norm. We will consider this application in more detail in \Cref{example:MOR}. Furthermore, in the same example we will also see that finding the state-space system of the form \eqref{eq:MOR_approx} that minimizes the mismatch in terms of \eqref{eq:MOR_mismatch} can be reformulated as a $\hinf{}$-controller design problem (which will be considered in \Cref{subsec:optimize_hinfnorm}). As such, \packageName{} can also be used to derive a delay-free approximation for a time-delay system.
    \item Finally, in robust control theory the small gain theorem is frequently used to guarantee stability in the presence of uncertainties \cite[Chapter 9]{zhou1996robust}. Application of this small gain theorem typically gives rise to constraints in terms of the \hinfnorm{} of an appropriately defined transfer function. 
\end{itemize}

To compute the \hinfnorm{} of a time-delay system, \packageName{} provides the function \matlabfun{tds_hinfnorm}. This function takes one mandatory argument, the \matlabfun{tds_ss} object of which we want to compute the \hinfnorm{}. Additional options can be specified using the optional argument \verb|options|. For more details on the available options, see \matlabfun{help tds_hinfnorm} or \Cref{sec:tds_hinfnorm}.

\begin{remark}
	The function \matlabfun{tds_hinfnorm} employs a combination of the BBBS-algorithm \cite{bruinsma1990fast,boyd1990regularity} and Newton's method to compute the \hinfnorm{} of a time-delay system. More specifically, in a first step it uses the BBBS-algorithm to compute the \hinfnorm{} of a finite dimensional approximation of the system (obtained using a similar spectral discretisation approach as was used in \eqref{eq:discrete_evp} to approximate the characteristic roots of a DDE). This result is subsequently refined using Newton corrections. The BBBS-algorithm itself consists of two step that are alternated iteratively. First, it computes the frequency intervals at which the spectral norm of the frequency response is larger than a given level. Next, it updates this level with the maximum value of the spectral norm of the frequency response at the geometric mean of these intervals. This two-step process is repeated until the change in the level is sufficiently small.
\end{remark}

\begin{example}
\label{example:hinfnorm1}
Consider the following single-input single-output (SISO) time-delay system
\begin{equation}
\label{eq:hinfnorm_example1}
\left\{
\begin{array}{rcl}
    \dot{x}(t) &=& \begin{bmatrix}
-4 & 2 \\
-3 & -3
\end{bmatrix}x(t) + \begin{bmatrix}
-2 & -1 \\
3 & -2
\end{bmatrix}x(t-1) + \begin{bmatrix}
1 \\
-1
\end{bmatrix} w(t-2), \\[10px]
    z(t) & = & \begin{bmatrix}
  -2 & 1  
    \end{bmatrix} x(t).
\end{array}
\right.
\end{equation}
As the underlying RDDE is exponentially stable, the asymptotic input-to-output behavior of this system can be characterized in terms of the following transfer function:
\begin{equation}
\label{eq:hinfnorm_example1_tf}
T_{zw}(s) = -\frac{9+3s+5e^{-s\tau_1}}{s^2+7s+18+(4s+5)e^{-s \tau_1}+7e^{-s2\tau_1}}e^{-s\tau_2}.
\end{equation}
The corresponding \hinfnorm{} is found using the following code.
\begin{lstlisting}
% Create a representation for the state-space system 
A0 = [-4 2;-3 -3]; A1 = [-2 -1;3 -2];
B = [1;-1]; C = [-2 1];
tds = tds_create({A0 A1},[0 1],'B2',{B},[2],'C2',{C},[0]);
% Compute the H-infinity norm
[hinf,wPeak] = tds_hinfnorm(tds);
\end{lstlisting}
The function \verb|tds_hinfnorm| returns both the \hinfnorm{} (\verb|hinf|) and the angular frequency at which this maximal gain is achieved (\verb|wPeak|), \ie the $\omega$ that maximizes $\|T_{zw}(\jmath \omega;\vec{\tau})\|_{2}$ (\textbf{note} that \verb|wPeak| might be equal to $\infty$ for transfer functions that are not strictly proper, meaning that the supremum in \eqref{eq:hinfnorm2} is only attained in the limit $\omega\rightarrow \infty$). In this case we find that the \hinfnorm{} is equal to 1.5388 and is attained at $\omega = 3.5571$.

We can use the function \matlabfun{tds_sigma(tds,omega)} to verify the obtained result. This function computes the singular values of the transfer function associated with \matlabfun{tds} at the given angular frequencies \matlabfun{omega}. For a SISO-system, this plot corresponds to the magnitude of the frequency response.
\begin{lstlisting}
% Plot the magnitude of the frequency response
omega = logspace(-2,2,10001);
[sv] = tds_sigma(tds,omega);
semilogx(omega,sv)
hold on
plot(xlim,[hinf hinf])
plot(wPeak,hinf,'x')
\end{lstlisting}
\Cref{fig:ex16_sv_hinfnorm} shows the result and indicates the values \verb|hinf| and \verb|wPeak| obtained above. (Luckily,) we see that the obtained result is indeed correct.
\begin{figure}[!h]
    \centering
    \includegraphics[width=0.6\linewidth]{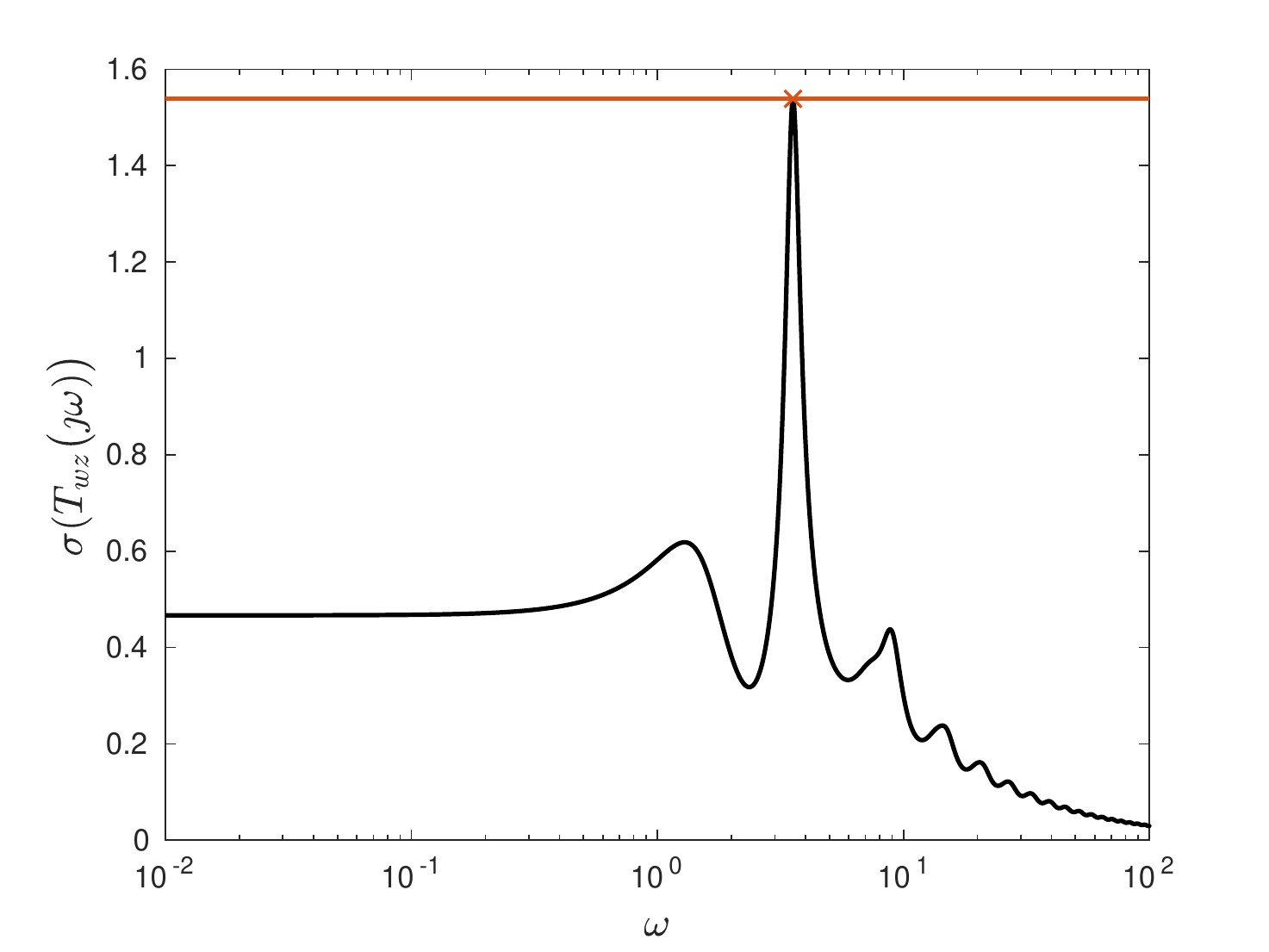}
    \caption{The singular value plot of \eqref{eq:hinfnorm_example1_tf} for $s\in\jmath\,[10^{-2}, \, 10^{2}]$ (black), the value of the \hinfnorm{} (red) and the angular frequency at which the maximal gain is attained (cross).}
    \label{fig:ex16_sv_hinfnorm}
\end{figure}
\end{example}

As mentioned before, the \hinfnorm{} might be discontinuous with respect to the delays. This observation implies that an infinitesimal change on the delays might cause a jump in the value of the \hinfnorm{}, as illustrated in the example below. 
\begin{example}
\label{example:hinfnorm2}
Consider the following single-input single-output (SISO) system
\begin{equation}
\label{eq:hinfnorm_example2}
\left\{
\begin{array}{rcl}
    \dot{x}(t) &=& \begin{bmatrix}
-4 & 2 \\
-3 & -3
\end{bmatrix}x(t) + \begin{bmatrix}
-2 & -1 \\
3 & -2
\end{bmatrix}x(t-\tau_1) + \begin{bmatrix}
1 \\
-1
\end{bmatrix} w(t-\tau_2) \\[13px]
    z(t) & = & \begin{bmatrix}
  -2 & 1  
    \end{bmatrix} x(t) + w(t) + w(t-\tau_1) - 2 w(t-\tau_2)
\end{array}
\right.
\end{equation}
with $\tau_1= 1$ and $\tau_2 = 2$. Notice that the considered system is quite similar to the one in \Cref{example:hinfnorm1}. The only difference is that we have added three direct feedthrough terms. The associated transfer function is now given by:
\begin{equation}
\label{eq:hinfnorm_example2_tf}
T_{zw}(s) = -\frac{9+3s+5e^{-s\tau_1}}{s^2+7s+18+(4s+5)e^{-s \tau_1}+7e^{-s2\tau_1}}e^{-s\tau_2} + 1+e^{-s\tau_1} - 2e^{-s\tau_2},
\end{equation}
which, in contrast to \eqref{eq:hinfnorm_example1_tf}, is not strictly proper. Furthermore, it can be shown that the \hinfnorm{} of \eqref{eq:hinfnorm_example2} is equal to 3.7019 (rounded to 4 decimal places). However, if we make a small perturbation to $\tau_1$, \eg $\tau_1 = 0.995$, the \hinfnorm{} suddenly jumps to approximately 4. Furthermore, even if we reduce the size of the perturbation, \eg $\tau_1 = 0.999$, we still observe a jump in the \hinfnorm{} to approximately 4. To understand what happens, let us use the \matlabfun{tds_sigma}-function to plot the magnitude of the frequency response.
\begin{lstlisting}
% Create a representation for the time-delay system
A0 = [-4 2;-3 -3]; A1 = [-2 1;3 -2];
B = [1;-1]; C = [-2 1];
D0 = 1; D1 = 1; D2 = -2;
tau1 = 1; tau2 = 2;
tds = tds_create({A0 A1},[0 tau1],'B2',{B},tau2,'C2',{C},0,...
       'D22',{D0,D1,D2},[0 tau1 tau2]);
% Plot the magnitude of the frequency response
omega = logspace(-2,3.5,10001);
[sv] = tds_sigma(tds,omega);
figure
semilogx(omega,sv)
\end{lstlisting}
From \Cref{fig:ex17_sv1}, we observe that the frequency responce indeed attains a maximum of approximately $3.7019$ (at $\omega = 2.7509$). Next, we repeat the same procedure for $\tau_1 = 0.995$.
\begin{lstlisting}
% Create a representation for the time-delay system
tau1 = 0.995;
tds2 = tds_create({A0 A1},[0 tau1],'B2',{B},tau2,'C2',{C},0,...
	'D22',{D0,D1,D2},[0 tau1 tau2]);
% Plot the magnitude of the frequency responce
[sv2] = tds_sigma(tds2,omega);
figure
semilogx(omega,sv2)
\end{lstlisting}
and $\tau_1=0.999$ 
\begin{lstlisting}
tau1 = 0.999;
tds3 = tds_create({A0 A1},[0 tau1],'B2',{B},tau2,'C2',{C},0,...
	'D22',{D0,D1,D2},[0 tau1 tau2]);
% Plot the magnitude of the frequency responce
[sv3] = tds_sigma(tds3,omega);
figure
semilogx(omega,sv3)
\end{lstlisting}
\Cref{fig:ex17_sv2,fig:ex17_sv3} show the result. We observe that in both cases the frequency response at high frequencies significantly differs from the one in \Cref{fig:ex17_sv1} and that the \hinfnorm{} in both cases jumps to approximately 4.
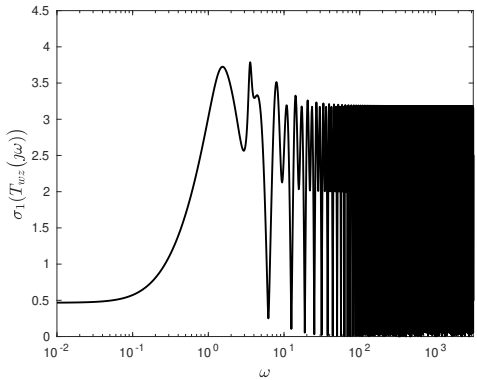
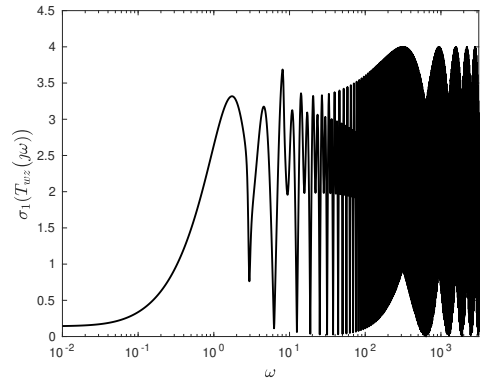
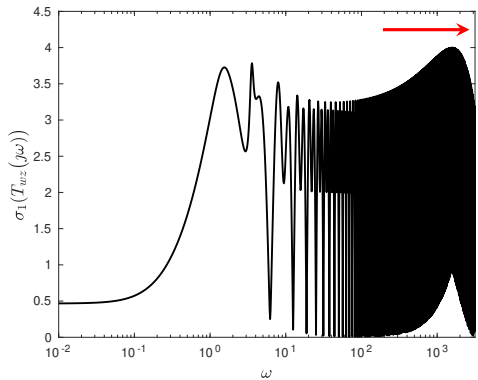
\begin{figure}[!h]
    \centering
    \begin{subfigure}{0.49\linewidth}
    \includegraphics[width=\linewidth]{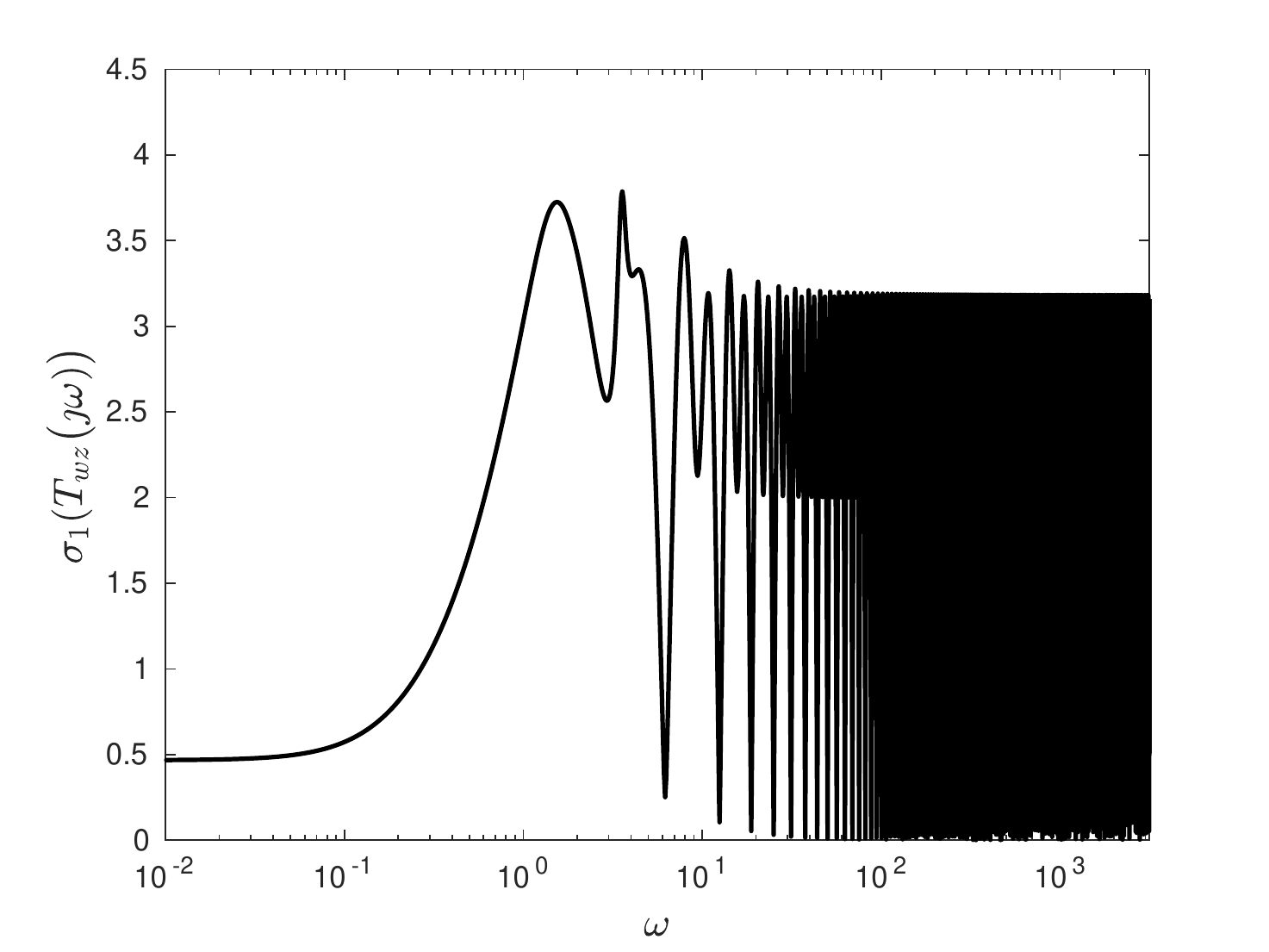}
    \caption{$\tau_1=1$ and $\tau_2=2$.}
    \label{fig:ex17_sv1}
    \end{subfigure}
    \begin{subfigure}{0.49\linewidth}
    \includegraphics[width=\linewidth]{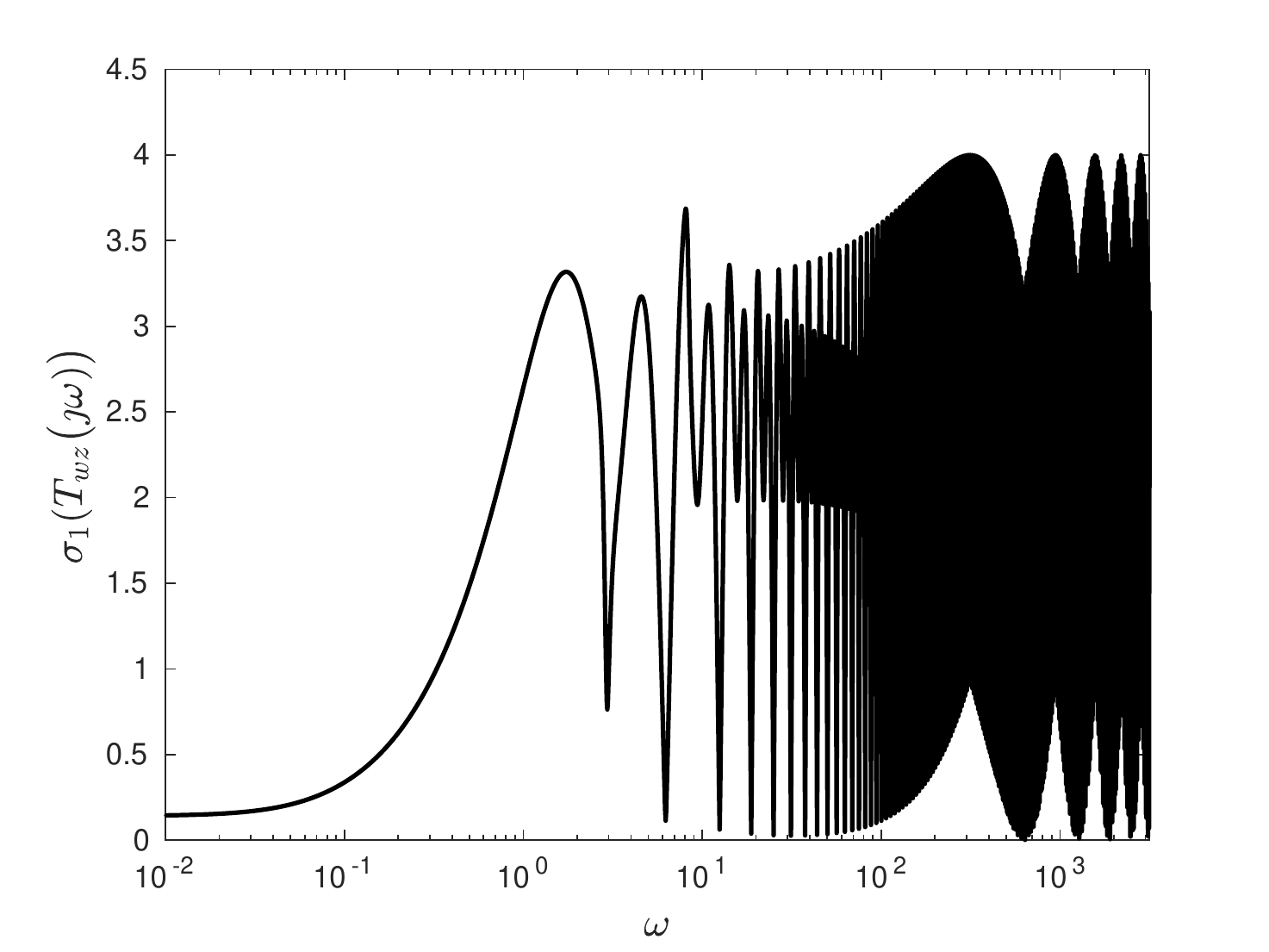}
    \caption{$\tau_1=0.995$ and $\tau_2=2$.}
    \label{fig:ex17_sv2}
    \end{subfigure}
    \begin{subfigure}{0.49\linewidth}
      \begin{tikzpicture}[      
        every node/.style={anchor=south west,inner sep=0pt},
        x=1mm, y=1mm,
      ]   
     \node (fig1) at (0,0)
       { \includegraphics[width=\linewidth]{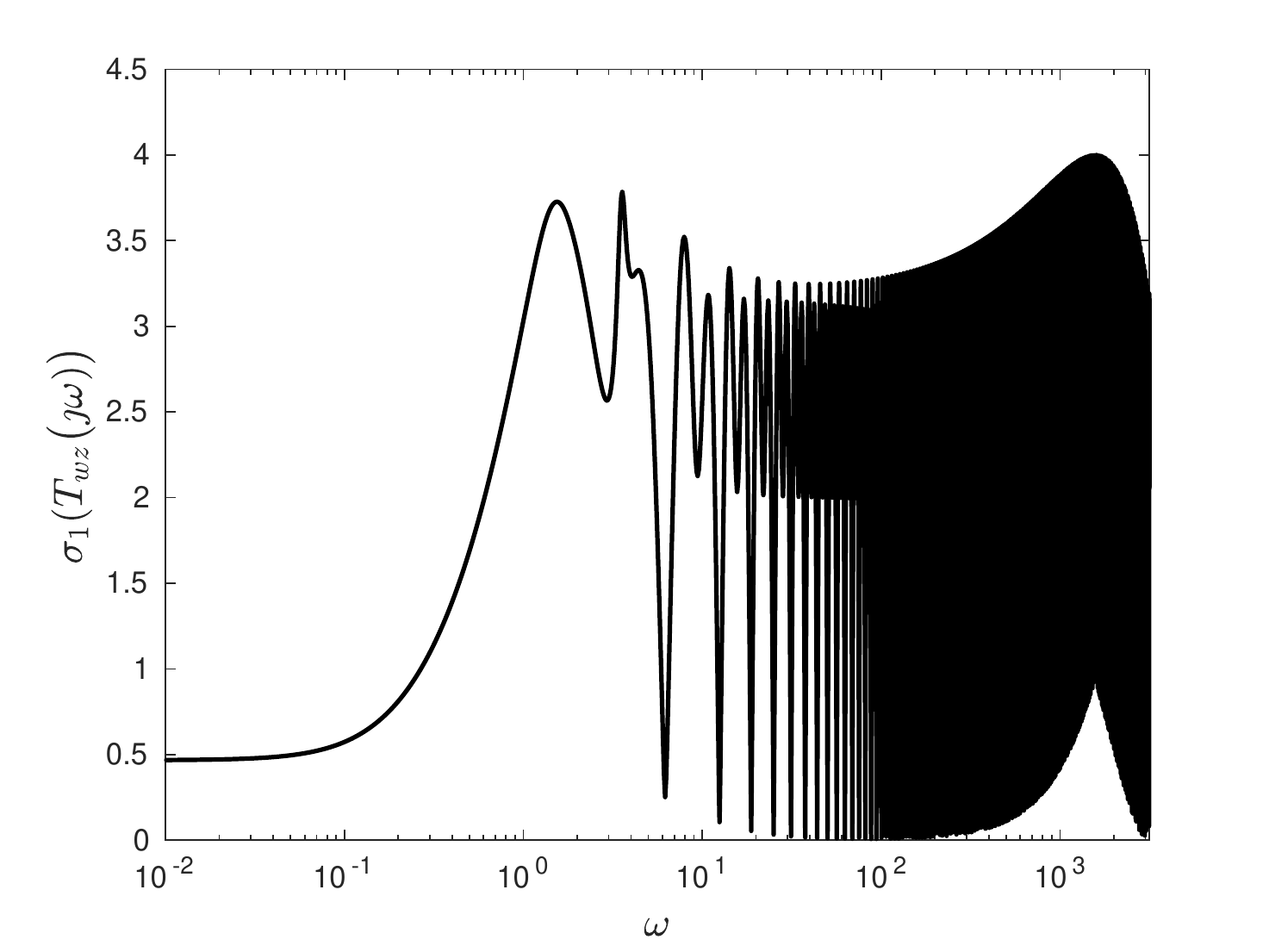}};
    \node[] (tt) at (64,47) {};
    \node[] (bt) at (52,47) {};
    \draw[-stealth,very thick,red] (bt) -- (tt);
\end{tikzpicture}
    \caption{$\tau_1=0.999$ and $\tau_2=2$.}
    \label{fig:ex17_sv3}
    \end{subfigure}
    \caption{The (maximum) singular value plot of \eqref{eq:hinfnorm_example2_tf} for several values of $\tau_1$ and $\tau_2$.}
\end{figure}
\end{example}
We can draw several parallels with \Cref{example:neutral_2} and \Cref{remark:sensitivity}. Firstly, as the spectral abscissa of a NDDE, the \hinfnorm{} of a time-delay system that is not strictly proper might be a discontinuous function of the delay values, even if the system is strongly exponentially stable. However, if we look at a fixed angular frequency, the singular values of the frequency response change in a continuous way when the delay parameters are varied. (Recall that we made a similar observation for the individual characteristic roots of a NDDE.) Nonetheless, we observe from \Cref{fig:ex17_sv2,fig:ex17_sv3} that the larger the angular frequency, the larger the sensitivity (\ie derivative) of the singular values of the frequency response to delay perturbations. In the example above, this derivative grows unbounded as $\omega$ goes to infinity. Furthermore, it can be shown that for any $\epsilon>0$ and $\delta>0$ (no matter how small), there exists a delay perturbation smaller than $\epsilon$ for which the \hinfnorm{} of \eqref{eq:hinfnorm_example2} jumps to at least $4-\delta$. Furthermore as this perturbation bound is decreased, the frequencies at which such a peak in the singular value plot can be attained, moves to infinity (as indicated with the arrow in \Cref{fig:ex17_sv3}). Note again the similarity with \Cref{remark:sensitivity}. Below, we will therefore introduce the strong \hinfnorm{} and the asymptotic transfer function, which have similar roles as the strong spectral abscissa and the delay difference equation underlying the NDDE in \Cref{sec:neutral_stability}.

The asymptotic transfer function of \eqref{eq:ol_hinfnorm2} is given by:
\begin{equation}
\label{eq:asymptotic_tf}
T_{zw,a}(s) := -C_{s}^{(2)} \left(A_0^{(22)} + \sum_{k=1}^{m} A_k^{(22)}e^{-s\tau_k}\right)^{-1} B_{s}^{(2)},
\end{equation}
with
\[
A_k^{(22)} = U^{T}A_k V \text{, } C_{s}^{(2)} =C_{s}V \text{ and } B_{s}^{(2)} = U^{T}B_{s}
\] in which the columns of the matrices $U$ and $V$ span an orthonormal basis for the left and right null space of $E$, respectively, see \Cref{subsec:stability_DDAE}.

This asymptotic transfer function owes its name to the fact that it converges to \eqref{eq:transferfunction} at high frequencies (see \cite[Equations (3.2)-(3.3) and Proposition~3.3]{gumussoy2011}.
\begin{remark}
For time-delay systems of the form
\begin{equation}
\label{eq:state_space_direct_feedthrough}
\left\{
\begin{array}{rcl}
\dot{x}(t) &=& \displaystyle \sum_{k=1}^{m_A} A_k\, x(t-h_{A,k}) + \sum_{k=1}^{m_{B_2}} B_{2,k}\, w(t-h_{B_2,k}) - \sum_{k=1}^{m_H} H_k\, \dot{x}(t-h_{H,k})\\ [12px]
z(t) &=& \displaystyle \sum_{k=1}^{m_{C_2}} C_{2,k}\, x(t-h_{C_2,k}) + \sum_{k=1}^{m_{D_{22}}} D_{22,k} \, w(t-h_{D_{22},k}),
\end{array}
\right.
\end{equation}
the asymptotic transfer function is given by
\[
T_{zw,a}(s) = \sum_{k=1}^{m_{D_{22}}} D_{22,k}e^{-s h_{D_{22},k}}.
\]
For example, the asymptotic transfer function associated with \eqref{eq:hinfnorm_example2}  is equal to 
\begin{equation}
\label{eq:hinfnorm_example2_ta}
T_{zw,a}(s) = 1 + e^{-s\tau_1}-2e^{-s\tau_2}.
\end{equation}
\end{remark}

Next, we define the strong \hinfnorm{} of \eqref{eq:transferfunction} as the smallest upper bound for the \mbox{\hinfnorm{}} which insensitive to (sufficiently) small delay changes.
\begin{definition}[{\cite[Definitions~4.2 and 4.4]{gumussoy2011}}]
The strong \hinfnorm{} of \eqref{eq:transferfunction} is given by
\begin{equation*}
|||T_{zw}(s;\vec{\tau})|||_{\hinf} :=  \lim_{\epsilon\searrow 0+} \sup \{\|T_{zw}(\cdot;\vec{\tau}_{\epsilon})\|_{\hinf}:  \vec{\tau}_{\epsilon} \in \mathcal{B}(\vec{\tau},\epsilon) \cap \R_{+}^{m} \}
\end{equation*}
with $\mathcal{B}(\vec{\tau},\epsilon)$ as defined in \Cref{def:strong_spectral_abscissa} and $T_{zw}(\cdot;\vec{\tau}_{\epsilon})$ the transfer function in \eqref{eq:transferfunction} where the delays are replaced by $\vec{\tau}_{\epsilon}$.
Similarly, the strong \hinfnorm{} of \eqref{eq:asymptotic_tf} is given by
\begin{equation*}
|||T_{zw,a}(s;\vec{\tau})|||_{\hinf} := \lim_{\epsilon\searrow 0+} \sup \{\|T_{zw,a}(\cdot;\vec{\tau}_{\epsilon})\|_{\hinf}:  \vec{\tau}_{\epsilon} \in \mathcal{B}(\vec{\tau},\epsilon) \cap \R_{+}^{m} \}.
\end{equation*}
\end{definition}

In contrast to the nominal \hinfnorm{}, the strong \hinfnorm{} continuously depends on both the entries of the matrices in and the delay values \cite[Thoerem~4.5]{gumussoy2011}.

Next, we will give a characterization for the strong \hinfnorm{} of both \eqref{eq:transferfunction} and \eqref{eq:asymptotic_tf} that is useful from a computational point of view.
\begin{proposition}
\label{prop:strong_hinfnorm}
Under the assumption of strong exponential stability, the strong \hinfnorm{} of the asymptotic tranfer function \eqref{eq:asymptotic_tf} is given by
\begin{equation}
\label{eq:strong_hinfn_asympt}
|||T_{zw,a}|||_{\hinf} = \max_{\vec{\theta}\in[0,2\pi)^{m}} \left\|C_{s}^{(2)} \left(A_0^{(22)} + \sum_{k=1}^{m} A_k^{(22)}e^{-\jmath\theta_k}\right)^{-1} B_{s}^{(2)}\right\|_2.
\end{equation}
For state-space systems of the form \eqref{eq:state_space_direct_feedthrough} this expression reduces to,
\[
|||T_{zw,a}(s)|||_{\hinf} = \max_{\vec{\theta}\in[0,2\pi)^{m_{D_{22}}}} \left\| \sum_{k=1}^{m_{D_{22}}} D_{22,k} e^{-\jmath \theta_k}\right\|_2.
\]
The strong \hinfnorm{} of \eqref{eq:transferfunction} is given by
\[
|||T_{zw}||_{\hinf} = \max\left\{\|T_{zw}\|_{\hinf},|||T_{zw,a}|||_{\hinf}\right\}.
\]
\end{proposition}
The strong \hinfnorm{} of \eqref{eq:transferfunction} can thus be computed using a two-step approach. First the strong \hinfnorm{} of the associated asymptotic transfer function is computed by solving the optimization problem in \eqref{eq:strong_hinfn_asympt}. Next, it is verified whether the spectral norm of the frequency response attains a higher value at a finite frequency using the approach mentioned above \Cref{example:hinfnorm1}. This procedure is described in more detail in \cite{gumussoy2011}. 

As small delay perturbations are inevitable in practice, the function \verb|tds_hinfnorm| will always return the strong \hinfnorm{} of a time-delay system. In the following example, we apply this function to the system in \eqref{eq:hinfnorm_example2}.

\begin{example}
	\label{example:strong_hinfnorm}
Let us reconsider \eqref{eq:hinfnorm_example2} from \Cref{example:hinfnorm2} with $\tau_1 = 1$ and $\tau_2 = 2$. The corresponding asymptotic transfer function is given by \eqref{eq:hinfnorm_example2_ta}. It follows from \Cref{prop:strong_hinfnorm} that 
\[
|||T_{zw,a}|||_{\hinf} = \max_{\vec{\theta}\in[0,2\pi)^2} \|1+e^{-\jmath\theta_1}-2e^{-\jmath\theta_2} \|_2 = 4.
\]
\Cref{fig:ex17_sv1} shows that the spectral norm of the frequency response does not attain a higher value, meaning that the strong \hinfnorm{} of \eqref{eq:hinfnorm_example2} is equal to 4. Let us verify this result using the \verb|tds_hinfnorm|-function.
\begin{lstlisting}
>> [hinf,wPeak] = tds_hinfnorm(tds)
hinf =

     4

wPeak =

   Inf
\end{lstlisting}
Notice that \matlabfun{wPeak} now equals $\infty$, indicating that strong \hinfnorm{} of \eqref{eq:hinfnorm_example2_tf} is equal to the strong \hinfnorm{} of the associated asymptotic transfer function in \eqref{eq:hinfnorm_example2_ta}. (Note however that it does not imply that $\lim_{\omega\rightarrow \infty} \sigma_1\left(T_{zw}(\jmath\omega\right)$ exists.)
\end{example}

\begin{remark}
Similar to the strong spectral abscissa of the associated delay-difference equation ($C_D$), the strong \hinfnorm{} of the asymptotic transfer function can be computed separately in \packageName{}. To this end, we first obtain a representation for the asymptotic transfer function using the function \matlabfun{get_asymptotic_transfer_function}. Next, we can apply \matlabfun{tds_hinfnorm} on the resulting object to compute its strong \hinfnorm{}.
\begin{lstlisting}
Ta = get_asymptotic_transfer_function(tds);
[hinfa,wPeaka] = tds_hinfnorm(Ta)
hinfa =

     4

wPeaka =

   Inf
\end{lstlisting}
\end{remark}
In this section, we introduced the \hinfnorm{} and the strong \hinfnorm{} and briefly discussed their applicability. Next, we will discuss how \packageName{} can be used for $\hinf$- synthesis problems. 

\section{Designing controllers that minimize the strong \hinfnorm{}}
\label{subsec:optimize_hinfnorm}
This subsection describes how \packageName{} can be used to design output feedback controllers that (locally) minimize the strong \hinfnorm{} of the corresponding closed-loop system. To this end, we will consider the following state-space representation
\begin{equation}
\label{eq:state_space_hinfnorm_opt}
\left\{
\begin{array}{rcc}
    \displaystyle E\dot{x}(t) &=& \displaystyle \sum\limits_{k=1}^{m_A} A_k\, x(t-h_{A,k}) + \sum_{k=1}^{m_{B_1}} B_{1,k}\, u(t-h_{B_1,k}) + \sum_{k=1}^{m_{B_2}} B_{2,k}\, w(t-h_{B_{2},k}) \\[13pt] && \displaystyle  -\sum_{k=1}^{m_H} H_k \dot{x}(t-h_{H,k}) \\[15pt]
     y(t) &=& \displaystyle  \sum\limits_{k=1}^{m_{C_1}} C_{1,k}\, x(t-h_{C_1,k}) + \sum_{k=1}^{m_{D_{11}}} D_{11,k}\, u(t-h_{D_{11},k}) + \sum_{k=1}^{m_{D_{12}}} D_{12,k}\, w(t-h_{D_{12},k}) \\[15pt]
     z(t) &=& \displaystyle  \sum\limits_{k=1}^{m_{C_2}} C_{2,k}\, x(t-h_{C_{2},k}) +\sum_{k=1}^{m_{D_{21}}} D_{21,k}\, u(t-h_{D_{21},k})+    \sum_{k=1}^{m_{D_{22}}} D_{22,k}\, w(t-h_{D_{22},k})
\end{array}
\right.
\end{equation}
with $x\in \R^{n}$ the state variable, $u\in\R^{p_1}$ the control input (related to the \textit{“actuators”} that steer the system), $y\in\R^{q_1}$ the measured output (representing the available measurements), $w\in\R^{p_2}$ the performance input (representing \eg noise signals), and $z\in\R^{q_2}$ the performance output (representing \eg error signals). As always, the system matrices are  real-valued and have appropriate dimensions, and the delays are non-negative (positive). To create a representation for a time-delay system of the form \eqref{eq:state_space_hinfnorm_opt}, we again use the functions \verb|tds_create|, \verb|tds_create_neutral| and \verb|tds_create_ddae|.
\begin{itemize}
    \item \texttt{tds\_create(A,hA,B1,hB1,C1,hC1,D11,hD11,B2,hB2,C2,hC2,D12,hD12,D21,...}\newline
    \texttt{hD21,D22,hD22)} allows to create retarded state-space models, \ie $E = I_n$ and $m_H=0$.
    \item \texttt{tds\_create\_neutral(H,hH,A,hA,B1,hB1,C1,hC1,D11,hD11,B2,hB2,C2,hC2,...}\linebreak\texttt{D12,hD12,D21,hD21,D22,hD22)} allows to create neutral state-space models, \ie $E=I_n$ and $m_H>0$.
    \item \texttt{tds\_create\_ddae(E,A,hA,B1,hB1,C1,hC1,D11,hD11,B2,hB2,C2,hC2,D12,...}\newline \texttt{hD12,D21,hD21,D22,hD22)} allows to create delay descriptor systems, \ie $E$ is not necessarily the identity matrix can can be  singular and $m_H = 0$. 
\end{itemize}
The arguments \texttt{H}, \texttt{A}, \texttt{B1}, \texttt{C1}, \texttt{D11}, \texttt{B2}, \texttt{C2}, \texttt{D12}, \texttt{D21}, and \texttt{D22} should be cell arrays containing the system matrices. The arguments \texttt{hH}, \texttt{hA}, \texttt{hB1}, \texttt{hC1}, \texttt{hD11}, \texttt{hB2}, \texttt{hC2}, \texttt{hD12},\texttt{hD21}, and \texttt{hD22} should be arrays containing the delay values. The argument \texttt{E} of \matlabfun{tds_create_ddae} should be a (real-valued) matrix.  As often only a limited number of fields are present, it is also possible to specify the necessary fields using key-value pairs. Furthermore, when using the key-value based syntax, zero delays do not need to be specified. For example, to create a representation for the retarded state-space system \[
\left\{
\begin{array}{rcl}
     \dot{x}(t) &= & A_1 x(t-h_{A,1}) + A_2 x(t-h_{A,2}) + B_1 u(t) + B_2 w(t) \\
     y(t) &=&C_1 x(t) + D_{12,1} w(t-h_{D_{12,1}}) + D_{12,2} w(t-h_{D_{12,2}}) \\
     z(t) &=& C_2 x(t)
\end{array}
\right.
\]
the following syntax can be used
\begin{lstlisting}
tds_create(A,hA,B1,[0],C1,[0],'B2',B2,'C2',C2,'D12',D12,hD12)
\end{lstlisting}
with \texttt{A}, \texttt{B1}, \texttt{C1}, \texttt{B2}, \texttt{C2}, and \texttt{D12} cell arrays and \texttt{hA} and \texttt{hD12} arrays.

The considered $\mathcal{H}_{\infty}$-synthesis problem consists of designing a feedback control law of the form \eqref{eq:dynamic_output_feedback_controller} or \eqref{eq:static_output_feedback_controller} (which generates inputs $u$ based on the measurements $y$) such that the (strong) \hinfnorm{} of the resulting closed-loop system from $w$ to $z$ is minimized, as depicted in \Cref{fig:hinf_control}. 
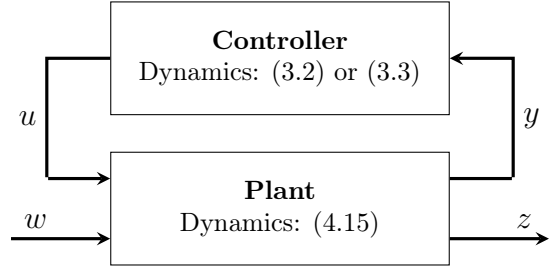
\begin{figure}[!h]
	\centering
	\begin{tikzpicture}
	\node[anchor=center,draw,rectangle,minimum width=4.5cm,minimum height=1.5cm] (plant) at (0,0) {\begin{tabular}{c} \textbf{Plant}\\ Dynamics: \eqref{eq:state_space_hinfnorm_opt}\end{tabular}};
	\node[anchor=center,draw,rectangle,minimum width=4.5cm,minimum height=1.5cm] (controller) at (0,2) {\begin{tabular}{c} \textbf{Controller}\\
		Dynamics: \eqref{eq:dynamic_output_feedback_controller}  or \eqref{eq:static_output_feedback_controller} \end{tabular}
	};
	\node (w) at (-3.7,-0.4) {};
	\node (z) at (3.7,-0.4) {};
	\node[inner sep=0,outer sep=0] (t1) at (-3.1,0.4) {};
	\node[inner sep=0,outer sep=0] (t2) at (3.1,0.4) {};
	\draw[stealth-,very thick] (plant.west)++(0,-0.4) -- (w) node [near end,above] {\large $w$};
	\draw[-stealth,very thick] (plant.east)++(0, -0.4) -- (z) node [near end,above] {\large $z$};
	\draw[-stealth,very thick] (plant.east)++(0, 0.4) -- (t2) -- (t2 |- controller.east) node[midway,right]{\large $y$} -- (controller.east);
	\draw[very thick] (t1) -- (t1 |- controller.west) node[midway,left]{\large $u$} -- (controller.west);
	\draw[stealth-,very thick] (plant.west)++(0,0.4) -- (t1);
	\end{tikzpicture}
	\caption{Set-up for the design of $\mathcal{H}_{\infty}$-controllers.}
	\label{fig:hinf_control}
\end{figure}
 As for stabilization in \Cref{sec:stabilization}, we will employ an optimization approach that consists of directly minimizing the (strong) \hinfnorm{} with respect to the controller parameters. To this end, the functions \matlabfun{tds_hiopt_static} and \matlabfun{tds_hiopt_dynamic} can be used. These functions take the same arguments as respectively \matlabfun{tds_stabopt_static} and \matlabfun{tds_stabopt_dynamic}. More specifically, the first argument of both functions is a \verb|tds_ss| representation for the open-loop plant, while \matlabfun{tds_hiopt_dynamic} takes the desired controller order as a second mandatory argument. The remaining arguments should be passed as key-value pairs. 
\begin{itemize}
    \item \verb|'options'| allows to specify additional options. The \verb|'options'|-argument should be a \verb|struct| with the necessary fields. A valid \verb|'options'|-structure can be created using the \matlabfun{tds_hiopt_options}-function  (for more details see \Cref{sec:tds_hiopt}).
    \item \verb|'initial'| allows to specify the initial values for the optimization variables. The \verb|'initial'|-argument should be a cell array containing the initial controllers.
    \item \verb|'mask'| allows to specify which entries of the matrices in \eqref{eq:dynamic_output_feedback_controller} or \eqref{eq:static_output_feedback_controller} that should remain fixed to their basis value (specified using the \verb|'basis'| argument), \ie the entries that should not be touched by the optimization algorithm. 
    \item \verb|'basis'| allows to specify the value of the entries of \eqref{eq:dynamic_output_feedback_controller} or \eqref{eq:static_output_feedback_controller} that are not optimized (indicated using the \verb|'mask'| argument). 
\end{itemize}
\begin{remark}
The functions \matlabfun{tds_hiopt_static} and \matlabfun{tds_hiopt_dynamic} will test whether the provided initial controllers are strongly stabilizing. If an initial controller is not strongly stabilizing, the functions \matlabfun{tds_stabopt_static} and \matlabfun{tds_stabopt_dynamic} are used to design a strongly stabilizing controller starting from the initial value. If this does not result in a strongly stable controller, this initial controller is skipped in the minimization of the \hinfnorm{}. If no strongly stabilizing controller can be found, an error is thrown.
\end{remark}
\begin{example}[{\cite[Section~7.1]{gumussoy2011}}]
	\label{example:Hinf_design}
Consider the following delay-descriptor system.
\begin{equation}
\label{eq:hiopt_ex1}
    \left\{
    \begin{array}{rcl}
        \begin{bmatrix}
            1 & 0 \\
            0 & 0
        \end{bmatrix} \dot{x}(t) & = & \begin{bmatrix}
            -0.1 & -1 \\ 1 & -1
        \end{bmatrix}x(t) + \begin{bmatrix}
            0 \\ 1
        \end{bmatrix}u(t) + \begin{bmatrix}
            0 \\ 1
        \end{bmatrix}w(t)\\[3mm]
         y(t) &=& \begin{bmatrix}
             0 & 1 \\
             0 & 0
         \end{bmatrix}x(t-\tau_1) + \begin{bmatrix}
             0 & 0 \\ 0 & 1
         \end{bmatrix} x(t-\tau_2)\\[4mm]
         z(t) &=& \begin{bmatrix}
             2 & -1
         \end{bmatrix}x(t)\\
        
    \end{array}
    \right.
\end{equation}
with $\tau_1 = 1$ and $\tau_2 = 2$. As always, we start with creating a \verb|tds_ss| representation.
\begin{lstlisting}
E = [1 0;0 0];
A = {[-0.1 -1;1 -1]}; B1 = {[0;1]}; B2 = {[0;1]};
C1 = {[0 1; 0 0],[0 0;0 1]};
C2 = {[2 -1]};
sys = tds_create_ddae(E,A,0,B1,0,C1,[1 2],'B2',B2,'C2',C2);
\end{lstlisting}
Next we run the optimization algorithm to design a static feedback controller.
\begin{lstlisting}
options = tds_hiopt_options('nstart',1);
K_initial = [0.25 -0.5];
[K,cl] = tds_hiopt_static(sys,'options',options,'initial',K_initial);
\end{lstlisting}
We obtain the following feedback law
\begin{equation}
    \label{eq:hiopt_ex1_static}
u(t) = [-0.3533 -0.1012]\ y(t), 
\end{equation}
and the strong \hinfnorm{} of the corresponding closed-loop system is equal to 1.8333. \Cref{fig:hiopt_ex1_svplot} shows the spectral norm of the closed-loop frequency response in function of $\omega$. We observe that it attains a maximum of 1.833 at $\omega = 0$. However, \Cref{fig:hiopt_ex1_svplot2} shows the same plot after we made a small perturbation to the first delay. At high frequencies, we now see a jump to almost the \hinfnorm{}. Let us therefore compute the strong H-infinity norm of the asymptotic transfer function using the following code.
\begin{lstlisting}
>> Ta = get_asymptotic_transfer_function(cl);
>> hinfa = tds_hinfnorm(Ta)

hinfa =

    1.8331
\end{lstlisting}
We conclude that (for this example) the standard \hinfnorm{} and the strong \hinfnorm{} of the asymptotic transfer function (almost) coincide.
\begin{figure}
    \centering
    \begin{subfigure}{0.49\linewidth}
      \includegraphics[width=\linewidth]{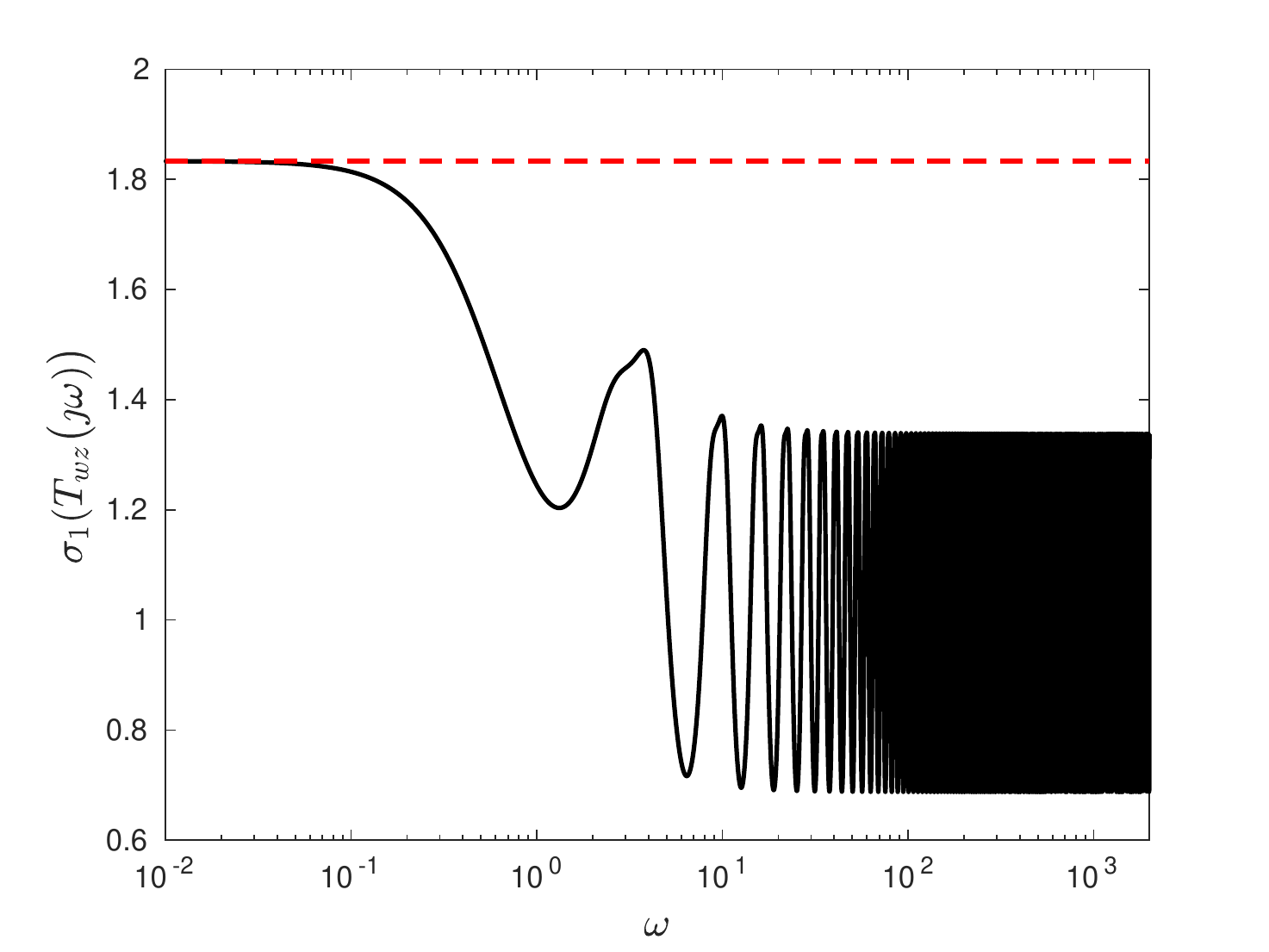}
      \caption{$\tau_1 = 1$ and $\tau_2 = 2$}
      \label{fig:hiopt_ex1_svplot}
    \end{subfigure}
    \begin{subfigure}{0.49\linewidth}
      \includegraphics[width=\linewidth]{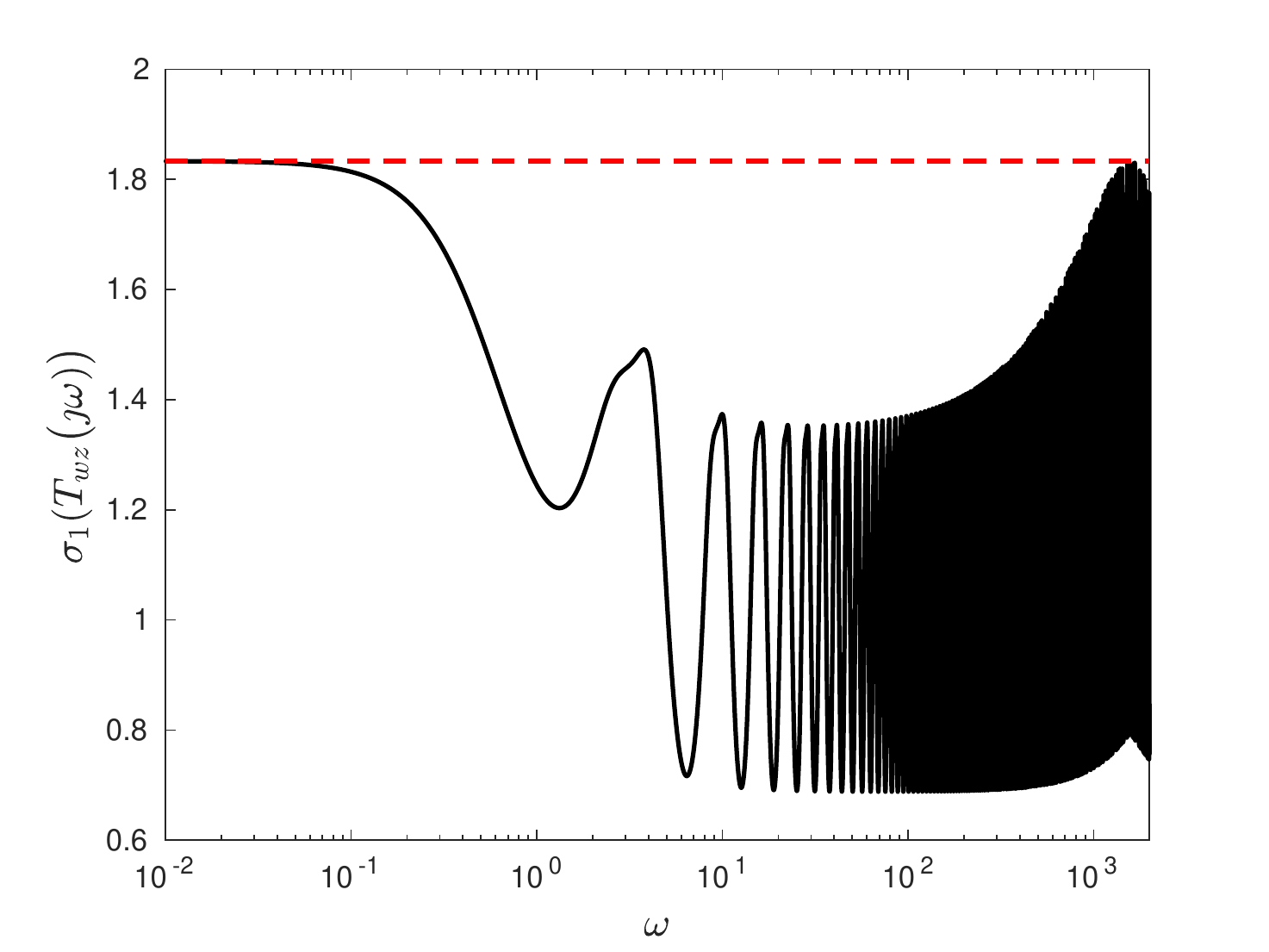}
      \caption{$\tau_1 = 0.999$ and $\tau_2 = 2$}
        \label{fig:hiopt_ex1_svplot2}
    \end{subfigure}
    \caption{Singular value plot of the frequency response of the feedback interconnection of \eqref{eq:hiopt_ex1} and \eqref{eq:hiopt_ex1_static} for different values of $\tau_1$.}
    \label{fig:hiopt_ex1}
\end{figure}
\end{example}

Next, we illustrate how to design a structured controller using \verb|tds_hiopt_dynamic|.
\begin{example}
\label{ex:measurement_noise}
In this academic example we will revisit the set-up from \Cref{example:fixed_entries}, in which we designed a stabilizing dynamic output feedback controller of order two under the constraint that the matrices $B_c$ and $D_c$ are diagonal. Although the obtained controller was stabilizing, we have no indication of the performance of the closed-loop system, \eg ``How well are external disturbances attenuated?''. For the sake of this example, imagine therefore that the output measurements of the plant are disturbed by a noise signal $w$ (see \Cref{fig:example_perf}). 
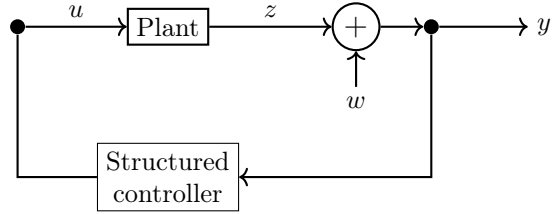
\begin{figure}[!h]
    \center
    \begin{tikzpicture}
      \node[draw,rectangle,thick,align=center] (plant) at (0,1) {Plant};
      \node[draw,thick,circle,inner sep=2pt] (sum) at (2.5,1) {\large $+$ };
      \node[fill,circle,inner sep=2pt] (node) at (3.5,1) {};
      \node[fill,circle,inner sep=2pt] (node2) at (-2,1) {};
      \node (w) at (2.5,0) {$w$};
      \node (y) at (5,1) {$y$};
      \draw[->,thick] (w) -- (sum);
      \draw[->,thick] (sum) -- (node);
      \draw[->,thick] (node) -- (y);
      \node[draw,rectangle,align=center] (controller) at (0,-1) {Structured \\ controller};
      \draw[->,thick] (node) |- (controller);
      \draw[-,thick] (controller) -| (node2);
      \draw[->,thick] (node2) -- node[above] {$u$} (plant);
      \draw[->,thick] (plant) -- node[above]{$z$} (sum);
    \end{tikzpicture}
    \caption{Situation sketch for \Cref{ex:measurement_noise}: the control output is disturbed by measurement noise.}
    \label{fig:example_perf}
\end{figure}
We now are interested in quantifying the effect of this measurement noise on the ``real'' output of the plant ($z$). To this end, let us extend the state-space model in \eqref{eq:fixed_entries_sys} with the  performance input and output channels:
\begin{equation}
\label{eq:hiopt_structured}
\left\{
\begin{array}{rcl}
    \dot{x}(t)&=&
         A_0 
         x (t) + 
         A_1 
         x(t-1) +  B u(t-5)   \\
     y(t) &=&   C 
         x(t) + w(t)\\ 
     z(t) &=& 
         C x(t) , 
\end{array}
\right.
\end{equation}
with the matrices $A_0$, $A_1$, $B$, and $C$ as defined in \eqref{eq:fixed_entries_sys}. In this case, $|||T_{zw}(s;\vec{\tau})|||_{\hinf}$ thus quantifies the worst-case amplification of measurement noise at the output of the plant. For the controller parameters obtained in \Cref{example:fixed_entries} the strong \hinfnorm{} is equal to 187.6219. Next we will design a dynamic output feedback controller (with the same structure) that locally minimizes the closed-loop \hinfnorm{}. To this end, we can use the following code. (For illustrative purposes, the initial controller is chosen such that the optimization procedure converges quickly.)
\begin{lstlisting}
% Create a representation for (4.18)
A0 = [-3.5 -2 3 -1.5;3 -6.5 3 -1;4 -2 -7 2.5;-2.5 1.5 -3.5 1];
A1 = [-0.25 0.5 -0.25 0.25;0.5 -0.5 1 -0.5;...
       1 -1 0.5 -0.15;-0.15 -0.15 0.1 0.2];
B = [-1 2;-2 1;3 -2.5;-4  0];
C = [0 1 1 0;1 0 0 1];
D12 = eye(2);
sys=tds_create({A0,A1},[0 1],{B},5,{C},0,'D12',{D12},'C2',{C});

% Define the mask and basis arguments 
mask = ss(ones(2),eye(2),ones(2),eye(2));
basis = ss(zeros(2),zeros(2),zeros(2),zeros(2));

% Design a structured controller
nc = 2;
initial = ss([-0.25 0.75;-1.5 0.5],diag([0 -0.5]),...
                [-0.1 0.1;-1 3.5],diag([0.01 -1]));
options_stabopt = tds_stabopt_options('nstart',1);
[K,cl] = tds_hiopt_dynamic(sys,nc,'options',options_stabopt,...
             'initial',initial,'basis',basis,'mask',mask);
\end{lstlisting}
The obtained controller matrices are
\[
\begin{array}{ll}
A_c = \begin{bmatrix}
   -0.3029 & 0.7356 \\
   -1.502 & 0.4897
\end{bmatrix}\text{, }&B_c = \begin{bmatrix}
-0.03063 & 0 \\
 0 & -0.271
\end{bmatrix}\text{, }\\[12pt] C_c = \begin{bmatrix}
-0.06593 & 0.04455 \\
-0.9919 & 3.475
\end{bmatrix} \text{, and } &
D_c = \begin{bmatrix}
 0.01272 & 0 \\
 0 & -0.944
\end{bmatrix}.
\end{array}
\]
The corresponding closed-loop system is strongly exponentially stable (the strong spectral abscissa is equal to -0.1046) and the strong \hinfnorm{} is now equal to 3.2424. The controller designed above thus (significantly) reduces the strong \hinfnorm{}, at the cost of an increase in the strong spectral abscissa. 
\end{example}

Below we retake the previous example in a more practical setting. In a typical $\hinf$-controller design problem we not only want to suppress measurement noise, but also have good disturbance rejection and good reference tracking. Furthermore, it is often desirable to limit the control effort to avoid saturating the actuators and to satisfy the physical constraints of the system. For good practical performance, one thus has to take multiple objectives into account at the same. However, considering these different objectives separately can lead to a waterbed effect in which reducing one objective, increases another one. For example, to have good disturbance rejection, a large loop gain is required. But, a large loop gain implies that sensor noise is barely suppressed and that the system is more sensitive to unmodelled dynamics. Therefore, a trade-off between the different control goals needs to be made. To achieve this, we will optimize a weighted combination of these performance objectives. Furthermore, this weighting needs to be frequency-dependent, as different objectives are relevant in different frequency ranges. For example, disturbance signals are typically significant for low frequencies, while measurement noise and modelling uncertainties are typically significant for higher frequencies.

The design approach sketched above gives rise to a so called mixed-sensitivity $\hinf$-synthesis problem, \ie
find a controller that minimizes the following objective function
\begin{equation}
\label{eq:mixed_sensitivity}
\left\|\begin{bmatrix}
W_1 S \\
W_2 KS \\
W_3 T
\end{bmatrix}\right\|_{\hinf}
\end{equation}
with $S := (I-PK)^{-1}$ the sensitivity function, $T:= (I-PK)^{-1}PK$ the complementary sensitivity function, $P$ the transfer function of the plant, $K$ the transfer function of the controller and $W_1$, $W_2$, and $W_3$ appropriately chosen weighting functions. The three components in this objective function, each play a different role. 
\begin{itemize}
	\item By chosing $W_1$ large for low frequencies we obtain a small $S$, which assures good reference tracking and a good disturbance rejection.
	\item By chosing $W_2$ large we obtain a small $KS$, which limits the control effort.
	\item By chosing $W_3$ large for high frequencies we obtain a small $T$, which implies good robustness against model uncertainties and good measurement noise attenuation.
\end{itemize}
 Below we briefly illustrate how such a design problem can be tackled in \packageName{}. For an extensive overview of how such mixed-sensitivity synthesis problems pop up in (robust and optimal) controller design and their interpretation in terms of performance and robust stability, we refer the interested reader to \cite{zhou1996robust} and \cite{ozbay2018frequency}.

\textbf{Note.} In the current implementation of \packageName{} the employed weighting functions need to be stable. (A similar constraint applies to the \matlabfun{mixsyn} function in \matlab{}.)

\begin{example} 
\label{example:mixed_sensitivity}
In this example we again consider the system from \Cref{example:fixed_entries}. We now want to find controller parameters that (locally) minimize \eqref{eq:mixed_sensitivity} for
\[
W_1(s) = 0.25 \frac{s+30}{s+2.5} I_2\text{, } W_2(s) = 0\text{, and }W_3(s) = 2 \frac{s+125}{s+1000}.
\]
 To tackle this design problem in \packageName{}, we consider the set-up depicted in \Cref{fig:mixed_sensitivity}. The objective function now corresponds to the \hinfnorm{} of the closed-loop transfer function from $w$ to $\begin{bmatrix}
z_1 & z_2 & z_3
\end{bmatrix}^{T}$.
\begin{figure}[!h]
    \center
    \begin{tikzpicture}
      \node[draw,rectangle,minimum size=0.8cm,thick] (plant) at (0,1) {$P$};
      \node[draw,thick,circle,inner sep=2pt] (sum) at (2,0) {\large $+$};
      \node[fill,circle,inner sep=2pt] (node3) at (2,1) {};
      \node[fill,circle,inner sep=2pt] (node1) at (2,-1) {};
      \node[fill,circle,inner sep=2pt] (node2) at (-2,-1) {};
      \node[draw,thick,minimum size=0.8cm,rectangle] (W2) at (-3.5,-1) {$W_2$};
      \node[draw,thick,minimum size=0.8cm,rectangle] (W1) at (3.5,-1) {$W_1$};
      \node[draw,thick,minimum size=0.8cm,rectangle] (W3) at (3.5,1) {$W_3$};
      \node (w) at (3,0) {$w$};
      \node (z1) at (5,-1) {$z_1$};
      \node (z2) at (-5,-1) {$z_2$};
      \node (z3) at (5,1) {$z_3$};
      
      \draw[->,thick] (w) -- (sum);
      \draw[->,thick] (sum) -- (node1);
      \node[draw,minimum size=0.8cm,rectangle] (controller) at (0,-1) {$K$};
      
      \draw[->,thick] (node1) -- node[above]{$y$} (controller);
      
      \draw[->,thick] (node1) -- (W1);
      \draw[->,thick] (W1) -- (z1);
      
      \draw[-,thick] (controller) --  (node2);
      \draw[->,thick] (node2) -- (W2);
      \draw[->,thick] (node2) |- node[near end,above]{$u$} (plant);
      \draw[->,thick] (W2) -- (z2);
      
      \draw[-,thick] (plant) --  (node3);
      \draw[->,thick] (node3) --  (sum);
      \draw[->,thick] (node3) --  (W3);
      \draw[->,thick] (W3) -- (z3);
    \end{tikzpicture}
    \caption{Situation sketch for \Cref{example:mixed_sensitivity}.}
    \label{fig:mixed_sensitivity}
\end{figure}
To incroporate the performance outputs $z_1$ and $z_3$\footnote{In this example, we ignore $z_2$ as $W_2(s)=0$, but a similar approach can be used to incorporate this output.}, we need to augment the model with two auxiliary state variables $\xi_1$ and $\xi_3$:
\[
\left\{
\begin{array}{rcl}
    \dot{\xi}_1(t) &=& -2.5 \xi_1(t) + (Cx(t) + w(t)) \\
    z_1(t) &=& 6.875 \xi_1(t) + 0.25 (Cx(t) + w(t))
\end{array}
\right.
\]
and
\[
\left\{
\begin{array}{rcl}
    \dot{\xi}_3(t) &=& -1000 \xi_3(t) + Cx(t)  \\
    z_3(t) &=& -1750 \xi_3(t) + 2 Cx(t).
\end{array}
\right.
\]
Adding these dynamics to \eqref{eq:fixed_entries_sys}, we obtain the following state-space model
\[
\setlength{\arraycolsep}{2pt}
\left\{
\begin{array}{rcl}
\begin{bmatrix}
\dot{x}(t) \\
\dot{\xi}_1(t) \\
\dot{\xi}_3(t)
\end{bmatrix}
 & = & \begin{bmatrix}
A_0 & 0 & 0 \\
C & -2.5 I & 0 \\
C & 0 & -1000 I
\end{bmatrix}
\begin{bmatrix}
x(t) \\
\xi_1(t) \\
\xi_3(t)
\end{bmatrix} + 
\begin{bmatrix}
A_1 & 0 & 0 \\
0 & 0 & 0 \\
0 & 0 & 0
\end{bmatrix} \begin{bmatrix}
x(t-1) \\
\xi_1(t-1) \\
\xi_3(t-1)
\end{bmatrix} + \begin{bmatrix}
B \\
0 \\
0 
\end{bmatrix} u(t-5) + \begin{bmatrix}
0 \\
I \\
0
\end{bmatrix} w(t)\\[15pt]
 y(t) &=& \begin{bmatrix}
     C & 0 & 0 
 \end{bmatrix} \begin{bmatrix}
     x (t)\\
     \xi_1(t) \\
     \xi_3(t)
 \end{bmatrix} + w(t)\\ [15pt]
\begin{bmatrix}
z_1(t)\\
z_3(t)
\end{bmatrix} &=&  \begin{bmatrix}
        0.25 C & 6.875I & 0 \\
        2 C & 0 & -1750 I
        \end{bmatrix}
        \begin{bmatrix}
         x(t) \\
         \xi_1(t) \\
         \xi_3(t)
     \end{bmatrix} + \begin{bmatrix}
        0.25 I \\
        0 
        \end{bmatrix}w(t).
\end{array}
\right.
\]
Now, we can apply the same procedure as before. 
\begin{lstlisting}
% Open-loop system matrices and delays
A0 = [-3.5 -2 3 -1.5;3 -6.5 3 -1;4 -2 -7 2.5;-2.5 1.5 -3.5 1];
A1 = [-0.25 0.5 -0.25 0.25;0.5 -0.5 1 -0.5;...
       1 -1 0.5 -0.15;-0.15 -0.15 0.1 0.2];
B = [-1 2;-2 1;3 -2.5;-4  0]; C = [0 1 1 0;1 0 0 1];
tau1 = 1; tau2 = 5; n = 4; p = 2; q = 2;

% weighting functions
K1 = 0.25; a1 = 30; b1 = 2.5; K3 = 2; a3 = 125; b3 = 1000;

% Augmented system matrices
A0e=[A0 zeros(n,2*q);C -b1*eye(q) zeros(q);...
     C zeros(q) -b3*eye(q)];
A1e = zeros(n+2*q); A1e(1:n,1:n) = A1;
B1e = [B;zeros(2*q,p)]; B2e = [zeros(n,q);eye(q);zeros(q)];
C1e = [C zeros(q,2*q)];
C2e = [K1*C K1*(a1-b1)*eye(q) zeros(q);...
       K3*C zeros(q) K3*(a3-b3)*eye(q)];
D12 = eye(q); D22 = [K1*eye(q);zeros(q)];

sys = tds_create({A0e,A1e},[0 tau1],{B1e},tau2,{C1e},0,...
           'D12',{D12},'B2',{B2e},'C2',{C2e},'D22',{D22});

%% Define the mask and basis arguments 
mask = ss(ones(2),eye(2),ones(2),eye(2));
basis = ss(zeros(2),zeros(2),zeros(2),zeros(2));

%% Design a structured controller
nc = 2;
initial = ss([-0.2 0.5;-5 -8],diag([0 -2]),...
             [-0.2 -0.5;-6 -5],diag([0.01 -3]));
options_stabopt = tds_stabopt_options('nstart',1);
[K,cl] = tds_hiopt_dynamic(sys,nc,'options',options_stabopt,...
                  'initial',initial,'basis',basis,'mask',mask);
\end{lstlisting}
We obtain the following controller parameters
\[
\begin{array}{ll}
A_c = \begin{bmatrix}
-0.2215 & 0.6798 \\
-4.992 &  -7.975
\end{bmatrix}\text{, } &B_c = \begin{bmatrix}
0.07326 &  0 \\
0 & -2.139
\end{bmatrix}\text{, } \\[12pt] C_c = \begin{bmatrix}
-0.04723 & -0.07836 \\
-6.029 & -5.074
\end{bmatrix}\text{, and } &
D_c = \begin{bmatrix}
5.226 \times 10^{-4}& 0 \\
0 & -2.628 
\end{bmatrix}
\end{array}
\]
and the corresponding value of the objective function is equal to 10.7073. 

\end{example}

To conclude this section, we will discuss how the function \matlabfun{tds_hiopt_dynamic} can be used to obtain a delay-free approximation for a stable and strictly proper time-delay system. To this end we will reformulate the minimization of the mismatch error in \eqref{eq:MOR_mismatch}  as a controller design problem by defining appropriate control inputs and outputs.
\begin{example}
\label{example:MOR}
In this example we look for a delay-free approximation of the form \eqref{eq:MOR_approx} for the following time-delay system
\begin{equation}
\label{eq:MOR_sys}
\left\{
\begin{array}{rcl}
     \dot{x}(t) & = & A_0 x(t) + A_1 x(t-1) + B_2 w(t)\\
     z(t) & = & C_{2} x(t)
\end{array}
\right.
\end{equation}
with
\[
A_0 = \begin{bmatrix}
-3 & 1 \\
2 & -3
\end{bmatrix},\ A_1 = \begin{bmatrix}
-2 & 0.5 \\
0.7 & -1 
\end{bmatrix},\ 
B_2 = \begin{bmatrix}
0 \\
1
\end{bmatrix}
\text{ and }
C_2 = \begin{bmatrix}
1 & 0
\end{bmatrix},
\]
such that the approximation error in \eqref{eq:MOR_mismatch} is minimized. To be able to tackle this problem using \verb|tds_hiopt_dynamic|, we need to reformulate the approximation problem as a controller design problem. Let us therefore append \eqref{eq:MOR_sys} with the following control input and output channels:
\begin{equation}
\label{eq:MOR_extended_sys}
\left\{
\begin{array}{rcl}
     \dot{x}(t) & = & A_0 x(t) + A_1 x(t-1) + B_2 w(t)\\
     y(t) &=& w(t) \\
     z(t) & = & C_{2} x(t) - u(t).
\end{array}
\right.
\end{equation}
Eliminating the control signals $u$ and $y$, the dynamics of the closed loop of  \eqref{eq:MOR_extended_sys} and \eqref{eq:dynamic_output_feedback_controller} (with $D_c = 0$) can be described by the 
\[
\left\{
\begin{array}{rcl}
\begin{bmatrix}
\dot{x}(t) \\
\dot{x}_c(t)
\end{bmatrix} &=& \begin{bmatrix}
A_0 & 0 \\
0 & A_c
\end{bmatrix}\begin{bmatrix}
x(t) \\
x_c(t)
\end{bmatrix} + 
\begin{bmatrix}
A_1 & 0 \\
0 & 0
\end{bmatrix}\begin{bmatrix}
x(t-1) \\
x_c(t-1)
\end{bmatrix} + \begin{bmatrix}
B_2 \\
B_c
\end{bmatrix}w(t) \\ [10px]
z(t) &=& \begin{bmatrix}
C_2 & -C_c
\end{bmatrix}
\begin{bmatrix}
x(t) \\
x_c(t)
\end{bmatrix}.
\end{array}
\right.
\]
The corresponding transfer function from $w$ to $z$ is given by
\begin{multline*}
\begin{bmatrix}
C_{2} & -C_{c}
\end{bmatrix}
\left(\begin{bmatrix}
sI-A_0-A_1e^{-s} & 0 \\
0 & sI-A_c
\end{bmatrix}\right)^{-1}\begin{bmatrix}
B_2 \\
B_c
\end{bmatrix} \\
 = C_2 \left(sI-A_0-A_1e^{-s}\right)^{-1}B_2 - C_c\left(sI-A_c\right)^{-1} B_c.
\end{multline*}
The \hinfnorm{} of this closed-loop transfer function thus corresponds to the approximation error defined in \eqref{eq:MOR_mismatch}. The output feedback controller with $D_c = 0$ that minimizes the closed-loop \hinfnorm{} of \eqref{eq:MOR_extended_sys} thus corresponds to the dynamic system of the form \eqref{eq:MOR_approx} that best approximates the delay system in the sense of \eqref{eq:MOR_mismatch}. Below we compute a delay-free approximation for \eqref{eq:MOR_sys} with $n_a = 5$. 
\begin{lstlisting}
% Create a representation for the augmented system
A0 = [-3 1;2 -3]; A1 = [-2 0.5;0.7 -1];
B = [0;1]; C = [1 0];
sys = tds_create({A0 A1},[0 1],'D12',1,'B2',B,'C2',C,'D21',-1);

% Compute the approximationn
na = 5; % order of the delay-free approximatio
p = 1; q = 1; % dimensions of the input and output channels
% Enforce D_c = 0
mask = ss(ones(na),ones(na,p),ones(q,na),zeros(q,p));
basis = ss(zeros(na),zeros(na,p),zeros(q,na),zeros(q,p));
% Set initial controller
initial=ss([-0.8 0 -0.35 0.14 -0.1;-0.3 -0.5 -0.5 0.14 0.7;...
        0.18 0.21 -0.5 -0.3 -1.2;-0.04 -0.2 -0.14 -0.9 0.4;...
        -0.2 -0.86 1 -0.33 -0.4],[0.66;-0.5;0.04;0.2;-0.3],...
        [2 0 1 2 0.4],1);
options = tds_hiopt_options('nstart',1);
[approx,CL] = tds_hiopt_dynamic(sys,na,'options',options,...
          'initial',initial,'mask',mask,'basis',basis);
\end{lstlisting}
\Cref{fig:MOR} compares the original time-delay system and its delay-free approximation for $n_a = 5$. We see that the approximation matches the original system quite well and the approximation error $\| T(s) - T_{\mathrm{approx.}}(s) \|_{\hinf}$ is equal to 7.6316e-03.

\textbf{Note. } Due to the employed definition of the \hinfnorm{} for unstable systems \ie the \hinfnorm{} of an unstable system is equal to infinity, the obtained approximation is always stable.
\begin{figure}[!h]
    \centering
    \includegraphics[width=0.6\linewidth]{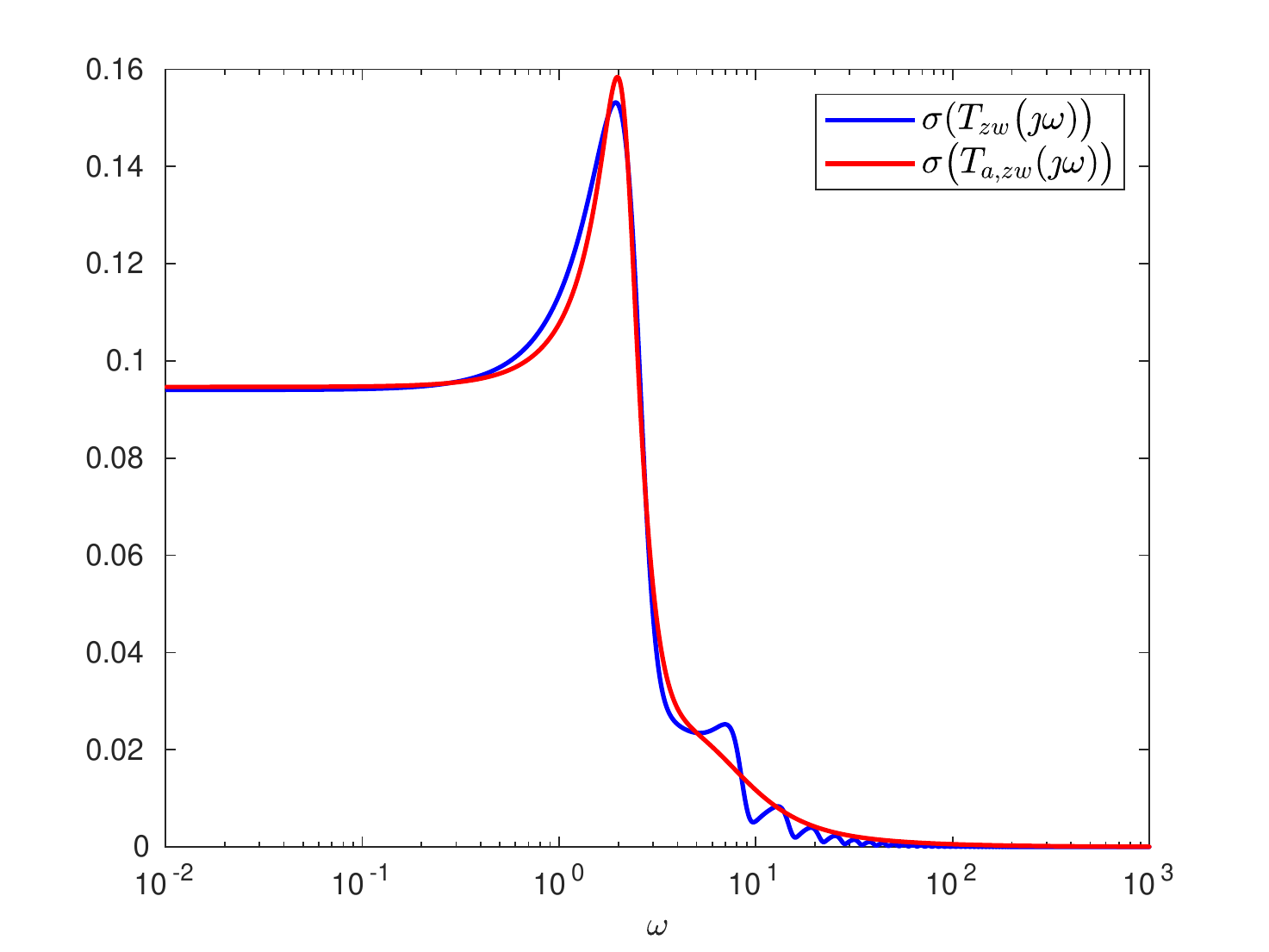}
    \caption{The magnitude frequency response of \eqref{eq:MOR_sys} and its delay-free approximation with $n_a=5$.}
    \label{fig:MOR}
\end{figure}

\end{example}
\subsection{Mixed performance criteria}
It is also possible to design controllers that minimize a combination of the strong spectral abscissa and the strong \hinfnorm{}. More specifically, we will consider the objective function
\begin{equation}
\label{eq:mixed_performance}
\alpha C + (1-\alpha) |||T_{zw}(s;\vec{\tau})|||_{\hinf}
\end{equation}
with $\alpha \in [0,1]$ a weighting coefficient that determines the trade-off between minimizing the strong spectral abscissa and the strong \hinfnorm{}. We will demonstrate how we can optimize this alternative objective function using \packageName{} in the following example.
\begin{example}
\label{example:mixed_performance}
Consider the following time-delay system
\begin{equation}
\label{eq:mixed_performance_ex}
\left\{
\begin{array}{rcl}
	\dot{x}(t) &=& \begin{bmatrix}
	1 & -2  & 4\\
	3 & 0.5 & -1 \\
	2 & 0.4 & -2
	\end{bmatrix} x(t) + \begin{bmatrix}
	1.5 & 0.3  & 2 \\
	0.7 & -0.8 & 0.4 \\
	0.5 & 0.4  & -0.9
	\end{bmatrix} x(t-1) \\[6mm] && \qquad\qquad\qquad \qquad + \begin{bmatrix}
	0.3 & 0.4 \\
	-0.7 & -0.5 \\
	0.7 & -0.1
	\end{bmatrix} u(t-0.1) + \begin{bmatrix}
	-0.7 \\
	-0.5 \\
	-0.3 
	\end{bmatrix} w(t)\\[7mm]
	y(t) &=& \begin{bmatrix}
	-1 & 0.3 & 0\\
	0.4 & 0.9 & 1
	\end{bmatrix} x(t) \\[4mm]
	z(t)&=& \begin{bmatrix}
	3 & -5 & -4
	\end{bmatrix} x(t).
\end{array} 
\right.
\end{equation}
The code below designs three static output feedback controllers for \eqref{eq:mixed_performance_ex} by minimizing the objective function in \eqref{eq:mixed_performance} for $\alpha = 1$ (\ie minimizing the  spectral abscissa), for $\alpha = 0$ (\ie minimizing the \hinfnorm{}), and for $\alpha = 0.5$ (\ie a trade-off between minimizing the spectral abscissa and the \hinfnorm{}).
\begin{lstlisting}
% Create a representation for the time-delay system
A0 = [1 -2 4;3 0.5 -1;-2 0.4 -2]; 
A1 = [1.5 0.3 2;0.7 -0.8 0.4;0.5 0.4 -0.9];
B1 = [0.3 0.4;-0.7 -0.5;0.7 -0.1]; B2 = [-0.7;-0.5;-0.3];
C1 = [-1 0.3 0;0.4 0.9 1]; C2 = [3 -5 -4];
sys=tds_create({A0,A1},[0 1],{B1},0.1,{C1},0,'B2',{B2},'C2',{C2});

% alpha = 1
options = tds_hiopt_options('nstart',1,'alpha',1);
[Dc1,cl1]=tds_hiopt_static(sys,'options',options,'initial',zeros(2));

% alpha = 0
options = tds_hiopt_options('nstart',1,'alpha',0);
[Dc2,cl2]=tds_hiopt_static(sys,'options',options,'initial',Dc1);

% alpha = 0.5
options = tds_hiopt_options('nstart',1,'alpha',0.5);
[Dc3,cl3]=tds_hiopt_static(sys,'options',options,'initial',Dc1);
\end{lstlisting}
The obtained controllers are
\begin{align}
	D_{c,1} &\, = \begin{bmatrix}
	8.4197 & -0.4036 \\
	4.3451 & 10.6842 
	\end{bmatrix}\text{,} \label{eq:ms_1}\\
	D_{c,2} &\, = \begin{bmatrix}
	7.0877 & \phantom{-}0.0571 \\
	6.5345 & 13.3927
	\end{bmatrix} \label{eq:ms_2}, \text{ and }\\
		D_{c,3} &\, = \begin{bmatrix}
	8.1324 & -0.7980 \\
	4.7536 & 11.2210 
	\end{bmatrix}. \label{eq:ms_3}
\end{align}
\Cref{tab:mixed_perf} gives the correspoding closed-loop spectral abscissa and \hinfnorm{}. As expected, when decreasing $\alpha$ more emphasis is put on minimizing the \hinfnorm{} at the cost of an increase in the spectral abscissa.

\begin{table}[!h]
    \centering
    \caption{The spectral abscissa and the \hinfnorm{} of the feedback interconnection of \eqref{eq:mixed_performance_ex} and the controllers in \eqref{eq:ms_1}, \eqref{eq:ms_2}, and \eqref{eq:ms_2}.}
    \label{tab:mixed_perf}
   	\begingroup
   	\renewcommand{\arraystretch}{1.2}
    \begin{tabular}{ccc}
    	\toprule
         & spectral abscissa & \hinfnorm{} \\\midrule
        $\alpha = 1$ & -0.8751 &  1.8061 \\
        $\alpha = 0$ & -0.6153 & 1.2908 \\
        $\alpha = 0.5$ & -0.8285 & 1.4869 \\\bottomrule
    \end{tabular}
	\endgroup
\end{table}

\end{example}
As mentioned before, the \hinfnorm{} corresponds to the reciprocal of the distance to instability with respect to a complex-valued perturbation acting on the zero-delay state matrix. From a practical point of view, the \hinfnorm{} thus only provides a crude approximation of the perturbation that is necessary to destabilize the system, as the perturbations on the model are typically real-valued and may act on all system matrices and even the delays. In the next section we will therefore consider a more realistic distance to instability.


\section{Pseudospectrum and the distance to instability}
\label{sec:robust_stability}
In this section we will consider a class of uncertain retarded time-delay systems with bounded perturbations on both the state matrices and the delays. These perturbations allow to account for mismatches between model and reality. For example, to describe a model parameter that is only known up to a certain precision (\eg due to an imprecise measuring device), we add a real-valued scalar perturbation to the model. Now instead of a single model, we obtain a collection of possible models. To examine the stability of this uncertain system for all possible parameter values, this section introduces the notions of the pseudospectral abscissa and the distance to stability.
\begin{example}[Inspired by \cite{borgioli2019novel}]
	\label{example:pseudo_definition}
	Consider the following simple time-delay model for a turning process
	\[
	\ddot{d}(t) + 2 \xi \omega\, \dot{d}(t) + \omega^2\, d(t) = \frac{k}{m}\,\big(d(t-\tau)-d(t)\big)
	\]
	with $\omega$ the natural angular frequency, $\xi$ the damping ratio, $m$ the modal mass, $k$ the cutting force coefficient, and $\tau$ the time of one revolution of the work-piece. This model can be transformed into a standard first-order delay differential equation of retarded type by considering the state vector $x(t) = \begin{bmatrix}
	d(t) & \dot{d}(t)
	\end{bmatrix}^{\top}$
	\begin{equation}
	\label{eq:turning_process}
	\dot{x}(t) = \begin{bmatrix}
	0 & 1 \\
	-\omega^2-\frac{k}{m} & -2\xi\omega
	\end{bmatrix}x(t) + \begin{bmatrix}
	0 & 0 \\
	\frac{k}{m} & 0
	\end{bmatrix}x(t-\tau).
	\end{equation}
	In this section we will examine the stability of \eqref{eq:turning_process} when the parameters $k$, $\xi$, and $\tau$ are uncertain. More specifically, we will assume that the parameter $k$ lies in the interval $[k-\Delta_k,k+\Delta_k]$, $\xi$ in $[\xi-\Delta_{\xi},\xi+\Delta_{\xi}]$, and $\tau$ in $[\tau-\Delta_{\tau},\tau+\Delta_{\tau}]$. To model this situation we need two perturbations ($\delta_1$ and $\delta_2$) acting on the state matrices and one perturbation acting on the delay value $\delta\tau_1$:
	\begin{multline}
	\label{eq:uncertain_dde}
	\small
	\dot{x}(t) = \left(\begin{bmatrix}
	0 & 1 \\
	{\small -\omega^2-\frac{k}{m}} & {\small -2\xi\omega}
	\end{bmatrix} + \begin{bmatrix}
	0 \\
	\frac{-1}{m}
	\end{bmatrix} \delta_1 
	\begin{bmatrix}
	\Delta_k &0
	\end{bmatrix} + 
	\begin{bmatrix}
	0\\
	-2\omega
	\end{bmatrix}
	\delta_2
	\begin{bmatrix}
	0 & \Delta_{\xi}
	\end{bmatrix}
	\right) x(t) + \\
	\left(\begin{bmatrix}
	0 & 0 \\
	\frac{k}{m} & 0
	\end{bmatrix}+ \begin{bmatrix}
	0 \\
	\frac{1}{m}
	\end{bmatrix} \delta_1 
	\begin{bmatrix}
	\Delta_k &0
	\end{bmatrix} \right)x\big(t-(\tau+ \Delta_\tau \delta\tau_1)\big)
	\end{multline}
	with $|\delta_1|\leq 1$, $|\delta_2|\leq 1$, and $|\delta\tau_1|\leq 1$. Notice that in this case the state matrix $A_0$ is affectected by both perturbations and that the perturbation $\delta_1$ affects both system matrices.
	\end{example}
Let us start by introducing the proper notation. Firstly, we will consider the set
\begin{equation}
\label{eq:matrix_perturbations}
\delta := (\delta_1,\dots,\delta_{L_M})
\end{equation}
which consists of $L_M$ matrix perturbations that can have aribrary dimensions and can either be real or complex-valued, \ie
\[
\delta_{l} \in \R^{\rho_l \times \varsigma_l} \text{ or } \delta_{l} \in \C^{\rho_l \times \varsigma_l} \text{ for } l=1,\dots,L_M.
\]
Similarly, we consider the set
\[
\delta\tau :=(\delta\tau_1,\dots,\delta\tau_{L_{D}})
\] which consists of $L_D$ real-valued delay perturbations. Further, we assume that these perturbations are bounded, \ie there exists a bound $\epsilon$ such that
\[
\|\delta_{l}\|_{2 / \mathrm{fro}} \leq \epsilon \text{ for } l = 1,\dots,L_{M}
\]
and
\[
|\delta\tau_l| \leq \epsilon \text{ for } l = 1,\dots,L_{D},
\]
in which $\|\delta_{l}\|_{2 / \mathrm{fro}}$ denotes that we can either use the spectral norm (the induced 2-norm) or the Frobenius norm (the element-wise two norm). (Furthermore, for different matrix perturbations, different norms may be used.) Now that we have introduced these perturbations, we can consider the following uncertain dealy differential equation:
\begin{equation}
\label{eq:uncertain_rdde}
\dot{x}(t) = \big(A_0+ \tilde{A}_{0}(\delta)\big)\,x(t) + \sum_{k=1}^{m} \big(A_k + \tilde{A}_{k}(\delta)\big)\,x\big(t-(h_{A,k}+\tilde{h}_{A,k}(\delta\tau))\big)
\end{equation}
in which the matrices $A_0,\dots,A_m$ all belong to $\R^{n\times n}$ and represent the nominal state matrices. Similarly, the scalars $h_{1,m}\dots,h_{A,m}$ represent the nominal delay values. The uncertain matrices $\tilde{A}_{0}(\delta),\dots,\tilde{A}_{m}(\delta)$ are defined as
\[
\tilde{A}_{k}(\delta) =  \textstyle\sum\limits_{j=1}^{N_{A_k}} G_{k,j}\  \delta_{\texttt{ind}_{A_k}[j]}\ H_{k,j}
\]
with $N_{A_k}$ the number of uncertain terms, $\texttt{ind}_{A_k}[j]$ the $j$\textsuperscript{th} element of the array $\texttt{ind}_{A_k}$ that indicates the index of the uncertainty in $\delta$ and $G_{k,j}$ and $H_{k,j}$ real-valued shape matrices of appropriate dimensions. (Note that this seemingly
cumbersome notation is necessary to account for the fact that one matrix may be affected by multiple
uncertainties and that one uncertainty
may affect multiple matrices.) Similarly, the uncertain delays $\tilde{h}_{A,k}(\delta\tau)$ are defined as
 \[
\tilde{h}_{A,k}(\delta\tau) = \max\left\{-h_{A,k}, \textstyle\sum\limits_{j=1}^{N_{\tau_k}} w_{k,j}\, \delta\tau_{\texttt{ind}_{h_{A,k}}[j]} \right\}
\]
with $N_{h_{A,k}}$ the number of uncertain terms, $\texttt{ind}_{h_{A,k}}[j]$ the $j$\textsuperscript{th} element of the array $\texttt{ind}_{\tau_k}$ that indicates the index of the uncertainty in $\delta\tau$ and $w_{k,j}$ a real-valued weighting coefficient.

Given the definitions for the perturbations above and the uncertain DDE \eqref{eq:uncertain_rdde}, the first question that we want to address is: ``\textit{Given a bound $\epsilon$, is the uncertain system \eqref{eq:uncertain_rdde} exponentially stable for all instances of the uncertainties smaller than $\epsilon$?}''. To answer this question, we will introduce the $\epsilon$-pseudospectral abscissa as the worst-case value of the spectral abscissa of \eqref{eq:uncertain_rdde} over all uncertainty values.
\begin{definition}
	The $\epsilon$-pseudospectral abscissa of \eqref{eq:uncertain_rdde} will be denoted by $c^{\mathrm{ps}}(\epsilon)$ and is defined as follows:
	\[
	c^{\mathrm{ps}}(\epsilon) := \max \{c(\delta,\delta\tau): \|\delta_l\| \leq \epsilon \text{ for }   l = 1,\dots,L_{M} \ \& \ |\delta \tau_l|\leq \epsilon  \text{ for } l = 1,\dots,L_{D} \}
	\]
	with $c(\delta,\delta\tau)$ the spectral abscissa of \eqref{eq:uncertain_rdde} for a given instance of the uncertainties.
\end{definition}
 The uncertain system \eqref{eq:uncertain_rdde} is thus exponentially stable for all instances of the uncertainties bounded by $\epsilon$ if (and only if) the $\epsilon$-pseudospectral abscissa is strictly negative. 

Secondly, we are interested in answering the following related question: \textit{``Given a set of uncertainties and an uncertain system of the form \eqref{eq:uncertain_rdde}, what is the largest $\epsilon$ such that \eqref{eq:uncertain_rdde} is exponentially stable for all instances of the uncertainties smaller than $\epsilon$?}'' or equivalently, ``\textit{What is the smallest $\epsilon$ such that there exists an instance of \eqref{eq:uncertain_rdde} with the norm uncertainties smaller than or equal to $\epsilon$ that is not exponentially stable?}''. This leads us to the notion of distance to instability.
\begin{definition}
	The distance to instability of \eqref{eq:uncertain_rdde} is defined as the smallest $\epsilon$ for which the spectral abscissa is positive, or in other words
	\[
	\dins = \min \{\epsilon\geq 0 : c^{ps}(\epsilon) \geq 0 \} .
	\]
\end{definition}
\textbf{Important:} Currently \packageName{} only supports the computation of the pseudospectral abscissa and the distance to instability for uncertain systems with real-valued (\verb|'r'|) uncertainties bounded in Frobenius norm (\verb|'f'|).
\begin{example}
	\label{example:pseudo_cont}
	In this example, we will retake the set-up in \Cref{example:pseudo_definition} using the following parameter values: $k=8\times 10^6\,\mathrm{N}\mathrm{m}^{-1}$, $\omega = 775\,\mathrm{s}^{-1}$ , $\tau=0.012\,\mathrm{s}$, $\xi = 0.05$, $m=50\,\textrm{kg}$, $\Delta_{k} = 1\times10^{6}\,\mathrm{N}\mathrm{m}^{-1}$, $\Delta_{\xi}=0.005$, and $\Delta_{\tau}=0.001\,\mathrm{s}$. \\

\noindent Let us start by representing the nominal system \eqref{eq:turning_process} in \packageName{}
\begin{lstlisting}
k=8e6; omega=775; xi=0.05; tau=0.008; m=50;
A0 = [0 1;-omega^2-k/m -2*xi*omega]; A1 = [0 0;k/m 0];
tds = tds_create({A0,A1},[0 tau]);
\end{lstlisting}
and computing its nominal spectral abscissa
\begin{lstlisting}
>> tds_strong_sa(tds,-30)

ans =

-16.3646	
\end{lstlisting}
We conclude that the system is exponentially stable. 

Next we will see how the uncertain RDDE \eqref{eq:uncertain_dde} can be represented in \packageName{}. We start by creating a representation for the uncertainties $\delta_1$ and $\delta_2$ using the function \linebreak \matlabfun{tds_create_delta}
\begin{lstlisting}
delta=tds_create_delta({1,1},{1,1},{'r','r'},{'f','f'});
\end{lstlisting}
This function takes four cell arrays of length $L_{M}$ (here equal to 2) as input. The first cell indicates the row dimension of the uncertainties, the second the column dimension, the third whether the uncertainty is real (\verb|'r'|) or complex-valued (\verb|'c'|), and the fourth whether its size should be measured using the spectral (\verb|'s'|) or Frobenius norm (\verb|'f'|). \\
Next, we create objects that represents the uncertain matrices $\tilde{A}_0(\delta)$ and $\tilde{A}_1(\delta)$ using the function \matlabfun{tds_uncertain_matrix}: 
\begin{lstlisting}
Delta_k = 1e6; Delta_xi = 0.005; 
G01 = [0;-1/m]; H01 = [Delta_k 0];
G02 = [0;-2*omega]; H02 = [0 Delta_xi];
uA0 = tds_uncertain_matrix({1,2},{G01,G02},{H01,H02});

G11 = [0;1/m]; H11 = [Delta_k 0];
uA1 = tds_uncertain_matrix({1},{G11},{H11});
\end{lstlisting}	 
The function \verb|tds_uncertain_matrix| takes three cell arrays as input argument: the first argument is the array $\texttt{ind}_{A_k}$, the second one is a cell array that stores the left shape matrices, and the third one is a cell array storing the right shape matrices.\\ Subsequently we create the uncertain delay term $\tilde{\tau}_1(\delta\tau) = \Delta_\tau \delta\tau_1$:
\begin{lstlisting}
Delta_tau = 0.001;
uhA1 = tds_uncertain_delay({1},{Delta_tau});
\end{lstlisting}
in which the function \verb|tds_uncertain_delay| takes two cell arrays as input: the first is the array $\texttt{ind}_{\tau_k}$ and the second is a cell array with the weights. Finally we use the method \verb|add_uncertainty| to add these uncertainties to the nominal time delay system \verb|tds|.
\begin{lstlisting}
utds = tds.add_uncertainty(delta,1,{uA0,uA1},{[],uhA1});
\end{lstlisting}
In this case, the method \verb|add_uncertainty| takes four input arguments: the set of matrix uncertainties $\delta$, the number of delay uncertainties (in this case 1) and two cell arrays of length $m_A$, with respectively the uncertain matrices and delays\footnote{If a term is not affected by the uncertainties, then one can just use an empty array instead of the uncertain matrix or delay.}. The function returns an \verb|utds_ss|-object that represents the uncertain delay differential equation of which we can compute the $\epsilon$-pseudospectral abscissa using the function \verb|tds_psa|.
\begin{lstlisting}
>> epsilon = 1;
>> tds_psa(utds,epsilon)

ans = 

20.7022
\end{lstlisting}
We thus conclude that the uncertain delay differential equation \eqref{eq:uncertain_dde} is not stable for all admissible uncertainty values. Let us therefore compute the corresponding distance to instability.
\begin{lstlisting}
tds_dins(utds,[0 1]);
\end{lstlisting}
We find that $r_{\mathrm{ins}} = 0.4125$, meaning that the uncertain system is stable for all $k$ in  $(k-r_{\mathrm{ins}}\Delta_k,k+r_{\mathrm{ins}}\Delta_k)$, $\xi$ in $(\xi-r_{\mathrm{ins}}\Delta_{\xi},\xi+r_{\mathrm{ins}}\Delta_{\xi})$, and $\tau$ in $(\tau-r_{\mathrm{ins}}\Delta_{\tau},\tau+r_{\mathrm{ins}}\Delta_{\tau})$.

\end{example}

\chapter*{Acknowledgments}
\addcontentsline{toc}{chapter}{\protect\numberline{}Acknowledgments}

The software package \packageName{} builds on twenty years of research in the group of Wim Michiels from the Numerical Analysis and Applied Mathematics (NUMA) section of KU Leuven, Belgium. The package integrates several methods and algorithms, some of which gave rise to specific software tools available from 
\url{https://twr.cs.kuleuven.be/research/software/delay-control/} 
Writing the package would not have been possible without the contribution of several (former) group members and collaborators.
Therefore, in the first place, the authors want to thank Suat Gumussoy, who strongly contributed to the methodology and software tools for $\mathcal{H}_{\infty}$ analysis and design. Also the adopted specification of systems and controllers in the package carries the stamp of Suat's work. Secondly, the authors want to acknowledge the contributions of Haik Silm from which this software package greatly benefited. Special thanks further goes to Zhen Wu (algorithm for automatically determining the resolution of the discretization such that all characteristic roots in a given half plane are captured), Francesco Borgioli (computation of real structured pseudospectral abscissa and stability radii), Joris Vanbiervliet (the first work on stabilization of time-delay systems via the minimization of the spectral abscissa) and Deesh Dileep (design of structurally constrained controllers). Dan Pilbauer, Adrian Saldanha and Tom{\'a}{\v{s}} Vyhl{\'{i}}dal contributed to the validation and first applications of the design tools to DDAE models, in the context of vibration control applications. \\

\noindent \packageName{} uses \verb|HANSO| (version 3.0) to solve the non-smooth, non-convex optimization problems arising in the context of controller-design. \verb|HANSO| is available for download from \url{https://cs.nyu.edu/~overton/software/hanso/}. \verb|HANSO| is based on the methods described in \cite{HANSO,HANSO2} and is licensed under the GNU General Public License version 3.0.

\printbibliography
\addcontentsline{toc}{chapter}{\protect\numberline{}References}%
\appendix
\chapter{Installation}
\label{sec:installation}
The \packageName{} package is available for downloaded from \link{}. To install the package, simply add the directory \verb|code| and all its subdirectories to your \matlab{} search path:
\begin{lstlisting}
>> addpath(genpath('code'));
\end{lstlisting}
\ \\
\noindent \textbf{Important.} To be able to use the controller design algorithms, such as \matlabfun{tds_stabopt_static} and \matlabfun{tds_stabopt_dynamic}, \verb|HANSO| (version 3.0) \footnote{For more information see \url{https://cs.nyu.edu/~overton/software/hanso/}.} must be present in your \matlab{} search path. \verb|HANSO| is included in the folder \verb|hanso3_0| and can be added to your \matlab{} search path using
\begin{lstlisting}
>> addpath(genpath('hanso3_0'));
\end{lstlisting}
To add \packageName{} and \verb|HANSO| permanently to your \matlab{}-path, use subsequently
\begin{lstlisting}
>> savepath
\end{lstlisting}
which saves the current search path. \\

\noindent \textbf{Important.} \packageName{} utilizes functionality from the MathWorks\textsuperscript{\textregistered}  Control System Toolbox\texttrademark, such as the \matlabfun{ss}-class and the \matlabfun{sigma}-function. If these functions are not present in your \matlab{} search path, \packageName{} will revert to a fallback implementation. Note however, both performance and accuracy might be reduced.
\chapter{Documentation}
\label{sec:documentation}
\section{\texorpdfstring{\texttt{tds\_roots}}{tds\_roots}}
\label{sec:tds_roots}


The function \matlabfun{tds_roots} computes the characteristic roots of a given delay differential equation or time-delay system in a specified right half-plane or rectangular region. To this end, it implements the method from \cite{roots}. First, a generalized eigenvalue problem of the form 
\begin{equation}
\label{eq:discrete_eigenvalue_problem}
\Sigma_N\,v  =  \lambda \Pi_N \,v
\end{equation}
with $\Sigma_N$ and $\Pi_N$ belonging to $\R^{(N+1)n\times (N+1)n}$, is constructed such that the eigenvalues of \eqref{eq:discrete_eigenvalue_problem} approximate the desired characteristic roots. More specifically, the generalized eigenvalue problem \eqref{eq:discrete_eigenvalue_problem} corresponds to a spectral discretization of the infinitesimal generator of the solution operator underlying the delay differential equation. (For more details on spectral discretisation methods for computing the characteristic roots of time-delay systems, we refer the interested reader to \cite{breda2014}.) The approximations for the characteristic roots of the time-delay system obtained by computing the eigenvalues of \eqref{eq:discrete_eigenvalue_problem} are subsequently improved by applying Newton's method. Note that the dimensions of the matrices $\Sigma_N$ and $\Pi_N$ depend on a parameter $N$, the degree of the spectral discretisation. This $N$ is related to the quality of the spectral approximation: the larger $N$, the better individual characteristic roots are approximated and the larger the region in the complex plane in which the characteristic roots are well approximated. However, the larger $N$, the larger the size of the generalized eigenvalue problem and hence the longer the computation time. An interesting feature of the method from \cite{roots} is that it includes an empirical strategy to find a suitable $N$ such that the desired characteristic roots are well-approximated, while keeping the generalized eigenvalue problem as small as possible. 

\textbf{Note:} As solving large eigenvalue problems is computationally costly, the maximal size of the generalized eigenvalue problem \eqref{eq:discrete_eigenvalue_problem} is limited. More specifically, if the inequality $N(n+1) > \matlabfun{options.max_size_evp}$ holds, then $N$ is lowered to the largest $N$ for which $N(n+1) \leq \matlabfun{options.max_size_evp}$ holds. However, lowering $N$ means that certain characteristic roots in the desired region are not sufficiently well approximated and might be missed by \matlabfun{tds_roots}. See also \Cref{example:warning}.

\subsection{Documentation}
\begin{itemize}
	\item \matlabfun{L = tds_roots(tds,region)} returns a vector \texttt{L} containing the characteristic roots of \matlabfun{tds} in the right half-plane
$
	\{z\in\C: \Re(z) \geq \text{ \texttt{region} }\}
$
	if \texttt{region} is a scalar and in the rectangular region
$
	[\texttt{region}(1),\,\texttt{region}(2)] \times \jmath\,[\texttt{region}(3),\,\texttt{region}(4)]
$
	if \matlabfun{region} is a vector (of length 4). 
	\item \matlabfun{[L,LD] = tds_roots(tds,region)} also returns a vector \texttt{LD} containing the eigenvalues of \eqref{eq:discrete_eigenvalue_problem}.
	\item \matlabfun{[L,LD,info] = tds_roots(tds,region)} also returns a structured array \matlabfun{info} containing the following information:
	\begin{itemize}
		\item \matlabfun{N} - the employed degree of the spectral discretisation.
		\item \matlabfun{N_original} - the degree for the spectral discretization deemed necessary  to capture all characteristic roots in the desired region (\matlabfun{N} is smaller than \matlabfun{N_original} if the maximal size of the generalized eigenvalue problem would be exceeded for \matlabfun{N_original}).
		\item \matlabfun{max_size_evp_enforced} - flag indicating that for \matlabfun{N_original}, the maximal size for the eigenvalue problem \eqref{eq:discrete_eigenvalue_problem} would be exceeded and that the degree of the spectral discretisation is therefore lowered to \matlabfun{N}. 
		
		\item \matlabfun{gamma_r_exceeds_one} - flag indicating that $\gamma(\text{ \texttt{region} })\geq 1$, meaning that the considered right half-plane contains infinitely many characteristic roots or that there exist infinitesimal delay perturbations such that it does. The default value $N =15$ is therefore used. See also the warning \matlabfun{Case: gamma(region) >= 1}.
		\item \matlabfun{newton_initial_guesses} - array with the initial guesses for Newton's method.
		\item \matlabfun{newton_final_values} - array of the same length with the values after Newton's method.
		\item \matlabfun{newton_residuals} - array of the same length with the norm of the residuals after the Newton corrections.
		\item \matlabfun{newton_unconverged_initial_guesses} - array containing the indices of the initial guesses in \matlabfun{newton_initial_guesses} for which Newton's method failed to converge.
		\item \matlabfun{newton_large_corrections} - array containing the indices for which the result after Newton's method lies far from the initial guess.
	\end{itemize} 
	\item \matlabfun{[..] = tds_roots(tds,region,options)} allows to specify additional options. 
\begin{itemize}
    \item \matlabfun{fix_N}: fixes the parameter $N$ of the generalized eigenvalue problem \eqref{eq:discrete_eigenvalue_problem}. Note that the constaint on the maximal size of the generalized eigenvalue problem is now ignored.
    \item \matlabfun{max_size_evp}: gives the maximal size of  \eqref{eq:discrete_eigenvalue_problem} (if \matlabfun{fix_N} is not specified).\\
    \hspace*{1em} \textit{Default value - } \texttt{600}.
    \item \matlabfun{newton_tol} tolerance on the norm of the residuals for Newton's method.\\
    \hspace*{1em} \textit{Default value - } \texttt{1e-10}.
    \item \matlabfun{newton_max_iter}: maximum number of Newton iterations for correcting a given root.  \\
    \hspace*{1em} \textit{Default value - } \texttt{20}.
    \item \matlabfun{commensurate_basic_delay}:  value of the basic delay $h$ in case of commensurate delays. May lead to a speed-up for computing \matlabfun{N} if the ratio $\tau_m/h$ is small.
     \item \matlabfun{quiet}: suppress all output (including warnings). \\
     \hspace*{1em} \textit{Default value - } \texttt{false}.
\end{itemize}
\end{itemize}
\subsection{Warnings}
The function \matlabfun{tds_roots} may give the following warnings.
\begin{itemize}
    \item \texttt{The provided time-delay system is an ODE/a DAE.} The characteristic function of the provided system does not contain delays. \packageName{} falls back to \matlabfun{eig} for computing the characteristic roots. Note that only the characteristic roots in the desired region are returned in \matlabfun{L}. 
    
    \item \texttt{Size of the discretized EVP would exceed its maximum value.} 
    The degree of the spectral discretisation needed to accurately capture all characteristic roots in the specified region (\matlabfun{N_original}) would result in an eigenvalue problem \eqref{eq:discrete_eigenvalue_problem} that is larger than \matlabfun{options.max_size_evp}. The degree of the spectral discretisation is therefore lowered to \matlabfun{N}.
    \item \matlabfun{Case: gamma(r) >= 1 (i.e., CD>r).} This warning is thrown when \matlabfun{region} is a scalar (one is interested in the characteristic roots in a given right half-plane) and $\gamma(\text{ \texttt{region} }) \geq 1$. This means that the desired right half-plane contains infinitely many characteristic roots or that there exist infinitesimal perturbations on the delays such that it does. As a consequence, the default value $N=15$ is used (lowered if $N(n+1)$ exceeds \linebreak\matlabfun{options.max_size_evp}, see warning above). This means that not all desired characteristic roots may be captured. The user can manually increase $N$ using the option \verb|fix_N| or specify a rectangular region to capture more of the desired characteristic root. See also \Cref{example:roots_NDDE_rhp}. 
\end{itemize}
\section{\texorpdfstring{\texttt{tds\_gamma\_r}}{tds\_gamma\_r} }
\label{sec:tds_gamma_r}
The function \matlabfun{tds_gamma_r} computes the quantity $\gamma(r)$ (as defined in \eqref{eq:gamma_r} for neutral DDEs and in \eqref{eq:gamma_r_ddae} for DDAEs). As mentioned in \Cref{proposition:strong_stability}, the quantity $\gamma(0)$ gives a necessary condition for strong stability (\ie stability is preserved for  sufficiently small delay perturbations). On the other hand, the inequality $\gamma(r) < 1$ guarantees that the right half-plane $\{\lambda \in \C: \Re(\lambda) > r \}$ only contains finitely many characteristic roots (even when infinitesimal delay perturbations are considered). To compute the quantity $\gamma(r)$ for a NDDE, the code will first obtain a prediction for the solution of the optimization problem in \eqref{eq:gamma_r} by restricting the search space for $\vec{\theta}$ to a grid in $\{0\} \times [0,2\pi)^{m_{H}-1}$ (note that $\theta_1$ can be fixed to zero). The resulting value is subsequently improved by solving the following overdetermined system of nonlinear equations in the unknowns $v\in \C^{n}$, $u \in \C^{n}$, $\lambda\in\C$, and $\theta \in \{0\} \times [0,2\pi)^{m_{H}-1}$
\begin{equation}
\label{eq:nonlinear_system_gamma_r}
\left\{
\begin{array}{l}
\left(H_1 e^{-r h_{H,1}} + \sum_{k=2}^{m_{H}} H_k e^{-r h_{H,k}} e^{\jmath \theta_k}\right) v - \lambda v= 0 \\
u^{H} \left(H_1 e^{-r h_{H,1}} + \sum_{k=2}^{m_{H}} H_k e^{-r h_{H,k}} e^{\jmath \theta_k}\right) - u^{H} \lambda = 0  \\
n(v) - 1 = 0 \\
u^{H} v - 1 = 0 \\
\Im\left(\bar{\lambda}(u^{H} H_{k} v) e^{-rh_{H,k}} e^{\jmath \theta_k}  \right) = 0, \text{ for } k = 2,\dots, m_{H}
\end{array}
\right.
\end{equation}
with $n(v) - 1 = 0$ a normalization constraint, using the \matlabfun{fsolve} function in Matlab. A solution $(v,u,\lambda,\theta)$ of \eqref{eq:nonlinear_system_gamma_r} corresponds to an eigenvalue  $\lambda$ of the matrix
\[
H_1 e^{-r h_{H,1}} + \sum_{k=2}^{m_{H}} H_k e^{-r h_{H,k}} e^{\jmath \theta_k}
\]
with $u$ and $v$, respectively, the corresponding left and right eigenvectors such that the derivatives of the modulus of this eigenvalue with respect to $\theta_2,\dots,\theta_{m_{H}}$ equals zero. The approach for computing the quantity $\gamma(r)$ of a DDAE is similar.

\subsection{Documentation}
\begin{itemize}
	\item \matlabfun{gammar = tds_gamma_r(tds,r)} computes the quantity $\gamma(\texttt{r})$ of \matlabfun{tds}. 
	\item \matlabfun{gammar = tds_gamma_r(tds,r,options)} allows to specify additional options. \matlabfun{options}  must be a structured array created using the function \matlabfun{tds_gamma_r_options}.  The available options are: 
	\begin{itemize}
		\item \matlabfun{quiet}: suppress all output (including warnings).\\
		\hspace*{1em} \textit{Default value - } \texttt{false}.
		\item \matlabfun{Ntheta}: the number of grid points in each direction. \\
		\hspace*{1em} \textit{Default value - } \texttt{10}.
	\end{itemize}
\end{itemize}
\section{\texorpdfstring{\texttt{tds\_CD}}{tds\_CD} }
\label{sec:tds_CD}
The function \matlabfun{tds_CD} computes the strong spectral abscissa of the delay difference equation underlying a delay differential equation. To compute this quantity we use the expression for $C_D$ in \Cref{proposition:strong_spectral_abscissa}. More specifically, we use \matlabfun{fsolve} to find a zero crossing of the function \eqref{eq:CD_zero_crossing}.
\subsection{Documentation}
\begin{itemize}
	\item \matlabfun{CD = tds_CD(tds)} computes the strong spectral abscissa of the delay difference equation associated with \matlabfun{tds}. 
	\item \matlabfun{CD = tds_CD(tds,options)} allows to specify additional options. \matlabfun{options} must be a structured array created using the function \matlabfun{tds_CD_options}. The available options are: 
	\begin{itemize}
		\item \matlabfun{tol}: tolerance \\
		\hspace*{1em} \textit{Default value - } \texttt{1e-10}
		\item \matlabfun{quiet}: suppress all output (including warnings) \\
		\hspace*{1em} \textit{Default value - } \texttt{false }
		\item \matlabfun{Ntheta}: the number of grid points in each direction \\
		\hspace*{1em} \textit{Default value - } \texttt{10} 
		\item \matlabfun{CD0}: initial estimate for the strong spectral abscissa of the underlying delay difference equation\\
		\hspace*{1em} \textit{Default value - } 0 
		
	\end{itemize}
	
\end{itemize}

\section{\texorpdfstring{\texttt{tds\_sa}}{tds\_sa} }
The function \matlabfun{tds_sa} computes the spectral abscissa (see \Cref{def:spectral_abscissa,def:spectral_abscissa_neutral}) of a time-delay system. To this ends it computes all characteristic roots in a given right half-plane using \matlabfun{tds_roots} and then returns the real part of the rightmost root.
\subsection{Documentation}
\begin{itemize}
	\item \matlabfun{C = tds_sa(tds,r)} computes the spectral abscissa of \matlabfun{tds} by computing its characteristic roots in the right half plane $\{z \in \C: \Re(z) \geq R\}$.
	\item \matlabfun{C = tds_sa(tds,r,options)} computes the spectral abscissa with the default options for the \matlabfun{tds_roots} replaced by \matlabfun{options}.  
	
\end{itemize}
\section{\texorpdfstring{\texttt{tds\_strong\_sa}}{tds\_strong\_sa} }
The function \matlabfun{tds_strong_sa} computes the strong spectral abscissa of a time-delay system. Recall from \Cref{def:strong_spectral_abscissa} that the strong spectral abscissa is defined as the smallest upper bound for the spectral abscissa which is insensitive to infinitesimal delay changes. Furthermore, remember that for a(n) (essentially) retarded time-delay system the strong spectral abscissa corresponds to the spectral abscissa and that for a(n) (essentially) neutral time-delay system the strong spectral abscissa is equal to the maximum of the spectral abscissa and the strong spectral abscissa of the associated difference equation (which can be computed using the function \matlabfun{tds_CD}). To compute the strong spectral abscissa, the code will therefore first determine the strong spectral abscissa of the associated delay difference equation (\texttt{CD}). Subsequently, it checks for characteristic roots in the right half-plane $\{z \in \C: \Re(z)>\texttt{CD}\}$.

\subsection{Documentation}
\begin{itemize}
	\item \matlabfun{C = tds_strong_sa(tds,r)} computes the strong spectral abscissa of \matlabfun{tds} by first computing \matlabfun{CD = tds_CD(tds)} and subsequently computing the characteristic roots of \matlabfun{tds} in the right half-plane $\{z \in \C: \Re(z)> \max(\texttt{CD},\texttt{r})\}$. 
	\item \matlabfun{C = tds_strong_sa(tds,r,options)} allows to specify additional \matlabfun{options}. The argument \matlabfun{options} should be a structure with the necessary fields. The available options are: 
	\begin{itemize}
		\item \matlabfun{roots}: a structure created using \matlabfun{tds_roots_options} containing the       options for \matlabfun{tds_roots}.
		\item CD - a structure created using \matlabfun{tds_CD_options} containing the options for \matlabfun{tds_CD}. 
	\end{itemize}
\end{itemize}
\section{\texorpdfstring{\texttt{tds\_tzeros}}{tds\_tzeros}}
The function \matlabfun{tds_tzeros} allows to compute the transmission zeros of a SISO time-delay system, \ie a point $z_{0}$ in the complex plane for which the transfer function becomes zero. To this end, the function exploits the fact that computing the transmission zeros of a SISO transfer function can be reformulated as computing the characteristic roots of a DDAE (see \Cref{subsec:tzeros}).

\subsection{Documentation}
\begin{itemize}
	\item \matlabfun{tzeros = tds_tzeros(tds,rect)} returns a vector \matlabfun{tzeros} containing the transmission zeros of the time-delay system \matlabfun{tds} in the rectangular region [\texttt{rect}(1) \texttt{rect}(2)] $\times\jmath$ [\texttt{rect}(3) \texttt{rect}(4)].  
	\item \matlabfun{tzeros = tds_tzeros(tds,rect,options)} allows to specify additional options created using the function \matlabfun{tds_roots_options}.
\end{itemize}

\section{\texorpdfstring{\texttt{tds\_stabopt\_static}}{tds\_stabopt\_static} and \texorpdfstring{\texttt{tds\_stabopt\_dynamic}}{tds\_stabopt\_dynamic}}
\label{sec:tds_stabopt}
The functions \matlabfun{tds_stabopt_static} and \matlabfun{tds_stabopt_dynamic} look for (strongly) stabilizing output feedback controller for a given time-delay system by minimizing a certain objective function (see below) in the entries of the controller matrices, \ie the matrices in \eqref{eq:static_output_feedback_controller} and \eqref{eq:dynamic_output_feedback_controller}. Below, we will briefly discuss the core working principles of these functions, for more details see \cite{vanbiervliet2008,stabilization_neutral,michiels2011spectrum}. For ease of notation, we will restrict our attention to the dynamic feedback control of the following retarded time-delay system
\[
\left\{
\begin{array}{rcl}
\dot{x}(t) &=& A_0\,x(t) + \sum_{k=1}^{m} A_k x(t-\tau_k) + B_{1,0}\,u(t) + \sum_{k=1}^{m} B_{1,k} \, u(t-\tau_k),\\[9pt]
    y(t) &=& C_{1,0}\,x(t) + \sum_{k=1}^{m} C_{1,k}\,x(t-\tau_k) + D_{11,0}\,u(t) + \sum_{k=1}^{m} D_{11,k} \,u(t-\tau_k).
\end{array}
\right.
\]
The discussion for static control or neutral and delay descriptor systems is however similar. The dynamics of the feedback interconnection of the system above with the \eqref{eq:dynamic_output_feedback_controller} can be described by
\begin{multline}
\label{eq:DDAE_cl2}
    \begin{bmatrix}
    I & 0&  0 & 0 \\
    0 & 0 & 0 & 0 \\
    0 & 0 & I & 0 \\
    0 & 0 & 0 & 0
    \end{bmatrix} \begin{bmatrix}
    \dot{x}(t) \\
    \dot{y}(t) \\
    \dot{x}_c(t) \\
    \dot{u}_u(t)
    \end{bmatrix} = \begin{bmatrix}
    A_0 & 0 & 0 & B_{1,0} \\
    C_{1,0} & -I & 0 & D_{11,0} \\
    0 & 0 & 0 & 0 \\
    0 & 0 & 0 & -I
    \end{bmatrix} \begin{bmatrix}
    x(t) \\
    y(t) \\
    x_c(t) \\
    u(t)
    \end{bmatrix} \ + \\[5px]  \sum_{k=1}^{m} \begin{bmatrix}
    A_k & 0 & 0 & B_{1,k} \\
    C_{1,k} & 0 & 0 & D_{11,k} \\
    0 & 0 & 0 & 0 \\
    0 & 0 & 0 & 0 
    \end{bmatrix}\begin{bmatrix}
    x(t-\tau_k) \\
    y(t-\tau_k) \\
    x_c(t-\tau_k) \\
    u(t-\tau_k)
    \end{bmatrix} +
    \begin{bmatrix}
    0 & 0 & 0 & 0 \\
    0 & 0 & 0 & 0 \\
    0 & B_c & A_c & 0 \\
    0 & D_c & C_c & 0 \\
    \end{bmatrix} \begin{bmatrix}
    x(t) \\
    y(t) \\
    x_c(t) \\
    u(t)
    \end{bmatrix}.
\end{multline}
In other words, the closed-loop system corresponds to a delay differential algebraic equation of the form
\begin{equation}
\label{eq:DDAE_cl}
E_{cl}\,\dot{x}_{cl}(t) =  \sum_{k=1}^{m_{cl}} A_{cl,k}(\mathbf{p})\ x_{cl}(t-\tau_{cl,k}) 
\end{equation}
in which $\mathbf{p}$ denotes the vector containing the free control parameters (\ie the optimization variables). In the case of an unstructured controller design problem, $\mathbf{p}$ consist of the entries of the matrices $A_c$, $B_c$, $C_c$ and $D_c$. For the design of a structured controller (see \Cref{subsec:structured_controllers}) only certain elements of these matrices are contained in $\mathbf{p}$. Note that the state matrices $A_{cl,1},\dots,A_{cl,m}$ depend affinely on the controller parameters in $\mathbf{p}$.

Next, we will give more information on the objective functions that are minimized. If \eqref{eq:DDAE_cl} is essentially retarded (see \Cref{subsec:stability_DDAE}), \texttt{tds\_stabopt\_dynamic} attempts to find a stabilizing controller by minimizing the spectral abscissa, or in other words, the objective function is given by
\[
f(\mathbf{p}) := c\left(E_{cl}, A_{1,cl}(\mathbf{p}),\dots,A_{m_{cl},cl}(\mathbf{p})\right).
\]
If \eqref{eq:DDAE_cl} is essentially neutral, closed-loop stability might be sensitive to infinitesimal delay perturbations (see \Cref{subsec:neutral_multiple_delays}). To guarantee strong exponential stability (as defined in \Cref{def:strong_stability}), we can consider either of the following two approaches (see also \Cref{subsec:neutral_stabilization}).
\begin{itemize}
    \item \textbf{Approach 1:} minimize the strong spectral abscissa of \eqref{eq:DDAE_cl}, or in other words, the objective function is given by
    \[
    f(\mathbf{p}) := C\left(E_{cl}, A_{0,cl}(\mathbf{p}),\dots,A_{k,cl}(\mathbf{p})\right).
    \]
    \item \textbf{Approach 2:} minimize the spectral abscissa of \eqref{eq:DDAE_cl} under the constraint that the inequality 
    \[
    \gamma_0(A_{0,cl}^{(22)}(\mathbf{p}),\dots,A_{m_{cl},cl}^{(22)}(\mathbf{p}))<1
    \]
    with $A_{0,cl}^{(22)},\dots,A_{m_{cl},cl}^{(22)}$ defined in a similar vain as in \Cref{subsec:stability_DDAE}, holds. To solve this constraint optimization problem, \packageName{} uses a log-barrier method, \ie a barrier function which grows to infinity as $\gamma(0)$ goes 1, is added to the cost function. More specifically, \packageName{} considers the following barrier function
    \begin{equation}
        \label{eq:barrier_function}
        -w_2\log\big(1-w_1-\gamma_0(A_{0,cl}^{(22)}(\mathbf{p}),\dots,A_{m_{cl},cl}^{(22)}(\mathbf{p}))\big)
    \end{equation}
    in which $0 \leq w_1 < 1$ dictates the acceptable region for $\gamma_0$ and $w_2> 0$ determines the weight of the barrier function in the overall cost function, \ie the larger $w_2$, the more emphasis is put on keeping $\gamma_0$ small. Note that in order to use this approach, an initial controller for which the inequality $\gamma_0 < 1-w_1$ holds, is required. \packageName{} implements therefore an initial optimization step in which $\gamma_0(E_{cl}, A_{0,cl}(\mathbf{p}),\dots,A_{k,cl}(\mathbf{p}))$ is minimized with respect to the control parameters. Once a controller found for which $\gamma_0<1-w_1$ the following extended cost function is minimized in $\mathbf{p}$:
    \begin{multline*}
    f(\mathbf{p}) := c\left(E_{cl}, A_{0,cl}(\mathbf{p}),\dots,A_{k,cl}(\mathbf{p})\right) \\ -w_2\log\big(1-w_1-\gamma_0(A_{0,cl}^{(22)}(\mathbf{p}),\dots,A_{m_{cl},cl}^{(22)}(\mathbf{p}))\big).
    \end{multline*}
\end{itemize}
In each case \packageName{} thus solves an unconstrained optimization problem of the form
\begin{equation}
\label{eq:obj_func}
\min_{\mathbf{p}} f(\mathbf{p}).
\end{equation}
The objective function $f$ is typically neither convex nor smooth. (However the object function is typically  almost everywhere differentiable.)
The \packageName{} therefore relies on HANSO v3.0 \cite{HANSO,HANSO2} for solving these nonsmooth, nonconvex optimization problems. 

\noindent \textbf{Important}. To avoid ending up in a local optimum of \eqref{eq:obj_func}, the optimization routine is restarted from different initial values for $\mathbf{p}$.
\subsection{Documentation}
\begin{itemize}
	\item \matlabfun{K = tds_stabopt_static(tds)} attempts to  design a strongly stablizing static output feedback gain matrix \texttt{K} for \texttt{tds}. 
	\item \matlabfun{K = tds_stabopt_static(tds,'key1',value1,'key2',value2,...)} allows to pass the following optional arguments using key-value pairs:
	\begin{itemize}
		\item \texttt{options} structured array created using \matlabfun{tds_stabopt_options}  containing additional options
		\item \texttt{initial} cell array containing the initial optimization values
		\item \texttt{mask} logical array indicating which entries of the feedback gain matrix are allowed to change (\texttt{true}) and which should remain fixed to their original value (\texttt{false})
		\item \texttt{basis} array specifying the values of the fixed entries of the feedback gain matrix
	\end{itemize}
	\item \matlabfun{[K,CL] = tds_stabopt_static(tds,...)} also returns a \matlabfun{tds_ss_ddae}-object that represents the closed-loop system.  
\end{itemize}
\begin{itemize}
	\item \matlabfun{K = tds_stabopt_dynamic(tds,nc)} attempts to  design a strongly stablizing dynamic output feedback controller \texttt{K} of order \texttt{nc} for \texttt{tds}.
	\item \matlabfun{K = tds_stabopt_dynamic(tds,nc,'key1',value1,'key2',value2,...)} allows to pass the following optional arguments using key-value pairs:
	\begin{itemize}
		\item \texttt{options} structured array created using \matlabfun{tds_stabopt_options}  containing additional options
		\item \texttt{initial} cell array containing the initial optimization values 
		\item \texttt{mask} an \texttt{ss}-object indicating which entries of the feedback gain matrix are allowed to change (\texttt{1}) and which should remain fixed to their original value (\texttt{0})
		\item \texttt{basis} an \texttt{ss}-object specifying the fixed entries of the feedback gain matrix
	\end{itemize}
			\item \matlabfun{[K,CL] = tds_stabopt_dynamic(tds,nc,...)} also returns a \matlabfun{tds_ss_ddae}-object that represents the closed-loop system. 
\end{itemize}

\subsection{Options}
The function \matlabfun{tds_stabopt_options} can be used to create a structured array specifying additional options for \matlabfun{tds_stabopt_static} and \matlabfun{tds_stabopt_dynamic}. 
\begin{itemize}
    \item \texttt{nstart}: number of starting points for the optimization procedure\\
    \hspace*{1em} \textit{Default value: } \texttt{5}.
    \item \texttt{roots}: options for \matlabfun{tds_roots} (see \Cref{sec:tds_roots})
    \item \texttt{Ntheta}: number of grid points in each direction for computing $\gamma_0$ (Approach 2, see \Cref{sec:tds_gamma_r}) or $C_D$ (Approach 1, see \Cref{sec:tds_CD}) \\
    \hspace*{1em} \textit{Default value - } \texttt{10}.
    \item \texttt{cpumax}: quit the BFGS-procedure if the CPU time (in seconds) exceeds the given value\\
    \hspace*{1em} \textit{Default value: } \texttt{inf}.
    \item \texttt{fvalquit}: quit the BFGS-procedure if the object values drops below the given value\\
    \hspace*{1em} \textit{Default value: } \texttt{-inf}.
    \item \matlabfun{gradient_sampling}: apply the gradient sampling algorithm from \cite{HANSO} after the BFGS-procedure. By default, gradient sampling is disabled as it is typically quite slow. \\
    \hspace*{1em} \textit{Default value: } \texttt{false}.
    \item \texttt{method}: for essentially neutral closed-loop systems select which cost function is used:
    \begin{itemize}
    	\item \verb|'CD'|: approach 1
    	\item \verb|'barrier'|: approach 2
    	\item \verb|'automatic'|: automatically select which approach is used based on the number of terms in the associated delay difference equation
    \end{itemize} 
    \hspace*{1em} \textit{Default value: } \verb|'automatic'|.
    \item \verb|w1|: gives the value of the constant $w_1$ in \eqref{eq:barrier_function}. \\
    \hspace*{1em} \textit{Default value: } \texttt{1e-3}.
    \item \verb|w2|: gives the value of the constant $w_2$ in \eqref{eq:barrier_function}. \\
    \hspace*{1em} \textit{Default value: } \texttt{1e-3}.
    \item \matlabfun{print_level}: 0 (no printing), 1 (minimal printing), 2 (verbose), 3 (debug)\\
    \hspace*{1em} \textit{Default value: } 2.

\end{itemize}

\subsection{Warnings}
The functions \matlabfun{tds_stabopt_static} and \matlabfun{tds_stabopt_dynamic} may give the following warnings.
\begin{itemize}
    \item \texttt{Warning: Resulting controller is not stabilizing.} This warning indicates that \matlabfun{tds_stabopt_static} or \matlabfun{tds_stabopt_dynamic} failed to find a (strongly) stabilizing controller. Note that this warning does not imply that there does not exist a (strongly) stabilizing controller. It merely indicates that starting from the chosen initial values, the optimization procedure ended up in local optima that do not correspond with a strictly negative (strong) spectral abscissa. A possible solution is to sample more initial values using the option \verb|'nstart'| and/or the optional argument \verb|'initial'|. If after sampling sufficiently many initial values, still no stabilizing controller is found, this might be an indication that there does not exists any. In this case, one might consider increasing the order of the controller.

\end{itemize}
 
\section{\texorpdfstring{\texttt{tds\_hinfnorm}}{tds\_hinfnorm}}
 \label{sec:tds_hinfnorm}
The function \matlabfun{tds_hinfnorm} computes the (strong) \hinfnorm{} of a time-delay system. To this end, it implements the method from \cite{gumussoy2011}. This method can be summerized as follows:

\begin{enumerate}
	\item Compute the strong \hinfnorm{} of the associated asymptotic transfer function by solving the optimization problem in \eqref{eq:strong_hinfn_asympt}. 
	\begin{itemize}
		\item \textbf{Prediction step.} As \eqref{eq:strong_hinfn_asympt} is generally not convex and thus may have multiple local optima, the object function is first sampled on a grid in the $\theta$-space with \verb|options.Ntheta| grid points in each direction.
		\item \textbf{Correction step}. Improve the obtained result using Newton's method.
	\end{itemize}
	\item Verify whether the maximal singular value function attains a higher value at a \textbf{finite frequency}.
	\begin{itemize}
		\item \textbf{Prediction step.} Form a descriptor system that approximates the time-delay system using a spectral discretisation of degree \verb|options.fixN|. Compute the \hinfnorm{} of this delay-free system using a level-set method. \\
		\textbf{Note.} In contrast to \verb|tds_roots|, no automatic selection of the discretization degree is performed.
		\item \textbf{Correction step}. Improve the obtained result using Newton's method.
	\end{itemize}  
\end{enumerate}   
\textbf{Important:} 
If the provided time-delay system is not strongly exponentially stable, the function \matlabfun{tds_hinfnorm} will return infinity.
\subsection{Documentation}
\begin{itemize}
	\item \matlabfun{[HINF] = TDS_HINFNORM(TDS)} returns the strong \hinfnorm{} of \matlabfun{TDS}.\\[1mm] \textbf{Note.} \matlabfun{HINF=Inf} when \matlabfun{TDS} is not strongly exponentially stable.
	\item \matlabfun{[HINF,WPEAK] = tds_hinfnorm(TDS)} also returns the angular frequency at which the H-infinity norm is attained. \matlabfun{WPEAK=Inf} when the strong H-infinity norm of \matlabfun{TDS} is equal to the strong H-infinity norm of the underlying asymptotic transfer function. \matlabfun{WPEAK=NaN} when \matlabfun{TDS} is not strongly exponentially stable.
	\item \matlabfun{[..] = tds_hinfnorm(TDS,OPTIONS)} allows to specify additional options (see below).
\end{itemize}
\subsection{Options}
The function \verb|tds_hinfnorm_options| can be used to create a structured array containing additional options. The following options are available:
\begin{itemize}
    \item \verb|Ntheta|: the number of grid points in each direction to predict the strong \hinfnorm{} of the asymptotic transfer function. \\
        \hspace*{1em} \textit{Default value: } 20.
    \item \verb|fix_N|: the degree spectral discretisation in the (finite frequency) prediction step. \\
        \hspace*{1em} \textit{Default value: } 20 
    \item \verb|omega_init|: vector containing the initial frequencies for the (finite frequency) prediction step. \\
        \hspace*{1em} \textit{Default value: } \verb|[0]|
    \item \verb|pred_tol|: the tolerance in the (finite frequency) prediction step. \\
    \hspace*{1em} \textit{Default value: } 1e-3
    \item \verb|newton_tol|: the tolerance on the residue after the (finite frequency) correction step. \\
   \hspace*{1em} \textit{Default value: } 1e-6 
\end{itemize}
\subsection{Warnings}
\begin{itemize}
    \item \verb|'Warning: correction step has nonconverging frequencies: %s\n'| The Newton correction for either the strong \hinfnorm{} of the asymptotic transfer function or the maximal singular value at a finite frequency did not converge. Try improving the prediction result by increasing the number of grid points (\verb|Ntheta|) or the degree of the spectral discretization (\verb|fixN|). 
    \item \verb|'Warning: norm is corrected more than .. in ..'|: A large Newton correction was needed after the prediction step (for either the strong \hinfnorm{} of the asymptotic transfer function $\omega=\infty$ or the maximal singular value at a finite frequency). This might be an indication that the prediction step was not sufficiently accurate. Try increasing the number of grid points (\verb|Ntheta|) or the degree of the spectral discretisation (\verb|fixN|)
\end{itemize}
\section{\texorpdfstring{\texttt{tds\_sigma}}{tds\_sigma}}
\label{sec:tds_sigma}
The function \matlabfun{tds_sigma} computes the singular values of the frequency response matrix of a time-delay system at the given angular frequencies.
\subsection{Documentation}
\begin{itemize}
	\item \matlabfun{SV = tds_sigma(TDS,W)} returns the singular values \matlabfun{SV} of the frequency response of \matlabfun{TDS} at the provided angular frequency vector \matlabfun{W}. The matrix \matlabfun{SV} has \matlabfun{length(W)} columns and the column \matlabfun{SV(:,k)} gives the singular values (in descending order) of the frequency response of \matlabfun{TDS} at the frequency \matlabfun{W(k)}. 
\end{itemize}
\section{\texorpdfstring{\texttt{tds\_hiopt\_static}}{tds\_hiopt\_static} and \texorpdfstring{\texttt{tds\_hiopt\_dynamic}}{tds\_hiopt\_dynamic}}
\label{sec:tds_hiopt}
The functions \matlabfun{tds_hiopt_static} and \matlabfun{tds_hiopt_dynamic} tune static and dynamic output feedback controllers for time-delay systems by minimizing the closed-loop \hinfnorm{} with respect to the controller parameters. Its working principles are quite similar to those of the functions \matlabfun{tds_stabopt_static} and \matlabfun{tds_stabopt_dynamic} (see \Cref{sec:tds_stabopt}). More specifically, its core working principles can be summarized as follows:
\begin{enumerate}
	\item Create a delay descriptor representation for the closed-loop system (depicted in \Cref{fig:hinf_control}).
	\item Take the user-provided initial values for the controller parameters or randomly generate them.
	\item Check whether the closed-loop system is strongly exponentially stable for the initial controller values. If not run \matlabfun{tds_stabopt_static} or \matlabfun{tds_stabopt_dynamic} with the entries of this controller as initial values to obtain a stabilizing starting point. If the result is not stabilizing, discard this initial values.
	\item For all stable starting points, minimize the closed-loop \hinfnorm{} with respect to the controller parameters using a BFGS algorithm with weak line search.
	\item Optional: improve the obtained optimum using gradient sampling.
\end{enumerate}
For more details we refer the interested reader to \cite{gumussoy2011}. \\

\noindent \textbf{Note.} The \packageName{} relies on HANSO v3.0 for the last two steps in the description above. 

\subsection{Documentation}
\begin{itemize}
	\item \matlabfun{K = tds_hiopt_static(tds)} designs a strongly stablizing static output feedback gain matrix \texttt{K} for \texttt{tds} by minimizing the closed-loop \hinfnorm{} with respect to the entries of \texttt{K}. 
	\item \matlabfun{K = tds_hiopt_static(tds,'key1',value1,'key2',value2,...)} allows to pass the following optional arguments using key-value pairs:
	\begin{itemize}
		\item \texttt{options} structured array created using \matlabfun{tds_hiopt_options}  containing additional options
		\item \texttt{initial} cell array containing the initial optimization values
		\item \texttt{mask} logical array indicating which entries of the feedback gain matrix are allowed to change (\texttt{true}) and which should remain fixed to their original value (\texttt{false})
		\item \texttt{basis} array specifying the values of the fixed entries of the feedback gain matrix
	\end{itemize}
	\item \matlabfun{[K,CL] = tds_hiopt_static(tds,...)} also returns a \matlabfun{tds_ss_ddae}-object that corresponds to the resulting closed-loop system.  
\end{itemize}
\begin{itemize}
	\item \matlabfun{K = tds_hiopt_dynamic(tds,nc)} designs a strongly stablizing dynamic output feedback controller \texttt{K} of order \texttt{nc} for \texttt{tds} by minimizing the closed-loop \hinfnorm{} with respect to the controller parameters. 
	\item \matlabfun{K = tds_hiopt_dynamic(tds,nc,'key1',value1,'key2',value2,...)} allows to pass the following optional arguments using key-value pairs:
	\begin{itemize}
		\item \texttt{options} structured array created using \matlabfun{tds_hiopt_options} containing additional options
		\item \texttt{initial} cell array containing the initial optimization values 
		\item \texttt{mask} an \texttt{ss}-object indicating which entries of the feedback gain matrix are allowed to change (\texttt{1}) and which should remain fixed to their original value (\texttt{0})
		\item \texttt{basis} an \texttt{ss}-object specifying the fixed entries of the feedback gain matrix
	\end{itemize}
	\item \matlabfun{[K,CL] = tds_hiopt_dynamic(tds,nc,...)} also returns a \matlabfun{tds_ss_ddae}-object that corresponds to the resulting closed-loop system. 
\end{itemize}

\subsection{Options}
The function \matlabfun{tds_hiopt_options} can be used to create a structured array specifying additional options for \matlabfun{tds_hiopt_static} and \matlabfun{tds_hiopt_dynamic}. 
\begin{itemize}
	\item \texttt{nstart}: number of starting points for the optimization procedure\\
	\hspace*{1em} \textit{Default value: } \texttt{5}.
	\item \texttt{cpumax}: quit the BFGS-procedure if the CPU time (in seconds) exceeds the given value\\
	\hspace*{1em} \textit{Default value: } \texttt{inf}.
	\item \texttt{fvalquit}: quit the BFGS-procedure if the object values drops below the given value\\
	\hspace*{1em} \textit{Default value: } \texttt{-inf}.
	\item \matlabfun{gradient_sampling}: apply the gradient sampling algorithm from \cite{HANSO} after the BFGS-procedure. By default, gradient sampling is disabled as it is typically quite slow. \\
	\hspace*{1em} \textit{Default value: } \texttt{false}.
	\item \matlabfun{print_level}: 0 (no printing), 1 (minimal printing), 2 (verbose), 3 (debug)\\
	\hspace*{1em} \textit{Default value: } 2.
	\item \texttt{alpha}: the value of $\alpha$ in \eqref{eq:mixed_performance} \\
	\hspace*{1em} \textit{Default value: } 0.
	\item \texttt{roots}: options for \matlabfun{tds_roots} (see \Cref{sec:tds_roots})
\end{itemize}

\subsection{Warnings and errors}
The functions \matlabfun{tds_hiopt_static} and \matlabfun{tds_hiopt_dynamic} may give the following warnings and errors.
\begin{itemize}
	\item \texttt{Failed to find a strongly stabilizing initial controller.} As the strong \hinfnorm{} of a system that is not strongly exponentially stable is equal to $\infty$, these two functions check for each starting point whether the resulting closed-loop system satisfies this stability requirement. If this is not the case, they will run \matlabfun{tds_stabopt_static} or \matlabfun{tds_stabopt_dynamic} starting the given starting point, to find a stabilizing initial controller. If for none of the starting points, a strongly stabilizing initial controller can be found, this error is thrown.
\end{itemize}
\section{\texorpdfstring{\texttt{tds\_create\_delta}}{tds\_create\_delta}}
The function \matlabfun{tds_create_delta} can be used to create a structured array that stores the relevant information of a set of matrix perturbations $\delta$, as defined in \eqref{eq:matrix_perturbations}, such as the dimensions, whether the perturbations are real or complex-valued and the norm in which the size of the perturbations should be measured.
\subsection{Documentation}
\begin{itemize}
	\item \matlabfun{delta = tds_create_delta(dim1,dim2,type,norm)} with \texttt{dim1}, \texttt{dim2}, \texttt{type} and \texttt{norm} cell arrays of equal length. The elements of \texttt{dim1} and \texttt{dim2} give respectively the first and second dimensions of the different perturbations. The elements of \texttt{type} indicate which the perturbations are real (\verb|'r'|) and complex-valued (\verb|'c'|). The elements of \texttt{norm} denote whether the spectral norm (\verb|'s'|) or the Frobenius norm (\verb|'f'|) should be used to measure the size of the perturbations. The return argument
	\texttt{delta} is a structured array storing the relevant information of the set of matrix perturbations.
\end{itemize}
\section{\texorpdfstring{\texttt{tds\_uncertain\_matrix}}{tds\_uncertain\_matrix} and \texorpdfstring{\texttt{tds\_uncertain\_delay}}{tds\_uncertain\_delay}}
The function \matlabfun{tds_uncertain_matrix} creates a structured array that stores all relevant information related to an uncertain matrix of the form:
\begin{equation}
\label{eq:uncertain_matrix}
	\tilde{A}_{k}(\delta) =  \textstyle\sum\limits_{j=1}^{N_{A_k}} G_{k,j}\  \delta_{\texttt{ind}_{A_k}[j]}\ H_{k,j}.
\end{equation}

\noindent The function \matlabfun{tds_uncertain_delay} creates a structured array that stores all relevant information related to an uncertain delay of the form:
 \begin{equation}
 \label{eq:uncertain_delay}
\tilde{h}_{A,k}(\delta\tau) = \max\left\{-h_{A,k}, \textstyle\sum\limits_{j=1}^{N_{\tau_k}} w_{k,j}\, \delta\tau_{\texttt{ind}_{h_{A,k}}[j]} \right\}.
\end{equation}

\subsection{Documentation}
\begin{itemize}
	\item \matlabfun{umat = tds_uncertain_matrix(ind,G,H)} creates a structured array \texttt{umat} that stores the relevant information of an uncertain matrix of the form \eqref{eq:uncertain_matrix} with \texttt{ind} a cell array storing the positions of the perturbations in the corresponding set of matrix perturbations (created using \matlabfun{tds_create_delta}), and \texttt{G} and \texttt{H} cell arrays containing the shape matrices.
	\item \matlabfun{udelay = tds_uncertain_delay(ind,w)} creates a structured array \texttt{udelay} that stores the relevant information of an uncertain delay of the form \eqref{eq:uncertain_delay} with \texttt{ind} a cell array storing the positions of the perturbations in the corresponding set of delay perturbations and \texttt{w} a cell array containing the weights.  
\end{itemize}
\section{\texorpdfstring{\texttt{tds\_psa}}{tds\_psa}}
The function \matlabfun{tds_psa} computes the $\epsilon$-pseudospectral abscissa of the uncertain retarded delay differential equation:
\begin{equation}
\label{eq:uncertain_sys}
\dot{x}(t) = \big(A_0+ \tilde{A}_{0}(\delta)\big)\,x(t) + \sum_{k=1}^{m} \big(A_k + \tilde{A}_{k}(\delta)\big)\,x\big(t-(h_{A,k}+\tilde{h}_{A,k}(\delta\tau))\big).
\end{equation}
Recall that the $\epsilon$-pseudospectral abscissa was defined as the worst-case value of the spectral abscissa over all allowable uncertainty values, \ie
\begin{equation}
\label{eq:psa}
c^{\mathrm{ps}}(\epsilon) := \max \{c(\delta,\delta\tau): \|\delta_l\| \leq \epsilon \text{ for }   l = 1,\dots,L_{M} \ \& \ |\delta \tau_l|\leq \epsilon  \text{ for } l = 1,\dots,L_{D} \}
\end{equation}
with $c(\delta,\delta\tau)$ the spectral abscissa of \eqref{eq:uncertain_sys} for a given instance of the uncertainties.
\subsection{Documentation}
\begin{itemize}
	\item \matlabfun{psa = tds_psa(utds,epsilon)} returns the \texttt{epsilon}-pseudospectral abscissa, as defined in \eqref{eq:psa}, of the uncertain retarded DDE represented by \texttt{utds}. 
\end{itemize}
\section{\texorpdfstring{\texttt{tds\_dins}}{tds\_dins}}
The function \matlabfun{tds_dins} computes the distance to instability of the uncertain RDDE \eqref{eq:uncertain_sys}. Recall that the distance to instability was defined as the smallest $\epsilon$ for which the $\epsilon$-pseudospectral abscissa is positive, \ie
\begin{equation}
\label{eq:distance_instability}
\dins = \min \{\epsilon\geq 0 : c^{ps}(\epsilon) \geq 0 \} .
\end{equation}
\subsection{Documentation}
\begin{itemize}
	\item \matlabfun{dins = tds_dins(utds,interval)} computes the distance to instability, as
	 defined in \eqref{eq:distance_instability}, of \texttt{utds} by finding the zero-crossing of the function $\epsilon \mapsto c^{\mathrm{ps}}(\epsilon)$ in the search interval specified by \texttt{interval}.
\end{itemize}

\end{document}